\documentclass[11pt]{report}
\usepackage[utf8]{inputenc}
\usepackage{graphicx}
\graphicspath{ {images/} }
\usepackage{caption}
\usepackage{subcaption}
\usepackage{float}
\usepackage{enumitem}
\usepackage{tikz}
\usepackage[width=150mm,top=35mm,bottom=25mm,bindingoffset=6mm]{geometry}
\usepackage{fancyhdr}
\usepackage{scalerel,stmaryrd} 
\renewcommand{\headrulewidth}{0pt}
\RequirePackage{amsthm,amsmath,amsfonts,amssymb}
\usepackage{comment}
\usepackage{mathabx}
\newcommand{\arr}{a}
\newcommand{\arrv}{\pmb{a}}
\newcommand{\srv}{s}
\newcommand{\srvv}{\pmb{s}}
\newcommand{\dep}{d}
\newcommand{\depv}{\pmb{d}}
\newcommand{\Qs}{\mathcal{U}}

\def\bq#1{\pmb{#1}}
\newcommand{\ml}{\mathcal{X}}
\newcommand{\Vmap}{\mathcal{V}}

\newcommand{\fl}[1]{\lfloor{#1}\rfloor} 
\newcommand{\ce}[1]{\lceil{#1}\rceil}

\newcommand{\sqtwopi}{\sqrt{2\pi}}
\newcommand{\fraclambda}{\frac{\dir}{\sqrt 2}}
\newcommand{\zfraclambda}{\frac{z -\dir}{\sqrt 2}}
\newcommand{\negzfraclambda}{-\frac{z + \dir}{\sqrt 2}}
\newcommand{\Var}{\mathbb V ar}
\def\XiSH{\Xi_{G}} 
\def\XiSHsig{\Xi_{G^{\sigma}}} 
\newcommand{\Rf}{\mathfrak s}
\def\SHpp{\Gamma} 
\def\SHHpp{\Lambda} 

\newcommand{\Erl}{\mathcal E^{\rho,\lambda}}

\usepackage[hidelinks]{hyperref}
\usepackage{setspace}

\allowdisplaybreaks

\newcommand{\N}{\mathbb{N}}
\newcommand{\R}{\mathbb{R}}
\newcommand{\Q}{\mathbb{Q}}
\newcommand{\Z}{\mathbb{Z}}
\newcommand{\F}{\mathcal{F}}
\newcommand{\G}{\mathcal{G}}

\newcommand{\X}{\mathcal{X}}
\newcommand{\Y}{\mathcal{Y}}

\newcommand{\C}{\mathcal{C}}
\newcommand{\D}{\mathcal{D}}
\newcommand{\E}{\mathbb{E}}
\newcommand{\Hh}{\mathcal{H}}

\newcommand{\ve}{\varepsilon}
\newcommand{\Ll}{\mathcal{L}}
\newcommand{\Pp}{\mathbb P}
\newcommand{\f}{\frac}
\newcommand{\Ss}{\mathcal S}
\newcommand{\T}{\mathcal T}
\newcommand{\Ff}{\mathcal F}

\newcommand{\deq}{\overset{d}{=}}

\newcommand{\mbf}{\mathbf}

\newcommand{\wt}{\widetilde}

\newcommand{\Rup}{\R_{{\scaleobj{0.7}{\uparrow}}}^4}
\newcommand{\wh}{\widehat}
\newcommand{\li}{\;{\le}_{\rm inc}\;}
\newcommand{\gi}{\;{\ge}_{\rm inc}\;}
\newcommand{\Nor}{\mathcal N}
\newcommand{\Exp}{\operatorname{Exp}}
\newcommand{\W}{W}
\newcommand{\NU}{\operatorname{NU}}
\newcommand{\DLBusedc}{\Xi}
\newcommand{\dir}{\xi}
\newcommand{\Split}{\mathfrak S}
\newcommand{\h}{\mathfrak h}

\newcommand{\graph}{\mathcal G}
\def\shdif{J}  
\def\ind{\mathbf{1}}
\def\ddd{\displaystyle} 
\newcommand{\UC}{\operatorname{UC}}
\newcommand{\CRpin}{C_{\text{pin}}(\R)}
\newcommand{\Ee}{\mathbb E}
\newcommand{\tf}{\tfrac}
\def\ind{\mathbf{1}}
\DeclareMathOperator*{\argmax}{arg\,max} 

\newcommand{\sig}{{\scaleobj{0.8}{\boxempty}}} 
\newcommand{\sigg}{{\scaleobj{0.9}{\boxempty}}} 

\newcommand{\cN}{\mathcal{N}}
\newcommand{\lzb}{\llbracket}   
\newcommand{\rzb}{\rrbracket}   
\DeclareMathOperator\Cls{Cls} 

\newcommand{\Qd}{Q}
\newcommand{\Rd}{R_d}
\newcommand{\Dd}{D_d}
\newcommand{\Ud}{U_d}

\def\bck#1{{\overleftarrow{#1}}}

\newcommand{\be}{\begin{equation}}
\newcommand{\ee}{\end{equation}}

\def\tsp{\hspace{0.5pt}}  
\def\tspa{\hspace{0.7pt}}
\def\tspb{\hspace{0.9pt}}
\def\tspc{\hspace{1.1pt}}

\newcommand\abullet{{\raisebox{1.5pt}{\scaleobj{0.5}{\bullet}}}}  
\newcommand\aabullet{{\tspc\raisebox{1.5pt}{\scaleobj{0.5}{\bullet}}\tspc}}  

\theoremstyle{plain}
\newtheorem{theorem}{Theorem}[section]
\newtheorem{proposition}[theorem]{Proposition}
\newtheorem{corollary}[theorem]{Corollary}
\newtheorem{lemma}[theorem]{Lemma}

\theoremstyle{remark}
\theoremstyle{definition}
\newtheorem{definition}[theorem]{Definition}

\newtheorem{remark}[theorem]{Remark}

\usepackage[style=numeric-comp]{biblatex}
\title{The stationary horizon as the central multi-type invariant measure in the KPZ universality class}
\author{Evan L. Sorensen}
\date{Day Month Year}

\begin{document}
\pagenumbering{roman} 

\begin{titlepage}
    \begin{center}
        \vspace*{1cm}

        \LARGE{\textbf{The stationary horizon as the central multi-type invariant measure in the KPZ universality class}}
        
        \vspace{0.5cm}

        by\\
    
        \textbf{Evan L. Sorensen}
        
        \vfill
        \Large
        A dissertation submitted in partial fulfillment\\
        of the requirements for the degree of\\
        Doctor of Philosophy\\
        (Mathematics)\\

        \vspace{1.8cm}

        \Large
        at the\\University of Wisconsin-Madison\\
        2023\\
        \vspace{1.0cm}
        \begin{flushleft}
        \large
        Date of Final Oral Exam:  06/09/2023\\
        The dissertation is approved by the following members of the Final Oral Committee: \\
        \setlength{\parindent}{10ex}
        Timo Sepp\"al\"ainen, Professor, Mathematics\\
        Benedek Valk\'o,  Professor, Mathematics\\
        Hao Shen, Assistant Professor, Mathematics\\ 
        Hung Vinh Tran, Associate Professor, Mathematics\\
        \end{flushleft}
        
    \end{center}
    
\end{titlepage}

\thispagestyle{empty}
\vspace*{\fill}
Copyright \textcopyright\ 2023 (Evan Sorensen) All rights reserved.

\pagebreak

\doublespacing

\setcounter{page}{1}

\fancyhf{} 
\fancyhead[RO,R]{\thepage} 
\renewcommand{\headrulewidth}{0pt}

\begin{center}
    \Large
    \textbf{The stationary horizon as the central multi-type invariant measure in the KPZ universality class}

    \vspace{0.4cm}
    \textbf{Evan L. Sorensen}
    
    \vspace{0.9cm}
    \textbf{Abstract}
\end{center}
The Kardar-Parisi-Zhang (KPZ) universality class describes a large class of 2-dimensional models of random growth, which exhibit universal scaling exponents and limiting statistics.  The last ten years has seen remarkable progress in this area, with the formal construction of two interrelated limiting objects, now termed the KPZ fixed point and the directed landscape (DL). This dissertation focuses on a third central object, termed the stationary horizon (SH). The SH was first introduced (and named) by Busani as the scaling limit of the Busemann process in exponential last-passage percolation. Shortly after, in the author's joint work with Sepp\"al\"ainen, it was independently constructed in the context of Brownian last-passage percolation. In this dissertation, we give an alternate construction of the SH, directly from the description of its finite-dimensional distributions and without reference to Busemann functions. From this description, we give several exact distributional formulas for the SH. Next, we show the significance of the SH as a key object in the KPZ universality class by showing that the SH is the unique coupled invariant distribution for the DL. A major consequence of this result is that the SH describes the Busemann process for the DL. From this connection, we give a detailed description of the collection of semi-infinite geodesics in the DL, from all initial points and in all directions. As a further evidence of the universality of the SH, we show that it appears as the scaling limit of the multi-species invariant measures for the totally asymmetric simple exclusion process (TASEP). This dissertation is adapted from two joint works with Sepp\"al\"ainen and two joint works with Busani and Sepp\"al\"ainen. 

\chapter*{Dedication}

To my incredible wife Blair and our two daughters, Abigail and Caroline. 


\chapter*{Acknowledgements}

Just over four years ago, as a first-year PhD student, I walked into Professor Timo Sepp\"al\"ainen's office and asked if I could work with him. Working with Timo has been an incredible fit that has made a significant lasting impact on my professional life. He has given me the latitude to work at the pace that was right for me, but has also pushed and challenged me to become a much better mathematician. Throughout my experience in graduate school, he has given priceless professional advice that has helped me to navigate all aspects of the academic world.  I am forever grateful to have him as my PhD advisor. 

There have been so many others who have blessed my life both before and during graduate school. I have been so fortunate to study at UW-Madison. In addition to my interactions with Timo, I have had the great fortune of taking courses from and interacting with Dave Anderson, Sebastien Roch, Hao Shen, Hung Tran, Jean-Luc Thiffeault, and Benedek Valk\'o. For the last three years, I have had the great fortune of interacting and collaborating with Erik Bates during his postdoc at UW-Madison. He has freely given of his time to answer my questions about the postdoc search process and to read my research statements. I am also grateful to Xiao Shen for his collaboration and advice. When I started working with Timo, Xiao kindly welcomed me into the group and willingly  helped me as I was learning about the field. I wish to also mention the kindness and friendship both Erik and Xiao showed me, along with Arjun Krishnan and Riddhi Basu, for sitting with me during my day in the hospital in India when we were there for a workshop. I wish to thank Billy Jackson for his valuable help in mentoring me as a teacher and helping me to prepare for the postdoc job market. I thank Wil Cocke, Jane Davis, Max Hill, Moises Herradon, Aidan Howells, Robert Laudone, Nathan Nicholson, Sun-Woo Park,   and the many other wonderful graduate students I have interacted with during my time at UW-Madison for their mentoring, their collaborations in our classes, and for helping to make my time at UW-Madison a memorable experience. I also give thanks to the outstanding members of the Madison 1st ward of the Church of Jesus Christ of Latter-Day Saints. Throughout my time in graduate school, they have served me and my family, have welcomed us with open arms, and have helped us to grow in faith.

At the end of 2021, Timo and I stared collaborating with Ofer Busani. The things we have accomplished together have made a huge difference in my career trajectory. Throughout our collaboration, Ofer has willingly shared his experiences and advice with me, for which I am incredibly grateful. I am also grateful for my wonderful experiences with my other collaborators, including Ken Alexander, David Clancy, Michael Damron,  Tam\'as Forg\'acs,  Arjun Krishnan, Sean Groathouse, Jack Hanson, David Harper, Andrzej Piotrowski, Firas Rassoul-Agha, Daniel Slonim, Shrivats Sudhir, and Jason White. I wish to additionally thank many other welcoming members of the probability community, with whom I have had the pleasure of sharing enlightening conversations and e-mail exchanges. These include Tom Alberts, M\'arton Bal\'azs Yuri Bakhtin, Ivan Corwin, Sayan Das, Duncan Dauvergne, Douglas Dow, Hindy Drillick, Elnur Emrah, Nicos Georgiou, Adam Jaffe, Chris Janjigian, Kostya Khanin, Alisa Knizel, Yier Lin, Neil O'Connell, Shalin Parekh, Jonathon Peterson, Leandro Pimentel, Ron Peled, Jeremy Quastel, Daniel Remenik, Phil Sosoe,  B\'alint Vir\'ag, and  Xuan Wu. 

Of course, my experience in mathematics did not start in graduate school. Throughout my schooling, I have been blessed to learn from incredible K-12 teachers who loved me and believe in me, including Mark Anderson, Keith Erickson, Wolfgang Griesinger, Windlan Hall, Brenda Johnson, Mark Kingsbury, Michael Maas, Kari McSherry, Lisa Pingrey, and Vince Thomas. During my undergraduate years at Brigham Young University, I cannot overstate the influence that David Cardon has had on me. I was fortunate to take several classes from him and had the great pleasure of working on mentored research with him during my time at BYU. There, I was also blessed to learn from several other outstanding instructors, including Blake Barker, Bradley Barney, Wayne Barrett, Denise Halverson, Paul Jenkins, Ken Kuttler, Chris Grant, Nathan Priddis, and Jared Ward.

 I am grateful for my interactions with Jeff Weintraub and Scott Warrick during two internships at Cirrus Logic, Inc. Jeff has been one of my greatest supporters, even after my internship experience. He has shown me such incredible kindness and gave me a wonderful opportunity to learn mathematics outside the classroom. Much of my early experience with probability came in numerous thought-provoking discussions with Jeff.

My family has been there for me to support me in every state of my education. My parents, Bryant and Marie Sorensen, helped me to develop a love for learning, both spiritually and academically. I am also so grateful for my siblings Kristina, Alex, Isaac, and Grace, as well as their spouses Michael, Jessica, and Rhianna for the love and support they have shown me. I am grateful for my nieces Rosalyn and Hannah and the one that is on the way. I also wish to thank my in-laws, Linda and Matt Harris, for their constant love and support, as well as all of my grandparents, aunts, uncles, and cousins.

My greatest thanks of all goes to my beautiful wife Blair and our two children, Abigail and Caroline. Blair believed in me through the most difficult times of graduate school. She has given everything to raise our two children and to encourage and support me as we navigated these past five years together, including a global pandemic. The work and sacrifice that Blair has given is much more worthy of recognition than anything I have accomplished. Our daughters Abigail and Caroline have showed me endless love and kindness as I have gone through my graduate school experience. I am so proud to be their father and am excited to see the wonderful things they will accomplish in their lives.

 With permission of the other authors, this dissertation contains material taken and adapted from from joint works with Sepp\"al\"ainen~\cite{Seppalainen-Sorensen-21a,Seppalainen-Sorensen-21b} and Busani and Sepp\"al\"ainen~\cite{Busa-Sepp-Sore-22arXiv,Busa-Sepp-Sore-22b}. During my time working on these projects and on this dissertation, I have been partially supported by Timo Sepp\"al\"ainen, through NSF grants DMS-1602846, DMS-1854619, and DMS-2152362.

\tableofcontents

\listoffigures


\chapter{A gentle introduction for a general audience} \label{chap:introgentle}
\pagenumbering{arabic} 
Imagine you are in a class studying probability, and each of the $30$ people in your class roll a fair $6$-sided die. We can't predict which number your die will land on, but we can predict what the aggregate statistics of the class will be. Since there are $30$ students and $6$ equally likely possibilities, we expect that each number would be rolled by $5$ people. There is an underlying theoretical average for this experiment. The possibilities are $1,2,3,4,5,6$, and each has an equal chance of being rolled. The average of these numbers is $3.5$, which we call the \textit{expectation}.  This terminology may seem strange; you can't roll a $3.5$, of course. However, if you take all the rolls from the $30$ students in your class and look at the average of the numbers rolled, that number will be close to $3.5$. You will likely not get exactly $3.5$ (since this is a random experiment), but this number of $3.5$ gives us a good non-random approximation.

Figure \ref{fig:dice1} is an example of what a histogram of the data might look like. In this simulated example, $6$ people rolled a $1$, $8$ people rolled a $2$, $4$ people rolled a $3$, $3$ people rolled a $4$, $5$ people rolled a $5$, and $4$ people rolled a $6$. The average of all the rolls is approximately $3.17$; this is fairly close to the expectation of $3.5$. By adding a large number of students to the class, our sample average is guaranteed to get much closer to $3.5$. This phenomenon is what we call the \textit{Law of Large Numbers}.  
\begin{figure}[t]
\centering
\includegraphics[height = 3in]{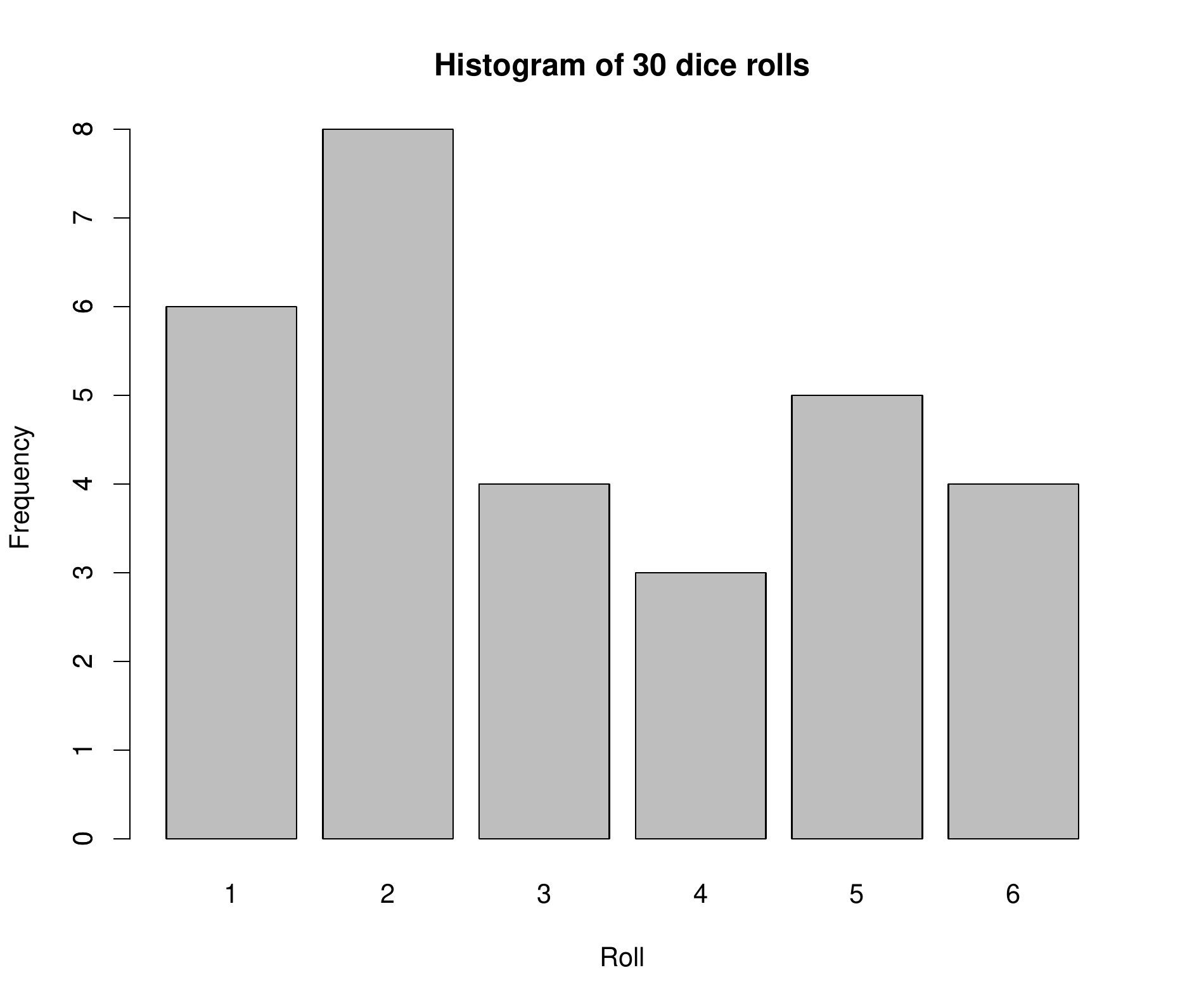}
\caption{Histogram of $30$ dice rolls}
\label{fig:dice1}
\end{figure}

Let's make things a little more interesting. Imagine that you and your classmates have a lot of time on your hands, and you can repeat your dice rolling experiment $1,000$ times. Each time, you record the average of your $30$ rolls. Remember that each of these averages will be close to $3.5$, but now each sample average is an independent random number. Of course, something like this would take too long, but we can simulate this quickly on a computer. Figure \ref{fig:dice} is a histogram of the $1,000$ simulated averages from $30$ dice rolls. We see a distribution that is centered around $3.5$ and forms a bell curve. We call this the normal, or \textit{Gaussian}, distribution. The fact that this bell curve appears is a manifestation of what we call the \textit{Central Limit Theorem}, and this allows us to quantitatively estimate how much the sample average will deviate from the expectation.

\begin{figure}[h]
\centering
\includegraphics[height = 3in]{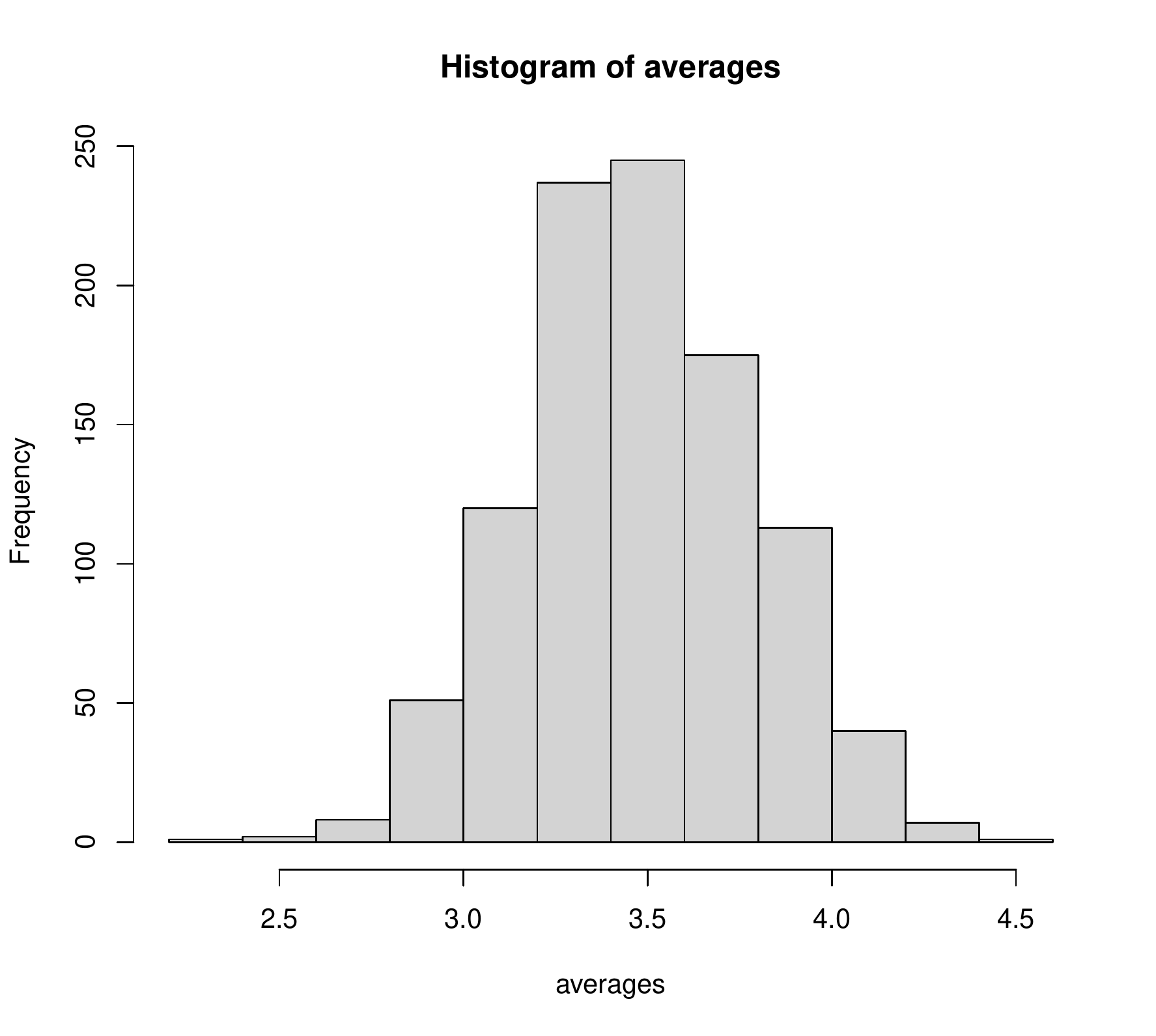}
\caption{Simulation of dice roll averages}
\label{fig:dice}
\end{figure}

Both the law of large numbers and the central limit theorem are examples of universal laws for random variables. While we can't predict the value of a single random quantity, on average, these will distribute themselves in ways that we can predict with high accuracy. There is nothing special about rolling dice in the examples we saw. For \textit{any} experiment that has a finite variance (a technical assumption we won't discuss in detail here), the averages for large samples will distribute themselves as a bell curve. For example, this will also work if you flip coins instead of rolling dice, or if you roll dice with a more than $6$ heads. The law of large numbers and central limit theorem also work if your dice or coins are unfair--that is, not all possibilities are equally likely. And this phenomenon is more than just an interesting fact about dice games. Random quantities from tree heights in a forest or electrical test data on a semiconductor chips are often modelled with the normal distribution. The statistical framework for modern science stands firmly on the mathematical groundwork provided by the central limit theorem.

We can think of the central limit theorem as a one-dimensional law of random behavior. In our dice experiment, we are measuring one number each time. But there are many instances when we are interested in modelling random behavior that captures more than just single quantities that are independent of their surroundings. Consider the following picture of the Mississippi River and its tributaries in Figure \ref{fig:Mississippi}.
\begin{figure}[h]
    \centering
    \includegraphics[height = 3in]{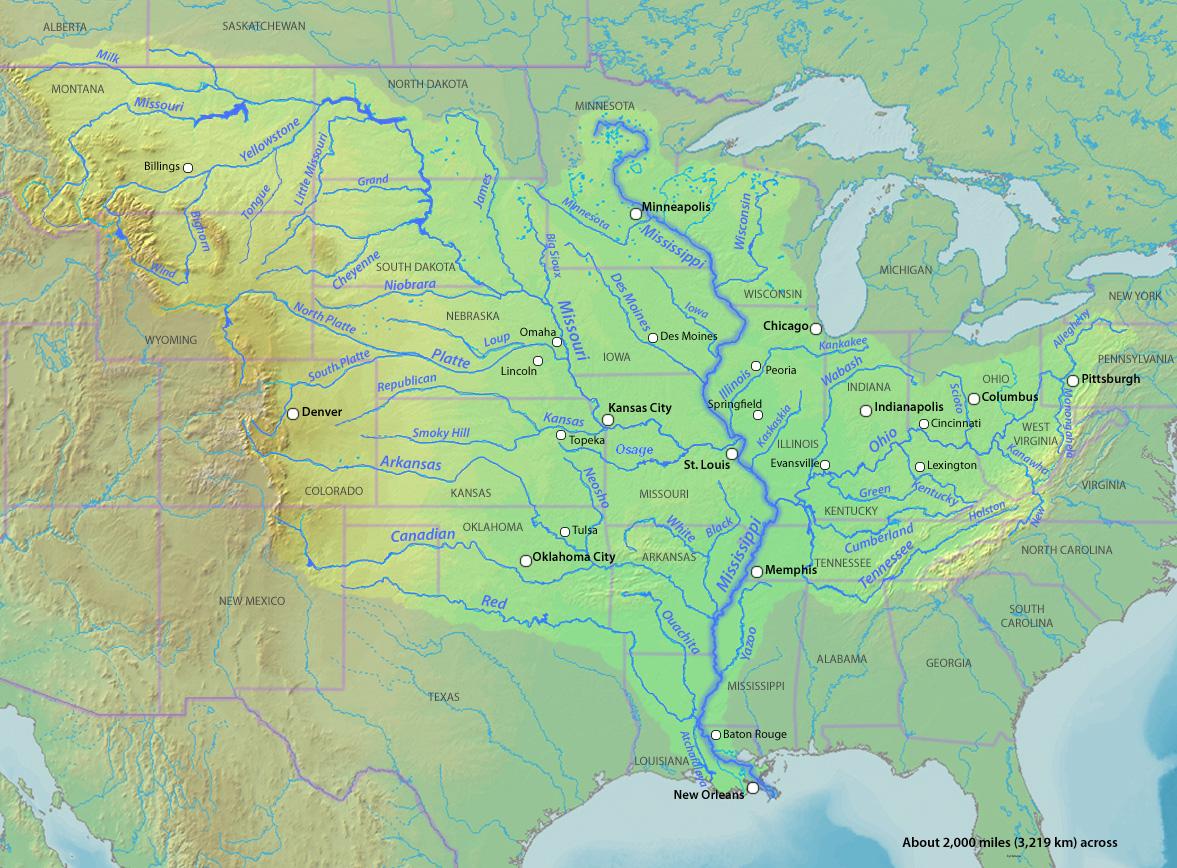}
    \caption{\tiny \copyright Shannon, 2010. Found at Wikimedia Commons \url{https://commons.wikimedia.org/wiki/File:Mississippirivermapnew.jpg.} Licensed under CC BY-SA 4.0 https://creativecommons.org/licenses/by-sa/4.0/legalcode/ Creative Commons Corporation (“Creative Commons”) is not a law firm and does not provide legal services or legal advice. Distribution of Creative Commons public licenses does not create a lawyer-client or other relationship. Creative Commons makes its licenses and related information available on an “as-is” basis. Creative Commons gives no warranties regarding its licenses, any material licensed under their terms and conditions, or any related information. Creative Commons disclaims all liability for damages resulting from their use to the fullest extent possible.}
    \label{fig:Mississippi}
\end{figure}
It may seem strange to say that the path of the river is random; there is in fact only one Mississippi River. However, its formation was the result of the surrounding landscape, with a huge number of small factors contributing to its development. This formation is an example of spatial random growth. A key observation is that the location of the river depends heavily on the surrounding areas. In this dice experiment, the result of your roll has no effect on your neighbor's roll. There is certainly much more correlation in this example. If you were to measure the statistics of spatial locations of the river, you would observe something different from a bell curve. Let's consider another picture I took of a system of cracks in a parking lot (Figure \ref{fig:cracks}).
\begin{figure}[h]
    \centering
    \includegraphics[height = 3in]{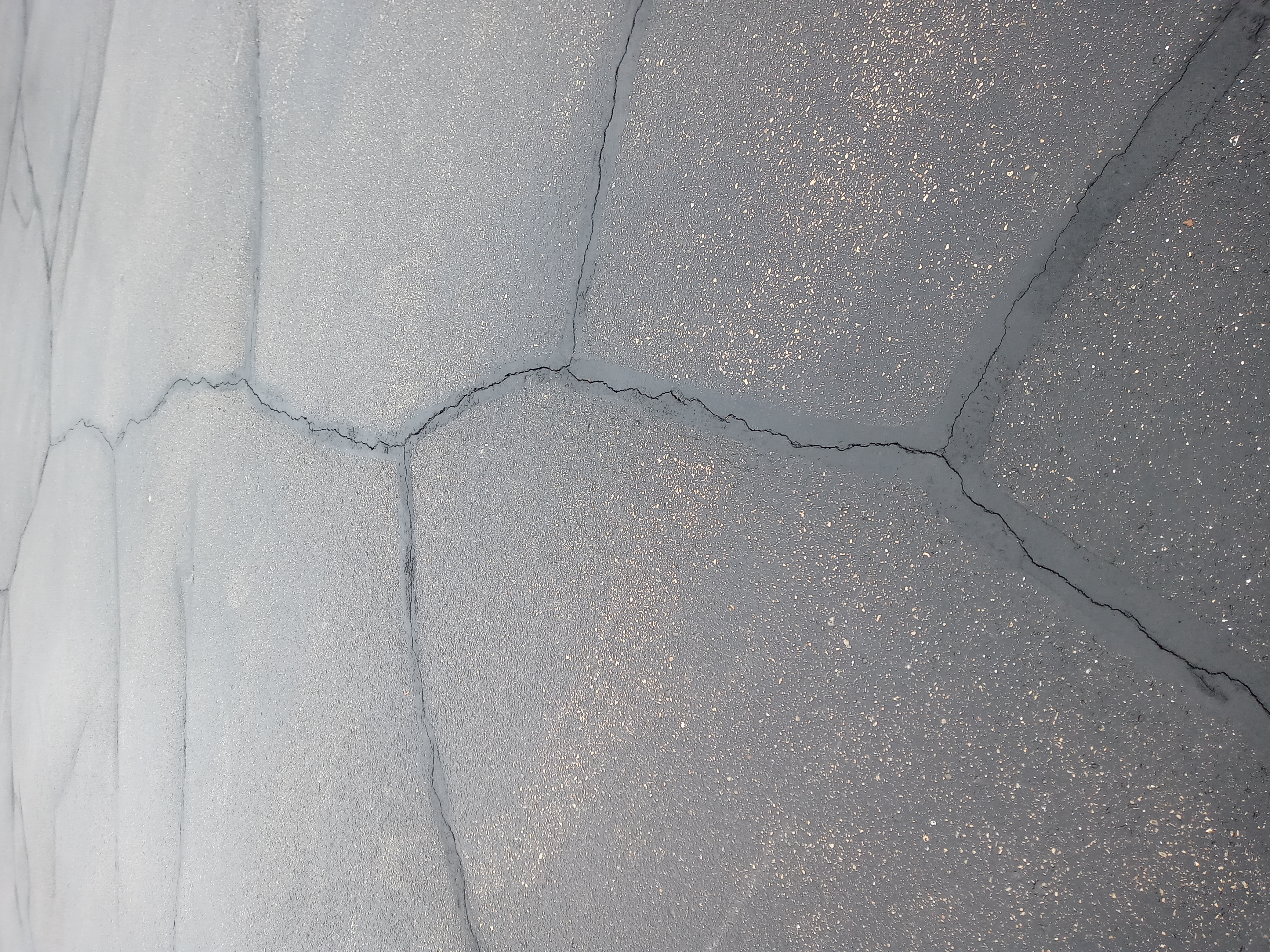}
    \caption{\small A system of cracks in a parking lot}
    \label{fig:cracks}
\end{figure}
In many ways, this picture resembles the picture of the Mississippi River. In particular, there are several paths merging onto a main path, which somewhat resembles the path of a river. While the time taken to form the crack is much different than the time taken to form the Mississippi River we see today, there are similarities in the dynamics of its formation. The random impurities in the asphalt govern where the crack will form, and the location of the crack at a given point depends heavily on the local structure of the pavement.

To motivate this further, imagine that you are travelling from San Diego, California to Boston, Massachusetts. Meanwhile, your friend is travelling from Los Angeles, California to New York, New York. If the paths you each took were straight lines, your paths would intersect at exactly one point. However, this is certainly not the case. Each of you will want to take the route that takes you to your destination in the shortest time possible; and that involves travelling on highways that were built around the random terrain. Figure \ref{fig:US_map} compares the fastest routes along the journey between the two pairs of cities, according to Google maps. If you and your friend both follow the recommended route, both of you will drive to Holbrook, Arizona, then travel the same route to Columbus, Ohio, and lastly, split off to travel to your respective destinations. 
\begin{figure}[h]
    \centering
    \includegraphics[height = 3in]{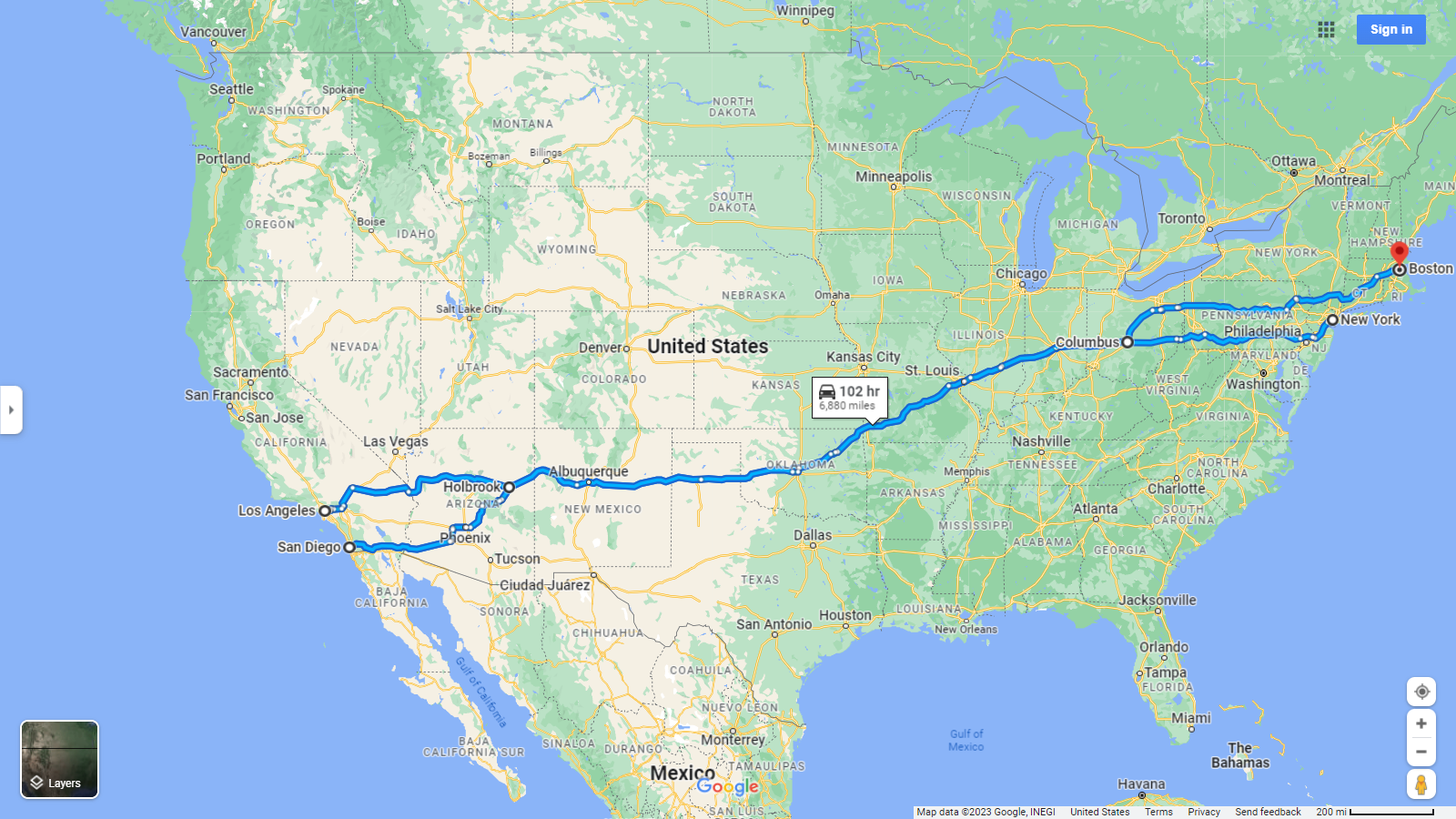}
    \caption{\tiny Picture created via google.com/maps, May 11, 2023. Map of fastest route from New York  to Columbus, Ohio to Los Angeles, California to Holbrook, Arizona to San Diego, California to Boston,  Massachusetts. 
    }
    \label{fig:US_map}
\end{figure}

Just as the Gaussian distribution describes many real-world phenomena via the central limit theorem, it is natural to ask if there is a mathematical object that can describe the similar behavior of these three pictures, as well as many others one might observe scientifically. The overwhelming answer, based on physical and numerical evidence, is, ``Yes!" (See Chapter \ref{chap:intro} for more discussion of the contexts in which this universal behavior has appeared). However, developing a mathematical proof of such universal behavior is currently out of reach, except for a few models that are equipped with formulas that allow for accessible computations.

Mathematically describing these phenomena is not so simple, either. There is an astronomical number of tiny factors that govern the geometry of these paths.  One natural approach is to do a discrete approximation. Let's imagine you come back to your probability class, but this time only $16$ people showed up, and your desks are arranged in a $4 \times 4$ square. Each of you rolls a die and records a number. Let's say that the numbers rolled appear as in Figure \ref{fig:dice_grid}. 

\begin{figure}
    \centering
    \includegraphics[height = 3in]{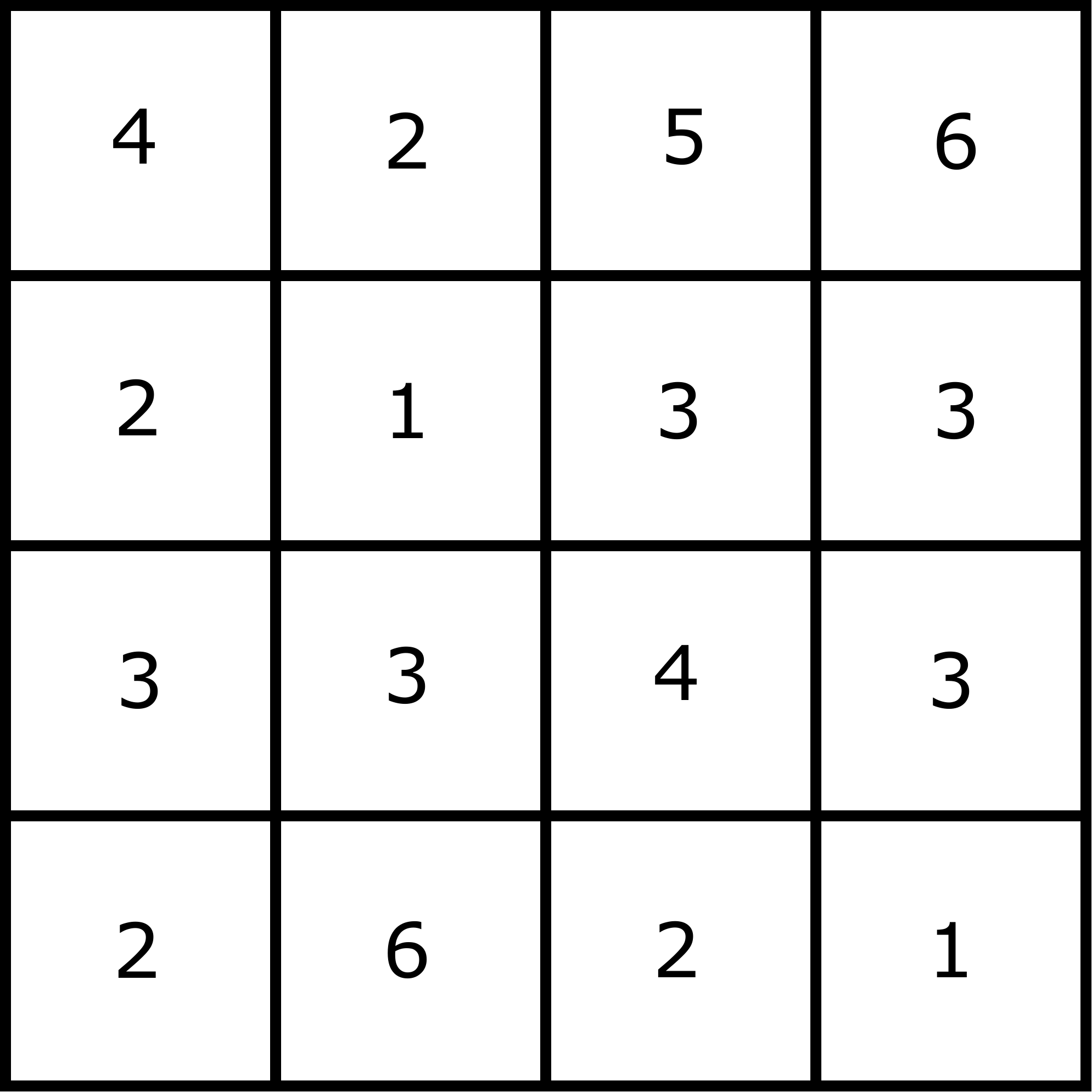}
    \caption{\small Results of dice rolls}
    \label{fig:dice_grid}
\end{figure}

Now, with these results, we will play a game. You want to travel between the bottom left corner to the upper right corner. We impose a rule that you may only move one step at a time, and that step must either be an upward step or a rightward step. Every time you visit a square, you collect the number in the box and add it to your score. You want to find the path that maximizes your score. There are 20 possible paths you could take, and Figure \ref{fig:paths} shows four possibilities, with the optimal one in the lower right corner. Just as in the previous experiment, the result of your dice roll does not affect the result of your neighbor's dice roll. However, unlike the previous experiment, the place where you sit does has an effect on the choice of optimal path. This introduces the spatial correlation similar to what we have observed in the real-world examples discussed.

\begin{figure}
    \centering
    \includegraphics[height= 2in]{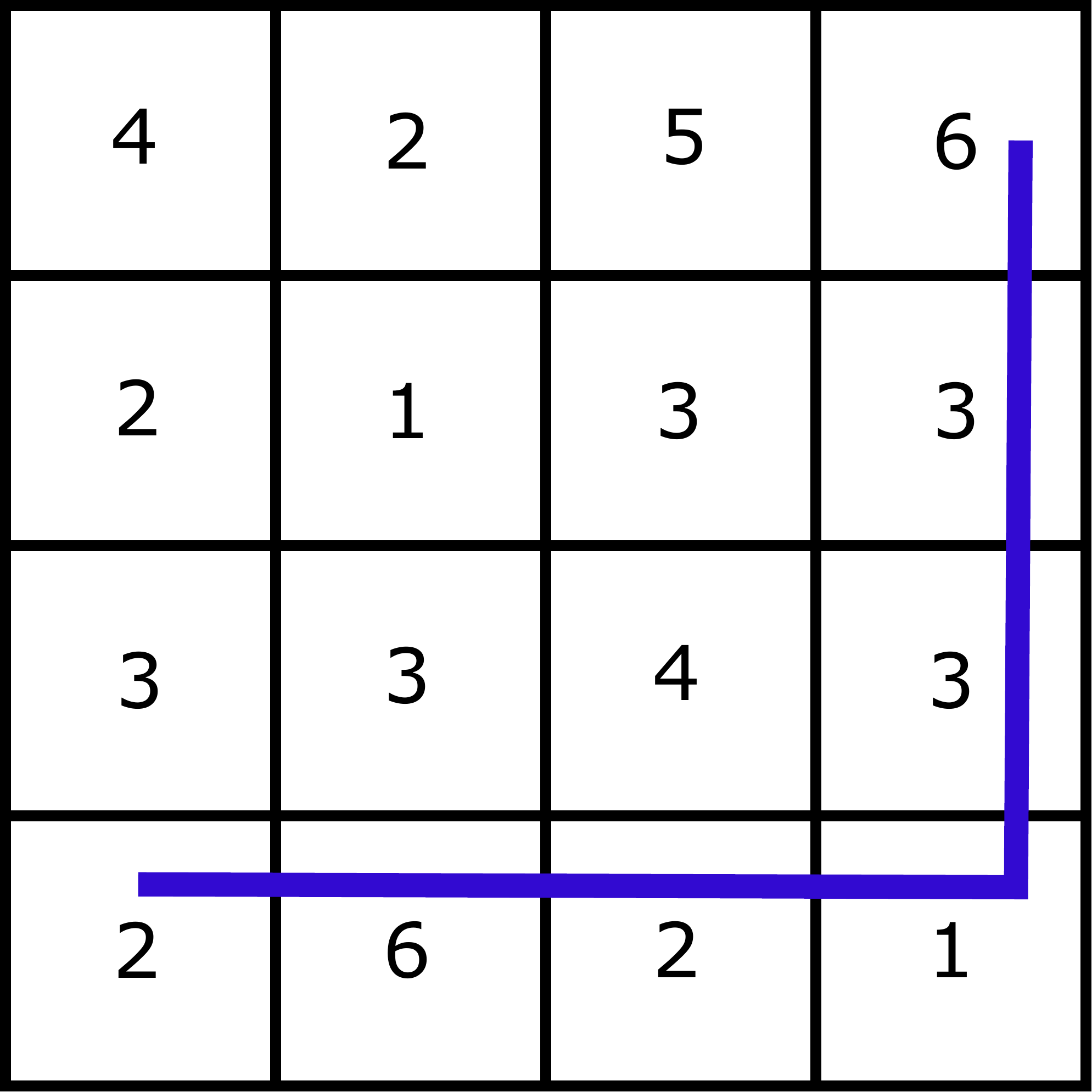}
    \includegraphics[height= 2in]{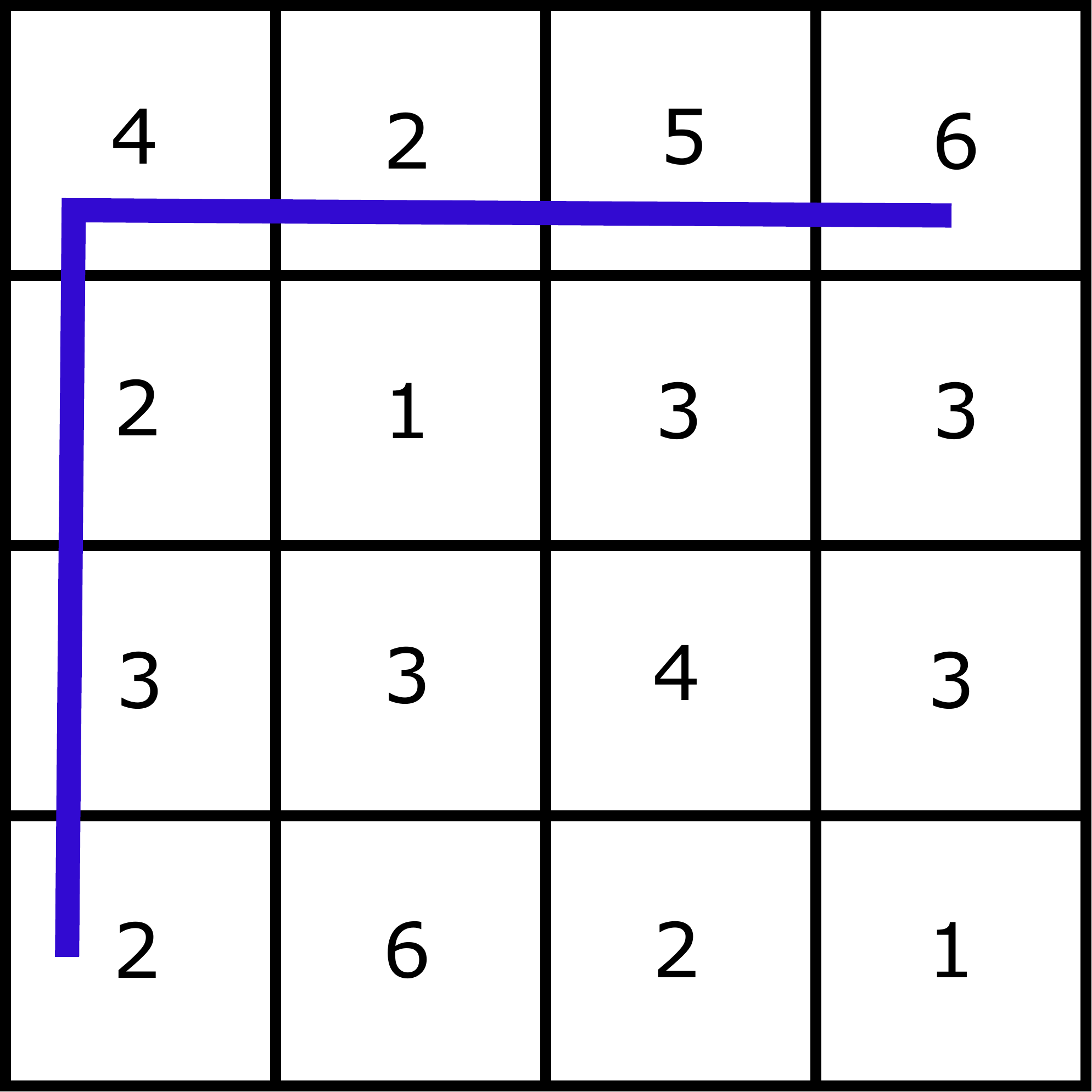}
    \includegraphics[height= 2in]{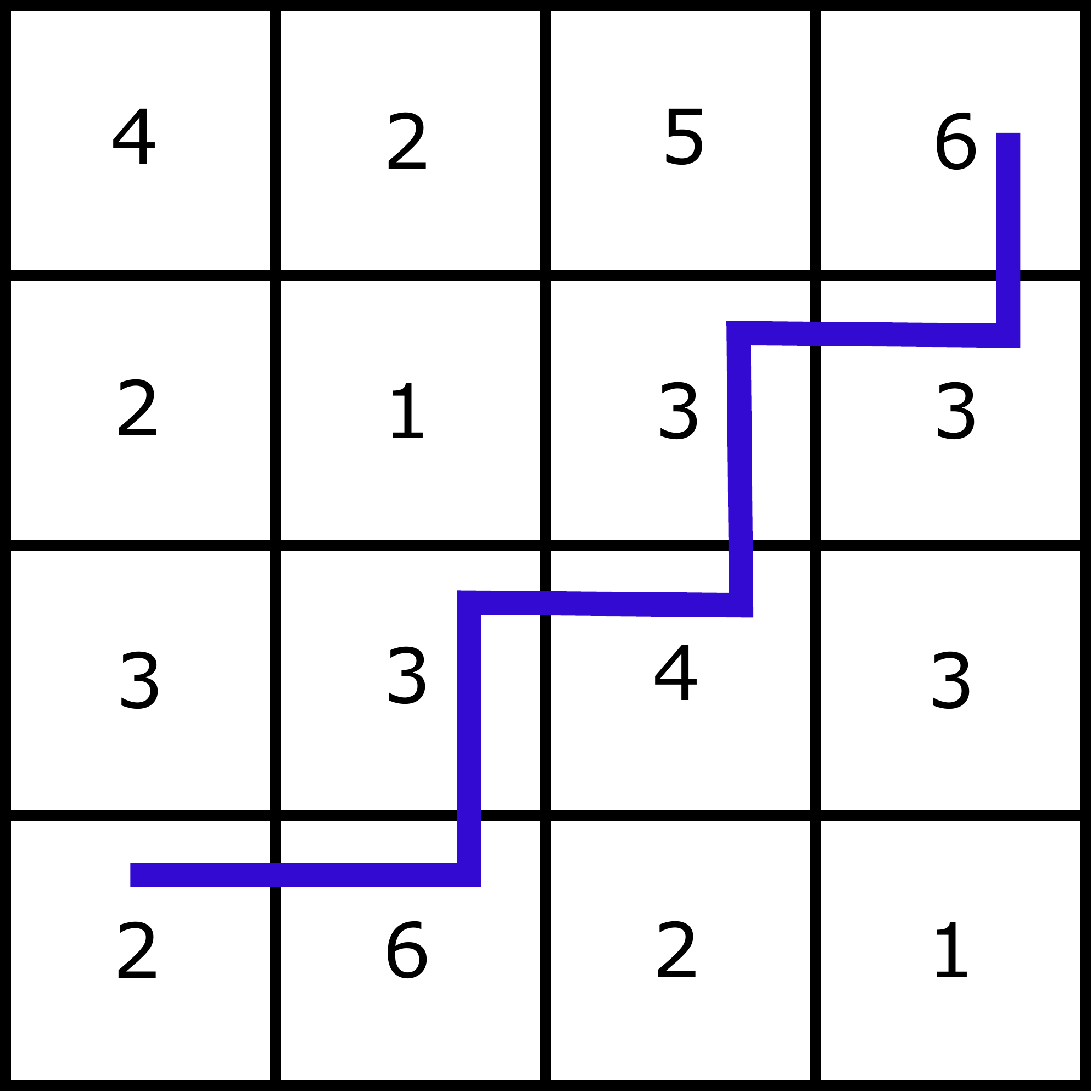}
    \includegraphics[height= 2in]{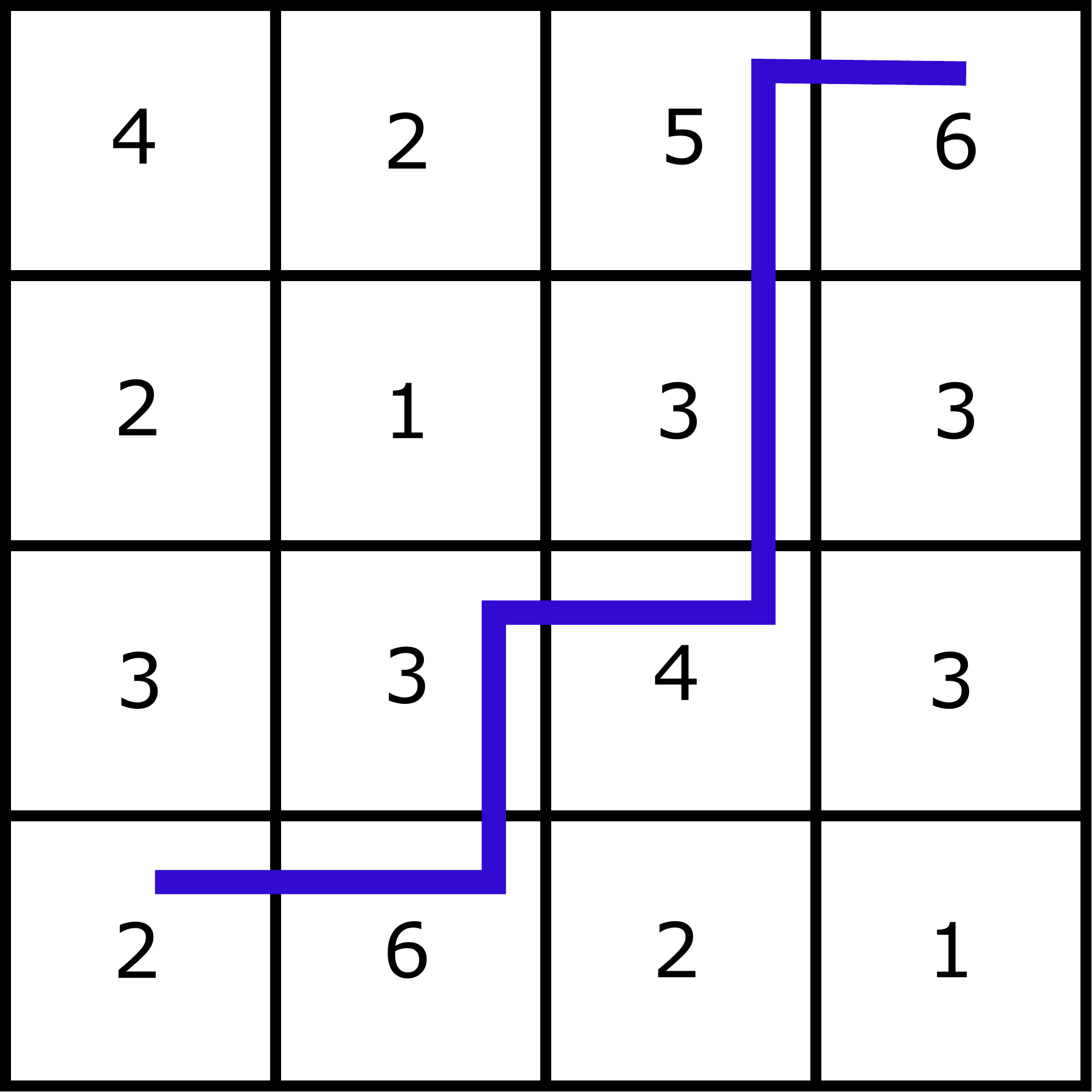}
    \caption{\small Examples of admissible paths. The lower right path has the maximum possible total sum of $29$}
    \label{fig:paths}
\end{figure}

\newpage
\begin{figure}
    \centering
    \includegraphics[height = 3in]{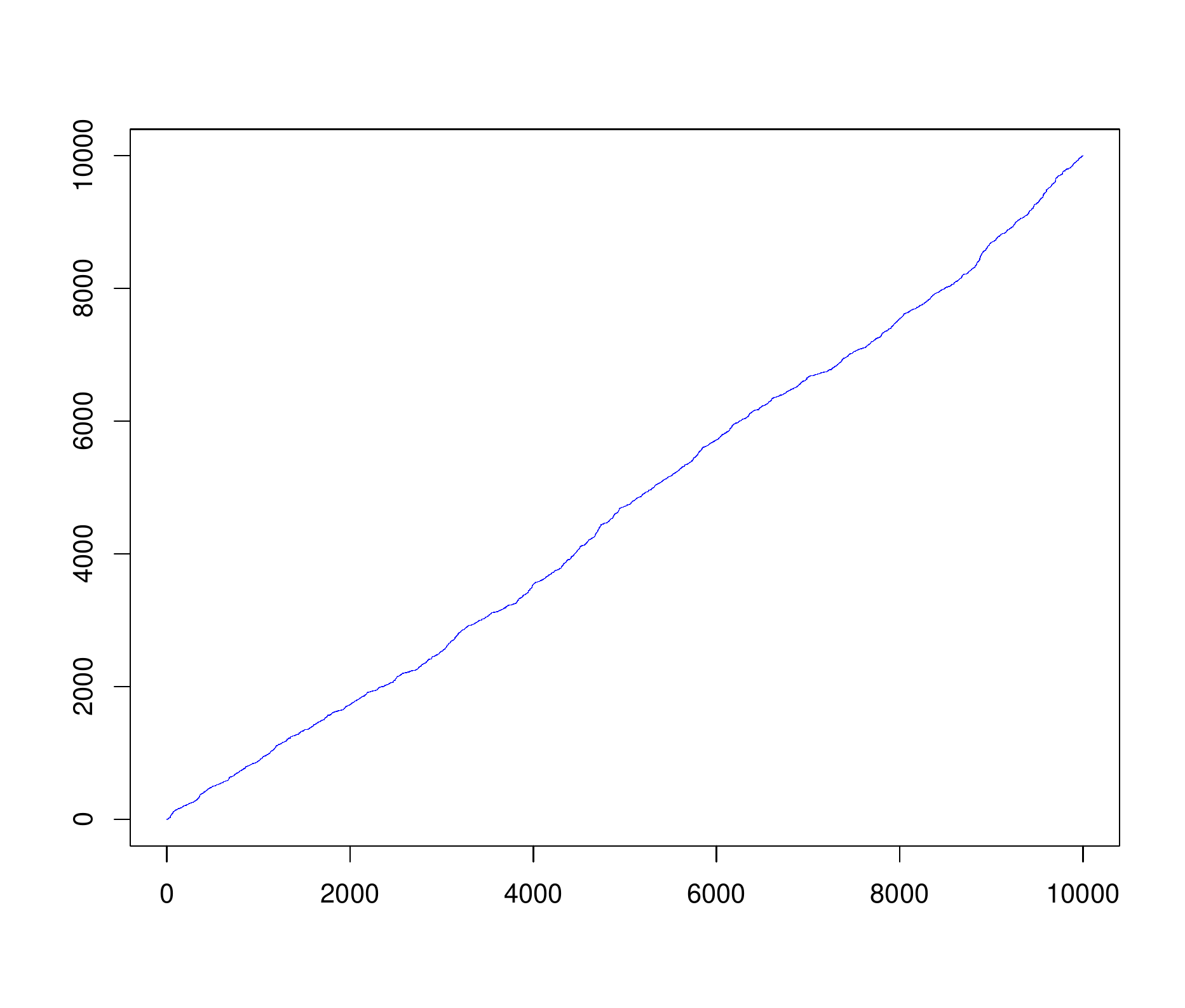}
    \caption{\small Example of a maximal path in a grid of size $1,000 \times 1,000$}
    \label{fig:LPP_path}
\end{figure}

We could look at this picture for a much larger box via a computer simulation. Figure \ref{fig:LPP_path} depicts a simulation of this process for a box of size $1,000 \times 1,000$. We may also be interested in travelling between other pairs of points. In Figure \ref{fig:LPP_two_path}, we overlay the maximal up-right path between a second pair of points close to the corners. What we see is something closely resembling the picture of the road trip we discussed earlier. We call the optimal paths in these pictures \textit{geodesics}. The phenomenon of geodesics merging, is called \textit{coalescence}.  

\begin{figure}
    \centering
    \includegraphics[height = 3in]{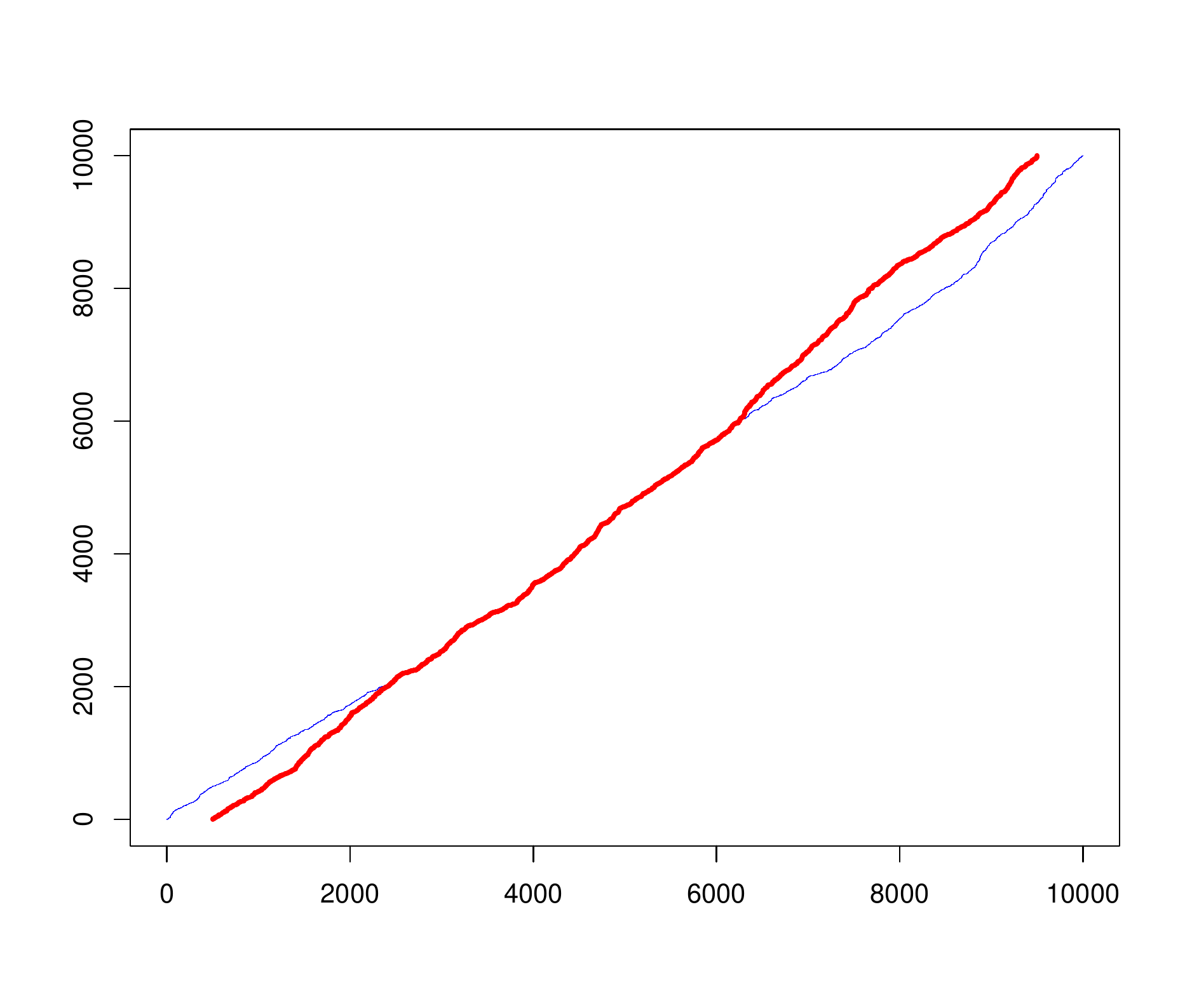}
    \caption{\small Example of maximal paths between two pairs of points. One path is blue/thin, and the other is red/thick }
    \label{fig:LPP_two_path}
\end{figure}

When the picture is zoomed out, the geodesics we see in Figure \ref{fig:LPP_two_path} are not much different than straight lines. If we look at the deviation of these paths from the straight line between points and then rotate and scale the picture, we get Figure \ref{fig:DL}. This object is called the \textit{directed landscape} (DL), and was first studied in \cite{Directed_Landscape}. While the picture seems to indicate that the process is stable under this centering and rescaling, it has not been mathematically proven that this is the case. However, if instead of rolling dice for each square, we sample from a different statistical distribution (either the geometric or exponential distribution) distribution, then it was proven recently by Dauvergne and Vir\'ag \cite{Dauvergne-Virag-21} that the DL does in fact appear after recentering and rescaling. There are a handful of variants of this model for which this has also been proven, but a proof of such universality  outside these special models remains out of reach. There are several other models that are conjectured to have the same behavior. For example, we could remove the rule that the admissible paths must only move upward and rightward. However, without the directional constraint, there are no known models for which we can make exact computations.  From numerical evidence, it seems that the scaling behavior should be the same, but no proof exists in the literature. Just as the Gaussian distribution is proven to be a universal object via the central limit theorem, the DL is believed to be a central object for these spatial growth models. Proving the universality of the DL is one of the major current areas of research in probability. 

There is another perspective for studying this universal behavior through a second object, called the \textit{KPZ fixed point}, which was first introduced in \cite{KPZfixed}. We can construct the KPZ fixed point from the directed landscape, so the two objects are closely interrelated (see Section \ref{sec:KPZ_fixed} for a mathematical introduction to the KPZ fixed point). The initials KPZ stand for Kardar-Parisi-Zhang, who introduced a partial differential equation in \cite{Kardar-Parisi-Zhang-86}, which we now call the KPZ equation. Mathematical models which exhibit the universal limiting statistics of the directed landscape and the KPZ fixed point are said to lie in the \textit{KPZ universality class}.

\begin{figure}
    \centering
    \includegraphics[height = 3in]{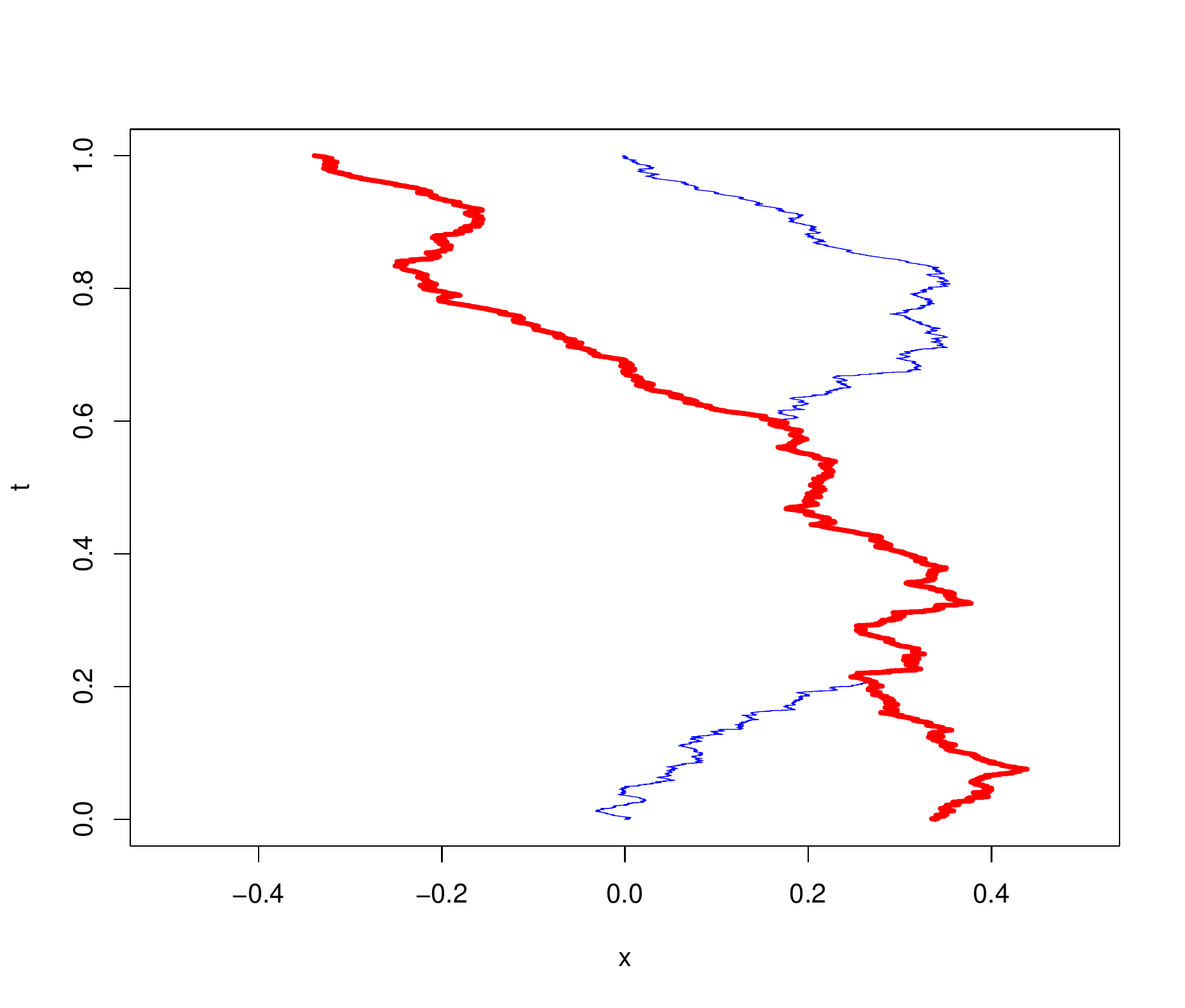}
    \caption{\small The recentered, rotated and scaled version of Figure \ref{fig:LPP_two_path} }
    \label{fig:DL}
\end{figure}

This dissertation is concerned with the long-term behavior of the models in the KPZ class. The main objective is to study a third central object, called the stationary horizon, which provides us with another perspective for understanding KPZ universality. The stationary horizon is mathematically constructed in Chapter \ref{chap:SHcons}. Chapter \ref{chap:invar} shows that the SH describes the long-term evolution of these growth models from a large class of initial conditions. Chapter \ref{chap:Buse} uses the SH to describe detailed geometrical features of long-term growth. One major contribution of this dissertation is a detailed study of the coalescence of geodesics in the DL. Previously, results regarding coalescence were obtained in \cite{Rahman-Virag-21}  for geodesics when restricting attention to a single direction in space. The results of this dissertation give a global picture for coalescence and splitting when observing all directions at once.  Lastly, Chapter \ref{chap:TASEP} shows how the SH appears in the separate context of particle systems, giving further evidence for the centrality of the SH in the KPZ universality class.

\chapter{Introduction} \label{chap:intro}
\section{KPZ universality}
Since the 1986 work of Kardar, Parisi, and Zhang \cite{Kardar-Parisi-Zhang-86}, it has been conjectured that there are universal asymptotic growth rates of the fluctuations and correlations of a large class of spatial stochastic growth models, which also exhibit universal limiting statistics. These models comprise what is now called the Kardar-Parisi-Zhang (KPZ) universality class.  Evidence of the characteristic fluctuations and statistics of these models has been uncovered, both experimentally and numerically, in several contexts, including the growth of liquid crystals \cite{Takeuchi-10,Takeuchi2011GrowingIU,Fukai-Takeuchi-17,Takeuchi-Sano-12,Iwatsuka-Fukai-Takeuchi-20,Fukai-Takeuchi-20}, colonies of living cells \cite{Mitsugu-98,Wakita-Itoh-Matsuyama-97,Huergo-10}, 
the slow burning of paper \cite{Miettinen-05,Myllys2001KineticRI,Maunuksela-97}, geological processes \cite{Halpin-Zhang-95,Yunker-13}, the interaction of molecules on smooth surfaces \cite{Barabasi-96,Degawa-97,WANG-03,WANG-05}, and quantum entanglement \cite{Nahum-17}. 

 The first mathematically rigorous proofs of these asymptotics were presented at the turn of the century by Baik, Deift, and Johansson \cite{Baik-Deift-Johansson-99,Johansson-2000} for some  specially constructed models with accessible formulas originating from representation theory. These special models are said to be \textit{exactly solvable}. The works \cite{Baik-Deift-Johansson-99,Johansson-2000} showed that the one-point distributions of these models converge to the Tracy-Widom distribution \cite{Tracy-Widom-94} from random matrix theory. Corwin, Quastel, and Remenik conjectured in 2011 \cite{Corwin-Quastel-Remenik-15} that models in the KPZ universality class should also converge to richer universal objects whose marginal distributions are described by the Tracy-Widom distribution. Recently, two such interrelated universal objects, known as the \textit{KPZ fixed point} and the \textit{directed landscape (DL)}, were rigorously constructed \cite{Directed_Landscape,KPZfixed} and shown to be the scaling limit of many exactly solvable models \cite{Dauvergne-Virag-21,KPZ_equation_convergence,Wu-23,heat_and_landscape,KPZ_equation_convergence}. The KPZ fixed point and DL are expected to be universal objects for random growth models, analogously to how the Gaussian distribution is a universal object for sums of independent random variables from the central limit theorem (See the survey article \cite{Corwin-survey} for more discussion). Proving convergence to the DL outside the exactly solvable cases remains as a major open problem. 

 This dissertation focuses on a third central object in the KPZ universality class, called the \textit{stationary horizon} (SH). The (SH) 
 is a cadlag process indexed by the real line  whose state space is continuous functions, each a Brownian motion with drift. The SH  was first introduced by Busani \cite{Busani-2021} as the diffusive scaling limit of the Busemann process of exponential last-passage percolation (LPP) from \cite{Fan-Seppalainen-20}, and was conjectured to be the universal scaling limit of the Busemann process of models in the KPZ universality class. Shortly afterward, the paper \cite{Seppalainen-Sorensen-21b} of the author's joint work with Sepp\"al\"ainen was posted. To derive results about semi-infinite geodesics in Brownian last-passage percolation (BLPP), we constructed the Busemann process in BLPP and made several explicit distributional calculations. Remarkably, after discussions with Busani, we discovered that the Busemann process of BLPP has the same distribution as the SH, restricted to nonnegative drifts. Furthermore, due to a rescaling property of the stationary horizon, when the direction is perturbed on order $N^{-1/3}$ from the diagonal, this process also converges to the SH, fully indexed by $\R$. These results were added to the second version of \cite{Seppalainen-Sorensen-21b}. 

 This dissertation details why the SH is a third central object in the KPZ class. We can understand its significance via three main points: (i) The SH is a coupled (or multi-type) invariant measure and an attractor of the KPZ fixed point. It therefore describes the long-term evolution of the KPZ fixed point, started from different initial conditions. (ii) The SH describes the Busemann process of the DL and thus provides a detailed geometric description of the collection of semi-infinite geodesics in the DL. (iii) The SH appears as the scaling limit of multi-type stationary measures for many exactly solvable models known to lie in the KPZ universality class, including exponential LPP, Brownian LPP, and TASEP.

\section{Semi-infinite geodesics and Busemann functions}
The presence of the SH in the KPZ universality class is tied to the study of semi-infinite geodesics and Busemann functions. The study  of semi-infinite geodesics was initiated by Licea and Newman in first-passage percolation in the 1990s \cite{licea1996,Newman} with  the first results on existence, uniqueness and coalescence.   Since the work of Hoffman \cite{Hoffman-2005,hoffman2008}, Busemann functions have been a key tool for studying semi-infinite geodesics (see, for example \cite{Damron_Hanson2012,Hanson-2018,Georgiou-Rassoul-Seppalainen-17a,Timo_Coalescence,Seppalainen-Sorensen-21a,Seppalainen-Sorensen-21b,Rahman-Virag-21,Ganguly-Zhang-2022a}, and Chapter 5 of \cite{50years}).

The study of semi-infinite geodesics began in directed last-passage percolation with the application of the Licea-Newman techniques  to the exactly solvable exponential model  by Ferrari and Pimentel \cite{ferr-pime-05}. 
Georgiou, Rassoul-Agha, and Sepp\"al\"ainen \cite{Georgiou-Rassoul-Seppalainen-17a,Georgiou-Rassoul-Seppalainen-17b} showed the existence of semi-infinite geodesics in directed  last-passage percolation with general weights under mild moment conditions. Using this, Janjigian, Rassoul-Agha, and Sepp\"al\"ainen \cite{Janjigian-Rassoul-Seppalainen-19} showed that geometric properties of the semi-infinite geodesics can be found by studying analytic properties of the Busemann process. In the special case of exponential weights,   the distribution of the Busemann process from \cite{Fan-Seppalainen-20} 
was used to show that all geodesics in a given direction coalesce if and only if that direction is not a discontinuity of the Busemann process.

In the author's collaboration with Sepp\"al\"ainen \cite{Seppalainen-Sorensen-21b}, we extended this work to the semi-discrete setting, by deriving the distribution of the Busemann process and analogous results for semi-infinite geodesics in BLPP. Again, all semi-infinite geodesics in a given direction coalesce if and only if that direction is not a discontinuity of the Busemann process. In each direction of discontinuity there are two coalescing families of semi-infinite geodesics  and from each initial point  \textit{at least} two semi-infinite geodesics. Compared to LPP on the discrete lattice, the semi-discrete setting of BLPP gives rise to additional non-uniqueness. In particular, \cite{Seppalainen-Sorensen-21b} developed a new coalescence proof to handle the non-discrete setting. This new technique is utilized in this dissertation to show coalescence of geodesics in the directed landscape.

While there is much research surrounding semi-infinite geodesics, it is natural to ask about the existence of bi-infinite geodesics. In first-passage percolation, Licea and Newman \cite{licea1996}  showed that there are no bi-infinite geodesics in fixed northeast and southwest directions. Howard and Newman \cite{howard2001} later proved similar results for Euclidean last-passage percolation. Around this time, Wehr and Woo \cite{wehr_woo_1998} proved that, under a first moment assumption on the edge weights, there are no bi-infinite geodesics for first-passage percolation that lie entirely in the upper-half plane. 

In 2016, Damron and Hanson \cite{Damron_Hanson2016} strengthened the result of Licea and Newman by proving that, if the weights have continuous distribution and the boundary of the limit shape is differentiable, for each fixed direction, there are no bi-infinite geodesics with one end having that direction. In 2018, Sepp\"al\"ainen \cite{Timo_Coalescence} showed nonexistence of bi-infinite geodesics in exponential LPP with fixed northeast and southwest directions. This technique was imported to BLPP in the author's joint work with Sepp\"al\"ainen \cite{Seppalainen-Sorensen-21a} for BLPP. Still, this leaves open this possibility of bi-infinite geodesics in exceptional directions. The full complete nonexistence was resolved in exponential LPP with two independent works: first by Basu, Hoffman, and Sly \cite{SlyNonexistenceOB} and shortly thereafter by Bal{\'a}zs, Busani, and Sepp\"al\"ainen \cite{Balzs2019NonexistenceOB}. The latter proof rested on understanding  the joint distribution of the Busemann functions from \cite{Fan-Seppalainen-20}, and this technique has since been used to show the nonexistence of bi-infinite polymer Gibbs measures in the log-gamma polymer \cite{Busani-Seppalainen-2020} and the non-existence of bi-infinite geodesics in geometric last-passage percolation \cite{Groathouse-Janjigian-Rassoul-21}. We discuss relevant literature to the study of infinite geodesics in the DL in the introduction to Chapter \ref{chap:Buse}.

The results for semi-infinite geodesics and Busemann functions have interpretations in the study of viscosity solutions for Hamilton-Jacobi equations. The evolution of Busemann functions can be described via variational formulas (see for example, Theorem \ref{thm:DL_Buse_summ}\ref{itm:Buse_KPZ_description} in Chapter \ref{chap:Buse}), and the maximizers are the locations of semi-infinite geodesics (Theorem \ref{thm:DL_SIG_cons_intro}). In this regard, Busemann functions can be viewed as global solutions to an appropriate Hamilton-Jacobi equation, and the semi-infinite geodesics are the backwards maximizers used to solve the equation. Theorem \ref{thm:DL_Buse_summ}\ref{itm:Buse_KPZ_description} is an example of the dynamic programming principle. The variational formula \eqref{eqn:KPZ_DL_rep} is an example of an optimal control formula, and $\Ll$ plays the role of the fundamental solution. Busemann functions and semi-infinite geodesics have been studied for the Burgers' equation with random forcing, both in compact and noncompact settings \cite{Sinai-1991,Sinai-1993,GIKP-2005,Iturriaga-Khanin-2003,Bakhtin-2007,Kifer-1997,Dirr-Souganidis-2005,Kifer-1997,Bakhtin-2013,Bakhtin-Cator-Konstantin-2014,Bakhtin-16chapter,Bakhtin-16,Bakhtin-Khanin-18,Bakhtin-Li-18,Bakhtin-Li-19,Hoang-Khanin-2003,Drivats-2022}. Specifically in the works of Bakhtin and coauthors \cite{Bakhtin-16,Bakhtin-Cator-Konstantin-2014,Bakhtin-16chapter,Bakhtin-2013,Bakhtin-Li-18,Bakhtin-Li-19}, one sees analogous results for Busemann functions and semi-infinite geodesics, restricted to a fixed direction in space. The uniqueness of the global solution in a fixed direction is called the \textit{one force--one solution} principle. One novelty of the study in the present work (along with the analogous results for exponential LPP in \cite{Janjigian-Rassoul-2020a} and Brownian LPP in \cite{Seppalainen-Sorensen-21b}) is that, for the directed landscape, there exists random exceptional directions, for which the one force--one solution principle fails. Analogous questions were recently explored for the KPZ equation in \cite{Janj-Rass-Sepp-22}. There, it was left open whether such exceptional directions exist, but it was shown that, either these directions exist with probability one and are dense in $\R$, or the set of such directions is almost surely empty. 

\section{Exclusion processes and multi-type stationary measures}
	Since the work of Spitzer in the 1970s \cite{spitzer1991interaction}, exclusion processes have been extensively studied (see, for example, the book  \cite{liggett1985interacting}). These processes can be mapped into growing interfaces, which, under appropriate assumptions, are believed to lie in the KPZ class \cite{Corwin-survey}.  The most famous example is that of the totally asymmetric simple exclusion process (TASEP), consisting of particles and holes on $\Z$, where particles attempt to move to the right when signaled by Poisson clocks. TASEP was the first model for which convergence of the evolving height function to the KPZ fixed point was shown \cite{KPZfixed}.
	Stationary measures of one-dimensional  exclusion processes are well-known \cite[Chapter VIII]{liggett1985interacting} under very  general assumptions on $p$. In this case, the i.i.d.\ Bernoulli product measures  $\nu^\rho$ on $\{0,1\}^\Z$ with particle density  $\rho\in[0,1]$ are the translation-invariant, extremal  stationary measures.

	The family $\{\nu^\rho\}_{\rho\in[0,1]}$ has been  instrumental for example in the study of hydrodynamic limits of exclusion processes \cite{kipnis1998scaling}. In \cite{benassi1987hydrodynamical,andjel1987hydrodynamic},  it was shown that when started from $\nu^{\lambda,\rho}$ (the product measure on $\Z$ with intensity $\lambda$ to the left of the origin and intensity $\rho$ to the right), the TASEP particle profile will converge to either a rarefaction fan or a moving shock depending on the values of $\rho$ and $\lambda$. When $\rho>\lambda$, i.e.\ the shock hydrodynamics,  \cite{ferrari1991microscopic}  showed the existence of a microscopic stationary  profile  as seen from the shock. These studies utilized couplings $\pi^{\lambda,\rho}$ of the measures $\nu^\rho$ and $\nu^\lambda$ that are themselves  stationary under the joint TASEP dynamics of two processes that evolve in   \textit{basic coupling}.  Basic  coupling means that two or more exclusion processes, each from their own initial state,  are run together with common Poisson clocks.

	The stationary measure  $\pi^{\lambda,\rho}$ is sometimes called the two-type stationary measure. This is because one can realize the basic coupling by introducing two types of particles on $\Z$. The location of the first class particles is described by $\nu^\lambda$. Next, when ignoring classes, the distribution of first and second class particles together  is   $\nu^\rho$. In this process, first class particles treat second class particles as holes, so the first class particles take priority. 
The two-type stationary measures
	$\pi^{\lambda,\rho}$ generalize  to 
		multi-type stationary measures $\pi^{\rho_1,\dotsc,\rho_n}$. These measures  and their Ferrari-Martin construction by queuing mappings  \cite{Ferrari-Martin-2007}   are key inputs in Chapter \ref{chap:TASEP}. The study of the joint distribution of Busemann functions for exponential LPP by Fan and Sepp\"al\"ainen \cite{Fan-Seppalainen-20}, can be traced back to the Ferrari-Martin construction. It has been known since the work of Rost \cite{Rost-1981} that exponential LPP can be mapped to TASEP, allowing one to infer limiting statistics about one model via the other. However, this connection does \textit{not} extend to the multi-type exclusion dynamics, and so the scaling limit of the Busemann process to the SH shown by Busani \cite{Busani-2021} does not imply the analogous result for TASEP. However, using the queuing representation of the multi-type stationary distribution, we arrive at the same scaling limit, showing further evidence of the universality of the SH. This was originally done in \cite{Busa-Sepp-Sore-22b} and is recorded in this dissertation in Chapter \ref{chap:TASEP}.

 \section{Summary of results, inputs, and organization of the dissertation}

 The remainder of this dissertation consists of four chapters. This dissertation has been written to be mostly self-contained, except for a few necessary inputs outlined in this section. Knowledge of the fundamentals of measure-theoretic probability (for example, found in Durrett \cite{Durrett}) will be assumed. The results of each chapter depend on those in the previous chapters, except for Chapter \ref{chap:TASEP}, which may be read after Chapter \ref{chap:SHcons}, independently of Chapters \ref{chap:invar} and \ref{chap:Buse}.

 Chapter \ref{chap:SHcons} gives a new construction of the SH that does not rely on existence of Busemann functions. Instead, we define the finite-dimensional distributions and show they are consistent, then use this to construct a process on the Skorokhod space $D(\R,C(\R))$ of functions $\R \to C(\R)$ that are right-continuous with left limits. After the construction, we give detailed probabilistic information about the distribution of this process. Of particular significance is the fact that the stationary horizon, restricted to compact intervals, is a jump process (Theorem \ref{thm:SH_jump_process}). While the particular approach of construction presented here is new, most of the necessary development of this chapter is adapted from the papers \cite{Seppalainen-Sorensen-21a,Seppalainen-Sorensen-21b}, written jointly with Sepp\"al\"ainen. There are two notable exceptions. One is the proof of the reflection invariance of the SH in Theorem \ref{thm:SH_dist_invar}\ref{itm:SH_reflinv}. This dissertation provides a new proof from construction of the SH, a bijection proved in Lemma \ref{DRbij}, and a triangular array construction analogous to its discrete counterpart from \cite{Fan-Seppalainen-20}. Previously, the reflection invariance was proved in \cite{Busa-Sepp-Sore-22arXiv} using reflection invariance of the DL from \cite{Dauvergne-Virag-21} and the fact that the SH is the unique coupled invariant measure for the DL. The other exception is Section \ref{sec:SHPalm}, which has been adapted from the author's joint work with Busani and Sepp\"al\"ainen \cite{Busa-Sepp-Sore-22arXiv}.  Chapter \ref{chap:SHcons} relies on queuing theory from \cite{glynn1991,harrison1990,harrison1992,Harrison1985,brownian_queues}, culminating in a single theorem we use, recorded here as Theorem \ref{thm:OCY_orig}. We also use some distributional calculations from \cite{BM_handbook} and the theory of Palm conditioning from \cite{Kallenberg-book}. 

 Chapter \ref{chap:invar} shows that the SH is a coupled invariant measure and an attractor for the KPZ fixed point. This was previously shown in the author's joint work with Busani and Sepp\"al\"ainen \cite{Busa-Sepp-Sore-22arXiv} using the following inputs: (i)  
The invariance of the Busemann process of  the exponential corner growth model  under the LPP dynamics \cite{Fan-Seppalainen-20}, (ii) Convergence of this Busemann process to SH  \cite{Busani-2021},  (iii) Exit point bounds for stationary exponential LPP  \cite{bala-busa-sepp-20, Balazs-Cator-Seppalainen-2006, Emrah-Janjigian-Seppalainen-20,Seppalainen-Shen-2020, Sepp_lecture_notes}, and (iv)   Convergence of exponential LPP to  DL \cite{Dauvergne-Virag-21}.

To make this dissertation more self-contained, Chapter \ref{chap:invar} gives an alternate proof of the invariance of the SH. We replace (i) with the invariance of the SH under BLPP (proved here as Lemma \ref{lem:BLPP_invar} and adapted from the results in the author's joint work with Sepp\"al\"ainen \cite{Seppalainen-Sorensen-21b}). Item (ii) is replaced by the scaling relations of the SH proved in Theorem \ref{thm:SH_dist_invar}\ref{itm:SHscale}, removing the need for a limit transition to the SH. Item (iii) is replaced with exit point bounds for BLPP, which are proved in this dissertation as Theorem~\ref{thm:BLPP_exit_pts}. The proof uses the technique of \cite{Emrah-Janjigian-Seppalainen-20}, but, to the best of this author's knowledge, this particular result for zero temperature BLPP has not appeared anywhere else in the literature. Item (iv) is replaced with the convergence of BLPP to the DL from \cite{Directed_Landscape}. The attractiveness of the SH is proved the same in this dissertation as it appears in \cite{Busa-Sepp-Sore-22arXiv}.

In Chapter \ref{chap:Buse}, we give a detailed study of the collection of infinite geodesics across all directions and initial points. This chapter is adapted from the author's joint work with Busani and Sepp\"al\"ainen \cite{Busa-Sepp-Sore-22arXiv}. To construct the global Busemann process, we start from the results in \cite{Rahman-Virag-21},   summarized in Section \ref{sec:RV_summ}. After  the first version of our paper \cite{Busa-Sepp-Sore-22arXiv}, Ganguly and Zhang \cite{Ganguly-Zhang-2022a} gave an independent construction of the Busemann function in a fixed direction, and we note that the results of this dissertation do not rely on this newer work. After characterizing the distribution of the Busemann process, we use the regularity of SH from Chapter \ref{chap:SHcons} to prove results about the regularity of the Busemann process and semi-infinite geodesics. A result from \cite{Dauvergne-22} implies  Lemma \ref{lm:horiz_shift_mix} and the mixing in Theorem \ref{thm:Buse_dist_intro}\ref{itm:stationarity}. In this chapter, we also describe the size of the  exceptional sets of points with non-unique geodesics (Theorems \ref{thm:Split_pts} and  \ref{thm:DLNU}\ref{itm:DL_NU_count}). For this, we use results about point-to-point geodesics from \cite{Bates-Ganguly-Hammond-22} and \cite{Dauvergne-Sarkar-Virag-2020}.

The techniques of Chapters \ref{chap:invar} and \ref{chap:Buse} are probabilistic, rather than integrable,  but some results we use come from  integrable inputs. We use results about the directed landscape proved in \cite{Directed_Landscape,Dauvergne-Virag-21,Dauvergne-Nica-Virag-2021,Busani-2021,Dauvergne-Zhang-2021}, which utilize the continuous  RSK correspondence 
\cite{OConnell-2003,rep_non_colliding}. The   results on  point-to-point geodesics in \cite{Bates-Ganguly-Hammond-22,Dauvergne-Sarkar-Virag-2020} rely on \cite{Hammond2}, who studied the   number of disjoint geodesics in BLPP using integrable inputs. The necessary inputs for the DL are recorded in Appendix \ref{appB}. Appendix \ref{appA} contains some miscellaneous facts that we use throughout the dissertation.

Chapter \ref{chap:TASEP} shows that the SH appears as the scaling limit of the TASEP speed process studied in \cite{Amir_Angel_Valko11}. This chapter is taken from the author's joint work with Busani and Sepp\"al\"ainen \cite{Busa-Sepp-Sore-22b}. To show convergence to the SH, there are two key steps: (i) convergence of the finite-dimensional distributions and (ii) tightness on the path space $D(\R,C(\R))$ of the SH. The key input for both of these steps is the Ferrari-Martin queuing representation of the multi class stationary measures for TASEP \cite{Ferrari-Martin-2007}. The technique for the tightness proof was first developed by Busani in the context of exponential last-passage percolation \cite{Busani-2021}. It requires several technical details about the space $D(\R,C(\R))$, which we omit in this dissertation and refer the reader to \cite{Busa-Sepp-Sore-22b}.

\section{Notation}
\label{sec:notat} 
The notation of this dissertation has been made consistent throughout the chapters, occasionally differing slightly from that used in the papers \cite{Seppalainen-Sorensen-21a,Seppalainen-Sorensen-21b,Busa-Sepp-Sore-22arXiv,Busa-Sepp-Sore-22b}. We summarize key notation used here. 

\begin{enumerate}
[label={\rm(\roman*)}, ref={\rm\roman*}]   \itemsep=2pt 
    \item  $\Z$, $\Q$ and $\R$ are restricted by subscripts, as in for example $\Z_{> 0}=\{1,2,3,\dotsc\}$.
    \item Equality in distribution is denoted by $\tspb\deq\tspb$ and convergence in distribution is denoted by $\tspb\Longrightarrow$.
    \item   $X \sim \Exp(\rho)$ means that $X$ is exponentially distributed with rate $\rho > 0$, or equivalently, with mean $\rho^{-1}$.  In other words,
    $\Pp(X>t)=e^{-\rho t}$ for $t>0$.
    \item $X \sim \Nor(\mu,\sigma^2)$ means that $X$ has the Gaussian distribution with mean $\mu$ and variance $\sigma^2$. 
    \item The increments of a function $f:\R \to \R$ are denoted by $f(x,y) = f(y) - f(x)$.
    \item Increment ordering of $f,g:\R \to \R$:   $f \li g$ means that  $f(x,y) \le g(x,y)$ for all $x < y$.
    \item For a function $f:\R \to \R$, define the reflected function $\Rf f:\R \to \R$ as $\Rf f(x) = f(-x)$.
    \item For $s \in \R$,   $\Hh_s=\{(x,s): x \in \R\}$ is  the set of space-time points at time level $s$.
    \item Two-sided standard Brownian motion is a random continuous  process $\{B(x): x \in \R\}$ such that $B(0) = 0$ almost surely and   $\{B(x):x \ge 0\}$ and $\{\Rf B(x):x \ge 0\}$ are two independent standard Brownian motions on $[0,\infty)$.
    \item\label{def:2BMcmu} If $B$ is a two-sided standard Brownian motion, then 
    for $\sigma > 0$ and $\lambda \in \R$, \\ 
    $\{\sigma B(x) + \lambda x: x \in \R\}$ is a two-sided Brownian motion with diffusivity $c>0$ and drift $\mu\in\R$. 
    \item The parameter domain of the directed landscape is  $\Rup = \{(x,s;y,t) \in \R^4: s < t\}$.
    \item Let $\nu_f$ denote the Lebesgue-Stieltjes measure of a non-decreasing function $f$ on $\R$.
    \item The Hausdorff dimension of a set $A$ is denoted by $\dim_H(A)$. 
    \item For $x \in \R$, $\lfloor x \rfloor$ denotes the largest integer less than or equal to $x$, $\lceil x \rceil$ denotes the smallest integer greater than or equal to $x$, $[x]$ denotes the integer closest to $x$ with $|[x]| \le |x|$.
\end{enumerate}

\chapter{Construction and properties of the stationary horizon} \label{chap:SHcons}

\section{Introduction} \label{sec:SH_bigintro}
This chapter is devoted to constructing the stationary horizon (SH) and studying its distributional properties. The SH is a stochastic process $G^\sigma = \{G^\sigma_\dir\}_{\dir \in \R}$, where $\sigma$ is a scaling parameter, and each $G_\dir^\sigma$ is a Brownian motion with drift. $G^\sigma$ lives in the Skorokhod space $D(\R,C(\R))$ of functions $\R \to C(\R)$ that are right-continuous with left limits (see Section \ref{sec:SH_intro}), and we let $G_{\dir -}:\R \to \R$ denote the left limit for this process at $\dir$. To define the SH, we first define queuing transformations in Section \ref{sec:queue_intro}.  In Section \ref{sec:SH_intro}, we state Proposition \ref{prop:SH_cons}, which defines the SH. Shortly after comes Theorem \ref{thm:SH_dist_invar}, which states several distributional invariances of the SH. Section \ref{sec:dist_calc_thm} uses the description of the finite-dimensional distributions of the SH to make several distributional calculations. These calculations are applied in Section \ref{sec:SH_jump} to show that the SH, restricted to a compact set of space, is a jump process (Theorem \ref{thm:SH_jump_process}). This leads to qualitative results on the set of discontinuities of the SH in the parameter $\dir$, as well as several path properties of the process, recorded in Theorem \ref{thm:SH_sticky_thm}. Section \ref{sec:SHPalm} is devoted to studying the process $x \mapsto J_{\dir}(x) := G_{\dir}(x) - G_{\dir -}(x)$ when $\dir$ is a direction of discontinuity. Each direction $\dir$ is a direction of discontinuity with probability $0$, so defining the distribution of this process requires tools from the theory of Palm conditioning \cite{Kallenberg-book}. In the appropriate sense, $J_\dir$ is equal in distribution to a Brownian local time (see Theorems \ref{thm:BusePalm} and \ref{thm:indep_loc} for a more precise statement).   

\section{The Brownian queue} \label{sec:queue_intro} 
In the classical M/M/1 queue, customers arrive at a service station at times determined by a Poisson process. Services are also available at times determined by an independent Poisson process, whose rate is strictly larger than that of the arrivals. The queue follows a first in-first out (FIFO) principle so that customers are served in the order in which they arrived at the queue. The remarkable classical result, now known as Burke's Theorem \cite{Burke1956} is that the times at which customers depart the queue is also a Poisson process with rate equal to that of the arrivals. 

In this setup, the number of customers who arrive at the queue in a finite interval must be an integer, naturally representing a number of people. However, when the rate of service is close to the rate of arrivals, one can obtain a scaling limit of the queue where the number of arrivals and services in an interval are now real-valued \cite{glynn1991,Harrison1985,harrison1990,harrison1992} (See also \cite[Section 6.9]{resnick}). We call this the Brownian queue. Here, we introduce the Brownian queue in the formulation of O'Connell and Yor \cite{brownian_queues}. For our purposes, we consider a queue that has no starting time; it has been running since time $x = -\infty$. Let $A$ and $S$ be two independent, two-sided standard Brownian motions, and let $\lambda > 0$. For $x < y$, $A(x,y)$ represents the arrivals to the queue in the time interval $(x,y]$, and $\lambda(y - x) - S(x,y)$ is the amount of service available   in  $(x,y]$.  For $y \in \R$, set
\be \label{qddef}
\begin{aligned}
q(y) &= \sup_{-\infty < x \le y}\{A(x,y) + S(x,y) - \lambda(y - x)\} \\
d(y) &= A(y) + q(0) - q(y), \\
e(y) &= S(y) + q(0) - q(y).
\end{aligned}
\ee
In queuing terms, $q(y)$ is the length of the queue at time $y$, and for $x < y$, $d(x,y)$ is the number of departures from the queue in the interval $(x,y]$. The following is due to \cite{brownian_queues}. Without the statements for the process $e$, the theorem is a special case of a more general result  previously   shown in \cite{harrison1990}.
\begin{theorem}\cite[Theorem 4]{brownian_queues} \label{thm:OCY_orig}
The processes $d$ and $e$ are independent, two-sided standard Brownian motions. Furthermore, for each $y \in \R$, 
$\{d(x,y),e(x,y): - \infty < x \leq y\}$ is independent of $\{q(x): x \ge y\}$.
\end{theorem}

We view the operations in Theorem \ref{thm:OCY_orig} in terms of mappings $C(\R) \times C(\R) \to C(\R)$, where we subsume the drift term into one of the functions. For continuous functions $Z,B:\R \rightarrow \R$ satisfying $\limsup_{x \rightarrow -\infty} [Z(x) - B(x)] = -\infty$, define  
\begin{align} 
Q(Z,B)(y) &= \sup_{-\infty < x \le y} \{B(x,y)-Z(x,y)\},  \label{Qdef}  \\
D(Z,B)(y) &= Z(y) + Q(Z,B)(y) - Q(Z,B)(0),\label{Ddef} \\
R(Z,B)(y) &= B(y) + Q(Z,B)(0) - Q(Z,B)(y).\label{Rdef}
\end{align}
\begin{theorem} \label{thm:OCY_DR} Let $\sigma 
 > 0$, and let $B$ be a two-sided Brownian motion with diffusion $\sigma$ and drift $\lambda_1$, independent of the two-sided Brownian motion $B$ with diffusion $\sigma$ and drift $\lambda_2 > \lambda_1$. Then, 
 $(B,Z) \deq (R(Z,B),D(Z,B))$, that is, $R(Z,B)$ and $D(Z,B)$ are independent Brownian motions, each with diffusion $\sigma$ and with drifts $\lambda_1$ and $\lambda_2$, respectively.  
 Furthermore, for all $y \in \R$, $\{(D(Z,B)(x,y),R(Y,C)(x,y)): - \infty < x \leq y\}$ is independent of $\{Q(Z,B)(x): x \geq y\}$. 
\end{theorem} 
\begin{proof}
Define $A,S:\R \to \R$ by $A(x) = (\lambda_2 x - Z(x))/\sigma$ and $S(x) = (B(x) - \lambda_1 x)/\sigma$. Then, $S$ and $A$ are i.i.d. two-sided standard Brownian motions. Let $q,d,e$ be defined as in \eqref{qddef} with $\lambda = \lambda_2 - \lambda_1$.  Then, 
\begin{align*}
\sigma q(y) &= \sup_{-\infty < x \le y}\{A(x,y) + S(x,y) - (\lambda_2 - \lambda_1)(y - x)\} \\
&= \sup_{-\infty < x \le y}\{B(x,y) - Z(x,y)\} = Q(Z,B)(y). 
\end{align*}
Next observe that 
\begin{align*}
D(Z,B)(y) &= Z(y) + Q(Z,B)(y) - Q(Z,B)(0)  \\
&\qquad =  \lambda_2 y - \sigma A(y) + \sigma q(y) - \sigma q(0) = \lambda_2 y - \sigma d(y),
  \\[2pt]
\text{and}\quad R(Z,B)(y) &= B(y) + Q(Z,B)(0) - Q(Z,B)(y) \\
&\qquad =\lambda_1 y +\sigma S(y) + \sigma q(0) - \sigma q(t) = \lambda_1 y + \sigma e(y).
\end{align*}
The result now follows from Theorem \ref{thm:OCY_orig}.
\end{proof}
We now iterate the mapping $D$ as follows:
\be \label{Diter}
\begin{aligned} 
    D^{(1)}(Z) &= Z,\quad  D^{(2)}(Z^2,Z^1) = D(Z^2,Z^1), \quad\text{and} \\
    D^{(n)}(Z^n,Z^{n - 1},\ldots,Z^1) &= D(D^{(n - 1)}(Z^n,\ldots,Z^{2}),Z^1)\quad \text{ for } n \geq 2. 
\end{aligned}
\ee
Given a Borel subset $A \subseteq \R$, we define two state spaces of $n$-tuples of functions.  
\be\label{Yndef}\begin{aligned} 
\ \;  &\Y_n^A := \Bigg\{\mathbf Z = (Z^1,\ldots, Z^n) \in \CRpin^n:  \\
&\qquad\qquad\qquad\text{ for } 1 \le i \le n, \lim_{x\rightarrow -\infty} \frac{Z^i(x)}{x} \text{ exists and lies in }A,  \\
&\qquad\qquad\qquad\qquad\text{and for } 2 \leq i \leq n , \lim_{x \rightarrow -\infty} \frac{Z^i(x)}{x} > \lim_{x \rightarrow -\infty} \frac{Z^{i - 1}(x)}{x}    \Bigg\},
\end{aligned}\ee
and
\begin{align} \label{Xndef}
\X_n^A := \Bigg\{\eta = (\eta^1,\ldots,\eta^n) \in \Y_n^A:  \eta^i \gi \eta^{i - 1}  \text{ for } 2 \leq i \leq n 
\Bigg\}.
\end{align}
 These spaces are Borel-measurable subsets (see below) of the space $C(\R,\R^n)$. These spaces are treated as subspaces of $C(\R,\R^n)$ where the  $\sigma$-algebra is the Borel $\sigma$-algebra under the topology of uniform convergence on compact sets, or equivalently, the $\sigma$-algebra generated by the coordinate projections (see, for example, \cite[Section 7]{billing}).

\begin{proof}[Measurability of the state spaces $\X_n^A$ and $\Y_n^A$]
 For $\Y_n^A$, it is sufficient to show that the variables 
\[
\liminf_{x \rightarrow -\infty}\f{Z^i(x)}{x}\qquad\text{and}\qquad\limsup_{x \rightarrow -\infty}\f{Z^i(x)}{x}
\]
are both measurable. We prove the first is measurable, and the second follows analogously. For $a \in \R$, 
\[
\Bigg\{\liminf_{x \rightarrow -\infty} \f{Z^i(x)}{x} > a    \Bigg\} = \bigcup_{k \in \N}\bigcup_{N \in \N} \bigcap_{x \le -N} \Big\{\f{Z^i(x)}{x} \ge a + \f{1}{k}\Big\}.
\]
By continuity of $Z^i$, the intersection over $x \le -N$ can be changed to an intersection over $x \in (-\infty,-N] \cap \Q$. Hence, the set on the left is measurable. By continuity, we may write $\X_n^A$ as 
\[
\X_n^A = \bigcap_{i = 2}^n \bigcap_{x < y, x,y \in \Q}\Big\{\eta \in \Y_N^A: \eta^i(x,y) \ge \eta^{i-1}(x,y)\Big\},
\]
completing the proof.
\end{proof}

Next, define a transformation  $\D^{(n)}$ on $n$-tuples of functions as follows. Let $A\subseteq \R$. For $(Z^1,\ldots,Z^n)\in \Y_n^A$, the image $\eta = (\eta^1,\ldots, \eta^n)=\D^{(n)}(Z^1,\ldots,Z^n)) \in \X_n^A$ is defined by
\begin{equation} \label{scrD}
\eta^i = D^{(i)}(Y^i,\ldots,Y^1) \quad \text{ for } 1 \le i \le n.
\end{equation}
Lemma \ref{image of script D lemma}\ref{itm:image} below proves that $\D^{(n)}: \Y_n^A \to \X_n^A$.

For $\sigma > 0$ and an increasing vector of real drifts $\bar \dir = (\dir_1 < \dir_2 < \cdots < \dir_n)$, define the measure $\nu_\sigma^{\bar \dir}$ on $\Y_n^\R$ as follows: $(Z^1,\ldots,Z^n) \sim \nu_\sigma^{\bar \dir}$ if $Z^1,\ldots,Z^n$ are mutually independent and $Z^i$ is a Brownian motion with diffusivity $\sigma$ and drift $\sigma^2 \dir_i$. Define the measure $\mu_\sigma^{\bar \dir}$ on $\X_n^\R$ as
\be \label{musigma}
\mu_\sigma^{\bar \dir} = \nu_\sigma^{\bar \dir} \circ (\D^{(n)})^{-1}.
\ee

\noindent That is, for $(Z^1,\ldots,Z^n) \sim \nu_\sigma^{\bar \dir}$, we have $\D^{(n)}(Z^1,\ldots,Z^n) \sim \mu_\sigma^{\bar \dir}$. 

\section{Construction and invariances of the stationary horizon} \label{sec:SH_intro}
In this section, we introduce the stationary horizon (SH) and state several of its basic properties. The subsequent sections contain more detailed probabilistic information.  \label{sec:SH_cons_intro}
The Skorokhod space $D(\R,C(\R))$ consists of functions $\R \to C(\R)$ that are right-continuous with left limits. $C(\R)$ is endowed with the topology of uniform convergence on compact sets. We denote a generic element of this space as $G =   \{G_\dir\}_{\dir \in \R}$, where each $G_\dir$ is a continuous function. Membership in $D(\R,C(\R))$ means that for each $\dir \in \R$, $\lim_{\alpha \searrow \dir} G_\alpha = G_\dir$, where convergence holds uniformly on compact sets. Additionally, the limits $\lim_{\alpha \nearrow \dir} G_\alpha$ exist, uniformly on compact sets, and we denote this limiting continuous function as $G_{\dir -}$.  We endow $D(\R,C(\R))$ with the  $\sigma$-algebra $\Ff'$ generated by the projections $\pi^{\dir_1,\ldots,\dir_n}:D(\R,C(\R))\to C(\R,\R^n)$, which are defined by  $\pi^{\dir_1,\ldots,\dir_n}(G) = (G_{\dir_1},\ldots,G_{\dir_n})$. Since the Borel $\sigma$-algebra on $C(\R)$ coincides with the $\sigma$-algebra generated by projections $p^{x_1,\ldots,x_n}: C(\R) \to \R^n$ defined by $p^{x_1,\ldots,x_n} f = (f(x_1),\ldots,f(x_n))$, the $\sigma$-algebra $\F'$ is the smallest $\sigma$-algebra on $D(\R,C(\R))$ so that for each $\dir,x \in \R, G_\dir(x)$ is measurable.

\begin{proposition} \label{prop:SH_cons}
     On the Skorokhod space $(D(\R,C(\R)),\Ff')$, there exists a family of  measures $\{\Pp^\sigma\}_{\sigma > 0}$, satisfying the following properties. If we don't reference the measure $\Pp^\sigma$, we  let $G^\sigma = \{G_\dir^\sigma\}_{\dir \in \R}$ denote the associated random element of $D(\R,C(\R))$. 
     \begin{enumerate} [label={\rm(\roman*)}, ref={\rm(\roman*)}]   \itemsep=3pt
     \item \label{itm:SHBM} For $\dir \in \R$, $G^\sigma_\dir$ is a two-sided Brownian motion with variance $\sigma$ and drift $\sigma^2\dir$. In particular, $\Pp^\sigma$- almost surely, for every $\dir \in \R$, $G_\dir(0) = 0$. 
     \item \label{itm:SH_dist} For $\bar \dir = (\dir_1 < \cdots < \dir_n)$, $\Pp^\sigma \circ (\pi^{\bar \dir })^{-1} = \mu_\sigma^{\bar \dir}$. That is,
     \[
(G^\sigma_{\dir_1},\ldots,G^\sigma_{\dir_n}) \sim \mu_\sigma^{\bar \dir}.
     \]
     More concretely, this random vector in $C(\R,\R^n)$ has the distribution of
     \[
     (Y^1,D^{(2)}(Y^2,Y^1),\ldots, D^{(n)}(Y^n,\ldots,Y^1)),
     \]
     where $Y^1,\ldots,Y^n$ are independent Brownian motions, each with diffusivity $\sigma$ and with drifts $\sigma^2 \dir_1,\ldots,\sigma^2 \dir_n$. The measure $\Pp^\sigma$ is the unique distribution on $D(\R,C(\R))$ with these finite-dimensional distributions. 
     \item \label{itm:SH_mont} $\Pp^\sigma$-almost surely, for all $\dir_1 < \dir_2$, $G_{\dir_1-} \li G_{\dir_1} \li G_{\dir_2 -} \li G_{\dir_2}$.
     \item \label{itm:SH_cont} For each $\dir \in \R$, $\Pp^\sigma(G_{\dir}(x) = G_{\dir -}(x)\;\; \forall x \in \R) = 1$. 
     \item \label{itm:SH_drifts} $\Pp^\sigma$-almost surely, simultaneously for every $\dir \in \R$, 
     \[
     \lim_{|x| \to \infty} \f{G_\dir(x)}{x} = \lim_{|x| \to \infty} \f{G_{\dir-}(x)}{x} = \sigma^2 \dir.
     \]
    \end{enumerate}
\end{proposition}

We call the process $G^\sigma$ the \textit{stationary horizon} with diffusivity $\sigma$ (or SH for short). The most common choices for $\sigma$ are $1$ (where the SH is the unique coupled invariant measure for Brownian last-passage percolation) and $\sqrt 2$ (where the SH is the unique coupled invariant measure for the KPZ fixed point). These invariances are both discussed in further detail in the following chapter. The parameterization of the SH first constructed in \cite{Busani-2021} is $\{G^{2}_{\dir/4}\}_{\dir \in \R}$. The choice of parameterization bears no real significance, as Item \ref{itm:SHscale} of the following theorem shows we can obtain one process as a simple scaling of another. 

\begin{theorem} \label{thm:SH_dist_invar}
The SH $G^\sigma$ satisfies the following. 
\begin{enumerate} [label={\rm(\roman*)}, ref={\rm(\roman*)}]   \itemsep=3pt 
\item \label{itm:shinv} Translation invariance: for each $x\in\R$, $\{G^\sigma_\dir(x,x + \aabullet)\}_{\dir \in \R} \deq G^\sigma$.

\item\label{itm:SHscale} Scaling invariance: let $b,c,\sigma > 0$ and $\nu \in \R$. Then,
\[
\{b G^\sigma_\dir(c^2 \aabullet) - (bc\sigma)^2 \nu \aabullet \}_{\dir \in \R} \deq \{G^{bc\sigma}_{\dir/b - \nu} 
 \}_{\dir \in \R}.
\]
In particular, $\{G_\dir^\sigma(c^2 \abullet)\}_{\dir \in \R} \deq G^{c\sigma}$, and $\{c^{-1}G_\dir^\sigma(c^2 \abullet)\}_{\dir \in \R} \deq \{G^\sigma_{c\dir}\}_{\dir \in \R}$
\item \label{itm:SH_inc_stat} Stationarity of increments: For $\dir_1 < \dir_2 < \cdots < \dir_n$ and $\dir^\star\in \R$,
\[
(G^\sigma_{\dir_2} - G^\sigma_{\dir_1},\ldots,G^\sigma_{\dir_n} - G^\sigma_{\dir_{n - 1}}) \deq (G^\sigma_{\dir_2 + \dir^\star} - G^\sigma_{\dir_1 + \dir^\star},\ldots,G^\sigma_{\dir_n + \dir^\star} - G^\sigma_{\dir_{n - 1} + \dir^\star}).
\]
\item \label{itm:SH_reflinv} Reflection invariance: $\{G^\sigma_{(-\dir)-}(-\aabullet)\}_{\dir \in \R} \deq G^\sigma$.
\end{enumerate}
\end{theorem}

The rest of this section is devoted to proving Proposition \ref{prop:SH_cons} and Theorem \ref{thm:SH_dist_invar}. The following subsections develop the theory of the queuing mappings and the measures $\mu_\sigma^{\bar \dir}$, which describe the finite-dimensional distributions of the process. We show these measures are consistent (Theorem \ref{weak continuity and consistency}\ref{consistency}), then build the measures on $D(\R,C(\R))$ through limiting procedures. The proofs of Proposition \ref{prop:SH_cons} and Theorem \ref{thm:SH_dist_invar} come in Section \ref{sec:SH_cons_proof}. 
\subsection{Deterministic properties of the queuing mappings} \label{sec:queue_proofs} 

We now prove some accessible properties of the queuing mappings. We note here that all lemmas in this subsection are deterministic statements about continuous functions.
\begin{lemma}\label{DRcont}
    Assume $Z,B \in C(\R)$ satisfy $\lim_{x \to -\infty}[Z(x) - B(x)] = -\infty$. Then, for $x < y$,
    \begin{align*}
    D(Z,B)(x,y) &= B(x,y) +\Bigl(\sup_{x \le u \le y}\{Z(x,u) - B(x,u)\} - \sup_{-\infty < u \le x}\{Z(x,u) - B(x,u)\}\Bigr)^+,
    \end{align*}
    and 
    \begin{align*}
    R(Z,B)(x,y) &= Z(x,y) -\Bigl(\sup_{x \le u \le y}\{Z(x,u) - B(x,u)\} - \sup_{-\infty < u \le x}\{Z^1(x,u) - B^1(x,u)\}\Bigr)^+,
    \end{align*}
    Furthermore, $D(Z,B)$ and $R(Z,B)$ are continuous.  
\end{lemma}
\begin{proof}
    From the definition of $D$ \eqref{Ddef}, for $x < y$,
    \begin{align}
    D(Z,B)(x,y) &= Z(x,y) + \sup_{-\infty < u \le y}\{B(u,y) - Z(u,y)\} - \sup_{-\infty < u \le x}\{B(u,x) - Z(u,x)\} \nonumber\\
    &= B(x,y) + \sup_{-\infty < u \le y}\{Z(u) - B(u)\} - \sup_{-\infty < u \le x}\{Z(u) - B(u)\} \label{DZB2} \\
    &= B(x,y) + \sup_{-\infty < u \le y}\{Z(x,u) - B(x,u)\} - \sup_{-\infty < u \le x}\{Z(x,u) - B(x,u)\} \nonumber \\
    &= B(x,y) + \Bigl(\sup_{x \le u \le y}\{Z(x,u) - B(x,u)\} - \sup_{-\infty < u \le x}\{Z(x,u) - B(x,u)\} \Bigr)^+. \nonumber
    \end{align}
    A similar calculation shows the identity for $R(Z,B)$. To show continuity, we observe that 
    \begin{align*}
     0 &\le \Bigl(\sup_{x \le u \le y}\{Z(x,u) - B(x,u)\} - \sup_{-\infty < u \le x}\{Z(x,u) - B(x,u)\} \Bigr)^+ \\
     &\le  \Bigl(\sup_{x \le u \le y}\{Z(x,u) - B(x,u)\}  \Bigr)^+ = \sup_{x \le u \le y}\{Z(x,u) - B(x,u)\}. 
    \end{align*} 
    where the second inequality follows because 
    \[
    \sup_{-\infty < u \le x}\{Z(x,u) - B(x,u)\} \ge Z(x,x) - B(x,x) =  0,
    \]
    and the last equality follows by similar reasoning. The continuity of $D(Z,B)$ and $R(Z,B)$ now follows from the continuity of $Z$ and $B$.
\end{proof}

\begin{lemma} \label{identity for multiple queueing mappings}
For $(Z^1,\ldots,Z^n) \in C(\R,\R^n)$ with $Z^i(0) = 0$ for $1 \le i \le n$, define
\[
A^{\mathbf Z}_{n}(y) = \sup_{-\infty < x_{n - 1} \leq x_{n - 2} \le \cdots \leq x_{1} \leq y} \Big\{\sum_{i = 1}^{n - 1} Z^{i + 1}(x_{i}) - Z^{i}(x_{i})   \Big\}.
\]
Then, for $n \geq 2$, if $A_n^{\mathbf Z}(0)$ is finite, 
\begin{align*}
D^{(n)}(Z^n,Z^{n - 1},\ldots, Z^1)(y) 
= &Z^1(y) + A^{\mathbf Z}_{n}(y) - A^{\mathbf Z}_{n}(0).
\end{align*}
Furthermore, $R(Z^2,Z^1)(y) = Z^2(y) + A_2^{Z^1,Z^2}(0) - A_n^{Z^1,Z^2}(y)$.
\end{lemma}
\begin{proof}
The $n = 2$ case follows from \eqref{DZB2} (observe that the assumption $Z^i(0) = 0$ is crucial).
The statement for the map $R$ follows analogously. 
Now, assume that the statement is true for some $n \ge 2$.  By definition of $D^{(n +1)}$ \eqref{Diter} and the $n = 2$ case,
\begin{align*}
 &\quad \;D^{(n + 1)}(Z^{n + 1},\ldots,Z^1)(y) = D(D^{(n)}(Z^{n + 1},\ldots,Z^2),Z^1)(y) \\
&= Z^1(y) + \sup_{-\infty < x  \le y}\{D^{(n)}(Z^{n + 1},\ldots,Z^2)(x) - Z^1(x)\} \\
&\qquad\qquad\qquad\qquad\qquad - \sup_{-\infty < x  \le 0}\{D^{(n)}(Z^{n + 1},\ldots,Z^2)(x) - Z^1(x)\} \\
&= Z^1(y) + \sup_{-\infty < x \le y}\Biggl\{Z^2(x) + \sup_{-\infty < x_n \le \cdots \le x_{2} \le  x}\Bigl\{\sum_{i = 2}^{n} Z^{i + 1}(x_{i}) - Z^{i}(x_{i})   \Bigr\} - Z^1(x)  \Biggr\} \\
&\qquad\qquad - \sup_{-\infty < x \le 0}\Biggl\{Z^2(x) + \sup_{-\infty < x_n \le \cdots \le x_{2} \le  x}\Bigl\{ \sum_{i = 2}^{n} Z^{i + 1}(x_{i}) - Z^{i}(x_{i})   \Bigr\} - Z^1(x)  \Biggr\}  \\
&= Z^1(y) + A_n^{\mbf Z}(y) - A_n^{\mbf Z}(0),
\end{align*}
where in the third line, the term $\sup_{-\infty < x_n \le \cdots \le x_{2} \le  0}\Bigl\{\sum_{i = 2}^{n} Z^{i + 1}(x_{i}) - Z^{i}(x_{i})   \Bigr\}$ is canceled out from both terms by subtraction. 
\end{proof}

\begin{lemma} \label{uniform convergence of queueing mapping}
Let $a < b$, and let $Z_k,B_k \in \CRpin$ be sequences such that $Z_k \rightarrow Z$ and $B_k \rightarrow B$ uniformly on compact sets. Assume further that 
\begin{equation} \label{eqn:lim_cond}
\limsup_{\substack{x \rightarrow -\infty \\ k \rightarrow \infty} } \biggl|\f{1}{x} Z_k(x) - b \,\biggr| = 0 = \limsup_{\substack{x \rightarrow -\infty \\ k \rightarrow \infty} } \biggl|\f{1}{x} B_k(x) - a \,\biggr|.
\end{equation}
Then, $Z' := D(Z,B),B' := R(Z,B), Z_k' := D(Z_k,B_k)$, and $ B_k' := R(Z_k,B_k)$ are well-defined for sufficiently large $k$. Furthermore,  
\[
\lim_{k \rightarrow \infty} Z_k' = Z' \qquad\text{and}\qquad \lim_{k \rightarrow \infty} B_k' = B',
\]
in the sense of uniform convergence on compact subsets of $\R$, and
\[
\limsup_{\substack{x \rightarrow -\infty \\ k \rightarrow \infty} } \biggl|\f{1}{x} Z_k'(x) - b \biggr| = 0 = \limsup_{\substack{x \rightarrow -\infty \\ k \rightarrow \infty} } \biggl|\f{1}{x} B_k'(x) - a \biggr|.
\]
\end{lemma}
\begin{proof}
We first show that for all $k$ sufficiently large,
\be \label{itlimcond}
\limsup_{x \rightarrow -\infty} [Z_k(x) - B_k(x)] = \limsup_{x \rightarrow -\infty} [Z(x) - B(x)] = -\infty. 
\ee
Let $2\ve < b - a$. The assumption \eqref{eqn:lim_cond} implies that there exists $R> 0$ so that whenever $k > R$ and $x < -R$,  $\f{1}{x}Z_k(x) > b -\ve$ and $\f{1}{x}B_k(x) < a + \ve$. Then, \be \label{eqn:ZBklim}
Z_k(x) - B_k(x) < (b - a -2\ve)x \quad \text{for all }-x,k > R,
\ee
and this goes to $-\infty$ as $x \to -\infty$. Now, assume, to the contrary,  that  \\ $\limsup_{x \to -\infty}[Z(x) - B(x)] > -\infty$. Then, there exists a sequence $x_n \to -\infty$ with $Z(x_n) - B(x_n) \ge A$ for some constant $A$. Choose $x_n < -R$ so that $(b - a - 2\ve)x_n < A - 1$. Then, for $k \ge R$, $Z_k(x_n) - B_k(x_n) < (b - a - 2\ve)x_n < A - 1$. By the convergence $Z_k\to Z$ and $B_k \to B$, $Z(x_n) - B(x_n) \le A - 1$, a contradiction.

With \eqref{itlimcond}, we have established that $Z',B',Z_k'$, and $B_k'$ are well-defined for sufficiently large $k$. 
By Lemma \ref{identity for multiple queueing mappings}, 
\begin{align*}
 Z_k'(y)
&= B_k(y) + \sup_{-\infty < x \le y}\{Z_k(x) - B_k(x)\} - \sup_{-\infty < x \le 0}\{Z_k(x) - B_k(x)\}, \text{ and } \\
B_k'(y) &= Z_k(y) + \sup_{-\infty < x \le 0}\{Z_k(x) - B_k(x)\}- \sup_{-\infty < x \le y}\{Z_k(x) - B_k(x)\}.
\end{align*}
It therefore suffices to show that $\sup_{-\infty < x \le y}\{Z_k(x) - B_k(x)\}$ converges uniformly, on compact subsets of $\R$, to $\sup_{-\infty < x \le y} \{Z(x) - B(x)\}$, and that 
\be \label{eqn:limit_pres}
\limsup_{\substack{x \rightarrow -\infty \\ k \rightarrow \infty} } \biggl|\f{1}{x}\sup_{-\infty \le x \le y}\{Z_k(x) - B_k(x)\}  - (b - a) \biggr| = 0.
\ee

We first prove pointwise convergence. Let $x_y$ be a maximizer of $Z(x) - B(x)$ over \\ $x \in (-\infty,y]$. Then, 
\begin{align*}
\liminf_{k \rightarrow \infty}\sup_{-\infty < x \le y}\{Z_k(x) - B_k(x)\} &\ge 
\liminf_{k \rightarrow \infty} [Z_k(x_y) - B_k(x_y)] \\&= Z(x_y)- B(x_y) = \sup_{-\infty < x \le y}\{Z(x) - B(x)\}.
\end{align*}
For the converse, for $k$ sufficiently large so \eqref{itlimcond} holds, let $x_y^k$ be a sequence of maximizers of $Z_k(x) - B_k(x)$ over $x \in (-\infty,y]$. Then,
\[
\limsup_{k \rightarrow \infty}\sup_{-\infty < x \le y}\{Z_k(x) - B_k(x)\} \ge \limsup_{k \to \infty} Z_k(x_y^k) - B_k(x_y^k).
\]
By uniform-on-compact convergence $Z_k \to Z$ and $B_k \to B$, if
\be \label{eqn:lb}
\limsup_{k \rightarrow \infty}\sup_{-\infty < x \le y}\{Z_k(x) - B_k(x)\} > \sup_{-\infty < x \le y}\{Z(x) - B(x)\},
\ee
then by the uniform convergence of $B_k$ to $B$ and $Z_k$ to $Z$, it must be that $x_y^{k_j} \rightarrow -\infty$ along some subsequence $k_j$. Then, by \eqref{eqn:ZBklim},
\[
\limsup_{j \rightarrow \infty} \sup_{-\infty < x \le y}\{Z_{k_j}(x) - B_{k_j}(x)\} = \limsup_{j \rightarrow \infty} [Z_{k_j}(x_y^{k_j}) - B_{k_j}(x_y^{k_j})] = -\infty,
\]
a contradiction to \eqref{eqn:lb}.
Therefore, $\{x_y^k\}_k$ is a bounded sequence, and for each $y \in \R$, there exists some $R_y \in \R$ such that, for all $k$ sufficiently large, 
\begin{align*}
\sup_{t-\infty < x \le y} \{Z_k(x) - B_k(x)\} &= \sup_{R_y \le x \le y}\{Z_k(x) - B_k(x)\}, \text{ and} \\
\sup_{t-\infty < x \le y} \{Z(x) - B(x)\} &= \sup_{R_y \le x \le y}\{Z(x) - B(x)\}.
\end{align*}
The quantity $\sup_{R_t \le x \le y}\{Z_k(x) - B_k(x)\}$ converges to $\sup_{-\infty < x \le y} \{Z(x) - B(x)\}$ by the assumed uniform convergence on compact subsets $Z_k \rightarrow Z$ and $B_k \rightarrow B$. The uniform convergence on compact subsets of the function $y \mapsto \sup_{-\infty < x \le y} \{Z_k(x) - B_k(x)\}$ is as follows: for $N\in \R$,
\begin{align*}
\limsup_{k \rightarrow \infty}&\sup_{y \in [-N,N]} \; \bigl\lvert \,\sup_{-\infty < x \le y} \{Z_k(x) - B_k(x)\} - \sup_{-\infty < x \le y} \{Z(x) - B(x)\}\bigr\rvert \\
= \limsup_{k \rightarrow \infty}&\sup_{y \in [-N,N]} |\sup_{R_{-N} \le x \le y} \{Z_k(x) - B_k(x)\} - \sup_{R_{-N} \le x \le y} \{Z(x) - B(x)\}| = 0.
\end{align*}
To prove \eqref{eqn:limit_pres}. Let $0 < 2\ve < b - a$. By \eqref{eqn:ZBklim}, for $-x,k \ge R$,
\[
\sup_{-\infty < x \le y}\{Z_k(x) - B_k(x)\} \le \sup_{-\infty < x \le y}\{(b -\ve)x - (a + \ve)x\} = (b - a - 2\ve)y.
\]
For the reverse inequality, we observe that
\[
\sup_{-\infty < x \le y}\{Z_k(x) - B_k(x)\} \ge B_k(x) - Z_k(x),
\]
and apply Assumption \eqref{eqn:lim_cond}.
\end{proof}

\begin{lemma} \label{D and R preserve limits}
Let $(B,Z) \in \Y_2^\R$ satisfy the following limits:
\be \label{BZlim}
\lim_{x \rightarrow -\infty}\frac{B(x)}{x} = a < b = \lim_{x \rightarrow -\infty}\frac{Z(x)}{x}.
\ee
Then, the mappings $D(Z,B)$ and $R(Z,B)$ are well-defined, and 
\[
\lim_{x \rightarrow -\infty}\frac{R(Z,B)(x)}{x} = a \qquad \text{and}\qquad \lim_{x \rightarrow -\infty}\frac{D(Z,B)(x)}{x} = b.
\]

\noindent On the other hand, if $Z(x) - B(x) \to -\infty$ as $x \to -\infty$ and if the following limits exist (regardless of whether the limits \eqref{BZlim} exist) 
\[
\lim_{x \to +\infty} \f{Z(x)}{x} = c < \lim_{x \to +\infty} \f{B(x)}{x} = d,
\]
then also
\[
\lim_{x \to +\infty} \f{D(Z,B)(x)}{x} = c, \qquad \text{and}\qquad \lim_{x \to +\infty} \f{D(Z,B)(x)}{x} = d.
\]
\end{lemma}
\begin{proof}
The preservation of limits as $x \to -\infty$ is a direct corollary of Lemma \ref{uniform convergence of queueing mapping}, setting $B_k = B$ and $Z_k = Z$ for all $k$.

For limits as $x \to +\infty$, similarly as in the proof of Lemma \ref{uniform convergence of queueing mapping}, it suffices to show that 
\[
\lim_{y \to \infty} \f{1}{y}\sup_{-\infty < x \le y}\{Z(x) - B(x)\} = d - c.
\]

Let $0 < \ve < d - c$, and let $R_\ve > 0$ be sufficiently large so that for all $x > R_\ve$, $(d - c - \ve)x \le Z(x) - B(x) \le (d - c + \ve)x$. It follows that $Z(x) - B(x) \to +\infty$ as $x \to +\infty$, so for all sufficiently large $y$, 
\[
\sup_{-\infty < x \le y}\{Z(x) - B(x)\} = \sup_{R_\ve \le x \le y}\{Z(x) - B(x)\}.
\]
Then, 
\begin{align*}
&\quad \; b - c - \ve =  \liminf_{y \to \infty} \f{1}{y} \sup_{R_\ve \le x \le y}\{(b - c -\ve)x\}  \le\liminf_{y \to \infty} \f{1}{y}\sup_{-\infty \le x \le y}\{Z(x) - B(x)\} \\
&\le \limsup_{y \to \infty} \f{1}{y}\sup_{-\infty \le x \le y}\{Z(x) - B(x)\} \le \limsup_{y \to \infty} \f{1}{y} \sup_{R_\ve \le x \le y}\{(b - c + \ve)x\} = b - c +\ve.
\end{align*}
Letting $\ve \searrow 0$ completes the proof. 
\end{proof}
We prove a related lemma.

\begin{lemma} \label{liminflem}
Let $B,Z \in \CRpin$ satisfy 
\[
\liminf_{x \to -\infty} \f{Z(x)}{x} = b  > a = \lim_{x \to -\infty} \f{B(x)}{x},
\]
where the limit on the right is assumed to exist. Then,
\[
\liminf_{x \to -\infty} \f{D(Z,B)(x)}{x} \ge b.
\]
\end{lemma}
\begin{proof}
Let $0 < 2\ve < b -a$. Let $R_\ve > 0$ be sufficiently large so that, for all $x < -R_\ve$, $Z(x) < (b - \ve)x$  and $B(x) > (a + \ve)x$ Then, for all $y < -R_\ve$, 
\[
\sup_{-\infty < x \le y}\{Z(x) - B(x)\} \le \sup_{-\infty < x \le y}\{(b - a - 2\ve)x\} = (b - a - 2\ve)y.
\]
Hence, 
\[
\liminf_{y \to -\infty} \f{1}{y} \sup_{-\infty < x \le y}\{Z(x) - B(x)\} \ge (b - a - 2\ve)
\]
Since 
\[
D(Z,B)(y) = B(y) + \sup_{-\infty < x \le y}\{Z(x) - B(x)\} - \sup_{-\infty < x \le 0}\{Z(x) - B(x)\},
\]
and $\ve > 0$ is arbitrary, this completes the proof. 
\end{proof}

\begin{lemma} \label{DRbij}
Let 
\be \label{Y2abcd}
\begin{aligned}
\Y_2^{a,b,c,d}&=\Biggl\{(B,Z) \in \Y_2^\R: 
\lim_{x\to -\infty} \f{B(x)}{x} = a,\;\lim_{x \to -\infty}\f{Z(x)}{x} = b, \\ 
&\qquad\qquad\qquad\qquad\qquad\qquad \lim_{x \to +\infty} \f{B(x)}{x} = c, \;\text{and} \lim_{x \to +\infty}\f{Z(x)}{x} = d.
\Biggr\}
\end{aligned}
\ee
Then, the mapping $(B,Z) \mapsto (R(Z,B),D(Z,B))$ is a bijection $\Y_2^{a,b,c,d} \to \Y_2^{a,b,c,d}$.

For a function $f:\R \to \R$, define the transformed function $\Rf f:\R \to \R$ as $\Rf f(x) = f(-x)$, as in Section \ref{sec:notat}. Then, the inverse of  the mapping $(B,Z) \mapsto (R(Z,B),D(Z,B))$ is the mapping $(C,Y) \mapsto (\hat R(Y,C),\hat D(Y,C))$, where $\hat R$ and $\hat D$ are defined by 
\be
\hat R(Y,C) = \Rf[D(\Rf C, \Rf Y)],\qquad \text{and}\qquad \hat D(Y,C) = \Rf[R(\Rf C, \Rf Y)].
\ee

\noindent In other words, if $(B,Z) \in \Y_2^{a,b,c,d}$, then 
\be \label{Did}
D( \Rf R(Z,B),\Rf D(Z,B)) = \Rf B,\qquad\text{and}\qquad R(\Rf R(Z,B),\Rf D(Z,B)) = \Rf Z.
\ee
\end{lemma}

\begin{proof}
To check that the mapping $(B,Z) \mapsto (R(Z,B),D(Z,B))$ preserves $\Y_n^{a,b,c,d}$, we first use Lemma \ref{identity for multiple queueing mappings} to verify that $D(Z,B)(0) = R(Z,B)(0) = 0$. The preservation of the limits follows from Lemma \ref{D and R preserve limits}. Furthermore, since
\[
\lim_{x \to -\infty} \f{\Rf R(Z,B)(x)}{x} = -c > -d = \lim_{x \to -\infty} \f{\Rf D(Z,B)(x)}{x},
\]
we have $\lim_{x \to -\infty}[\Rf R(Z,B)(x) - \Rf D(Z,B)(x)] = -\infty$ so that $D( \Rf R(Z,B),\Rf D(Z,B))$ and $R( \Rf R(Z,B),\Rf D(Z,B))$ are well-defined. 

We show the first equality of \eqref{Did}, which is equivalent to $\hat R(D(Z,B),R(Z,B)) = B$. The proof that $\hat D(D(Z,B),R(Z,B)) = Z$, as well as the proofs that 
\[
D(\hat D(Y,C),\hat R(Y,C)) = Y,\qquad\text{and}\qquad R(\hat D(Y,C),\hat R(Y,C)) = C
\]
follow a similar argument. Using the $n = 2$ case of Lemma \ref{identity for multiple queueing mappings},
\begin{align*}
&\quad \; D(\Rf R(Z,B),\Rf D(Z,B))(y)  \\
&= \Rf D(Z,B)(y) + \sup_{-\infty < x \le y}\{\Rf R(Z,B)(x) - \Rf D(Z,B)(x)\}  \\
&\qquad\qquad\qquad\qquad - \sup_{-\infty < x \le 0}\{\Rf R(Z,B)(x) - \Rf D(Z,B)(x)\} \\
&\qquad= B(-y) + \sup_{-\infty < x \le -y}\{Z(x) - B(x)\} - \sup_{-\infty  < x \le 0}\{Z(x) - B(x)\} \\
&\qquad\qquad\qquad + \sup_{-\infty < x \le y}\{Z(-x) - B(-x) - 2\sup_{-\infty < u \le - x}[Z(u) - B(u)]\} \\
&\qquad\qquad\qquad\qquad\qquad - \sup_{-\infty < x \le 0}\{Z(-x) - B(-x) - 2\sup_{-\infty < u \le - x}[Z(u) - B(u)]\} \\
&= B(-y) + \sup_{-\infty < x \le -y}\{Z(x) - B(x)\} - \sup_{-\infty  < x \le 0}\{Z(x) - B(x)\} \\
&\qquad\qquad\qquad+ \sup_{-y \le x < \infty}\{Z(x) - B(x) - 2\sup_{-\infty < u \le x}[Z(u) - B(u)]\} \\
&\qquad\qquad\qquad\qquad\qquad - \sup_{0 \le x < \infty}\{Z(x) - B(x) - 2\sup_{-\infty < u \le x}[Z(u) - B(u)]\} \\
&= B(-y) + \sup_{-\infty < x \le -y}\{Z(x) - B(x)\} - \sup_{-\infty  < x \le 0}\{Z(x) - B(x)\} \\
&\qquad\qquad\qquad+ \sup_{-y \le x < \infty}\{Z(x) - B(x) - 2\sup_{-\infty < u \le x}[Z(u) - B(u)]\} \\
&\qquad\qquad\qquad\qquad\qquad - \sup_{0 \le x < \infty}\{Z(x) - B(x) - 2\sup_{-\infty < u \le x}[Z(u) - B(u)]\} \\
&= B(-y) + \sup_{-\infty < x \le -y}\{Z(x) - B(x)\} - \sup_{-\infty  < x \le 0}\{Z(x) - B(x)\} \\
&\qquad\qquad\qquad - \inf_{-y \le x < \infty}\{2\sup_{-\infty < u \le x}[Z(u) - B(u)] - (Z(x) - B(x))\} \\
&\qquad\qquad\qquad\qquad\qquad - \inf_{0 \le x < \infty}\{2\sup_{-\infty < u \le x}[Z(u) - B(u)] - (Z(x) - B(x))\}. 
\end{align*}
To show that this last line equals $B(-y) = \Rf B(y)$, it suffices to show that for all $y \in \R$,
\[
\sup_{-\infty < x \le -y}\{Z(x) - B(x)\} = \inf_{-y \le x < \infty}\{2\sup_{-\infty < u \le x}[Z(u) - B(u)] - (Z(x) - B(x))\}. 
\]
This is exactly the statement of Lemma \ref{pitman Representation Lemma} with $f = Z -B$. 
\end{proof}

\begin{lemma} \label{queue_order}
Assume that $Z,Z',B \in \CRpin$ satisfy $Z \li Z'$ and
\[
\limsup_{x \rightarrow -\infty} Z(x) - B(x) = -\infty.
\]
Then $B \li D(Z,B) \li D(Z',B)$. 
\end{lemma}
\begin{proof} 
By Lemma \ref{DRcont},
\[
D(Z,B)(x,y) = B(x,y) +\Bigl(\sup_{x \le u \le y}\{Z(x,u) - B(x,u)\} - \sup_{-\infty < u \le x}\{Z(x,u) - B(x,u)\}\Bigr)^+,
\]
from which both inequalities follow. 
\end{proof}

\begin{lemma} \label{image of script D lemma}
Let $A \subseteq \R$ be a Borel set. Then, the mapping $\D^{(n)}$ \eqref{scrD} satisfies the following properties: 
\begin{enumerate} [label=\rm(\roman{*}), ref=\rm(\roman{*})]  \itemsep=3pt 
    \item \label{itm:Ynlim} If $(Z^1,\ldots,Z^n) \in \Y_n^A$ satisfies
    \[
    \lim_{x \rightarrow -\infty} \f{Z^i(x)}{x} = a_i\qquad\text{for }  1 \le i \le n,
    \]
    then the image $(\eta^1,\ldots,\eta^n) = \D^{(n)}(Z^1,\ldots,Z^n)$ also satisfies
    \[
    \lim_{x \rightarrow -\infty} \f{\eta^i(x)}{x} = a_i\qquad\text{for }  1 \le i \le n.
    \]
    \item \label{itm:image}  $\D^{(n)}$ maps $\Y_n^A$ into $\X_n^A$.
\end{enumerate}
\end{lemma}
\begin{proof}
Item \ref{itm:Ynlim} follows by definition of $\D^{(n)}$ \eqref{scrD} and repeated application of Lemma \ref{D and R preserve limits}. To show Item \ref{itm:image}, we must show that for $(Z^1,\ldots,Z^n) \in \Y_n^A$, 
$(\eta^1,\ldots,\eta^n) = \D^{(n)}(Z^1,\ldots,Z^n),
$
each $\eta^i$ is continuous, satisfies $\eta^i(0) = 0$, and $\eta^i \li \eta^{i + 1}$. Continuity follows from Lemma \ref{DRcont}. Since $Z^1(0) = 0$, Lemma \ref{identity for multiple queueing mappings} proves $\eta^i(0)  = 0$ for $1 \le i \le n$. 
By  Lemma \ref{queue_order},  
$\eta^2 = D(Z^2,Z^1) \gi Z^1 = \eta^1$.
Assume inductively that
\[
\eta^{i} = D^{(i)}(Z^{i},\ldots,Z^1) \gi D^{(i - 1)}(Z^{i - 1},\ldots,Z^1) = \eta^{i - 1}.
\]
Then, after applying this assumption with $Z^2,\ldots,Z^{i + 1}$ in place of $Z^1,\ldots,Z^i$ and using Lemma \ref{queue_order}, we get that 
\begin{align*}
\eta^{i + 1} = &D^{(i + 1)}(Z^{i + 1},\ldots,Z^1) = D(D^{(i)}(Z^{i + 1},\ldots,Z^2),Z^1) \\
\gi &D(D^{(i - 1)}(Z^{i},\ldots,Z^2),Z^1) = D^{(i)}(Z^i,\ldots,Z^1) = \eta^i, 
\end{align*}
and the proof is complete.
\end{proof}

\begin{lemma} \label{analogue of Lemma 4.4 in Joint Buse Paper}
Let $(B^1,Z^1,Z^2) \in \Y_3^\R$, and set $B^2 = R(Z^1,B^1)$.
Then,
\[
D(D(Z^2,B^2),D(Z^1,B^1)) = D(D(Z^2,Z^1),B^1).
\]
\end{lemma}
\begin{proof}
We first note that by definition of $D^{(3)}$ \eqref{Diter} and Lemma \ref{identity for multiple queueing mappings},
\begin{align}
    D(D(Z^2,Z^1),B^1)(y) &= D^{(3)}(Z^2,Z^1,B^1)(y) \nonumber \\
    &= B^1(y) + \sup_{-\infty < w \le x \le y}\{Z^1(x) - B^1(x) + Z^2(w) - Z^1(w)\} \nonumber \\
    &\qquad\qquad- \sup_{-\infty < w \le x \le 0}\{Z^1(x) - B^1(x) + Z^2(w) - Z^1(w)\}. \label{expression for rhs in 4.4 analogue}
\end{align}
On the other hand, by definitions of the mappings $D$ and $R$ \eqref{Ddef}--\eqref{Rdef},
\begin{align}
   &\quad \; D(D(Z^2,B^2),D(Z^1,B^1))(y) \nonumber \\
   &= D(Z^1,B^1)(y) + \sup_{-\infty < x \le y}\{D(Z^2,B^2)(x) - D(Z^1,B^1)(x) \} \nonumber \\
   &\qquad\qquad\qquad\qquad- \sup_{-\infty < x \le 0 }\{D(Z^2,B^2)(x) - D(Z^1,B^1)(x)  \}  \nonumber \\
   &= B^1(y) + \sup_{-\infty < x \le y}\{Z^1(x) - B^1(x)\} - \sup_{-\infty < x \le 0}\{Z^1(x) - B^1(x)\} \nonumber \\
    &+ \sup_{-\infty < x \le y}\Big[B^2(x) + \sup_{-\infty < w \le x}\{Z^2(w) - B^2(w)\} - B^1(x) - \sup_{-\infty < w \le x}\{Z^1(w) - B^1(w)\}   \Big] \nonumber \\
    &- \sup_{-\infty < x \le 0}\Big[B^2(x) + \sup_{-\infty < w \le x}\{Z^2(w) - B^2(w)\} - B^1(x) - \sup_{-\infty < w \le x}\{Z^1(w) - B^1(w)\}   \Big]  \nonumber \\
    &= B^1(y) + \sup_{-\infty < x \le y}\{Z^1(x) - B^1(x)\} - \sup_{-\infty < x \le 0}\{Z^1(x) - B^1(x)\} \nonumber  \\
    &\qquad+ \sup_{-\infty < x \le y}\Big[Z^1(x) - B^1(x) - 2\sup_{-\infty < w \le x}\{Z^1(w) - B^1(w)\}\} \nonumber \\
    &\qquad \qquad\qquad\qquad\qquad+ \sup_{-\infty < v \le w \le x}\{Z^2(w) - Z^1(w) + Z^1(v) - B^1(v) \}   \Big] \nonumber \\
    &\qquad- \sup_{-\infty < x \le 0}\Big[Z^1(x) - B^1(x) - 2\sup_{-\infty < w \le x}\{Z^1(w) - B^1(w)\}\} \nonumber \\
    &\qquad \qquad\qquad\qquad\qquad+ \sup_{-\infty < v \le w \le x}\{Z^2(w) - Z^1(w) + Z^1(v) - B^1(v) \}   \Big] \label{expression for lhs in 4.4 analogue}.
\end{align}
Comparing \eqref{expression for rhs in 4.4 analogue} with \eqref{expression for lhs in 4.4 analogue}, it is sufficient to show that, for arbitrary $y \in \R$,
\begin{align}
    &\sup_{-\infty < x \le y}\{Z^1(x) - B^1(x)\} \nonumber \\
    &\qquad + \sup_{-\infty < x \le y}\Big[Z^1(x) - B^1(x)  - 2\sup_{-\infty < w \le x}\{Z^1(w) - B^1(w)\} \label{lhs}  \\
    &\qquad\qquad\qquad + \sup_{-\infty < v \le w \le x}\{Z^2(w) - Z^1(w) + Z^1(v) - B^1(v) \}   \Big] \nonumber\\
      = &\sup_{-\infty < w \le x \le y}\{Z^1(x) - B^1(x) + Z^2(w) - Z^1(w)\}. \label{rhs}
\end{align}

We will first prove that $\eqref{lhs} \leq\eqref{rhs}$. We note that
\begin{align} 
\eqref{lhs} &\leq \sup_{-\infty < x \le y}\{Z^1(x) - B^1(x)\} \\
&\qquad + \sup_{-\infty < x \le y}\Big[Z^1(x) - B^1(x) - 2\sup_{-\infty < w \le x}\{Z^1(w) - B^1(w)\} \\
 &\qquad\qquad + \sup_{-\infty < w \le x}\{Z^2(w) - Z^1(w)\} + \sup_{-\infty < w \le x}\{ Z^1(w) - B^1(w) \}   \Big]  \nonumber\\
=&\sup_{-\infty < x \le y}\{Z^1(x) - B^1(x)\}  \label{greater than lhs} \\
+ &\sup_{-\infty < x \le y}\Big[Z^1(x) - B^1(x)
 - \sup_{-\infty < w \le x}\{Z^1(w) - B^1(w)\}  + \sup_{-\infty < w \le x}\{Z^2(w) - Z^1(w)\}   \Big] \nonumber   
\end{align}
Now, we let $x^\star \le y$ be a point such that 
\begin{multline*}
Z^1(x^\star) - B^1(x^\star)
 - \sup_{-\infty < w \le x^\star}\{Z^1(w) - B^1(w)\}  + \sup_{-\infty < w \le x^\star}\{Z^2(w) - Z^1(w)\}  \\
=\sup_{-\infty < x \le y}\Big[Z^1(x) - B^1(x)
 - \sup_{-\infty < w \le x}\{Z^1(w) - B^1(w)\}  + \sup_{-\infty < w \le x}\{Z^2(w) - Z^1(w)\}   \Big]. 
\end{multline*}
We consider two cases.
\begin{flalign*}
&\textbf{Case 1:}  \sup_{-\infty < w \le x^\star}\{Z^1(w) - B^1(w)\} = \sup_{-\infty < w \le  y}\{Z^1(w) - B^1(w)\} &&
\end{flalign*}
Then,
\begin{align*}
    \eqref{lhs} \le \eqref{greater than lhs} &= \sup_{-\infty < x \le y}\{Z^1(x) - B^1(x)\} 
+ Z^1(x^\star) - B^1(x^\star) \\
 &\qquad - \sup_{-\infty < w \le x^\star}\{Z^1(w) - B^1(w)\}  + \sup_{-\infty < w \le x^\star}\{Z^2(w) - Z^1(w)\}    \\
    &= Z^1(x^\star) - B^1(x^\star) + \sup_{-\infty < w \le x^\star}\{Z^2(w) - Z^1(w)\}  \\
    &\le \sup_{-\infty < w \le x \le y}\{Z^1(x) - B^1(x) + Z^2(w) - Z^1(w)\} = \eqref{rhs}.
\end{align*}
\begin{flalign*}
\textbf{Case 2:}  \sup_{-\infty < w \le x^\star}\{Z^1(w) - B^1(w)\} < \sup_{-\infty < w \le  y}\{Z^1(w) - B^1(w)\}. &&
\end{flalign*}
Then, we have that 
\[
\sup_{-\infty < w \le  y}\{Z^1(w) - B^1(w)\} = \sup_{x^\star \le w \le y}\{Z^1(w) - B^1(w)\},
\]
so, noting that $Z^1(x^\star) - B^1(x^\star) \le \underset{-\infty < w \le x^\star}{\sup}\{Z^1(w) - B^1(w)\}$,
\begin{align*}
    \eqref{lhs} \le \eqref{greater than lhs} &= \sup_{-\infty < x \le y}\{Z^1(x) - B^1(x)\} 
+ Z^1(x^\star) - B^1(x^\star)  \\
&\qquad - \sup_{-\infty < w \le x^\star}\{Z^1(w) - B^1(w)\}  + \sup_{-\infty < w \le x^\star}\{Z^2(w) - Z^1(w)\}    \\
    &\le \sup_{x^\star \le  x \le y}\{Z^1(x) - B^1(x)\} + \sup_{-\infty < w \le x^\star}\{Z^2(w) - Z^1(w)\}\\
    &= \sup_{-\infty < w \le x^\star \le x \le y }\{Z^1(x) - B^1(x) + Z^2(w) - Z^1(w)  \} \\
    &\le \sup_{-\infty < w \le x \le y }\{Z^1(x) - B^1(x) + Z^2(w) - Z^1(w)  \} = \eqref{rhs}.
\end{align*}

Now, we prove that $\eqref{rhs} \le \eqref{lhs}$. Let $-\infty < w^\star \le x^\star \le y$ be such that 
\[
Z^1(x^\star) - B^1(x^\star) + Z^2(w^\star) - Z^1(w^\star)= \sup_{-\infty < w \le x \le y }\{Z^1(x) - B^1(x) + Z^2(w) - Z^1(w)  \}.
\]
We consider two new cases. 
\begin{flalign*}
\textbf{Case 1:} \sup_{-\infty < v \le x^\star}\{Z^1(v) - B^1(v)\} = \sup_{-\infty < v \le w^\star}\{Z^1(v)  -B^1(v)\}. &&
\end{flalign*}
Then,
\begin{align*}
    \eqref{rhs}
    = &Z^1(x^\star) - B^1(x^\star) + Z^2(w^\star) - Z^1(w^\star) \\
    &\qquad \qquad\qquad+ \sup_{-\infty < v \le w^\star}\{Z^1(v)  -B^1(v)\} -\sup_{-\infty < v \le x^\star}\{Z^1(v) - B^1(v)\}  \\
    \le &\sup_{-\infty < x \le y}\Big[Z^1(x) - B^1(x) - \sup_{-\infty < v \le x}\{Z^1(v) - B^1(v)\}\\
    &\qquad \qquad\qquad+ \sup_{-\infty < v \le w \le x}\{Z^1(v) - B^1(v) + Z^2(w) - B^2(w)\}  \Big] \\
    \le &\sup_{-\infty < x \le y}\Big[Z^1(x) - B^1(x)+ \sup_{-\infty < v \le y}\{Z^1(v) - B^1(v)\}    \\
    - &2\sup_{-\infty < v \le x}\{Z^1(v) - B^1(v)\} + \sup_{-\infty < v \le w \le x}\{Z^2(w) - Z^1(w) + Z^1(v) - B^1(v) \}   \Big]
    \\ =&\sup_{-\infty < v \le y}\{Z^1(v) - B^1(v)\} \nonumber \\
    &\qquad + \sup_{-\infty < x \le y}\Big[Z^1(x) - B^1(x)  - 2\sup_{-\infty < v \le x}\{Z^1(v) - B^1(v)\} \label{lhs}  \\
    &\qquad\qquad\qquad + \sup_{-\infty < v \le w \le x}\{Z^2(w) - Z^1(w) + Z^1(v) - B^1(v) \}   \Big] =\eqref{lhs}.
\end{align*}

\begin{flalign*}
\textbf{Case 2:} \sup_{-\infty < v \le x^\star}\{Z^1(v) - B^1(v)\} > \sup_{-\infty < v \le w^\star}\{Z^1(v)  -B^1(v)\}. &&
\end{flalign*}
Then, we have that
\[
\sup_{-\infty < v \le x^\star}\{Z^1(v) - B^1(v)\} = \sup_{w^\star \le v \le x^\star}\{Z^1(v) - B^1(v)\} > Z^1(w^\star) - B^1(w^\star).
\]
Consequently, there is a point $v^\star \in (w^\star,x^\star]$ such that 
\[
Z^1(v^\star) - B^1(v^\star) = \sup_{-\infty < v \le x^\star}\{Z^1(v) - B^1(v)\}. 
\]
Next, we define $z^\star \in [w^\star,v^\star)$ as
\[
z^\star = \inf \Big\{z \in [w^\star,v^\star): Z^1(z) - B^1(z) = \sup_{-\infty < v \le w^\star}\{Z^1(v) - B^1(v) \} \Big\}.
\]
To see that this set is nonempty, recall that
\begin{align*}
Z^1(v^\star) - B^1(v^\star) &= \sup_{-\infty < v \le x^\star}\{Z^1(v) - B^1(v)\} \\
&> \sup_{-\infty < x \le w^\star}\{Z^1(v)  -B^1(v)\} \ge Z^1(w^\star) - B^1(w^\star),
\end{align*}
and use the Intermediate Value Theorem. By continuity, we observe that 
\begin{equation} \label{sup_eq}
Z^1(z^\star) - B^1(z^\star) = \sup_{-\infty < v \le z^\star}\{Z^1(v) - B^1(v)\} = \sup_{-\infty < v \le w^\star}\{Z^1(v) - B^1(v)  \}.
\end{equation}

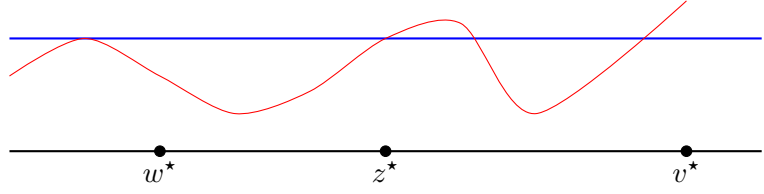
\begin{figure}[t]
\centering
    \begin{tikzpicture}
    \draw[black,thick] (10,0)--(20,0);
    \draw[blue,thick] (10,1.5)--(20,1.5);
    \filldraw[black] (12,0) circle (2pt) node[anchor = north] {$w^\star$};
    \filldraw[black] (15,0) circle (2pt) node[anchor = north] {$z^\star$};
    \filldraw[black] (19,0) circle (2pt) node[anchor = north] {$v^\star$};
    \draw [red] plot [smooth] coordinates {(10,1)(11,1.5)(12,1)(13,0.5)(14,0.8)(15,1.5)(16,1.7)(17,0.5)(19,2)};
    \end{tikzpicture}
    \caption{\small Example graph of the function $Z^1(z) - B^1(z)$. The upper horizontal (blue) line represents the value of $\underset{-\infty < v \le w^\star}{\sup}\{Z^1(v) - B^1(v)\}$. The value of $z^\star$ is the first location greater than or equal to $w^\star$ where the function $Z^1(z) - B^1(z)$ takes the value on the blue line.}
    \label{fig:existence of w^*}
\end{figure}

Refer to figure \ref{fig:existence of w^*} for clarity. Now, by \eqref{sup_eq}, we have 
\begin{align*}
    \eqref{rhs} = &Z^1(x^\star) - B^1(x^\star) + Z^2(w^\star) - Z^1(w^\star) \\
    = &Z^1(x^\star) - B^1(x^\star) + Z^1(z^\star) - B^1(z^\star) - 2\sup_{-\infty < v \le z^\star}\{Z^1(v) - B^1(v)\} \\
    &\qquad+Z^2(w^\star) - Z^1(w^\star) + \sup_{-\infty <v \le w^\star}\{Z^1(v) - B^1(v)\} \\
    &\le \sup_{-\infty < x \le y}\{Z^1(x) - B^1(x)\}\\
    + &\sup_{-\infty < z \le y}\Big\{Z^1(z) - B^1(z) - 2\sup_{-\infty < v \le z}\{Z^1(v) - B^1(v)\} \\
    &\qquad\qquad\qquad + \sup_{-\infty < v \le w \le z}\{Z^2(w) - Z^1(w) + Z^1(v) - B^1(v)\} \Big\} = \eqref{lhs}. 
\end{align*}
This concludes all cases of the proof.
\end{proof}

\newpage
\begin{lemma} \label{lem:Alternate Rep of Iterated Queues}
Let $n \ge 2$, and assume $(B^1,Z^1,Z^{2},\ldots,Z^n) \in \Y_{n + 1}^\R$.   For $2 \le j \le n$ define
$
B^j = R(Z^{j - 1},B^{j - 1}).
$  
Then, for $1 \le k \le n - 1$,
\begin{align*}
    &D^{(n + 1)}(Z^n,Z^{n - 1},\ldots,Z^1,B^1) \\
    &\qquad\qquad = D^{(k + 1)}(D^{(n - k + 1)}(Z^n,\ldots,Z^{k + 1},B^{k + 1}),D(Z^k,B^k),\ldots,D(Z^1,B^1)).
\end{align*}
\end{lemma}
\begin{proof}
With Lemma \ref{analogue of Lemma 4.4 in Joint Buse Paper} in place, we can now follow the argument of Theorem 4.5 in \cite{Fan-Seppalainen-20}.
By Lemma \ref{multiline process well-defined}, all the given operations are well-defined. Lemma \ref{analogue of Lemma 4.4 in Joint Buse Paper} gives us the statement for $n = 2$. Assume, by induction, that the statement is true for some $n - 1 \ge 2$. We will show the statement is also true for $n$. We first prove the case $k = 1$. 
\begin{align*}
    &\;\;\;D^{(2)}(D^{(n)}(Z^n,\ldots,Z^2,B^2),D(Z^1,B^1)) \\
    &= D(D(D^{(n -1)}(Z^n,\ldots,Z^2),B^2),D(Z^1,B^1)) \\
    &= D(D(D^{(n - 1)}(Z^n,\ldots,Z^2),Z^1),B^1) \\
    &= D(D^{(n)}(Z^n,\ldots,Z^1),B^1) \\
    &= D^{(n + 1)}(Z^n,\ldots,Z^1,B^1).
\end{align*}
The second equality above was a consequence of Lemma \ref{analogue of Lemma 4.4 in Joint Buse Paper}.
Now, let $2 \le k \le n - 1$. Then, applying the definition of $D^{(k + 1)}$ followed by the induction assumption,
\begin{align*}
    &D^{(k + 1)}(D^{(n - k + 1)}(Z^n,\ldots,Z^{k + 1},B^{k + 1}),D(Z^k,B^k),\ldots,D(Z^1,B^1)) \\
    = &D(D^{(k)}(D^{(n - k + 1)}(Z^n,\ldots,Z^{k + 1},B^{k + 1}),D(Z^k,B^k),\ldots,D(Z^2,B^2)),D(Z^1,B^1)) \\
    = &D(D^{(n)}(Z^n,\ldots,Z^2,B^2),D(Z^1,B^1)) = D^{(2)}(D^{(n)}(Z^n,\ldots,Z^2,B^2),D(Z^1,B^1)).
\end{align*}
Hence, we have reduced this to the $k = 1$ case. 
\end{proof}
\noindent We note that the case $k= n - 1$ of Theorem \ref{lem:Alternate Rep of Iterated Queues} gives us
\begin{equation} \label{intertwining}
    D^{(n + 1)}(Z^n,\ldots,Z^1,B^1) = D^{(n)}(D(Z^n,B^n),\ldots,D(Z^1,B^1)).
\end{equation}

\subsection{Multiclass Markov chains}
\label{section:multiline-Markov process}
We define here two discrete-time Markov chains on the state spaces $\Y_n^{A}$ and $\Y_n^A$ \eqref{Yndef}--\eqref{Xndef}. 

The multiline process is a discrete-time Markov chain on the state space $\Y_n^{(\sigma^2 a,\infty)}$ of \eqref{Yndef} for some choice of $\sigma > 0$ and $a \in \R$. Analogous processes are defined in discrete settings in \cite{Ferrari-Martin-2007} and \cite{Fan-Seppalainen-20}. 
The transition from the time $m-1$ state $\mathbf Z_{m - 1} =\mathbf Z =  (Z^1,Z^2,\ldots, Z^n) \in \Y_n^{(\sigma^2 a,\infty)}$ to the time $m$ state 
\begin{equation*} 
\mathbf Z_{m} = \overline{\mathbf Z} = (\overline Z^1,\overline Z^2,\ldots, \overline Z^n) \in \Y_n^{(\sigma^2 a,\infty)}
\end{equation*}
is defined as follows. The driving force is an auxiliary   function $B \in \CRpin$ that satisfies
\[
\lim_{x \rightarrow -\infty} \f{B(x)}{x} = \sigma^2 a \in \R. 
\]
for some fixed choice of $a$.
First, set $B^1 = B$,  and $\overline Z^1 = D(Z^1,B^1)$. 
    Then, iteratively for $i = 2,3,\ldots,n$:
    \be \label{multiline process}
    B^i = R(Z^{i - 1},B^{i - 1}), \qquad\text{and}\qquad
    \overline Z^i = D(Z^i,B^i).
    \ee

\begin{lemma} \label{multiline process well-defined}
The multiline process \eqref{multiline process} is well-defined on the state space $\Y_n^{(\sigma^2 a,\infty)}$.
\end{lemma}
\begin{proof}
This follows from Lemma \ref{D and R preserve limits}: Inductively, each $B^i$ satisfies
\[
\lim_{x \rightarrow -\infty}\frac{B^i(x)}{x} = \sigma^2 a 
\]
so since $\mathbf Z \in \Y_n^{(\sigma^2 a,\infty)}$, we have that, for $1 \le i \le n$,
\[
\limsup_{x \rightarrow -\infty}Z^i(x) - B^i(x) = -\infty. \qedhere
\]
\end{proof}

\begin{theorem} \label{multiline invariant distribution}
For each choice of $a \in \R$ and $\bar \dir = (\dir_1 < \cdots <\dir_n) \in \R^n_{> a}$, the measure $\nu_\sigma^{\bar \dir}$ on $\Y_n^{(\sigma^2 a,\infty)}$ is invariant for the multiline process \eqref{multiline process} if the driving function $B$ at each step of the evolution is taken to be an independent two-sided Brownian motion with diffusivity $\sigma$ and drift $\sigma^2 a$.
\end{theorem}
\begin{proof}
Assume that $\mathbf Z = (Z^1,\ldots, Z^n) \in \Y_n^{(\sigma^2 a,\infty)}$ has distribution $\nu_\sigma^{\bar \dir}$. We will show that $\overline{\mathbf Z}$ also has distribution $\nu_\sigma^{\bar \dir}$. The assumption on $\mbf Z$ means that $Z^1,\ldots,Z^n$ are independent two-sided Brownian motions with diffusivity $\sigma$ and drift $\sigma^2\dir_i$. By Theorem \ref{thm:OCY_DR}, 
$
\overline Z^1 = D(Z^1,B^1)
$
is a two-sided Brownian motion with diffusivity $\sigma$ and  drift $\sigma^2\dir_1$, independent of
$
B^2 = R(Z^1,B^1),
$
which is a two-sided Brownian motion with diffusivity $\sigma$ and drift $\sigma^2 a$. Hence, the random paths $\overline Z^1, B^2,Z^2,\ldots, Z^n$ are mutually independent. We iterate this process as follows: Assume, for some $2 \leq k \leq n - 1$, that the random paths $\overline Z^1,\ldots,\overline Z^{k - 1},B^k,Z^k,\ldots, Z^n$ are mutually independent, where for $1 \leq i \leq k - 1$, $\overline Z^i$ is a Brownian motion with diffusivity $\sigma$ and drift $\sigma^2\dir_i$. Then, by another application of Theorem \ref{thm:OCY_DR},
$
\overline Z^k = D(Z^k,B^k)
$
is a two-sided Brownian motion with diffusivity $\sigma$ and drift $\sigma^2\dir_k$, independent of 
$
B^{k + 1} = R(Z^k,B^k),
$
which is a two-sided Brownian motion with diffusivity $\sigma$ and zero drift.
Since $(\overline Z^k,B^{k + 1})$ is a function of $(B^k,Z^k)$, we have that $\overline Z^1,\ldots,\overline Z^k,B^{k + 1},Z^{k + 1},\ldots,Z^n$ are mutually independent, completing the proof. 
\end{proof}

\noindent For a fixed diffusivity $\sigma > 0$ and a constant $a > 0$, we now define a Markov chain \[
\eta := \{\eta_m\}_{m \in \Z_{\ge 0}} = \{(\eta_m^1,\ldots,\eta_m^n)\}_{m \in \Z_{\ge 0}}
\]
with state space $\X_n^{(\sigma^2 a,\infty)}$. Henceforth,  $\mbf F =\{F_m\}_{m \ge 1}$ denotes an i.i.d.\ sequence of two-sided Brownian motions with diffusivity $\sigma$ and drift $\sigma^2 a$, independent of the initial configuration $\eta_0\in\X_n^{(\sigma^2 a,\infty)}$. At each discrete time step $m \ge 1$, set $ F_m$ to be the driving Brownian motion. Given the time $m-1$ state $\eta_{m-1}$,   define the time $m$   state of the chain as  
\begin{equation} \label{Busemann Markov chain}
\eta_{m} = \bigl(D(\eta^1_{m - 1},F_m),D(\eta^2_{m - 1},F_m),\ldots,D(\eta^n_{m - 1},F_m)\bigr).
\end{equation}
Lemmas \ref{queue_order} and \ref{image of script D lemma} imply that if $\eta_{m -1} \in \X_n^{(\sigma^2 a,\infty)}$, then $\eta_{m} \in \X_n^{(\sigma^2 a,\infty)}$ as well. 
\begin{theorem} \label{existence of an invariant measure for Busemann MC}
For $a,\sigma > 0$ and $\bar \dir = (\dir_1 < \cdots < \dir_n) \in \R^n_{> a}$, the measure $\mu_\sigma^{\bar \dir} $ of \eqref{musigma} is invariant for the Markov chain \eqref{Busemann Markov chain}.
\end{theorem}
\newpage
\begin{proof}
This follows by an intertwining argument developed by Ferrari and Martin for particle systems in \cite{Ferrari-Martin-2007} and carried out for exponential last-passage percolation in \cite{Fan-Seppalainen-20}. Let $\mathbf Z \sim \nu_\sigma^{\bar \dir}$. Assume that $\eta$ has distribution $\mu_\sigma^{\bar \dir}$: the distribution of $\D^{(n)}(\mathbf Z)$. Without loss of generality, we assume that $\eta = \D^{(n)}(\mathbf Z)$. Then, for Brownian motion $B$, let $\Ss^B$ denote the mapping of a single evolution step of $\mathbf Z$ according to the multiline process \eqref{multiline process} and $\T^B$ denote the mapping of a single evolution step of $\eta$ according to the Markov chain \eqref{Busemann Markov chain}. Then, using the definition of $D^{(k)}$ and Equation \eqref{intertwining},
\begin{align*}
\T_k^B(\eta) = &D(\eta^k,B) = D(D^{(k)}(Z^k,\ldots,Z^1),B^1) = D^{(k + 1)}(Z^k,\ldots,Z^1,B^1) \\
= &D^{(k)}(D(Z^k,B^k),D(Z^{k - 1},B^{k - 1}),\ldots,D(Z^1,B^1))  \\
= &D^{(k)}(\Ss_k^B(\mathbf Z),\Ss_{k - 1}^B(\mathbf Z),\ldots,\Ss_1^B(\mathbf Z)) = \D_k^{(n)}(\Ss^B(\mathbf Z)).
\end{align*}
Hence, $\T^B(\eta) = \D^{(n)}(\Ss^B(\mathbf Z))$. Since $\eta = \D^{(n)}(\mathbf Z)$, we have that
\[
\T^B(\D^{(n)}(\mathbf Z)) = \D^{(n)}(\Ss^B(\mathbf Z)).
\]
By Theorem \ref{multiline invariant distribution}, $\Ss^B(\mathbf Z) \overset{d}{=} \mathbf Z \sim \nu_\sigma^{\bar \dir}$. Therefore, $\T^B(\eta) \overset{d}{=} \D^{(n)}(\mathbf Z) \sim \mu_\sigma^{\bar \dir}$.
\end{proof}

\subsection{A triangular array representation of the map $\D^{(n)}$}
Here, we give a triangular array construction of the mapping $\D^{(n)}$ of \eqref{scrD}. This construction has previously appeared for a discrete prelimiting analogue of the map in \cite{Fan-Seppalainen-20} and \cite{Busani-2021}. For $(Z^1,\ldots,Z^n) \in \Y_n^\R$, define the following triangular array $\{\eta^{i,j},\zeta^{i,j}:1 \le j \le i \le n\}$ inductively as follows: set $\eta^{1,1} = \zeta^{1,1} = Z^1$, and assuming that $\eta^{i - 1,j}$ and $\zeta^{i - 1,j}$ have been defined for $j \in \{1,\ldots,i - 1\}$,
\be \label{eqn:triang_array}
\begin{aligned}
\eta^{i,1} &= Z^i \\
\eta^{i,j} &= D(\eta^{i,j -1},\zeta^{i - 1,j - 1}),\qquad j  \in \{2,\ldots,i\} \\
\zeta^{i,j - 1} &= R(\eta^{i,j -1},\zeta^{i - 1,j - 1}), \qquad j  \in \{2,\ldots,i\}  \\
\zeta^{i,i} &= \eta^{i,i}.
\end{aligned}
\ee
To show this construction is well-defined, we need to show that for $i \in \{1,\ldots,n\}$, and $j \in \{2,\ldots,i\}$,
\[
\lim_{x \to -\infty} [\eta^{i,j - 1}(x) - \zeta^{i - 1,j - 1}(x)] = -\infty.
\]
The following lemma ensures this is the case. 
\begin{lemma} \label{lem:triang_wd}
    Let $(Z^1,\ldots,Z^n) \in \Y_n^\R$, and defined the triangular array as in \eqref{eqn:triang_array}. Let $\dir_1 < \cdots < \dir_n$ denote the asymptotic drifts:
    $
    \dir_i = \lim_{x \to -\infty} \f{Z^i(x)}{x}.
    $
    Then, for each $1 \le j \le i \le n$,
    \[
    \lim_{x \to -\infty} \f{\eta^{i,j}(x)}{x} = \dir_i,\qquad\text{and}\lim_{x \to -\infty} \f{\zeta^{i,j}(x)}{x} = \dir_j.    
    \]
    If in addition, the following limits exist:
    $
        \lim_{x \to +\infty} \f{Z^i(x)}{x} = \wt \dir_i
    $
    with $\wt \dir^1 < \cdots <\wt \dir^n$, then for each $1 \le j \le i \le n$,
    \[
    \lim_{x \to +\infty} \f{\eta^{i,j}(x)}{x} = \wt \dir_i,\qquad\text{and}\lim_{x \to +\infty} \f{\zeta^{i,j}(x)}{x} = \wt \dir_j.    
    \]
\end{lemma}
\begin{proof}
    This follows by Lemma \ref{image of script D lemma} and induction.  
\end{proof}
We prove the following facts about this triangular array. Given the preceding lemmas about the queuing mappings (which require different proofs than in the discrete case), the proofs of Lemmas \ref{triang_iii} and \ref{lem:triang_dist} follow the same procedure as in \cite{Fan-Seppalainen-20}. We give full proofs here for completeness and to give clarity to which of the prior results are used. Lemma \ref{lem:backwards_map} does not have an analogue in \cite{Fan-Seppalainen-20}, although I expect one to hold in the discrete case. The basic building block of the proof is the bijection in Lemma \ref{DRbij}. The analogue of the mapping $(B,Z) \mapsto (R(Z,B),D(Z,B))$ in the discrete case is also a bijection on the appropriate space, as shown to me in unpublished notes of Timo Sepp\"al\"ainen. However, the proof of Lemma \ref{DRbij} requires different techniques than in the discrete case. 

\begin{lemma} \label{triang_iii}
    Let $(Z^1,\ldots,Z^n) \in \Y_n^\R$. Let $(\wt \eta^1,\ldots,\wt \eta^n) = \D^{(n)}(Z^1,\ldots,Z^n)$ as defined in \eqref{scrD}. Then, for $i = 1,\ldots,n$, $\wt \eta^i = \eta^{i,i}$, where $\eta^{i,i}$ is defined in \eqref{eqn:triang_array}. 
\end{lemma}
\newpage
\begin{proof}
We first recall that $(\eta^{1,1},\eta^{2,1},\ldots,\eta^{n,1}) = (Z^1,\ldots,Z^n)$ by definition. Thus, the $i = 1$ case is immediate. For $i = 2$, 
\[
\eta^{2,2} = D(\eta^{2,1},\zeta^{1,1}) = D(\eta^{2,1},\eta^{1,1}) = D(Z^2,Z^1) = \wt \eta^2.
\]
Now, assume that $i \ge 3$. For $m =1,2,3,\ldots,i - 1$, it follows from the iterative definition that
\be \label{eqn:etatoD}
\begin{aligned}
    &(\eta^{i,m},\eta^{i - 1,m},\ldots,\eta^{m + 1,m},\eta^{m,m})  \\
    &=\bigl(D(\eta^{i,m - 1},\zeta^{i - 1,m - 1}),D(\eta^{i - 1,m - 1},\zeta^{i - 2,m - 1}),\ldots,D(\eta^{m,m - 1},\zeta^{m - 1,m - 1})\bigr). 
\end{aligned}
\ee
By definition, we also have $\zeta^{m - 1,m - 1} = \eta^{m - 1,m - 1}$ and $\zeta^{j,m - 1} = R(\eta^{j,m - 1},\zeta^{j - 1,m - 1})$ for $j = m,\ldots,i - 1$. Substituting this into \eqref{eqn:etatoD}, when $m = 1,\ldots,i - 2$,
\begin{align}
&\quad \; D^{(i - m + 1)}(\eta^{i,m},\eta^{i - 1,m},\ldots,\eta^{m,m})\nonumber  \\
&=D^{(i - m + 1)}(D(\eta^{i,m - 1},\zeta^{i - 1,m - 1}),D(\eta^{i - 1,m - 1},\zeta^{i - 2,m - 1}),\ldots,D(\eta^{m,m - 1},\zeta^{m - 1,m - 1})\bigr) \nonumber \\
&\overset{\eqref{intertwining}}{=} D^{(i - m + 2)}(\eta^{i,m - 1},\eta^{i - 1,m - 1},\ldots,\eta^{m - 1,m - 1}).\label{Detalevel}
\end{align}
Then, using the definition of $D^{(i)}$ \eqref{Diter} and applying \eqref{Detalevel} inductively for $m = 1,\ldots,i - 1$,
\[
\begin{aligned}
\wt \eta^i &= D^{(i)}(I^1,\ldots,I^1) 
= D^{(i)}(\eta^{i,1},\eta^{i - 1,1}\ldots,\eta^{i,1}) 
= D^{(i - 1)}(\eta^{i,2},\eta^{i - 1,2},\ldots,\eta^{2,2}) \\
&= \cdots = D^{(2)}(\eta^{i,i - 1},\eta^{i - 1,i - 1}) = D(\eta^{i,i - 1},\zeta^{i - 1,i - 1}) = \eta^{i,i},
\end{aligned}
\]
where in the last two steps, we simply used the definitions \eqref{eqn:triang_array}. 
\end{proof}

\begin{lemma} \label{lem:backwards_map}
Recall the mapping $\Rf f(x) = f(-x)$. Let $(Z^1,\ldots,Z^n) \in \Y_n^\R$, and assume also that $(\Rf Z^n, \Rf Z^{n -1},\ldots,\Rf Z^1) \in \Y_n^\R$,. This means that for each $1 \le i \le n$, $\lim_{x \to +\infty} x^{-1} Z^i(x)$ exists, and
\[
\lim_{x \to +\infty}\f{Z^i(x)}{x} > \lim_{x \to +\infty} \f{Z^{i - 1}(x)}{x}.
\]
Consider the triangular array $\{\eta^{i,j},\zeta^{i,j}: 1\le j \le i \le n\}$ defined in \eqref{eqn:triang_array}.  For $i \in \{1,\ldots,n\}$ and $j \in \{1,\ldots,i\}$, define
\be \label{Yijdef}
Y^{i,j} = \Rf \zeta^{i,i - j + 1}.
\ee
Then, for $i \in \{1,\ldots,n\}$ and $j \in \{1,\ldots,i\}$,
\be \label{Drev}
D^{(j)}(Y^{i,j},\ldots,Y^{i,1}) = \Rf D^{(i - j + 1)}(Z^{i - j + 1},\ldots,Z^1).
\ee
\end{lemma}
\begin{remark}
Lemma \ref{lem:backwards_map} shows that we can reverse the order of the queuing mappings. In the $i = n$ case, $Y^{n,1}$ is the reflected version of the last component of $\D^{(n)}(Z^1,\ldots,Z^n)$ by Lemma \ref{triang_iii}. This now becomes the first component of the input into the mapping $\D^{(n)}$. This is used later, along with the distributional invariance in Lemma \ref{lem:triang_dist} to show reflection invariance of the measures $\mu_\sigma^{\bar \dir}$ (Lemma \ref{weak continuity and consistency}\ref{reflect})
\end{remark}
\begin{proof}The preservation of limits in Lemma \ref{lem:triang_wd} ensures that all these operations are well-defined. In this proof, we several times will use the $k = 1$ case of Lemma \ref{lem:Alternate Rep of Iterated Queues}, which for $r \ge 3$, we write as
\be \label{Dflipr}
D^{(r)}(Z^r,\ldots,Z^1) = D(D^{(r - 1)}(Z^r,\ldots,Z^3,R(Z^2,Z^1)),D(Z^2,Z^1)).
\ee
We prove this by induction. For all $i \in \{1,\ldots,n\}$, the $j = 1$ case follows from  we  the identity 
\[
D^{(1)}(Y^{i,1}) = Y^{i,1} = \Rf \zeta^{i,i} = \Rf D^{(i)}(Z^i,\ldots,Z^1),
\]
where the last line holds by Lemma \ref{triang_iii}, and the others are just the definitions. In particular, the $i = j = 1$ case of \eqref{Drev} holds. We turn to the $j = 2$ case. Recall the definitions \eqref{eqn:triang_array} and \eqref{Yijdef}:
\[
\begin{aligned}
Y^{i,1} &= \Rf \zeta^{i,i} = \Rf D(\eta^{i,i - 1},\zeta^{i - 1,i - 1}),\\
Y^{i,2} &= \Rf \zeta^{i,i - 1} = \Rf R(\eta^{i,i - 1},\zeta^{i - 1,i - 1}).
\end{aligned}
\]
Lemma \ref{DRbij} and the definition $\eta^{i - 1,i -1} = \zeta^{i - 1,i - 1}$ implies now that 
\be \label{DRind_step}
D(Y^{i,2},Y^{i,1}) = \Rf \zeta^{i - 1,i - 1} = \Rf \eta^{i - 1,i - 1},\qquad\text{and}\qquad R(Y^{i,2},Y^{i,1}) = \Rf \eta^{i,i - 1}. 
\ee
Since Lemma \ref{triang_iii} states that $\eta^{i - 1,i -1} = D^{(i - 1)}(Z^{i - 1},\ldots,Z^1)$, the first equality of \eqref{DRind_step} shows the $j = 2$ case of \eqref{Drev}.

Now, assume that for some $i \ge 2$ that \eqref{Drev} holds for $i - 1$ and all $j \in \{1,\ldots,i - 1\}$. We showed already that the  $j = 1,2$ case holds for general $i$, so let $j \ge 3$. We show that the $i,j$ case holds. By \eqref{Dflipr} and \eqref{DRind_step},
\be \label{Djk1}
\begin{aligned}
D^{(j)}(Y^{i,j},\ldots,Y^{i,1}) &= D(D^{(j - 1)}(Y^{i,j},\ldots,Y^{i,3},R(Y^{i,2},Y^{i,1})),D(Y^{i,2},Y^{i,1})) \\
&= D(D^{(j - 1)}(Y^{i,j},\ldots,Y^{i,3},\Rf \eta^{i,i - 1}),\Rf \zeta^{i - 1,i - 1}).
\end{aligned}
\ee
Now, we prove, in a second layer of induction that, for $m = 1,\ldots,j - 2$,
\be \label{Dkind}
\begin{aligned}
&\quad \; D^{(j)}(Y^{i,j},\ldots,Y^{i,1}) \\
&= D(D(\cdots (D(D^{(j - m)}(Y^{i,j},\ldots,Y^{i,m + 2},\Rf \eta^{i,i - m}),\Rf \zeta^{i - 1,i - m}),\Rf \zeta^{i - 1,i - m  +1}),\cdots),\Rf \zeta^{i - 1,i - 1}).
\end{aligned}
\ee
Equation \eqref{Djk1} proves the $m = 1$ case of \eqref{Dkind}. Now, assume that \eqref{Dkind} holds for some $m \in 1,\ldots,j - 3$. Then, using \eqref{Dflipr}, 
\be \label{Dmstep1}
\begin{aligned}
&\quad \; D^{(j - m)}(Y^{i,j},\ldots,Y^{i,m + 2},\Rf \eta^{i,i - m}) \\
&= D(D^{(j - (m + 1))}(Y^{i,j},\ldots,Y^{i,m + 3},R(Y^{i,m + 2},\Rf \eta^{i,i - m})  ),D(Y^{i,m + 2},\Rf \eta^{i,i - m})).
\end{aligned}
\ee
Again, recall the definitions \eqref{eqn:triang_array} and \eqref{Yijdef}:
\[
\begin{aligned}
    \Rf \eta^{i,i - m} &= \Rf D(\eta^{i,i - m - 1},\zeta^{i - 1,i - m - 1}),\\
    Y^{i,m + 2} &= \Rf \zeta^{i,i - m - 1} = \Rf R(\eta^{i,i - m - 1},\zeta^{i - 1,i - m - 1}).
\end{aligned}
\]
Then, Lemma \ref{DRbij} gives us
\be \label{DRzeta2}
D(Y^{i,m + 2},\Rf \eta^{i,i - m}) = \Rf \zeta^{i - 1,i - m - 1},\qquad\text{and}\qquad R(Y^{i,m + 2},\Rf \eta^{i,i - m}) = \Rf \eta^{i,i -m - 1}.
\ee
Substituting \eqref{DRzeta2} into \eqref{Dmstep1} and then \eqref{Dmstep1} into \eqref{Dkind} completes the second layer of induction. We now return to the original layer of induction. We consider \eqref{Dkind} in the extremal case $m = j - 2$. Recalling $D^{(2)} = D$ \eqref{Diter}, equation \eqref{DRzeta2} tells us that 
\[
D^{(2)}(Y^{i,j},\Rf  \eta^{i,i - (j - 2)}) = \Rf \zeta^{i - 1,i - (j - 1)}
\]
Substituting this into the $m = j - 2$ case of \eqref{Djk1}, and using the iterative definition of $D^{(j)}$ \eqref{Diter}, we obtain
\begin{align*}
    &\quad \;D^{(j)}(Y^{i,j},\ldots,Y^{i,1}) \\
        &= D(D(\cdots (D(\Rf \zeta^{i - 1,i - (j - 1)},\Rf \zeta^{i - 1,i - (j - 2)}),\zeta^{i - 1,i - (j - 3)}),\cdots),\Rf \zeta^{i - 1,i - 1}) \\
            &= D^{(j - 1)}( \Rf \zeta^{i - 1,i - (j - 1)},\Rf \zeta^{i - 1,i - (j - 2)},\ldots,\Rf \zeta^{i - 1,i - 1}) \\
            &= D^{(j - 1)}(Y^{i - 1,j - 1},\ldots,Y^{i - 1,1}).
\end{align*}
By the induction assumption that \eqref{Drev} holds for $i -1$ and $j - 1$, this equals
\[
\Rf D^{(i - 1 - (j - 1) + 1)}(Z^{i - 1 - (j - 1) + 1},\ldots,Z^1) = \Rf D^{(i - j + 1)}(Z^{i - j + 1},\ldots,Z^1),
\]
and this gives exactly \eqref{Drev}, completing the induction. 
\end{proof}

\begin{lemma} \label{lem:triang_dist}
For $\sigma > 0$ and $\dir_1 < \cdots < \dir_n$, assume that $(Z^1,\ldots,Z^n)$ has distribution $\nu_\sigma^{(\dir_1,\ldots,\dir_n)}$. Let the triangular array $\{\eta^{i,j},\zeta^{i,j}:1 \le j \le i \le n\}$ be defined as in \eqref{eqn:triang_array}. Then, for each $1 \le i,j \le n$, the random vector $(\eta^{j,j},\eta^{j+1,j},\ldots,\eta^{n,j})$ has distribution $\nu_\sigma^{(\dir_j,\ldots,\dir_n)}$, and the random vector $(\zeta^{i,1},\ldots,\zeta^{i,i})$ has distribution $\nu_\sigma^{(\dir_1,\ldots,\dir_i)}$
\end{lemma}
\begin{proof}
By definition, $(\eta^{1,1},\ldots,\eta^{n,1}) = (Z^1,\ldots,Z^n)$, which has distribution$\sim \nu_\sigma^{(\dir_1,\ldots,\dir_n)}$ by assumption. Recall that this means that the $Z^1,\ldots,Z^n$ are mutually independent two-sided Brownian motions, each with diffusivity $\sigma$, and with drifts $\sigma^2 \dir_1,\ldots,\sigma^2 \dir_n$. We assume, by way of induction, that, for some $j \in \{2,\ldots,n - 1\}$,
\[
(\eta^{j - 1,j - 1},\eta^{j,j - 1},\ldots,\eta^{n,j - 1}) \sim \nu_\sigma^{(\dir_{j - 1},\ldots,\dir_n)}.
\]
By definition \eqref{eqn:triang_array}, for $m \in \{0,\ldots,n - j\}$,
\[
\eta^{j + m,j} = D(\eta^{j + m,j - 1},\zeta^{j + m - 1,j - 1}),\qquad\text{and}\qquad \zeta^{j + m,j - 1} = R(\eta^{j + m,j - 1},\zeta^{j + m - 1,j - 1}).
\]
From this, we see that we obtain the random vector $(\eta^{j,j},\ldots,\eta^{n,j})$ from $(\eta^{j,j - 1},\ldots,\eta^{n,j - 1})$ via a single iteration of the multiline process \eqref{multiline process}, where $\zeta^{j - 1,j - 1} = \eta^{j - 1,j - 1}$ is the driving Brownian motion. Theorem \ref{multiline invariant distribution} now implies that $(\eta^{j,j},\ldots,\eta^{n,j}) \sim \nu_\sigma^{(\dir_j,\ldots,\dir_n)}$.

Next, we prove $(\zeta^{i,1},\ldots,\zeta^{i,i})\sim\nu_\sigma^{(\dir_1,\ldots,\dir_i)}$. To start the induction, note that $\zeta^{1,1} = Z^1 \sim \nu_\sigma^{(\dir_1)}$. Now, assume that, for some $i \in \{2,\ldots,n - 1\}$, 
\[
(\zeta^{i - 1,1},\zeta^{i - 1,2},\ldots,\zeta^{i - 1,i - 1}) \sim \nu_\sigma^{(\dir_1,\ldots,\dir_i)}.
\] 
We now show via a second layer of induction that, for $m \in \{1,\ldots,i\}$, 
\be \label{zeind2}
(\zeta^{i,1},\ldots,\zeta^{i,m - 1},\zeta^{i - 1,m},\ldots,\zeta^{i - 1,i - 1},\eta^{i,m}) \sim \nu_\sigma^{(\dir_1,\ldots,\dir_{m - 1},\dir_m,\ldots,\dir_{i - 1},\dir_i)}.
\ee
In the $m = i$ case of \eqref{zeind2}, since $\eta^{i,i} = \zeta^{i,i}$, this will complete the first layer of induction. 

By construction of the triangular array \eqref{eqn:triang_array}, $(\zeta^{i - 1,1},\zeta^{i - 1,2},\ldots,\zeta^{i - 1,i - 1})$ is a function of $(Z^1,\ldots,Z^{i - 1})$ and is therefore independent of $Z^i = \eta^{i,1}\sim \nu_\sigma^{(\dir_i)}$. This gives the $m = 1$ case of \eqref{zeind2}.  Now, assume that \eqref{zeind2} holds for some $m \in \{1,\ldots,i - 1\}$. By definition, 
\[
\eta^{i,m + 1} = D(\eta^{i,m},\zeta^{i - 1,m}),\qquad\text{and}\qquad \zeta^{i,m} = R(\eta^{i,m},\zeta^{i - 1,m}).
\]

By Theorem \ref{thm:OCY_DR},  $\zeta^{i,m}$ and $\eta^{i,m + 1}$ are independent two-sided Brownian motions with diffusivity $\sigma$ and drifts $\sigma^2 \dir_m$ and $\sigma^2\dir_i$, respectively. Substituting the mapping of $(\zeta^{i - 1,m},\eta^{i,m}) \mapsto (\zeta^{i,m},\eta^{i,m + 1})$ into \eqref{zeind2} completes the inductive step. 
\end{proof}

\subsection{Distributional invariances of the queuing measures}

\begin{lemma} \label{weak continuity and consistency}
Let $\sigma > 0$ and $\bar \dir = (\dir_1 < \cdots < \dir_n)$. Let $(\eta^1,\ldots,\eta^n) \sim \mu_\sigma^{\bar \dir}$. The following hold.
\begin{enumerate} [label=\rm(\roman{*}), ref=\rm(\roman{*})]  \itemsep=3pt 
\item \label{Shift invariance} \rm{(}Shift invariance{\rm)} For $x \in \R$,
\[
\{(\eta^1(x,x + y),\ldots,\eta^n(x,x + y)): y \in \R\} \sim  \mu_\sigma^{\bar \dir}.
\]
\item \label{scaling relations}{\rm(}Scaling relations{\rm)} Let  and $b,c > 0$, and $\nu \in \R$. Then,
\[
\{(b \eta^1(c^2 y) - (bc\sigma)^2 \nu y,\ldots,b \eta^n(c^2 y) - (bc\sigma)^2 \nu y):y \in \R\}\sim \mu_{b c \sigma}^{(\dir_1/b - \nu,\ldots,\dir_n/b - \nu)}. 
\]
    \item \label{consistency}{\rm(}Consistency{\rm)} Any subsequence $(\eta^{j_1},\ldots,\eta^{j_k})$ has distribution $\mu_\sigma^{(\dir_{j_1},\ldots,\dir_{j_k})}$. 
    \item \label{reflect} {\rm(}Reflection invariance{\rm)} $(\Rf \eta^n,\ldots, \Rf \eta^1)\sim  \mu_{\sigma}^{(-\dir_n,\ldots,-\dir_1)}$. 
\item \label{weak continuity} {\rm(}Weak continuity{\rm)}
    Let $\bar \dir^k = (\dir_1^k,\dir_2^k,\ldots,\dir_n^k) \in \R^n$ be a sequence converging to $\bar \dir$ and let $\sigma^k$ be a sequence converging to $\sigma > 0$.  Then, $\mu_{\sigma_k}^{\bar \dir^k} \rightarrow \mu_\sigma^{\bar \dir}$ weakly, as probability measures on $\X_n^\R$. 
\end{enumerate}
\end{lemma}
\begin{proof}[Proof of Lemma \ref{weak continuity and consistency}]
For all items of this proof, we set 
\[\eta = (\eta^1,\ldots,\eta^n) = \D^{(n)}(Z^1,\ldots,Z^n),\] where $Z^i$ are independent Brownian motions with diffusivity $\sigma$ and drifts $\sigma^2 \dir_i$. Then, \\$\eta \sim \mu_\sigma^{\bar \dir}$ by definition \eqref{musigma}. Recall the $j$th component of $\D^{(n)}(Z^1,\ldots,Z^n)$ \\is $D^{(j)}(Z^j,\ldots,Z^1)$ \eqref{scrD}. 

We also make heavy use of Lemma \ref{identity for multiple queueing mappings}, which states
\be \label{Djrep}
\begin{aligned}
D^{(j)}(Z^j,\ldots,Z^1)(y) &= Z^1(y) + \sup_{-\infty < x_{j - 1} \leq x_{n - 2} \le \cdots \leq x_{1} \leq y} \Big\{\sum_{i = 1}^{j - 1} Z^{i + 1}(x_{i}) - Z^{i}(x_{i})   \Big\}  \\
&\qquad\qquad - \sup_{-\infty < x_{j - 1} \leq x_{j - 2} \le \cdots \leq x_{1} \leq 0} \Big\{\sum_{i = 1}^{j - 1} Z^{i + 1}(x_{i}) - Z^{i}(x_{i})   \Big\}
\end{aligned}
\ee

\noindent \textbf{Item \ref{Shift invariance}:}
From \eqref{Djrep}, observe that 
\begin{align*}
&\quad \;D^{(j)}(Z^j,\ldots,Z^1)(x,x + y)\\ &= Z^1(x,x + y) + \sup_{-\infty < x_{j - 1} \leq x_{j - 2} \le \cdots \leq x_{1} \leq x + y} \Big\{\sum_{i = 1}^{j - 1} Z^{i + 1}(x_{i}) - Z^{i}(x_{i})   \Big\} \\
&\qquad\qquad - \sup_{-\infty < x_{j - 1} \leq x_{j - 2} \le \cdots \leq x_{1} \leq x} \Big\{\sum_{i = 1}^{j - 1} Z^{i + 1}(x_{i}) - Z^{i}(x_{i})   \Big\} \\
&=Z^1(x,x + y) +  \sup_{-\infty < x_{j - 1} \leq x_{j - 2} \le \cdots \leq x_{1} \leq y} \Big\{\sum_{i = 1}^{j - 1} Z^{i + 1}(x_{i} + x) - Z^{i}(x_{i} + x)   \Big\} \\
&\qquad\qquad - \sup_{-\infty < x_{j - 1} \leq x_{j - 2} \le \cdots \leq x_{1} \leq 0} \Big\{\sum_{i = 1}^{j - 1} Z^{i + 1}(x_{i} + x) - Z^{i}(x_{i} +x)   \Big\} \\
&= Z^1(x,x + y) +  \sup_{-\infty < x_{j - 1} \leq x_{j - 2} \le \cdots \leq x_{1} \leq y} \Big\{\sum_{i = 1}^{j - 1} Z^{i + 1}(x,x_{i} + x) - Z^{i}(x,x_{i} + x)   \Big\} \\
&\qquad\qquad - \sup_{-\infty < x_{j - 1} \leq x_{j - 2} \le \cdots \leq x_{1} \leq 0} \Big\{\sum_{i = 1}^{j - 1} Z^{i + 1}(x,x_{i} + x) - Z^{i}(x,x_{i} +x)   \Big\} \\
&= D^{(j)}(Z^j(x,x + \aabullet),\ldots,Z^1(x,x + \aabullet))(y)
\end{align*}
where the penultimate step follows by adding and subtracting $\sum_{i = 1}^{j - 1} [Z^i(x) - Z^{i +1}(x)]$. The desired result follows because the law of $(Z^1,\ldots,Z^n)$ is preserved under shifts.

\noindent \textbf{Item \ref{scaling relations}}
From \eqref{Djrep}, 
\begin{align*}
    &\quad \; b D^{(j)}(Z^j,\ldots,Z^1)(c^2 y) - (bc\sigma)^2 \nu y \\
    &= b Z^1(c^2 y) - (bc\sigma)^2 \nu y  + \sup_{-\infty < x_{j - 1} \leq x_{j - 2} \le \cdots \leq x_{1} \leq c^2 y} \Big\{\sum_{i = 1}^{j - 1} bZ^{i + 1}(x_{i}) - b Z^{i}(x_{i})   \Big\}  \\
&\qquad\qquad - \sup_{-\infty < x_{j - 1} \leq x_{j - 2} \le \cdots \leq x_{1} \leq 0} \Big\{\sum_{i = 1}^{j - 1} bZ^{i + 1}(x_{i}) - bZ^{i}(x_{i})   \Big\}  \\
&=  b Z^1(c^2 y) - (bc\sigma)^2 \nu y +  \sup_{-\infty < x_{j - 1} \leq x_{j - 2} \le \cdots \leq x_{1} \leq y} \Big\{\sum_{i = 1}^{j - 1} bZ^{i + 1}(c^2 x_{i}) - b Z^{i}(c^2 x_{i})   \Big\} \\
&\qquad\qquad - \sup_{-\infty < x_{j - 1} \leq x_{j - 2} \le \cdots \leq x_{1} \leq 0} \Big\{\sum_{i = 1}^{j - 1} bZ^{i + 1}(c^2 x_{i}) - b Z^{i}(c^2 x_{i})   \Big\} \\
&= b Z^1(c^2 y) - (bc\sigma)^2 \nu y \\
&+  \sup_{-\infty < x_{j - 1} \leq x_{j - 2} \le \cdots \leq x_{1} \leq y} \Big\{\sum_{i = 1}^{j - 1} bZ^{i + 1}(c^2 x_{i}) - (bc\sigma)^2 \nu x_i - (b Z^{i}(c^2 x_{i}) - (bc\sigma)^2\nu  x_i)   \Big\} \\
&-\sup_{-\infty < x_{j - 1} \leq x_{j - 2} \le \cdots \leq x_{1} \leq 0} \Big\{\sum_{i = 1}^{j - 1} bZ^{i + 1}(c^2 x_{i}) - (bc\sigma)^2 \nu x_i - (b Z^{i}(c^2 x_{i}) - (bc\sigma)^2\nu x_i)   \Big\}\\
&= D^{(j)}(\wt Z^j,\ldots,\wt Z^1),
\end{align*}
where $\wt Z^i(x) = bZ^i(c^2 x) - (bc\sigma)^2\nu x$. Because the $\wt Z^i$ are independent two-sided Brownian motions with diffusivity $bc\sigma$ and drift $bc^2 \sigma^2 \dir_i - (bc\sigma)^2 \nu = (bc\sigma)^2(\dir_i/b - \nu)$, we have
\[
(\wt Z^1,\ldots,\wt Z^j) \sim \mu_{b c\sigma}^{\dir_1/b - \nu,\ldots, \dir_n/b - \nu}.
\]

\noindent \textbf{Item \ref{consistency}:} It suffices to show that, for $1 \le i \le n$,
\[(\eta^1,\ldots,\eta^{i - 1},\eta^{i + 1},\ldots, \eta^n) \sim \mu_\sigma^{\dir_1,\ldots,\dir_{i - 1},\dir_{i + 1},\ldots,\dir_n}.\]  

\noindent For $i = n$, the statement is immediate from the definition of the map $\D^{(n)}$. Next, we show the case $i = 1$. For $2 \le j \le n$, we use \eqref{intertwining} to write
\[
D^{(j)}(Z^j,\ldots,Z^1) = D^{(j - 1)}(D(Z^j,\wt Z^{j - 1}),\ldots, D(Z^3,\wt Z^{2}),D(Z^2,\wt Z^{1})),
\]
where $\wt Z^{1} = Z^1$, and for $i > 1$, $\wt Z^{i} = R(Z^{i},\wt Z^{i - 1})$. Then, $(\eta^2,\ldots,\eta^n) = \D^{(n - 1)}(\hat Z^2,\ldots,\hat Z^n)$, where $\hat Z^i = D(Z^i,\wt Z^{i - 1})$ for $2 \le i \le n$. By Theorem \ref{multiline invariant distribution}, $\hat Z^2,\ldots,\hat Z^n$ are independent Brownian motions, each with diffusivity $\sigma$ and with drifts $\sigma^2 \dir_2,\ldots,\sigma^2 \dir_n$, so
this completes the proof of the $i = 1$ case. By definition of $D^{(j)}$ \eqref{Diter}, for $i < j \le n$,
\begin{equation} \label{queue_iter}
D^{(j)}(Z^j,\ldots,Z^1) = D(D(\cdots D(D^{(j - i + 1)}(Z^j,\ldots,Z^{i}),Z^{i-1}),\ldots,Z^2),Z^1).
\end{equation}
 Similarly as in the $i = 1$ case, we apply \eqref{intertwining} to get that
\begin{align} \label{queue_iter_2}
&\;\;\;D^{(j - i + 1)}(Z^j,\ldots,Z^i)\nonumber \\
&= D^{(j - i)}(D(Z^j,\wt Z^{j - 1}),\ldots,D(Z^{i + 1},\wt Z^{i})) = D^{(j - i)}(\hat Z^j,\ldots,\hat Z^{i + 1}),
\end{align}
where, $\wt Z^{i} = Z^i$, and for $j > i$, $\wt Z^{j} =R(Z^j,\wt Z^{j- 1})$ and $\hat Z^j = D(Z^j,\wt Z^{j - 1})$.
Then, by \eqref{queue_iter} and \eqref{queue_iter_2},  for $i < j \le n$,
\[D^{(j)}(Z^j,\ldots,Z^1) = D^{(j - 1)}(\hat Z^j,\ldots,\hat Z^{i + 1}, Z^{i - 1},\ldots,Z^1),
\]
 and so 
\be \label{eqn:marg}
(\eta^1,\ldots,\eta^{i - 1},\eta^{i + 1},\ldots,\eta^n) = \D^{(n - 1)}(Z^1,\ldots,Z^{i - 1},\hat Z^{i + 1},\ldots,\hat Z^n).
\ee 
By Theorem \ref{multiline invariant distribution}, $\hat Z^{i + 1},\ldots,\hat Z^{n}$ are independent Brownian motions with diffusivity $\sigma$ and drifts $\sigma^2\dir_{i + 1},\ldots,\sigma^2\dir_n$. Since these are functions of $Z^{i},\ldots,Z^n$, the functions $Z^1,\ldots,Z^{i - 1},\hat Z^{i + 1},\ldots,\hat Z^j$ are independent as well, and by \eqref{eqn:marg}, \[
(\eta^1,\ldots,\eta^{i - 1},\eta^{i + 1},\ldots, \eta^n) \sim \mu_\sigma^{(\dir_1,\ldots,\dir_{i - 1},\dir_{i + 1},\ldots,\dir_n)}. 
\]

 \noindent \textbf{Item \ref{reflect}:} For $(Z^1,\ldots,Z^n)$ given, let the triangular array $\{\eta^{i,j},\zeta^{i,j}: 1 \le j \le i \le n\}$ be defined by \eqref{eqn:triang_array}. For $1 \le j \le n$, set $Y^j = \Rf \zeta^{n - j + 1}$. Then, by Lemma \ref{lem:triang_dist},
\[
(Y^1,\ldots,Y^n) \sim \nu_\sigma^{(-\dir_n,\ldots,-\dir_1)}.
\]
Note that the limits $\lim_{x \to \pm \infty} x^{-1} Z^i(x)$ exist and equal $\sigma^2 \dir_i$ almost surely. Then, we may apply the $i = n$ case of Lemma \ref{lem:backwards_map}, which states, for $1 \le j \le n$,
\[
D^{(j)}(Z^j,\ldots,Z^1) = \Rf \eta^{(n - j + 1)}.
\]
Hence, 
\[
(\Rf \eta^n,\ldots,\Rf \eta^1)  = \D^{(n)}(Y^1,\ldots,Y^n) \sim \mu_{\sigma}^{(-\dir_n,\ldots,-\dir_1)}.
\]

 \noindent \textbf{Item \ref{weak continuity}}
We show the existence of $\eta_k \sim \mu_{\sigma_k}^{\bar \dir^k}$ such that, almost surely, for $1 \le i \le n$,  $\eta_k^i \rightarrow \eta^i$, uniformly on compact sets. For the sequence $\mbf Z$ defined at the beginning of the proof, define $Z_k^i(x) = \f{\sigma_k}{\sigma} Z^i(x) + (\sigma_k^2 \dir_i^k  - \sigma \sigma_k\dir_i) x$. Then, $(Z_k^1,Z_k^2,\ldots,Z_k^n) \sim \nu_{\sigma_k}^{\dir^k}$. Set $\eta = \D^{(n)}(Z) \sim \mu_\sigma^{\bar \dir}$ and $\eta_k = \D^{(n)}(Z_k)$. By construction, $Z_k \rightarrow Z^i$, uniformly on compact sets. thus, the convergence $\eta_k^1 = Z_k^1 \rightarrow Z^1 = \eta^1$ is immediate.  Further, since $Z^i$ satisfies the almost sure asymptotic 
\[
\lim_{x \to -\infty} \f{Z^i(x)}{x} = \sigma^2 \dir_i,
\]
for $1 \le i \le n$, $Z_k^i$ satisfies
\[
\limsup_{\substack{x \rightarrow -\infty \\ k \rightarrow \infty} } \Bigl|\f{1}{x} Z_k^i(x) - \sigma^2 \dir_i  \Bigr| = 0.
\]

By Lemma \ref{uniform convergence of queueing mapping}, $\eta_k^2 = D(Z_k^2,Z_k^1)$ converges to $\eta^2 = D(Z^2,Z^1)$ uniformly on compact sets, and
\[
 \limsup_{\substack{z \rightarrow -\infty \\ k \rightarrow \infty} } \Bigl|\f{1}{x} \eta_k^2(x) - \sigma^2\dir_2  \Bigr| = 0.
 \]
Now, assume by induction that for $i \ge 2$, $\eta_k^i = D^{(i)}(Z_k^i,\ldots,Z_k^1)$ converges uniformly on compact sets to $\eta^i$ and that 
\[
\limsup_{\substack{x \rightarrow -\infty \\ k \rightarrow \infty} } \Bigl|\f{1}{x} \eta_k^i(x) - \sigma^2 \dir_i  \Bigr| = 0.
\]
Then, by shifting indices and setting $\wt \eta_k^i = D^{(i)}(Z_k^{i + 1},\ldots,Z_k^2)$ and $\wt \eta^i = D^{(i)}(Z^{i + 1},\ldots,Z^2)$, it also holds by induction that $\wt \eta_k^{i}$ converges uniformly on compact sets to $\wt \eta^{i}$, and
\[
\limsup_{\substack{x \rightarrow -\infty \\ k \rightarrow \infty} } \Bigl|\f{1}{x} \wt \eta_k^i(x) - \sigma^2\dir_i \Bigr| = 0.
\]
By definition of $D^{(i + 1)}$ \eqref{Diter} and the $i = 2$ case,
\[
\eta_k^{i + 1} = D^{(i + 1)}(Z_k^{i + 1},\ldots,Z_k^1) = D(\wt \eta_k^i,Z_k^1) \to D(\wt \eta^i,Z^1) = D^{(i + 1)}(Z^{i + 1},\ldots,Z^1)= \eta^{i + 1},
\]
where the convergence is almost sure, uniformly on compact sets. The $i = 2$ case also guarantees
\[
\limsup_{\substack{x \rightarrow -\infty \\ k \rightarrow \infty} } \Bigl|\f{1}{x} \eta_k^{i + 1}(x) - \sigma^2 \dir_{i+1}  \Bigr| = 0,
\]
thereby completing the inductive step.  
\end{proof}

\subsection{Proof of Proposition \ref{prop:SH_cons} and Theorem \ref{thm:SH_dist_invar}} \label{sec:SH_cons_proof}
 
\begin{proof}{Proof of Proposition \ref{prop:SH_cons}}
We remind the reader that this proposition shows the existence of the process $G^\sigma$, which we call the stationary horizon (SH), with a fixed diffusivity parameter $\sigma > 0$. 

Lemma \ref{weak continuity and consistency}\ref{consistency} establishes that the measures $\mu^{\bar \dir}_\sigma$ are consistent. Hence, for $\bar \dir = (\dir_1 < \cdots < \dir_n)$, if $(\eta^1,\ldots,\eta^n) \sim \mu_\sigma^{\bar \dir}$, each $\eta^i$ has distribution $\mu_\sigma^{\dir_i}$. By definition \eqref{musigma} (see also \eqref{Diter}), this is the distribution of $D^{(1)}(Z) = Z$ which is a two-sided Brownian motion with diffusion coefficient $\sigma$ and drift $\sigma^2\dir_i$. In particular, $\eta^i(0) = 0$, almost surely. Kolmogorov's extension theorem implies the existence of a unique measure $\mu^\Q_\sigma$ on $C(\R)^\Q = \prod_{\Q}C(\R)$ under which, for $\{\wt G_\alpha\}_{\alpha \in \Q} \in C(\R)^\Q$ 
and any increasing finite vector $\bar \alpha = (\alpha_1,\dotsc,\alpha_n) \in \Q^n$, $(\wt G_\alpha,\ldots,\wt G_{\alpha_n}) \sim \mu^{\bar \alpha}_\sigma$. In particular, under $\mu_\sigma^\Q$, each $\wt G^\alpha$ is a Brownian motion with diffusivity $\sigma$ and drift $\sigma^2\alpha$. Furthermore, the measures $\mu_\sigma^{\bar \dir}$ are supported on the set $\X_n^\R$ \eqref{Xndef}, so 
\be \label{mubetamont}
\mu^\Q_\sigma(\wt G_{\alpha_1} \li \wt G_{\alpha_2} \;\; \forall \alpha_1 < \alpha_2 \in \Q) = 1.
\ee
Hence, under $\mu^\Q_\sigma$, on a single event of probability one, for each $\dir\in \R$ and $x \in \R$, the limits
\be \label{eqn:SHdef}
G^\sigma_\dir(x) := \lim_{\Q \ni \alpha \searrow \dir}\wt G^\sigma_\alpha(x) \qquad\text{and}\qquad  G^\sigma_{\dir -}  := \lim_{\Q \ni \alpha \nearrow \dir} \wt G_\alpha^\sigma(x)
\ee
exist. Furthermore, 
\be \label{eqwtGmont}
\mu^\Q_\sigma(G_{\dir} \li \wt G_\alpha \;\; \forall \dir \in \R, \alpha \in \Q \text{ with }\alpha > \dir ) = 1.
\ee
Then, on the event of \eqref{eqwtGmont}, Lemma \ref{lem:ext_mont} implies that, for $A < a < b < B$,  
\[
0 \le \wt G_\alpha(a,b) -  G_\dir(a,b) \le  \wt G_\alpha(A,B) - G_\dir(A,B), 
\]
so the convergence is uniform on compact sets. A symmetric argument holds for limits from the left. Monotonicity implies that $\mu^\Q_\sigma( G_{\dir -} \li \wt G_\dir \li G_\dir \;\;\forall \dir \in \Q) = 1$. Furthermore, by uniform convergence, for each $\dir \in \R$, $G_{\dir-}$ and $G_\dir$ are both Brownian motions with drift $\dir$. Hence,
\[
\mu^\Q_\sigma(G_{\dir -}(x) = \wt G_{\dir}(x) = G_\dir(x) \;\;\forall x \in \Q, \dir \in \Q) = 1,
\]
but by continuity, we also have
\[
\mu^\Q_\sigma(G_{\dir -}(x) = \wt G_{\dir}(x) = G_\dir(x) \;\;\forall x \in \R, \dir \in \Q) = 1,
\]

We have now defined a stochastic process $\{ G_\dir\}_{\dir \in \R}$ whose projection to the rationals agrees with the process $\{\wt G_\dir\}_{\dir \in \R}$ originally constructed under the measure $\mu_\sigma^\Q$. Then, $\mu_\sigma^\Q$- almost surely, for $\dir < \dir_2 < \dir_3$, and rational values $\alpha_1 < \dir_1 < \alpha_2 < \dir_2 < \dir_3 < \alpha_4$. 
\be \label{eqn:SHbig_mont}
G_{\alpha_1} \li  G_{\dir_1} \li  G_{\alpha_2} \li  G_{\dir_2} \li  G_{\alpha_3} \li  G_{\dir_3} \li  G_{\alpha_4}.
\ee
This implies that, simultaneously for every $\dir \in \R$, the following limits exist, in the sense of uniform convergence on compact sets, and they agree with the limits along rational directions.
\[
G^\sigma_{\dir} = \lim_{\alpha \searrow \dir} G^\sigma_\alpha \qquad\text{and}\qquad G^\sigma_{\dir -} = \lim_{ \alpha \nearrow \dir} G^\sigma_\alpha.
\]

Hence, the random process $\{G_\dir\}_{\dir \in \R}$ almost surely lies in the space $D(\R,C(\R))$. Let $\Pp^\sigma$ be the pushforward of the measure $\mu_\sigma^\Q$ under the map to $D(\R,C(\R))$ defined by \eqref{eqn:SHdef}. Without reference to the measure, we use $\{G^\sigma_\dir\}_{\dir \in \R}$ to denote the process. 

We now verify that this process satisfies the items of the theorem. Item \ref{itm:SHBM} follows because of the corresponding property for rational parameters and the uniform convergence. Item \ref{itm:SH_dist} follows because the finite-dimensional measures for rational directions was defined to be $\mu_\sigma^{\bar \dir}$. The convergence in \eqref{eqn:SHdef} and weak convergence (Lemma \ref{weak continuity and consistency}\ref{weak continuity}) extend this property to all $n$-tuples of real directions. The uniqueness then follows because the $\sigma$-algebra on $D(\R,C(\R))$ is generated by the projection maps.  The monotonicity of Item \ref{itm:SH_mont} follows from \eqref{eqn:SHbig_mont}. For Item \ref{itm:SH_cont}, for any $\dir \in \R$, both $G_{\dir-}$ and $G_{\dir +}$ are Brownian motions with diffusivity $\sigma$ and drift $\sigma^2 \dir$, and they are ordered by Item \ref{itm:SH_mont}. Thus, with $\Pp^\sigma$-probability one, $G_{\dir-}(x) = G_{\dir}(x)$ for all $x \in \Q$. But both functions are continuous, so the equality extends to all $x \in \R$. 

We finish by proving Item \ref{itm:SH_drifts}. Since each $G_\dir$ is a Brownian motion with drift $\sigma^2 \dir$ under $\Pp^\sigma$,
\[
\Pp^\sigma\Bigl(\lim_{|x| \to \infty}\f{G_\dir(x)}{x} = \sigma^2 \dir \;\;\forall \dir \in \Q\Bigr) = 1.
\]
Then, by the monotonicity of Item \ref{itm:SH_mont}, and since $G_\dir(0,x) = G_\dir(x)$, for $x > 0$ and $\ve > 0$, 
\[
\begin{aligned}
\sigma^2(\dir - \ve) = \lim_{x \to +\infty} \f{G_{\dir - \ve}(x)}{x} &\le \liminf_{x \to +\infty} \f{G_{\dir-}(x)}{x}\\
&\le \limsup_{x \to +\infty} \f{G_{\dir +}(x)}{x} \le \lim_{x \to +\infty} \f{G_{\dir + \ve}(x)}{x} = \sigma^2(\dir + \ve).
\end{aligned}
\]
Sending $\ve \searrow 0$ completes the proof. The limits as $x \to -\infty$ are proved similarly.  
\end{proof}

\begin{proof}[Proof of Theorem \ref{thm:SH_dist_invar}]
Here, we verify the distributional invariances of the SH. Items \ref{itm:shinv},\ref{itm:SHscale}, and \ref{itm:SH_reflinv}  follow because $(G^\sigma_{\dir_1},\ldots,G^\sigma_{\dir_n}) \sim \mu_\sigma^{\bar \dir}$ (Proposition \ref{prop:SH_cons}\ref{itm:SH_dist}) and from the corresponding invariances of the measures $\mu_\sigma^{\bar \dir}$ (Lemma \ref{weak continuity and consistency}\ref{Shift invariance},\ref{scaling relations}, and \ref{reflect}). The increment-stationarity of Item \ref{itm:SH_inc_stat} follows now from Item \ref{itm:SHscale}: Setting $c = b = 1$,
we obtain 
  \[
(G^\sigma_{\dir_1},\ldots,G^\sigma_{\dir_n}) \deq (G^\sigma_{\dir_1 + \dir^\star},\ldots,G^\sigma_{\dir_n+ \dir^\star}),
  \]
  and therefore,
\[
(G^\sigma_{\dir_2} - G^\sigma_{\dir_1},\ldots,G^\sigma_{\dir_n} - G^\sigma_{\dir_{n - 1}}) \deq (G^\sigma_{\dir_2 + \dir^\star} - G^\sigma_{\dir_1 + \dir^\star},\ldots,G^\sigma_{\dir_n + \dir^\star} - G^\sigma_{\dir_{n - 1} + \dir^\star}). \qedhere
\]
\end{proof}

\section{Distributional calculations for the SH} \label{sec:dist_calc_thm}
We now state some explicit distributional calculations for the SH to be proved in Section \ref{sec:dist_calc}.  
\begin{theorem} \label{thm:SH_inc_dist}
Let $a,\dir_0,\sigma_0 \in \R$, and $z \ge 0$ and $\sigma,d,\dir > 0$. Then, the following probabilities are all equal to the quantity $F(z;\dir,\sigma,d)$, which is defined as
\[
 \Phi\Bigl(\f{z - \dir \sigma^2 d}{\sqrt{2\sigma^2 d}}\Bigr) + e^{z \dir}\left( (1 + \dir z + \dir^2 \sigma^2 d)\Phi\Bigl(-\f{z + \dir \sigma^2 d}{\sqrt{2\sigma^2 d}} \Bigr) - \dir \sqrt{\f{\sigma^2 d}{\pi}}  \exp\Bigl(-\frac{(z + \dir \sigma^2 d)^2}{4 \sigma^2 d}\Bigr)    \right).
\]
\begin{enumerate} [label=\rm(\roman{*}), ref=\rm(\roman{*})]  \itemsep=3pt
\item \label{Pp1} $\Pp^\sigma(\sup_{x,y \in [a,a + d]} |G_{\dir_0 + \dir}(x,y) - G_{\dir_0}(x,y)| \le z)$ 
\item \label{Pp2} $\Pp^\sigma(G_{\dir_0 + \dir}(a,a + d) - G_{\dir_0}(a,a + d) \le z)$
\item \label{Pp3} $\Pp^\sigma(G_{\dir_0}(-d) - G_{\dir_0 + \dir}(-d)\le z)$
\item \label{Pp4} $\Pp^\sigma(G_{\dir_0 + \dir}(d) - G_{\dir_0}(d)\le z)$ 
\item \label{Pp5} $\Pp^1\Bigl(G_{(\dir_0 + \dir) \sqrt{\sigma^2 d}}(1) -G_{\dir_0 \sqrt{\sigma^2 d}}(1) \le \f{z}{ \sqrt{\sigma^2 d}} \Bigr)$. 
\end{enumerate}
Furthermore, 
\be \label{F0}
F(0;\dir,\sigma,d) > 0,\;\; \forall \dir,d,\sigma > 0, \qquad \text{and}\qquad \lim_{\dir \to \infty} F(0;\dir.\sigma,d) = 0.
\ee
Hence, for $d > 0$, the distribution of $G_{\dir_0 + \dir}(d) - G_{\dir_0}(d)$ can be written as a mixture of probability measures
\[
p \delta_0 + (1 - p)\pi 
\]
where $p = F(0,\dir,\sigma,d)$, $\delta_0$ is the point mass at $z = 0$, and $\pi$ is a continuous probability measure on $(0,\infty)$ with density
\[
(1 - p)^{-1} \Big[\dfrac{\partial}{\partial z}F(z;\dir,\sigma,d)\Big]\ind(z > 0).
\]
\end{theorem}
\begin{remark}
We also compute the Laplace transform/moment generating function of this random variable in \eqref{Laplace}.
\end{remark}
\newpage

\begin{corollary} \label{cor:SH_dist_cor}
The following hold.
\begin{enumerate} [label=\rm(\roman{*}), ref=\rm(\roman{*})]  \itemsep=3pt 
\item \label{itm:conv_to_Gaus} For $\dir_0,x,y \in \R$, as $|\dir| \to \infty$, $G^\sigma_{\dir_0 + \dir}(x,x + y) - G^\sigma_{\dir_0}(x,x + y) - \sigma^2 \dir y$ converges in distribution to a Gaussian random variable with mean $0$ and variance $2\sigma^2 |y|$.
\item \label{itm:splitting_time} Let $\dir_0,x \in \R$. For $\dir > 0$, define
\be \label{Split_pts_def}
\begin{aligned}
S_{x}^-(\dir_0,\dir_0 + \dir) &= \inf\{ y > 0: G_{\dir_0 + \dir}(x,x - y) < G_{\dir_0}(x,x - y) \}, \\
S_x^{+}(\dir_0,\dir_0 + \dir) &= \inf\{y > 0: G_{\dir_0 + \dir}(x,x + y) > G_{\dir_0}(x,x + y) \}.
\end{aligned}
\ee
In words, $S_x^-$ and $S_x^+$ are the splitting times of the trajectories $y \mapsto G_{\dir_0}(x,x + y)$ and $y \mapsto G_{\dir_0 + \dir}(x,x + y)$ to the left and the right of the origin, respectively.
For $y > 0$, define
\begin{align*}
T_{\dir_0}^-(x,x+y) &= \inf\{\dir > 0: G_{\dir_0 -\dir}(x,x+y) < G_{\dir_0}(x,x + y)\}, \\
T_{\dir_0}^+(x,x+y) &= \inf\{\dir > 0: G_{\dir_0 + \dir}(x,x+y) > G_{\dir_0}(x,x+y)\},
\end{align*}
so that $T_{\dir_0}^-(x,x+y)$ and $T_{\dir_0}^+(x,x+y)$ are the first points of decrease and increase, respectively, of the function $\dir \mapsto G_{\dir_0 + \dir}(x,x+y)$ to the left and right of $\dir = 0$. Then, for $y > 0$ and $\dir > 0$,
\begin{align}
&\quad \;\Pp^\sigma(S_x^-(\dir_0,\dir_0 + \dir) \ge y) = \Pp^\sigma(S_x^{+}(\dir_0,\dir_0 + \dir) \ge y) \nonumber \\
&= \Pp^\sigma(T_{\dir_0}^-(x,x + y) \ge \dir) = \Pp^\sigma(T_{\dir_0}^+(x,x + y) \ge \dir) \nonumber  \\
&= \Pp^\sigma(G_{\dir_0 + \dir}(x,x+y) = G_{\dir_0}(x,x+y)) \nonumber \\
&= (2 + \dir^2 \sigma^2 y)\Phi\Bigl(-\dir\sqrt{\f{\sigma^2 y}{2}}\Bigr) - \dir \sqrt{\f{\sigma^2 y}{\pi}} e^{-\frac{ \dir^2 \sigma^2 y }{4}}  \label{split_prob}
\end{align}
\item \label{itm:not_indep} 
For $\sigma > 0$, $x \in \R$, $y > 0$, and $\dir_1 < \dir_2$, $G^\sigma_{\dir_2}(x,x+y) - G^\sigma_{\dir_1}(x,x+y)$ is not independent of $G^\sigma_{\dir_1}(x,x+y)$. The process $\dir \mapsto G^\sigma_\dir(x,x+y)$ is not Markov and does not have independent increments. For $\dir_1 < \dir_2$, the joint process $(G_{\dir_1},G_{\dir_2})$ is not Markov.  
\end{enumerate}
\end{corollary}

\newpage
\begin{remark}
Corollary 2 of \cite{Rogers-Pitman-81} describes another coupling of two  Brownian motions with drift
such that they agree for a finite amount of time.  Their result   is related to our work because it is used to show the stability of the Brownian queue (see, for example, page 289 in \cite{brownian_queues}).
In a similar vein, Section 5 of \cite{MEXIT} explores maximum exit couplings--couplings of two Brownian motions such that the law of the separation point is stochastically maximal.  The couplings of \cite{Rogers-Pitman-81,MEXIT} are different from the SH, as can be seen from a calculation of the separation time. In \cite{Rogers-Pitman-81}, the separation time is exponential, while in \cite{MEXIT}, the separation time has the distribution of the absolute value of a Gaussian.

The fact that  $(G_{\dir_1},G_{\dir_2})$ follows because the splitting time of the trajectories is not exponential. Alternatively, this fact follows from Theorem 28 of \cite{MEXIT}. Specifically, two elliptic diffusions with continuous coefficients cannot be coupled in a jointly Markovian manner with nonzero separation time. I thank Adam Jaffe for pointing me to the paper \cite{MEXIT}. 
\end{remark}

\subsection{Proofs} \label{sec:dist_calc}
\begin{proof}[Proof of theorem \ref{thm:SH_inc_dist}]
We first prove that each of the probabilities \ref{Pp1}--\ref{Pp5} is the same. We know from Proposition \ref{prop:SH_cons}\ref{itm:SH_mont} that for $\dir > 0$, $G_{\dir_0} \li G_{\dir_0 + \dir}$. Lemma \ref{lem:ext_mont} implies that for $a \le x \le y \le a + d$, $\Pp^\sigma$-almost surely, 
\[
0 \le G_{\dir_0 + \dir}(x,y) - G_{\dir_0}(x,y) \le G_{\dir_0 + \dir}(a,a + d) - G_{\dir_0}(a,d).
\]
Hence, $\ref{Pp1}=\ref{Pp2}$. The probabilities \ref{Pp3} and \ref{Pp4} are both equal to \ref{Pp2} by translation invariance (Theorem \ref{thm:SH_dist_invar}\ref{itm:shinv}).
To see that $\ref{Pp4} = \ref{Pp5}$, we apply the scaling invariance of Theorem \ref{thm:SH_dist_invar}\ref{itm:SHscale} with $b = \sqrt{\sigma^2 d}$ and $c = d^{-1/2}$ to obtain
\[
\{\sqrt{\sigma^2 d} G_{\dir \sqrt{\sigma^2 d}}^1(1)\}_{\dir \in \R} = \{\sqrt{\sigma^2 d} G_{\dir \sqrt{\sigma^2 d}}^1(d/d)\}_{\dir \in \R} \deq \{G_\dir^\sigma(d)\}_{\dir \in \R},
\]
and thus, \ref{Pp4} = \ref{Pp5}.
Now, to compute this probability, we use \ref{Pp5} with $d = \sigma = 1$. Then, the general case follows from replacing $z$ with $z/\sqrt{\sigma^2 d}$ and $\dir$ with $\dir \sqrt{\sigma^2 d}$.

\noindent We know from the construction in Proposition \ref{prop:SH_cons}\ref{itm:SH_dist} that 
\[
(G_{\dir_0}^1,G_{\dir_0 + \dir}^1) \deq (B,D(Z,B)),
\]
where $B$ and $Z$
are independent two-sided Brownian motions with diffusivity $1$ and drifts $\dir_0,\dir_0 + \dir$. Then, by Lemma \ref{DRcont}, 
\begin{align*}
G_{\dir_0 + \dir}^1(1) - G_{\dir_0}^1(1) &\deq \Bigl(\sup_{0 \le u \le 1}\{Z(u) - B(u)\} - \sup_{-\infty < u \le 0}\{Z(u) - B(u)\} \Bigr)^+.\\
&\deq \Bigl(\sup_{0 \le u \le 1}\{\sqrt 2 B(u) + \dir u\} - \sup_{-\infty < u \le 0}\{\sqrt 2 B(u) + \dir u\} \Bigr)^+.
\end{align*}
Notice that the two suprema on the right-hand side are independent, as the first depends on $B$ for positive $u$ and the second depends on $B$ for negative $u$. For $x \ge 0$, Lemma \ref{lem:BM_drift_sup} states
\begin{align*}
&\quad\;\Pp\bigl(\sup_{0 \le u \le 1}\{\sqrt 2 B(u) + \dir u\} \le x\bigr)
=  \Phi\Bigl(\frac{x - \dir}{\sqrt{2}}\Bigr) - e^{\dir x}\Phi\Bigl(\frac{-x - \dir}{\sqrt{2}}\Bigr).
\end{align*}
On the other hand, by Lemma \ref{lemma:sup of BM with drift}, 
\[
\sup_{-\infty < u \le 0}\{\sqrt 2 B(u) + \dir u\} \sim \operatorname{Exp}(\dir).
\]
The conclusion of the theorem follows by a simple, but tedious convolution. 
\begin{align*}
&\Pp^1(G_{\dir_0 + \dir}(1) - G_{\dir_0}(1) \le z) 
= \int_{-\infty}^0 \left(\Phi\bigl(\frac{z-y - \dir }{\sqrt{2}}\bigr) - e^{\dir (z - y)}\Phi\bigl(\frac{-z + y - \dir }{\sqrt{2}}\bigr)\right)\dir e^{\dir y}\, dy  \\[1em]
&= \int_{-\infty}^0\int_{-\infty}^{\frac{z-y - \dir}{\sqrt 2}} \frac{\dir}{\sqrt{2\pi}}e^{-x^2/2}e^{\dir y}\, dx\, dy - e^{ \dir z}\int_{-\infty}^0\int_{-\infty}^{\frac{-z + y - \dir}{\sqrt 2}} \frac{\dir}{\sqrt{2\pi}} e^{-x^2/2}\, dx\, dy.
\end{align*}
We now use Fubini's Theorem to switch the order of integration. This results in
\begin{align}
&= \int_{-\infty}^{\zfraclambda}\int_{-\infty}^0 \frac{\dir}{\sqrt{2\pi}} e^{-x^2/2}e^{\dir y}\, dy\, dx  + \int_{\zfraclambda}^\infty \int_{-\infty}^{z - \dir - \sqrt 2 x} \frac{\dir}{\sqrt{2\pi}} e^{-x^2/2}e^{\dir y}\, dy\, dx  \nonumber \\
&\qquad\qquad\qquad\qquad\qquad\qquad\qquad\qquad\qquad-e^{ \dir z}\int_{-\infty}^{\negzfraclambda}\int_{ \sqrt 2 x +z + \dir }^0 \frac{\dir}{\sqrt{2\pi}} e^{-x^2/2}\, dy\, dx \nonumber \\
&= \int_{-\infty}^{\zfraclambda} \frac{1}{\sqrt{2\pi}} e^{-x^2/2}\, dx + \int_{\zfraclambda}^\infty \frac{1}{\sqrt{2\pi}} e^{-x^2/2}e^{z\dir -\sqrt 2 \dir x - \dir^2}\, dx \nonumber \\
&\qquad\qquad\qquad\qquad\qquad\qquad\qquad\qquad\qquad +e^{ \dir z}\int_{-\infty}^{\frac{-z - \dir}{\sqrt 2}} (\sqrt 2 x + z + \dir)\frac{\dir}{\sqtwopi} e^{-x^2/2}\, dx \nonumber  \\
&= \Phi\bigl(\zfraclambda\bigr) + e^{\dir z}\int_{\zfraclambda}^\infty \frac{1}{\sqtwopi} e^{-\frac{(x + \sqrt 2 \dir)^2}{2}}\, dx  \nonumber \\
&\qquad\qquad\qquad\qquad\qquad\qquad\qquad + e^{\dir z}\left(\int_{-\infty}^{\negzfraclambda} \frac{\dir }{\sqrt{\pi}} xe^{-x^2/2}\, dx +   (\dir z + \dir^2) \Phi\bigl(\negzfraclambda\bigr)  \right) \nonumber  \\
&= \Phi\bigl(\zfraclambda\bigr) + e^{\dir z}\int_{\frac{z + \dir}{\sqrt 2}}^\infty \frac{1}{\sqtwopi} e^{-u^2/2}\, du
+ e^{\dir z}\left(-\frac{\dir}{\sqrt{\pi}}e^{-\frac{(z + \dir)^2}{4}} +   (\dir z + \dir^2) \Phi\bigl(\negzfraclambda\bigr)  \right) \nonumber  \\
&= \Phi\bigl(\zfraclambda\bigr) + e^{\dir z}\left( (1 + \dir z + \dir^2)\Phi\bigl(\negzfraclambda\bigr) - \frac{\dir}{\sqrt \pi}e^{-\frac{(z + \dir)^2}{4}}    \right). \label{CDF11}
\end{align}

To prove \eqref{F0}, it again suffices to show the $d = \sigma = 1$ case. 
Observe that 
\be \label{eqn:0inc}
F(0;\dir,1,1)= (2 + \dir^2 ) \Phi\bigl(-\dir/\sqrt 2\bigr) - \f{\dir}{\sqrt \pi} e^{-\f{\dir^2 }{4}}.
\ee
from which it readily follows that $\lim_{\dir \to \infty} F(0;\dir,1,1) = 0$.
Further,  from \eqref{eqn:0inc}, we can compute
    \be \label{eqn:Dlambda}
    \dfrac{\partial }{\partial \dir} F(0;\dir,1,1) = 2 \dir \Phi(-\dir/\sqrt 2) - \f{2}{\sqrt \pi}e^{-\f{\dir^2 }{4}}.
    \ee
    By \cite[Theorem 1.2.6]{Durrett}, for all $y > 0$, $\int_{-\infty}^{-y} e^{-x^2/2}\,dx < y^{-1} e^{-y^2/2}$. (The theorem is stated with a weak inequality, but the proof shows that the equality is strict.)  Applying this to \eqref{eqn:Dlambda}, we see that $\lambda \mapsto F(0;\dir,1,1)$ is strictly decreasing. 
    Hence, $F(0;\dir,1,1) > 0$ for all $\dir > 0$, and also $F(0;\dir,\sigma,d) > 0$ for all $\dir,\sigma,d > 0$ by the equality~$\ref{Pp5}=\ref{Pp4}$.  
    \end{proof}

    \noindent We now compute the Laplace transform/moment generating function of $G_{\dir_0 + \dir}^\sigma(d) - G_{\dir_0}^\sigma(d)$. For $\alpha \neq \dir$, 
    
\be \label{Laplace}
\begin{aligned}
&\Ee^\sigma\Big[\exp\Big(-\alpha(G_{\dir_0 + \dir}(a,a + d) - G_{\dir_0}(a,a + d))\Big)\Big] \\
&=e^{\alpha^2\sigma^2 d  - \alpha \dir \sigma^2 d }\Phi\Big((\dir - 2\alpha)\sqrt{\f{\sigma^2 d}{2}}\Big) \Big(1 - \f{\alpha^2}{(\dir - \alpha)^2}\Big) \\
&\qquad\qquad\qquad+\Phi\Big(-\dir \sqrt{\f{\sigma^2 d}{2}}\Big)\Big(1 + \f{\alpha \dir}{(\dir - \alpha)^2} - \f{\alpha(1 + \sigma^2\dir^2 d)}{\dir - \alpha}\Big) \\
&\qquad\qquad\qquad\qquad\qquad\qquad+\f{\dir \alpha}{(\dir - \alpha)}\sqrt{\f{\sigma^2 d}{\pi}} e^{-\dir^2\sigma^2 d/4}.
\end{aligned}
\ee
Again, we start with $\sigma = d = 1$. By \eqref{CDF11}, we obtain
\be \label{eqn:1-BuseCDF}
\Pp^1(G_{\dir_0 + \dir}(1) - G_{\dir_0}(1) > z) = \Phi\Big(-\zfraclambda\Big) - e^{\dir z}(1 + \dir z + \dir^2)\Phi\Big(\negzfraclambda\Big) + \f{\dir e^{\dir z}}{\sqrt \pi} e^{-\f{(z + \dir)^2}{4}}.
\ee
we compute the Laplace transform via Lemma \ref{Identity for Laplace Transform}: that is 
\be \label{Laplace_lemma_rep}
\Ee^1\Bigl[-\alpha(G_{\dir_0 + \dir}(a,b) - G_{\dir_0}(a,b))\Bigr]  = 1 - \alpha \int_0^\infty e^{-\alpha z}\Pp^1(G_{\dir_0 + \dir}(a,b) - G_{\dir_0}(a,b) > z)\, dz. 
\ee 
handling each of the three terms on in the sum on the right-hand side of \eqref{eqn:1-BuseCDF}.  Using Fubini's theorem, we have
\begin{align}
    &\int_0^\infty -\alpha e^{-\alpha z} \Phi\Big(-\f{z - \dir}{\sqrt 2}\Big)\,dz = \int_0^\infty \int_{-\infty}^{-\zfraclambda} -\alpha e^{-\alpha z} \f{1}{\sqtwopi} e^{-x^2/2}\, dxdz \nonumber \\
    = &\int_{-\infty}^{\fraclambda}\int_0^{\dir - \sqrt 2 x}-\alpha e^{-\alpha z} \f{1}{\sqtwopi} e^{-x^2/2}\,dz\,dx = \int_{-\infty}^{\fraclambda} (e^{-\alpha(\dir - \sqrt 2 x)} - 1)\f{1}{\sqtwopi} e^{-x^2/2}\, dx \nonumber  \\
    = &\int_{-\infty}^{\fraclambda} \f{e^{-\alpha(\dir - \sqrt 2 x)}e^{-x^2/2}}{\sqtwopi} \, dx - \Phi\Big(\fraclambda\Big) 
    = e^{\alpha^2 - \alpha \dir}\Phi\Big(\f{\dir - 2\alpha}{\sqrt 2}\Big) - \Phi\Big(\fraclambda\Big).  \label{Laplace part 1}
\end{align}
Next, we compute the second term, using integration by parts in the second step
\begin{align}
    &\int_0^\infty \alpha e^{(\dir - \alpha)z}(1 + \dir z + \dir^2)\Phi\Big(\negzfraclambda\Big)\,dz \nonumber  \\
    =&\int_{-\infty}^{-\fraclambda} \int_0^{-\sqrt 2 x - \dir} \alpha e^{(\dir - \alpha)z}\f{(1  + \dir z + \dir^2)e^{-x^2/2}}{\sqtwopi}\, dzdx \nonumber  \\
    = &\int_{-\infty}^{-\fraclambda} \Bigg((1 + \dir(-\sqrt 2 x - \dir) + \dir^2)\f{\alpha}{\dir - \alpha} e^{(\alpha - \dir)(\sqrt 2 x + \dir)} \nonumber  \\
    &\qquad\qquad - \f{\alpha(1 + \dir^2)}{\dir - \alpha} -\int_0^{-\sqrt 2 x - \dir} \f{\alpha \dir}{\dir - \alpha}e^{(\dir - \alpha)z}\, dz\Bigg) \f{e^{-x^2/2}}{\sqtwopi}\,dx \nonumber \\
    = &\int_{-\infty}^{-\fraclambda} \Bigg((1 -\sqrt 2 x \dir)\f{\alpha}{\dir - \alpha} e^{(\alpha - \dir)(\sqrt 2 x + \dir)} \nonumber \\
    &\qquad\qquad - \f{\alpha(1 + \dir^2)}{\dir - \alpha}-\f{\alpha \dir}{(\dir - \alpha)^2}e^{(\alpha - \dir)(\sqrt 2 x + \dir)} + \f{\alpha \dir}{(\dir - \alpha)^2}\Bigg)\f{e^{-x^2/2}}{\sqtwopi}\,dx \nonumber \\
    = &\f{\alpha(\alpha \dir^2 + \alpha - \dir^3)}{(\dir - \alpha)^2}\Phi\Big(-\fraclambda\Big) \nonumber \\
    &\qquad\qquad+ \int_{-\infty}^{-\fraclambda}\f{\alpha(-\sqrt 2 x \dir^2 - \alpha + \sqrt 2 x \alpha \dir)e^{(\alpha - \dir)(\sqrt 2 x + \dir)}}{(\dir - \alpha)^2}\f{e^{-x^2/2}}{\sqtwopi}\,dx \nonumber \\
    = &\f{\alpha(\alpha \dir^2 + \alpha - \dir^3)}{(\dir - \alpha)^2}\Phi\Big(-\fraclambda\Big)  \nonumber \\
    &\qquad\qquad+ \int_{-\infty}^{-\fraclambda}\f{\alpha(-\sqrt 2 x \dir^2 - \alpha + \sqrt 2 x \alpha \dir)e^{-(x - \sqrt 2(\alpha - \dir))^2/2 + \alpha(\alpha - \dir)}}{(\dir - \alpha)^2\sqtwopi}\,dx \nonumber\\
    = &\f{\alpha(\alpha \dir^2 + \alpha - \dir^3)}{(\dir - \alpha)^2}\Phi\Big(-\fraclambda\Big) \nonumber\\
    &\qquad\qquad+ e^{\alpha^2 - \alpha \dir}\Bigg( \f{ \dir \alpha}{\sqrt \pi(\dir - \alpha)}e^{-\dir^2/4 + \alpha \dir - \alpha^2} + \Big(2\dir \alpha - \f{\alpha^2}{(\dir - \alpha)^2}\Big) \Phi\Big(\f{\dir - 2\alpha}{\sqrt 2}\Big) \Bigg) \nonumber\\
    = &\f{\alpha(\alpha \dir^2 + \alpha - \dir^3)}{(\dir - \alpha)^2}\Phi\Big(-\fraclambda\Big) \nonumber \\
    &\qquad\qquad+ \f{\dir \alpha}{\sqrt \pi(\dir - \alpha)}e^{-\dir^2/4} + e^{\alpha^2 - \alpha \dir}\Big(2\dir \alpha - \f{\alpha^2}{(\dir - \alpha)^2}\Big)\Phi\Big(\f{\dir - 2\alpha}{\sqrt 2}\Big). \label{Laplace part 2}
\end{align}
Lastly, we handle the third term:
\begin{align}
    \int_0^\infty -\alpha e^{(\dir - \alpha)z}\f{\dir}{\sqrt \pi }e^{-(z + \dir)^2/4}\, dz = -2\alpha \dir e^{\alpha^2 - \alpha \dir} \Phi\Big(\f{\dir - 2\alpha}{\sqrt 2}\Big). \label{Laplace part 3}
\end{align}
Adding \eqref{Laplace part 1},\eqref{Laplace part 2}, and \eqref{Laplace part 3} to the number $1$, we get 
\begin{align*}
    \Ee[\exp(-\alpha Z)] &= e^{\alpha^2 - \alpha \dir}\Phi\Big(\f{\dir - 2\alpha}{\sqrt 2}\Big)\Big(1 - \f{\alpha^2}{(\dir - \alpha)^2}\Big) \nonumber\\
    &\qquad\qquad+ \Phi\Big(-\fraclambda\Big)\Big(1 + \f{\alpha \dir}{(\dir - \alpha)^2} - \f{\alpha(1 + \dir^2)}{\dir - \alpha}\Big) + \f{\dir \alpha}{\sqrt \pi (\dir - \alpha)}e^{-\dir^2/4},
\end{align*}
and we see this matches \eqref{Laplace} with $d = \sigma = 1$. The general case follows by \ref{Pp4} = \ref{Pp5} after replacing $\alpha$ with $\alpha \sqrt{\sigma^2 d}$ and $\dir$ with $\dir \sqrt{\sigma^2 d}$.

\begin{proof}[Proof of Corollary \ref{cor:SH_dist_cor}
]
 \noindent \textbf{Item \ref{itm:conv_to_Gaus} (Convergence to Gaussian):} For $\dir,y > 0$, Theorem \ref{thm:SH_inc_dist}, tells us that
\begin{multline*}
    \Pp^\sigma(G_{\dir + \dir_0}(x,x + y) - G_{\dir_0}(x,x + y) - \sigma^2 \dir d \le z) \\ 
    =\Phi\bigl(\f{z}{\sqrt{2\sigma^2 y}}\bigr) + e^{(z + \sigma^2 \dir y)\dir}\Bigl((1 + (z + \sigma^2 \dir y) \dir + \dir^2 \sigma^2 y) \Phi\bigl(-\f{z + 2\sigma^2 \dir y}{\sqrt{2\sigma^2 y}}\bigr) \\
     - \dir \sqrt{\f{\sigma^2 y}{\pi}}\exp\bigl(-\f{(z + 2\sigma^2 \dir y)^2}{4 \sigma^2 y}\bigr) \Bigr)
\end{multline*}
In the limit as $\dir \to +\infty$, the only nonvanishing term is 
\[
\Phi\bigl(\f{z}{\sqrt{2\sigma^2 y}}\bigr) = \int_{-\infty}^{z/\sqrt{2\sigma^2 y}} \f{1}{\sqrt{2\pi}} e^{-x^2/2}\,dx = \int_{-\infty}^z \f{1}{\sqrt{4\pi \sigma^2 y}} e^{-x^2/(4 \sigma^2 y)}\,dx,
\]
which we now recognize as the distribution function of a Gaussian random variable with $0$ mean and variance $2\sigma^2 y$. For $y < 0$, the result follows by symmetry of the Gaussian.

\noindent For the limit as $\dir \to -\infty$, the stationarity of increments in Theorem \ref{thm:SH_dist_invar}\ref{itm:SH_inc_stat} implies
\begin{align*}
    \quad \; G^\sigma_{\dir + \dir_0}(x,x +y) - G^\sigma_{\dir_0}(x,x + y) - \sigma^2 \dir y 
    &= - (G^\sigma_{\dir_0}(x,x + y) - G^{\sigma}_{\dir_0 + \dir}(x,x + y) - \sigma^2|\dir|y  )\\
    &\deq -(G_{-\dir}^\sigma(x,x + y) - G_{0}^\sigma(x,x + y) - \sigma^2 |\dir|y).
\end{align*}
The result now follows by the symmetry of the Gaussian distribution.

 \noindent \textbf{Item \ref{itm:splitting_time} (Splitting times):} We first recall the definitions: for $\dir > 0$,
\begin{align*}
S_{x}^-(\dir_0,\dir_0 + \dir) &= \inf\{ y > 0: G_{\dir_0 + \dir}(x,x - y) < G_{\dir_0}(x,x - y) \}, 
\\
S_x^{+}(\dir_0,\dir_0 + \dir) &= \inf\{y > 0: G_{\dir_0 + \dir}(x,x + y) > G_{\dir_0}(x,x + y) \},
\end{align*}
and for $y > 0$,
\begin{align*}
T_{\dir_0}^-(x,x+y) &= \inf\{\dir > 0: G_{\dir_0 -\dir}(x,x+y) < G_{\dir_0}(x,x + y)\}, \\
T_{\dir_0}^+(x,x+y) &= \inf\{\dir > 0: G_{\dir_0 + \dir}(x,x+y) > G_{\dir}(x,x+y)\}.
\end{align*}

Let $y,\dir > 0$. We work on the $\Pp$-almost sure event (depending on a fixed $\dir_0 \in \R,\dir > 0)$ on which $G_{(\dir_0 + \dir)-}(x) = G_{\dir_0 + \dir}(x)$  for all $x \in \R$, and for arbitrary $\dir_1 < \dir_2$, $G_{\dir_1} \li G_{\dir_2}$ (Proposition \ref{prop:SH_cons}\ref{itm:SH_mont}--\ref{itm:SH_cont}). On this event, we first show that
\begin{enumerate} [label=\rm(\alph{*}), ref=\rm(\alph{*})]  \itemsep=3pt
\item \label{eqa} $S_x^-(\dir_0,\dir_0 + \dir) \ge y \iff G_{\dir_0 + \dir}(x,x - y) = G_{\dir_0}(x,x - y)$
\item \label{eqb}$T_{\dir_0}^-(x,x + y) \ge \dir \iff G_{\dir_0 -\dir}(x,x+y) = G_{\dir_0}(x,x + y)$
\item \label{eqc} $S_x^+(\dir_0,\dir_0 + \dir) \ge y\iff T_{\dir_0}^+(x,x+y) \ge \dir \iff G_{\dir_0 + \dir}(x,x+y) = G_{\dir}(x,x+y)$.
\end{enumerate}
We prove \ref{eqa} and \ref{eqb}, then \ref{eqc} follows similarly. For \ref{eqa}, if $S_x^-(\dir_0,\dir_0 + \dir) \ge y$, then since $G_{\dir_0 + \dir}(x,x - z) \le G_{\dir_0}(x,x-z)$ holds for general $z > 0$, we must have equality for $z < y$. Continuity implies the equality $G_{\dir_0 + \dir}(x,x - y) = G_{\dir_0}(x,x-y)$. On the other hand, if $G_{\dir_0 + \dir}(x,x - y) = G_{\dir_0}(x,x-y)$, then for all $z \in (0,y]$, by Lemma \ref{lem:ext_mont}, 
\[
0 \ge  G_{\dir_0 + \dir}(x,x - z) - G_{\dir_0}(x,x-z)  \ge G_{\dir_0 + \dir}(x,x - y) - G_{\dir_0}(x,x-y)   = 0,
\]
so $S_x^-(\dir_0,\dir_0 + \dir) \ge y$.

 \noindent Now, we turn to \ref{eqb}. If $T_{\dir_0}^-(x,x+y) \ge \dir > 0$, then $G_{\dir_0 - \eta}(x,x + y) = G_{\dir_0}(x,x+y)$ for all $\eta \in (0,\dir)$. But the right-continuity of $G$ implies that $G_{\dir_0 - \dir}(x,x+y) = G_{\dir_0}(x,x+y)$ as well. On the other hand, if $G_{\dir_0 - \dir}(x,x + y) = G_{\dir_0}(x,x+y)$, then for $\eta \in (0,\dir]$,  $G_{\dir_0 - \eta}(x,x+y) \ge G_{\dir_0 - \dir}(x,x+y)$, so
\[
0 \ge  G_{\dir_0 - \eta}(x,x+y) - G_{\dir_0}(x,x+y) \ge  G_{\dir_0 - \dir}(x,x + y) - G_{\dir_0}(x,x+y) = 0.
\]
Hence, $T_{\dir_0}^-(x,x+y) \ge \dir$. By stationarity of $G$ in the $x$ and $\dir$ parameters (Theorem \ref{thm:SH_dist_invar}\ref{itm:shinv},\ref{itm:SH_inc_stat}),
\begin{align*}
    &\quad \; \Pp^\sigma(G_{\dir_0 + \dir}(x,x - y) = G_{\dir_0}(x,x - y)) = \Pp^\sigma(G_{\dir_0 + \dir}(x - y,x) = G_{\dir_0}(x-y,x) \\
    &= \Pp^\sigma(G_{\dir_0 + \dir}(x,x + y) = G_{\dir_0}(x,x+y) = \Pp^\sigma(G_{\dir_0}(x,x + y) = G_{\dir_0 - \dir}(x,x+y),
\end{align*}
where in the first equality we simply used the fact that $G_\dir(x,y) = -G_\dir(y,x)$. Hence, each of the events in Items \ref{eqa}--\ref{eqb} has the same probability, which is computed by setting $d = y$ and $z = 0$ in the formula of Theorem \ref{thm:SH_inc_dist}. We repeat this formula here:
\[
(2 + \dir^2 \sigma^2 y)\Phi\Bigl(-\dir\sqrt{\f{\sigma^2 y}{2}}\Bigr) - \dir \sqrt{\f{\sigma^2 y}{\pi}} e^{-\frac{ \dir^2 \sigma^2 y }{4}}.
\]

 \noindent \textbf{Item \ref{itm:not_indep} (Non-independence and non-Markovian structure):} 

Assume, by way of contradiction, that for $y > 0$, $G^\sigma_{\dir_2}(x,y) - G^\sigma_{\dir_1}(x,y)$ is independent of $G^\sigma_{\dir_1}(x,y)$. Then, since each $x \mapsto G^\sigma_\dir$ is a Brownian motion with diffusivity $\sigma$,
\begin{align*}
\sigma^2 y &=  \Var^\sigma(G_{\dir_2}(x,x + y)) = \Var^\sigma(G_{\dir_2}(x,x+y) - G_{\dir_1}(x,x+y)) + \Var^\sigma(G_{\dir_1}(x,x+y)) \\
&= \Var^\sigma(G_{\dir_2}(x,x+y) - G_{\dir_1}(x,x+y)) + \sigma^2 y.
\end{align*}
Hence, $\Var^\sigma(G_{\dir_2}(x,x+y) - G_{\dir_1}(x,x+y)) = 0$. But this is not true, as Theorem \ref{thm:SH_inc_dist} shows that $G_{\dir_2}^\sigma(x,x+y) - G_{\dir_1}^\sigma(x,x+y)$ does not have a degenerate distribution. 

We turn to showing the non-independence of increments. Assume, by way of contradiction, $\dir \mapsto G^\sigma_\dir(x,x + y)$ has independent increments. Then for $0 < \eta < \dir$,
\begin{align*}
&\quad \;\Pp^\sigma(G_{\dir}(x,x+y) - G_0(x,x+y) = 0) \\
&= \Pp^\sigma(G_{\eta}(x,x+y) - G_0(x,x+y) = 0) \Pp^\sigma(G_{\dir}(x,x+y) - G_\eta(x,x+y) = 0),
\end{align*}
by monotonicity. Equivalently,
\begin{multline*}
\Pp^\sigma(G_{\dir}(x,x+y) = G_0(x,x+y) |G_{\eta}(x,x+y) = G_0(x,x+y)) \\
= \Pp^\sigma( G_{\dir}(x,x+y) = G_\eta(x,x+y)).
\end{multline*}
As we showed in the proof of Item \ref{itm:splitting_time}, this is equivalent to
\begin{multline*}
\Pp^\sigma(T_0^+(x,x+y) \ge \dir | T_0^+(x,x+y) \ge \eta) \\= \Pp^\sigma(T_\eta^+(x,x+y) \ge \dir - \eta) = \Pp^\sigma(T_0^+(x,x+y) \ge \dir - \eta). 
\end{multline*}
However, our computation in Item \ref{itm:splitting_time} shows that $T_0^+(x,x+y)$ does not have a memoryless distribution, a contradiction. The process $\dir \mapsto G^\sigma_\dir(x,x+y)$ therefore is also not Markov because the time to its first point of increase is not memoryless. 

Lastly, the joint process $(G_{\dir_1},G_{\dir_2})$ is not Markov because the splitting time of $G_{\dir_1}(x)$ and $G_{\dir_2}(x)$ for $x \ge 0$ is not memoryless, as seen in Item \ref{itm:splitting_time}.
\end{proof}

\section{The SH as a jump process and its path properties} \label{sec:SH_jump}
The trajectories of $\dir \mapsto G^\sigma(x,y)$ are step functions. This fact is the fundamental ingredient for the analysis of the infinite geodesics in the directed landscape studied in Chapter \ref{chap:Buse}. We state this theorem now. 
\begin{theorem} \label{thm:SH_jump_process}
Fix $\sigma > 0$ and let $x \in \R, y > 0$. Then, $\Pp^\sigma$-almost surely, the paths of $\dir \mapsto G_{\dir}(x,x + y)$ are nondecreasing step functions, whose jumps are a discrete subset of $\R$. The expected number of jumps in an interval $[\dir_0,\dir_0 + \dir] \subseteq \R$ is
\[
\Ee^\sigma[\# \{\eta \in [\dir_0,\dir_0 + \dir]: G_{\eta- }(x,x + y) < G_{\eta +}(x,x+y)\}] = 2 \dir\sqrt{\f{\sigma^2 y}{\pi}}. 
\]
Furthermore, $\lim_{\dir \to \pm \infty} G_\dir(x,x+y) = \pm \infty$ almost surely; consequently, the set of jumps over all $\dir \in \R$ is infinite and unbounded for both positive and negative $\dir$. 
\end{theorem}

We now detail the path properties of the SH. Given $G \in D(\R,C(\R))$ and $x,y \in \R$, we define the random sets
\be \label{XiSHdef}
\begin{aligned}
\XiSH(x,y) &= \{\dir \in \R: G_{\dir -}(x,y)\neq G_{\dir +}(x,y)\},\;\text{and} \\
 \XiSH &= \bigcup_{x,y \in \R} \XiSH(x,y) = \bigcup_{x \in \R} \{\dir \in \R: G_{\dir -}(x) \neq G_{\dir +}(x)\}.
 \end{aligned}
\ee
Observe that $\XiSH(x,y) = \XiSH(y,x)$. We define $\XiSH(x) = \XiSH(0,x)$. When we wish to refer to the set $\XiSH$ for a realization of the SH $G^\sigma$ without referring to the measure $\Pp^\sigma$, we write $\XiSHsig$.

\newpage
\begin{theorem} \label{thm:SH_sticky_thm}
Let $\sigma > 0$. The following hold $\Pp^\sigma$-almost surely. 
\begin{enumerate}[label=\rm(\roman{*}), ref=\rm(\roman{*})]  \itemsep=3pt 
\item \label{itm:SH_dist_grows} For $\dir_1 < \dir_2$ and $x \in \R$, $y \mapsto G_{\dir_2}(x,x + y) - G_{\dir_1}(x,x + y)$ is a nondecreasing function. 
\item \label{itm:SH_set_contain} For $a \le x < y \le b$, $\XiSH(x,y) \subseteq \XiSH(a,b)$. 
 \item \label{itm:SH_stick} For every $\dir \in \R$ and every compact set $K \subseteq \R$, there exists an $\ve = \ve(K,\dir) > 0$ so that when $x,y \in K$, and $\dir - \ve < \alpha < \dir < \beta < \dir + \ve$,
$G_{\alpha}(x,y) = G_{\dir-}(x,y)$ and $G_{\dir}(x,y) = G_{\beta}(x,y)$.
\item \label{itm:SH_all_jump} For every $x < y$, $\dir \mapsto G_\dir(x,y)$ is a step function converging to $\pm \infty$ as $\dir \to \pm \infty$. The set $\XiSH(x,y)$ is countably infinite and contains only finitely many points in each compact interval. 
\item \label{itm:SH_split_Xi} Let $S_x^-(\dir_1,\dir_2)$ and $S_x^+(\dir_1,\dir_2)$ be defined as in \eqref{Split_pts_def}. For \textit{every} pair $\dir_1 < \dir_2$ and $x \in \R$, there exists $\alpha,\beta \in [\dir_1,\dir_2] \cap \XiSH$ so that
\[
\begin{aligned}
    S_x^-(\dir_1,\dir_2) &= \inf\{y > 0: G_{\alpha}(x,x-y) < G_{\alpha_-}(x,x-y) \},\qquad\text{and} \\
    S_x^+(\dir_1,\dir_2) &= \inf\{y > 0: G_{\beta}(x,x+y) > G_{\beta_-}(x,x+y) \}.
\end{aligned}
\]
That is, when the trajectories $G_{\dir_1},G_{\dir_2}$ split (to the left or right of the origin), there is an exceptional $\dir \in \XiSH$ so that $G_{\dir-}$ and $G_{\dir}$ also split at that same point.
\item \label{itm:SH_Xi_dense}  The set $\XiSH$ is dense in $\R$. 
\end{enumerate}
\end{theorem}

Figure \ref{fig:SH} shows a simulation of the stationary horizon for $\sigma = 1$ and \\ $\dir \in \{0,\pm 1,\pm 2,\pm 3, \pm 5, \pm 10\}$. The items of Theorem \ref{thm:SH_sticky_thm} can be seen from the graph. Item \ref{itm:SH_set_contain} states that the discontinuities increase as we move away from the origin. This is manifest in the picture by the splitting of trajectories. Item \ref{itm:SH_split_Xi} states that when the two trajectories split for positive $x$, there exists a value $\dir \in \XiSH$ so that $G_{\dir +}$ follows the upper trajectory, and $G_{\dir -}$ follows the lower trajectory. This results in the density of Item \ref{itm:SH_Xi_dense}. The reverse happens for negative $x$. Once the trajectories split, the distance between the two never decreases, as stated in Item \ref{itm:SH_dist_grows}. Items \ref{itm:SH_stick} and \ref{itm:SH_all_jump} reflect the fact that two trajectories $G_{\dir_1}$ and $G_{\dir_2}$ stick together in a neighborhood of the origin before splitting. 
\begin{figure}
\centering
\includegraphics[height = 3in]{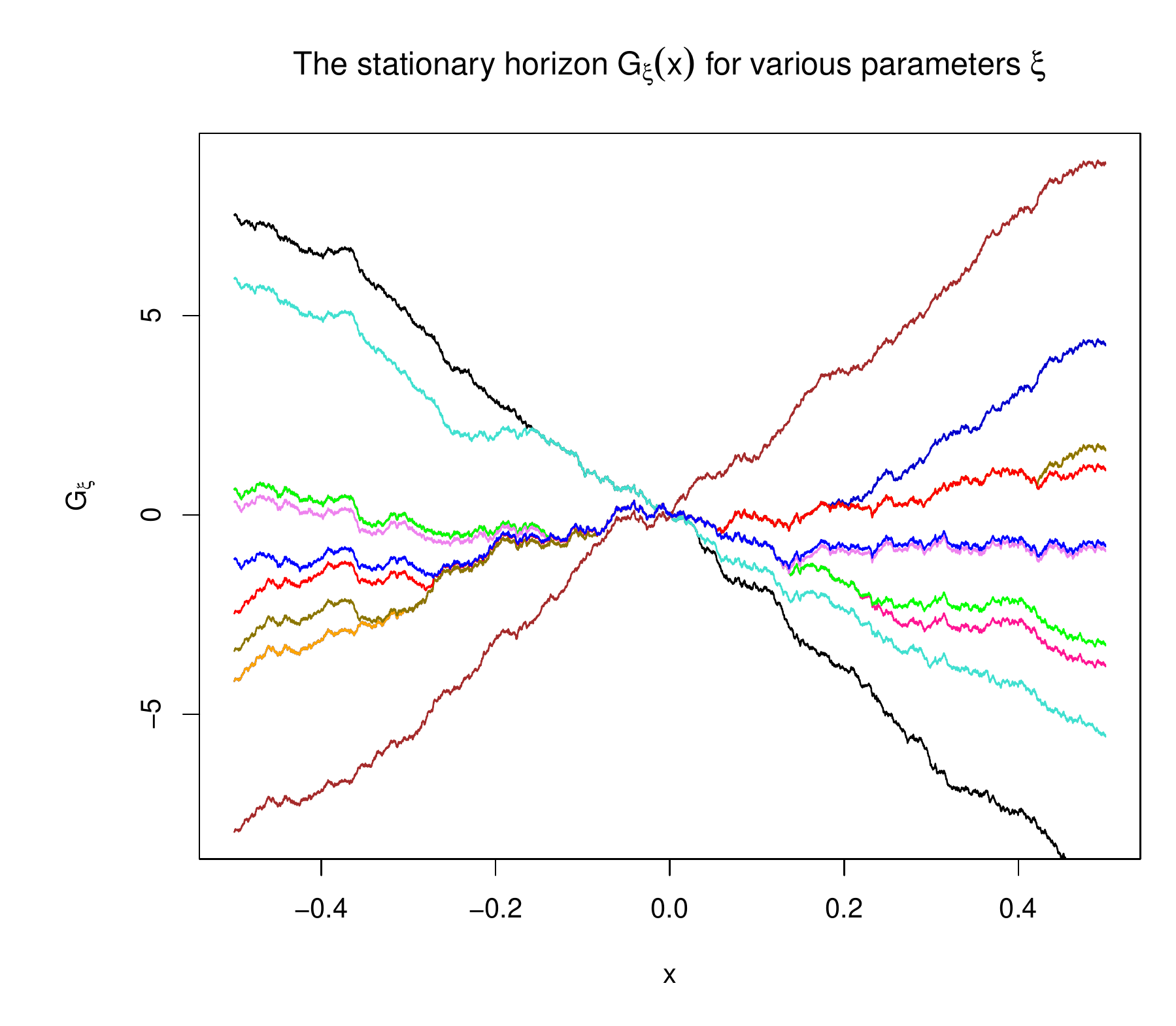}
\caption{A simulation of the stationary horizon}
\label{fig:SH}
\end{figure}

The discreteness of the sets $\XiSH(x,y)$ in Theorem \ref{thm:SH_sticky_thm}\ref{itm:SH_all_jump} implies that, under $\Pp^\sigma$, the set $\XiSH(x,y)$ is a well-defined point processes.  The set $\XiSH$ itself is dense, and it is not easy, a priori, to interpret as a random object. However, By Theorem \ref{thm:SH_sticky_thm}\ref{itm:SH_set_contain}, $\XiSH$ is the increasing union of the sets $\XiSH(x,y)$  as $x \to -\infty$ and $y \to +\infty$ (In fact, it suffices to fix either $x$ or $y$ and send the other to $\pm \infty$ by Corollary \ref{cor:dcLR} in Section \ref{sec:SHPalm}). We state some distributional invariances of these sets in the following corollary to Theorem \ref{thm:SH_dist_invar}. 
\begin{corollary}
For $\sigma > 0$, let $G^\sigma$ be the SH. For $c,b > 0$ and $\nu \in \R$, and $(x,y) \in \R$, the sets $\XiSH$ satisfy the following distributional invariances:
\[
\XiSHsig(x,y) \deq -\XiSHsig(-x,-y) 
\deq b\Xi_{G^{bc\sigma}}(c^{-2} x,c^{-2} y) + b\nu.  
\]
\end{corollary}
\begin{proof}
The first equality follows by the reflection invariance of Theorem \ref{thm:SH_dist_invar}\ref{itm:SH_reflinv}, and the second follows by the scaling relations in Theorem \ref{thm:SH_dist_invar}\ref{itm:SHscale}. More specifically, let $G^\sigma$ be the SH, and let $\wt G^{bc\sigma}$ be defined by $\wt G^\sigma_{\dir/b - \nu}(x) = bG_{\dir}^\sigma(c^2 x) - (bc\sigma)^2 \nu x$. Then, $\wt G^{bc\sigma}$ is distributed as $G^{bc\sigma}$, and $\dir \in \Xi_{G^\sigma}(c^2 x,c^2 y)$ if and only if $\dir \in b \Xi_{\wt G^{bc\sigma}}(x,y) + b\nu$.  
\end{proof}

\subsection{Proofs} \label{sec:step_path}
The following more general theorem is the key to showing step function behavior of the SH. It gives a general condition for an increment-stationary process to be a jump process, and is of independent interest. 
\begin{theorem} \label{thm:jump process condition}
On a probability space $(\Omega,\F,\Pp)$, let $Y = \{Y(t): t \ge 0\}$ be a nondecreasing, increment-stationary process such that the following three conditions hold:
\begin{enumerate} [label=\rm(\roman{*}), ref=\rm(\roman{*})]  \itemsep=3pt    
    \item \label{finitemean} For $a < b$, $ \Ee[Y(b)-Y(a)] < \infty$.
    \item  \label{p0>0} $\Pp[Y(t) = Y(0)] \in (0,1)$ for all sufficiently small $t > 0$.
    \item \label{liminf assumption} $ \liminf_{t \searrow 0} \Ee[Y(t) - Y(0)|Y(t) > Y(0)] > 0$.
\end{enumerate} 
Then, $\Pp$-almost surely the paths of  $t \mapsto Y(t)$ are step functions with finitely many jumps in each bounded interval. For each $t \ge 0$, there is a jump at $t$ with probability $0$. For  $a < b \in I$, the expected number of jump points in the interval $[a,b]$ equals
\[
\f{\Ee[Y(b) - Y(a)]}{\lim_{n \rightarrow \infty}\Ee[Y(2^{-n}) - Y(0)|Y(2^{-n}) > Y(0)]  }.
\]
\end{theorem}

\begin{remark}
Heuristically, we can think of Condition \ref{liminf assumption} in the following way: on average, the size of the jumps are bounded away from $0$, and therefore the jumps cannot accumulate because an increment of the process itself has finite expectation by Condition \ref{finitemean}.   
\end{remark}

\begin{proof}[Proof of Theorem \ref{thm:jump process condition}]
 We first note that for $b > a$,  
 \be \label{eqn:EY01}
 (b - a)\Ee[Y(1) - Y(0)] = \Ee[Y(b) - Y(a)].
 \ee
 Indeed, by increment-stationarity, it suffices to show that $\Ee[Y(t) - Y(0)] = t\Ee[Y(1) - Y(0)]$ for all $t > 0$. Since $Y$ is nondecreasing, $t \mapsto \Ee[Y(t) - Y(0)]$ is nondecreasing, so it further suffices to show that $\Ee[Y(t) - Y(0)] = t\Ee[Y(1) - Y(0)]$ just for rational $t > 0$. For any positive integer $k$,
\[
\Ee[Y(k) - Y(0)] = \sum_{i = 1}^k \Ee[Y(i) - Y(i -1)] = k\Ee[Y(1) - Y(0)] = k\Ee[Y(1) - Y(0)].
\]
Then for positive integers $r$ and $k$,
\[
r\Ee[Y(1) - Y(0)] = \Ee[Y(r) - Y(0)] = \sum_{i = 1}^k\Ee[Y(ri/k) - Y(r(i -1 )/k)] = k\Ee[Y(r/k) - Y(0)],
\]
and $\Ee[Y(r/k)-Y(0)] = \tf{r}{k} \Ee[Y(1) - Y(0)]$.

We now show that there are finitely many jumps in the interval $[0,1]$, and the general result holds by increment-stationarity and \eqref{eqn:EY01}. Consider discrete versions of the process $Y$ as follows. For $n \in \Z_{>0}$, let $D_n = \{\f{j}{2^n}: j \in \Z, 0 \le j \le 2^n\}$, and consider the process $Y_n := \{Y(t): t \in D_n\}$. Let $J_n$ be the number of jumps of $Y_n$, i.e.,
\[
J_n = \sum_{j =1}^{2^n} \ind \Bigl(Y\bigl(\tf{j}{2^n}\bigr) > Y\bigl(\tf{j -1}{2^n}\bigr)\Bigr)
\]
Then, $J_n$ is nondecreasing in $n$, so it has a limit, denoted as the random variable $J_\infty$. Let $K \in \{0,1,2,\ldots\} \cup \{\infty\}$ be the number of points of increase of $Y$ on the interval $[0,1]$. Specifically, a point $t \in (0,1)$ is a point of increase, if  $Y(t + \ve)  > Y(t -\ve)$ for all $\ve > 0$. We say $0$ is a point of increase if $Y(t) > Y(0)$ for all $t > 0$, and we likewise say that $1$ is a point of increase if $Y(t) < Y(1)$ for all $t < 1$.

There are three steps we need to complete the proof. The first is to show $K \le J_\infty$. Next, we show that $J_\infty$ is finite almost surely by computing its mean. Lastly, we show that $J_\infty \le K$.

We now show that $K \le J_\infty$. If $K < \infty$, let $k = K$, and otherwise, let $k$ be an arbitrary positive integer. It suffices to show that $J_\infty \ge k$.  By definition of $k$, we may choose $k$ points of increase $t_1<\dotsm<t_k$. First, we handle the case where $t_i \in (0,1)$ for all $i$. Then, for all sufficiently large $n$, there exist $n$-dependent positive integers $0 < j_1 < \cdots < j_k < 2^n$ so that for each $i$, $j_{i + 1} > j_i + 2$,  and 
\be \label{eqn:ti in interval}
\f{j_i -  1}{2^n} < t_i < \f{j_i + 1}{2^n}.
\ee
 Since $t_i$ is a point of increase and $Y$ is nondecreasing, $Y(\f{j_i + 1}{2^n}) > Y(\f{j_i - 1}{2^n})$. Therefore, $Y(\f{j_i + 1}{2^n}) > Y(\f{j_i}{2^n})$ or $Y(\f{j_i}{2^n}) > Y(\f{j_i - 1}{2^n})$. By assumption that $j_{i + 1} > j_i + 2$, the intervals $[\f{j_i - 1}{2^n},\f{j_i + 1}{2^n}]$ are mutually disjoint, so $J_n \ge k$ and therefore $J_\infty \ge k$. The case where $t_1 = 0$ or $t_k = 1$ is handled similarly.

\noindent Next, we compute the finite mean of $J_\infty$.
Let 
\[
c_n = \Ee[Y(2^{-n}) - Y(0)|Y(2^{-n}) > Y(0)]
\]
Then, using increment-stationarity,
\begin{align}
 \Ee[Y(1) - Y(0)] &= \sum_{j = 1}^{2^n} \Ee\Big[Y\big(\tf{j}{2^n}\big) - Y\big(\tf{j - 1}{2^n}\big)\Big]   \nonumber \\
    &= \sum_{j = 1}^{2^n} \Ee\Big[Y\big(\tf{j}{2^n}\big) - Y\big(\tf{j - 1}{2^n}\big)\Big|Y\big(\tf{j}{2^n}\big) > Y\big(\tf{j - 1}{2^n}\big) \Big]\Pp\Big(Y\big(\tf{j}{2^n}\big) > Y\big(\tf{j - 1}{2^n}\big)\Big) \nonumber \\
    &= c_n \sum_{j = 1}^{2^n}  \Pp\Big(Y\big(\tf{j}{2^n}\big) > Y\big(\tf{j - 1}{2^n}\big)\Big) \nonumber
    = c_n \Ee[J_n].
\end{align}
By assumptions \ref{finitemean} and \ref{liminf assumption} and the monotone convergence theorem,
\be \label{eqn:Jinf_mean}
\Ee[J_\infty] = \lim_{n \rightarrow \infty} \Ee[J_n] = \lim_{n \rightarrow \infty} \f{\Ee[Y(1) - Y(0)]}{c_n} < \infty. 
\ee
Therefore, $\Pp(J_\infty < \infty) = 1$. Since $K \le J_\infty$, with probability one, $Y$ has only finitely many points of increase on $[0,1]$. Therefore, with probability one, $Y:[0,1]\to \R$ is locally constant except at the finitely many jump points. Hence, for each $t \in (0,1)$, the left and right limits of $Y$ at $t, Y(t \pm)$ exist. The limits $Y(0+)$ and $Y(1-)$ exist as well. Since $Y$ is increasing, for each $t \in(0,1)$ and $\ve > 0$, we can apply \eqref{eqn:EY01} and Assumption \ref{finitemean} to get
\[
\Ee[Y(t+) - Y(t-)] \le \Ee[Y(t + \ve) - Y(t - \ve)] = 2\ve \Ee[Y(1) - Y(0)] < \infty.
\]
Sending $\ve \searrow 0$, the left-hand side is $0$ and therefore, a jump occurs at time $t$ with probability $0$. Similar arguments apply to $t = 0$ and $t = 1$. Therefore, there exists an event of probability one, $\Omega_{\Q_2}$ on which $Y$ has no jumps at points of the form $\f{j}{2^n}$ for positive integers $j$ and $n$.

To compute the mean number of jumps, we show that $J_\infty = K$ on the event $\Omega_{\Q_2}$. We already showed that $K \le J_\infty$, so it remains to show $J_\infty \le K$.

 We start by showing that if $Y(b) > Y(a)$ for some $a < b$, there must be some point of increase in the interval $[a,b]$. We prove this as follows: let $c$ be the midpoint of $a$ and $b$. Then, since $Y$ is nondecreasing, either $Y(b) > Y(c)$ or $Y(c) > Y(a)$. If, without loss of generality, $Y(b) > Y(c)$, then we can bisect the interval again with midpoint $d$ and get that $Y(b) > Y(d)$ or $Y(d) > Y(c)$, where $d$ is the midpoint of $a$ and $b$. Inductively, this constructs a sequence of nested intervals $[a_n,b_n] \subseteq [a_{n - 1},b_{n - 1}]\subseteq [a,b]$, where $[a_n,b_n]$ is either the left or right half of the previous interval. Then, $a_n$ is nondecreasing and $b_n$ is nonincreasing and $b_n - a_n \rightarrow 0$. Then, set $t = \lim_{n \rightarrow \infty} a_n = \lim_{n \rightarrow \infty} b_n$, and we have that $t \in [a_n,b_n]$ for all $n$. If $t \in (0,1)$, then for all $\ve > 0$,  we may choose $n$ large enough so that, because $Y$ is nondecreasing,
\[
Y(t + \ve) - Y(t - \ve) \ge Y(b_n) - Y(a_n) > 0.
\]
Hence, $t$ is a point of increase. The case where $t = 0$ or $1$ is handled similarly.

Now, we show that on $\Omega_{\Q_2}$, $J_n \le K$ for all $n$. By definition, $J_n$ is the number of integers $0 < j \le 2^n$ such that $Y(j2^{-n}) > Y((j -1)2^{-n})$. For each such $j$, we just showed that there must be a point of increase in $[(j - 1)2^{-n},j2^{-n}]$, and on the event $\Omega_{\Q_2}$, that point of increase must lie in the interior of the interval. Thus, $J_n \le K$, and $J_\infty \le K$, so $J_\infty = K$ on $\Omega_{\Q_2}$. Equation \eqref{eqn:Jinf_mean} computes the mean number of jump points. 
\end{proof}

\begin{proof}[Proof of Theorem \ref{thm:SH_jump_process}]
We verify the conditions of Theorem \ref{thm:jump process condition} for the process 
\[
\{G_{\dir_0 + \dir}^\sigma(x,x+y):\dir \ge 0\},
\]
where $\dir_0$ is arbitrary.  This  process is nondecreasing by Proposition \ref{prop:SH_cons}\ref{itm:SH_mont}, and has stationary increments by Theorem \ref{thm:SH_dist_invar}\ref{itm:SH_inc_stat}. 

Condition \ref{finitemean} of Theorem \ref{thm:jump process condition} follows because each $G^\sigma_\dir$ is a Brownian motion with diffusivity $\sigma$ and drift $\sigma^2 \dir$. In particular, for $\dir,y > 0$,
\be \label{GExp}
\Ee^\sigma[G_{\dir_0 + \dir}(x,x+y) - G_{\dir_0}(x,x+y)] = \sigma^2 \dir y < \infty.   
\ee
Theorem \ref{thm:SH_inc_dist} states that $\Pp^\sigma(G_{\dir_0 + \dir}(x,x+y) = G_{\dir_0}(x,x+y))  \in (0,1)$ for all $\dir,y > 0$. Hence, Condition \ref{p0>0} is satisfied. 

Lastly, we verify Condition \ref{liminf assumption}, which, combined with \eqref{GExp} allows us to compute the mean number of jumps in the interval $[\dir_0,\dir_0 + \dir]$. Observe that, because \[\Pp^\sigma[G_{\dir_0 + \dir}(x,x+y) -G_{\dir_0}(x,x+y) \ge 0] = 1,\] 
and using \eqref{split_prob},
\begin{align*}
&\quad \; \Ee^\sigma[G_{\dir_0 + \dir}(x,x+y) -G_{\dir_0}(x,x+y)|G_{\dir_0 + \dir}(x,x+y) >G_{\dir_0}(x,x+y)  ] \\
&= \f{\Ee^\sigma[G_{\dir_0 + \dir}(x,x+y) -G_{\dir_0}(x,x+y)]}{\Pp^\sigma[G_{\dir_0 + \dir}(x,x+y) > G_{\dir_0}(x,x+y)]} \\
&= \f{\sigma^2 \dir y}{1 - (2 + \dir^2 \sigma^2 y) \Phi\bigl(-\dir \sqrt{\f{\sigma^2 y}{2}}\bigr) + \dir \sqrt{\f{\sigma^2 y}{\pi}}\exp(-\dir^2 \sigma^2 y /4) }.
\end{align*}
An application of L'H\^opital's rule gives
\begin{align*}
&\quad \; \lim_{\dir \searrow 0} \Ee^\sigma[G_{\dir_0 + \dir}(x,x+y) -G_{\dir_0}(x,x+y)|G_{\dir_0 + \dir}(x,x+y) >G_{\dir_0}(x,x+y)  ]  \\
&=\lim_{\dir \searrow 0} \f{\sigma^2 y}{-2\dir \sigma^2 y \Phi\bigl(-\dir \sqrt{\f{\sigma^2 y}{2}}\bigr)  +2\sqrt{\f{\sigma^2 y}{\pi}} } = \f{\sqrt{\pi \sigma^2 y}}{2} > 0
\end{align*}
Along with \eqref{GExp}, Theorem \ref{thm:jump process condition} gives us
\[
\Ee^\sigma[\# \{\eta \in [\dir_0,\dir]: G_{\eta- }(x,x + y) < G_{\eta +}(x,x+y)\}] = 2 \dir\sqrt{\f{\sigma^2 y}{\pi}}. 
\]

The fact that there are infinitely many jumps over all $\dir \in \R$ is as follows: We recall that $\dir \mapsto G^\sigma_{\dir}(x,x+y)$ is nondecreasing. Thus, it has almost sure limits as $\dir \to \pm \infty$. Since $G^\sigma_{\dir}(x,x+y)$ is Gaussian with mean $\sigma^2 \dir y$ and variance $\sigma^2 y$, for all $z \in \R$, Markov's inequality implies, for any $z \in \R$
\begin{align*}
\lim_{\dir \to + \infty} \Pp^\sigma(G_\dir(x,x+y) \le z) &\le \lim_{\dir \to + \infty} \Pp^\sigma(|G_\dir(x,x+y) -\sigma^2 \dir y| > \sigma^2 \dir y - z)  \\
&\le \lim_{\dir \to +\infty} \f{\sigma^2 y}{(\sigma^2 \dir y - z)^2} = 0.
\end{align*}
so $\lim_{\dir \to \pm \infty} G^\sigma_\dir(x,x+y) = \pm\infty$, almost surely. Similarly, $\lim_{\dir \to -\infty} G^\sigma_\dir(x,x+y) = -\infty$. Since the set of jumps is discrete, there must be infinitely many of them on the real line, and they are unbounded for both positive and negative directions. 
\end{proof}

\begin{proof}[Proof of Theorem \ref{thm:SH_sticky_thm}]
 \noindent \textbf{Item \ref{itm:SH_dist_grows}:} The monotonicity of $y \mapsto f(y) := G_{\dir_2}(x,x+y) - G_{\dir_1}(x,x+y)$ follows by $G_{\dir_1} \li G_{\dir_2}$ (Proposition \ref{prop:SH_cons}\ref{itm:SH_mont}): for $w < y$,
\[
f(w,y) = G_{\dir_2}(x + w,x+  y) - G_{\dir_1}(x + w,x + y) \ge 0.
\]

 \noindent \textbf{Item \ref{itm:SH_set_contain}}
Because $G_{\dir-} \li G_{\dir +}$, Lemma \ref{lem:ext_mont} implies that for $a \le x < y \le b$,
\[
0 \le G_{\dir +}(x,y) - G_{\dir -}(x,y) \le G_{\dir +}(a,b) - G_{\dir -}(a,b),
\]
giving the inclusion $\XiSH(x,y) \subseteq \XiSH(a,b)$.

 \noindent \textbf{Items \ref{itm:SH_stick}--\ref{itm:SH_all_jump}:} For a fixed $x \neq y$, Theorem \ref{thm:SH_jump_process} states that, $\Pp^\sigma$ almost surely, $\XiSH(a,b)$ is a discrete and infinite set, and $\dir \mapsto G_\dir(a,b)$ is a right-continuous step function. Thus, $\Pp^\sigma$-almost surely, for any $\dir \in \R$, there exists $\ve > 0$ such that, if $\dir - \ve < \alpha < \dir < \beta < \dir + \ve$, 
\be \label{Gab=}
G_{\alpha}(a,b) = G_{\dir -}(a,b),\qquad\text{and}\qquad G_{\dir}(a,b) = G_{\beta}(a,b). 
\ee
We consider the $\Pp^\sigma$ almost sure event such a $\ve > 0$ exists for all $\dir \in \R$ and each rational pair $a\neq b$.  To show Item \ref{itm:SH_stick}, it suffices to take the compact set to be $[a,b]$ for rational $a < b$. Then, for every $\dir \in \R$, there exists $\ve = \ve(\dir, a,b)$ so that for $\dir - \ve < \alpha < \dir < \beta < \dir + \ve$ \eqref{Gab=} holds. Then, by $G_{\dir} \li G_{\beta}$ and Lemma \ref{lem:ext_mont}, for $a \le x \le y \le b$, and all such $\alpha,\beta$, 
\be \label{Gbdir-}
0 \le G_{\beta}(x,y) - G_{\dir}(x,y) \le G_{\beta}(a,b) - G_\dir(a,b) = 0.
\ee
A symmetric argument shows that $G_{\alpha}(x,y) = G_{\dir-}(x,y)$. 

The result of Item \ref{itm:SH_all_jump} that $\XiSH(x,y)$ is countably infinite and discrete for all $x < y$ is as follows. Without loss of generality, assume $x < y$. Choose rational values $a,b,c,d$ with $a < x < c < d  < y < b$. Then Item \ref{itm:SH_set_contain} implies
\[
\XiSH(c,d) \subseteq \XiSH(x,y) \subseteq \XiSH(a,b),
\]
and recall that $\XiSH(b,c)$ and $\Xi(a,d)$ are both infinite and discrete. Furthermore, \eqref{Gbdir-} implies that $\dir\mapsto G_\dir(x,y)$ can only increase when $\dir \mapsto G_{\dir}(a,b)$ increases, so $\dir \mapsto G_\dir(x,y)$ is also a step function. The limits as $\dir \to \pm \infty$ follow by a similar monotonicity argument.  

 \noindent \textbf{Item \ref{itm:SH_split_Xi}:} We handle the case for $S_x^+(\dir_1,\dir_2)$, and the other follows analogously. Recall the definition \eqref{Split_pts_def}
\[
S_x^+(\dir_1,\dir_2) = \inf\{y > 0: G_{\dir_2}(x,x+y) > G_{\dir_1}(x,x+y) \}.
\]
For shorthand, let $S = S_x^+(\dir_1,\dir_2)$. By definition of the nondecreasing property in Item \ref{itm:SH_dist_grows}, the function $\dir \mapsto G_\dir(x,x + S)$ is constant in the interval $[\dir_1,\dir_2]$, but for every $\ve > 0$, the function $\dir \mapsto G_{\dir}(x,x + S + \ve)$ is not constant in the interval $[\dir_1,\dir_2]$. Since $\dir \mapsto G_{\dir}(x,x + S + \ve)$ is a step function by Item \ref{itm:SH_all_jump}, it follows that $[\dir_1,\dir_2] \cap \XiSH(x,x + S) = \varnothing$, but $[\dir_1,\dir_2] \cap \XiSH(x,x + S + \ve) \ne \varnothing$ for all $\ve > 0$.  Since each $\XiSH(x,x + S + \ve)$ contains only finitely many values in $[\dir_1,\dir_2]$ (Item \ref{itm:SH_all_jump}), and since $\XiSH(x,x + S + \ve)$ decreases as $\ve$ decreases (Item \ref{itm:SH_set_contain}), there must be some $\dir^\star \in [\dir_1,\dir_2]$ so that, for every $\ve > 0$ $\dir^\star \in [\dir_1,\dir_2] \cap \XiSH(x,x + S + \ve)$. Then, $G_{\dir^\star -}(x,x + S) = G_{\dir^\star +}(x,x + S)$, but $G_{\dir^\star -}(x,x + S + \ve) < G_{\dir^\star +}(x,x + S + \ve)$ for all $\ve > 0$. Hence, 
\[
S = \inf\{y > 0: G_{\dir^\star +}(x,x + y) > G_{\dir^\star -}(x,x + y)\}.
\]

 \noindent \textbf{Item \ref{itm:SH_Xi_dense}:} The density of $\XiSH$ follows directly from Item \ref{itm:SH_split_Xi}.
\end{proof}

\section{Random measures and Palm kernels} \label{sec:SHPalm}
\begin{figure}[t]
\centering
\includegraphics[height = 3in]{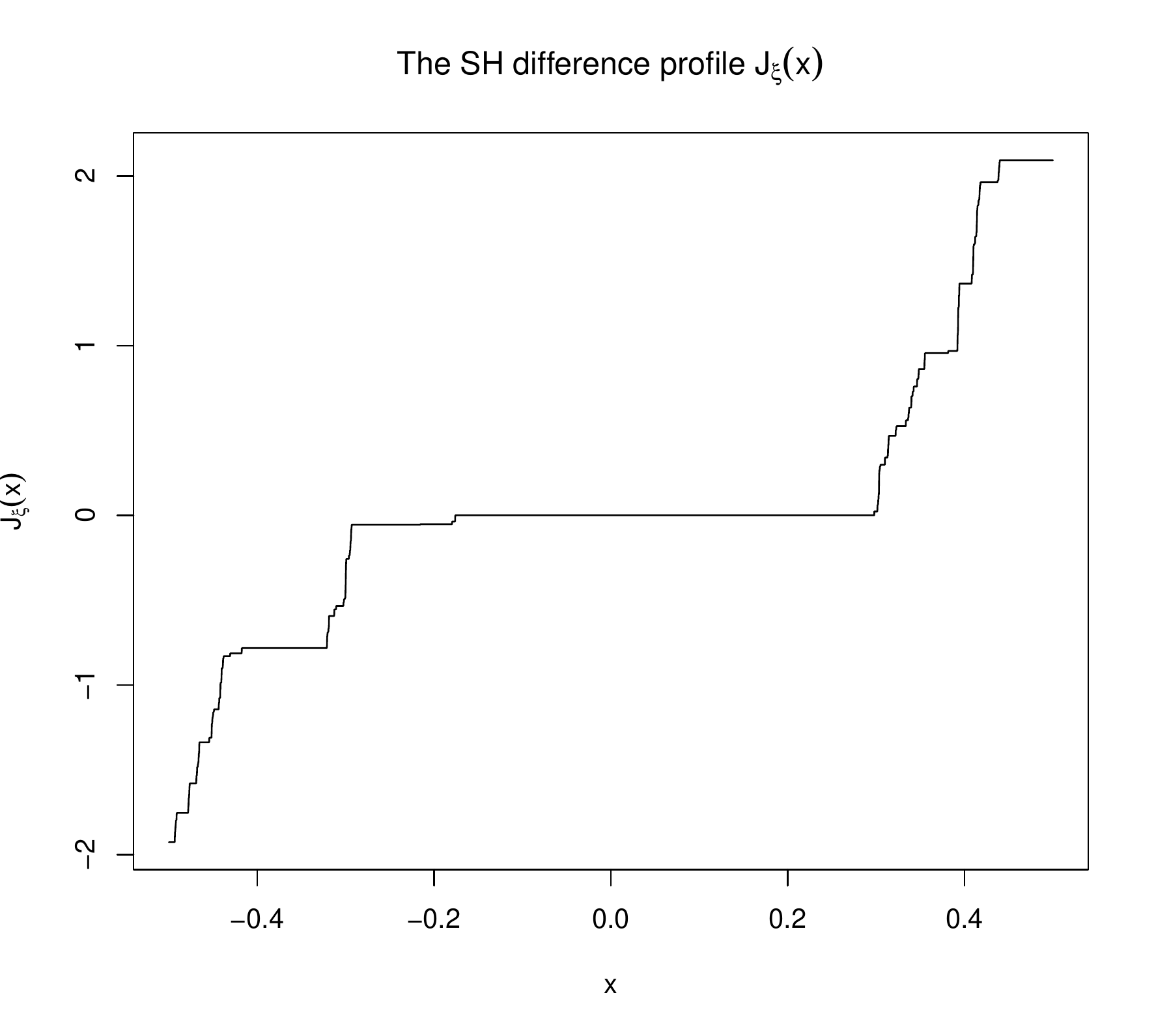}
\caption{\small The SH difference profile $J_\dir(x)$. The function vanishes   in a nondegenerate  random neighborhood
of $x = 0$ and evolves as two independent Brownian local times to the left and right.} 
\label{fig:loc_time}
\end{figure}

For $G \in D(\R,C(\R))$, define the process of jumps 
\begin{equation} \label{Jdir}
	\shdif := \{\shdif_\dir\}_{\dir\tsp\in\tsp\R} = \{G_\dir-G_{\dir-}\}_{\dir\tsp\in\tsp\R}
\end{equation}
By Theorem \ref{thm:SH_sticky_thm}, for each $\sigma > 0$, $\Pp^\sigma$-almost surely, either
$\shdif_\dir$ vanishes identically (when $\dir \notin \XiSH$ or $\shdif_\dir$ is a   nondecreasing continuous function that vanishes in a nondegenerate (random) neighborhood of the origin. We use $J^\sigma = \{J^\sigma_\dir\}_{\dir\in \R}$ to denote the process when $G = G^\sigma$, without reference to the measure $\Pp^\sigma$.   By Theorem \ref{thm:SH_dist_invar}\ref{itm:shinv},\ref{itm:SH_inc_stat},
\be\label{H65}
\{\shdif^\sigma_{\dir+\eta}(y+x)-\shdif^\sigma_{\dir+\eta}(y): x \in \R\}_{\dir \in \R}\;\deq\;\{\shdif^\sigma_{\dir}(x):x \in \R\}_{\dir \in \R}  \quad \forall\tsp y, \eta\in\R.  
\ee

The goal of this section is to prove the following. 
\begin{theorem} \label{thm:BusePalm}
Let $\sigma > 0$, and let $G^\sigma$ be the SH. For $\dir \in \R$  and $\shdif^\sigma_\dir$ defined above, let
\[
\tau_\dir^\sigma = \inf\{x > 0: \shdif^\sigma_{\dir}(x) > 0 \}\qquad\text{and}\qquad \bck{\tau^\sigma_\dir} = \inf\{x > 0: -\shdif^\sigma_{\dir}(-x) > 0\}
\]
denote the points to the right and left of the origin beyond which $G_{\dir +}$ and $G_{\dir -}$ separate, if ever. Then,  conditionally on $\dir \in \Xi_{G^\sigma}$ in the appropriate Palm sense, the restarted functions 
\[
x\mapsto \shdif^\sigma_{\dir}(x + \tau_\dir^\sigma) -\shdif^\sigma_{\dir}(\tau_\dir^\sigma)
\quad\text{ and } \quad 
x \mapsto -\shdif_{\dir}^\sigma(-x - \bck{\tau_\dir^\sigma}) + \shdif^\sigma_{\dir}(-\bck{\tau_\dir^\sigma}),\quad x \in\R_{\ge0},  
\]
 are equal in distribution to two independent running maximums of Brownian motion with diffusivity $\sigma^2$ and zero drift. In particular, they are equal in distribution to two independent appropriately normalized versions of Brownian local time. See Figure \ref{fig:loc_time}.
\end{theorem}

Defining the appropriate sense of Palm conditioning requires some care. The set $\Xi_{G^\sigma}$ is almost surely dense in $\R$, so we instead condition on pairs $(\dir,\tau_\dir^\sigma)$, where $\dir \in \Xi_{G^\sigma}$. A more precise version of Theorem \ref{thm:BusePalm}  is proved as Theorem \ref{thm:indep_loc}.

Theorem \ref{thm:BusePalm} has deep geometric significance for the study of semi-infinite geodesics in the directed landscape in Chapter \ref{chap:Buse}. For $\dir \in \XiSH$, the function $x \mapsto J_\dir(x)$ is a nondecreasing function and therefore defines a random Lebesgue-Stieltjes measure. The support of this measure, up to the removal of a countable set, is interpreted in Theorem \ref{thm:random_supp} as the set of points with two disjoint semi-infinite geodesics in direction $\dir$.

\subsection{The difference profile for positive $x$}
We first study the functions $\shdif_\dir(x)$ for $x\ge0$, under the measure $\Pp^\sigma$.  Approximate $\shdif$ by a process $\shdif^{N}$ defined on dyadic rational $\dir$. For $N\in \Z_{>0}$ let 
\begin{equation}\label{Hd} 
	\shdif^{N}_{\dir_i}=G_{\dir_i}-G_{\dir_{i-1}} \qquad\text{for } \  \dir_i=\dir^N_i=i2^{-N} \ \text{ and } \  i\in \Z.
\end{equation}For $i \in \Z$, let
\begin{equation}\label{taui}
	\tau_{\dir_i}^{N}= \inf\{x > 0: \shdif^{N}_{\dir_i}(x) > 0\} = S_0^+(\dir_{i - 1},\dir_i).
\end{equation}
Since the $G_{\dir_i}$ have different drifts for different values of $i$, $\tau_{\dir_i}^{N}<\infty$ almost surely. For $f\in C(\R)$ and $\tau \in \R$,   let 
$[f]^{\tau}\in C(\R_{\ge0})$ denote the restarted function
\begin{equation} \label{shiftnot}
[f]^{\tau}(x) =f(\tau+x) -f(\tau) \  \text{ for }  x\in [0,\infty) .
\end{equation}

Denote by $\mathcal{D}_{\alpha}^\sigma$
the distribution on $C(\R_{\ge0})$  of  the running maximum of a Brownian motion with drift $\alpha\in\R$ and diffusivity $\sqrt{2\sigma^2}$. That is, if $X$ denotes standard Brownian motion, then 
\[   \mathcal{D}_{\alpha}^\sigma(A)  =  \Pp\bigl\{  \bigl[ \sup_{0\leq u\leq s} \sqrt{2\sigma^2} X(u)+\alpha u\bigr]_{s\in[0,\infty)} \in A  \bigr\}  
\]
for Borel sets $A\subset C(\R_{\ge0})$.  When the drift vanishes  ($\alpha=0$)  we  abbreviate  $\mathcal{D}^\sigma=\mathcal{D}_0^\sigma$.

\begin{lemma} \label{lem:WBSM}
On a probability space $(\Omega, \F,P)$, 
Let $Y =\{Y(x): x \ge 0\}$ be a Brownian motion with drift $\alpha$ and diffusivity $\sqrt{2\sigma^2}$. Let $W$ be an almost surely negative random variable independent of $B_\alpha^\sigma$. Let 
\[\theta = \inf\{x > 0:  W+B_{\alpha}^\sigma(x) \ge 0 \}.
\]
Then, for all $x > 0$,
\be \label{condD}
P\Big(\Big[\sup_{0\leq s\leq \theta+u} W+Y(s)\Big]^+_{u\in[0,\infty)}\in \aabullet \,\Big|\,\theta=x\Big) = \D_\alpha^\sigma(\aabullet).
\ee
In particular, 
\be \label{uncondD}
P\Big(\Big[\sup_{0\leq s\leq \theta+u} W+Y(s)\Big]^+_{u\in[0,\infty)}\in \aabullet\Big) = \D_\alpha^\sigma(\aabullet)
\ee
\end{lemma}
\begin{proof}

Let $A\in\mathcal B(C(\R_{\ge0}))$ and  $\theta>0$.  Below, notice that  $Y(\theta)=-W$. Then, noting that $\theta$ is a stopping time with respect to the filtration $\mathcal{F}_y=\sigma\big(W,\{Y(x)\}_{x\in[0,y]}\big)$, we use the strong Markov property to restart at time $\theta$. 
	\begin{equation}\label{Oh2}
	\begin{aligned}
	&\quad\; P\Big(\Big[\sup_{0\leq s\leq \theta+u} W+Y(s)\Big]^+_{u\in[0,\infty)}\in A \,\Big|\,\theta=x\Big)\\
	&=P\Big( \Bigl[\,\sup_{\theta\leq s\leq \theta+u} W+Y(s)\Bigr]_{u\in[0,\infty)} \in A \,\Big|\,\theta=x\Big)\\
	&=P\Big( \Bigl[\,\sup_{0\leq s\leq u} Y(\theta+s)- Y(\theta)\Bigr]_{u\in[0,\infty)} \in A \,\Big|\,\theta=x\Big)\\
	&=P\Big( \Bigl[\,\sup_{0 \leq s\leq u} Y(s)\Bigr]_{u\in[0,\infty)} \in A\Big)=\mathcal{D}_{\alpha}^\sigma(A).
	\end{aligned}
	\end{equation}
	The first equality came from 
	\[  \sup_{s\in[0,\theta]}\{W+Y(s)\}  =  0 \le \sup_{s\in[\theta\!, \, \theta+u]}\{W+Y(s)\}\quad \text{for all } u\geq0.
	\] 
	The claim of \eqref{condD} has now been verified. Equation \eqref{uncondD} follows because  
	\begin{equation*}
	\begin{aligned}	&\quad \;P\Big(\Big[\sup_{0\leq s\leq \theta+u} W+Y(s)\Big]^+_{u\in[0,\infty)}\in A\Big) \\	&= \int P\Big(\Big[\sup_{0\leq s\leq \theta+u} W+Y(s)\Big]^+_{u\in[0,\infty)}\in A \,\Big|\,\theta=x\Big)\, d\theta(x).  	\end{aligned}    \qedhere 	
 \end{equation*}
\end{proof}

\begin{corollary} \label{cor:discrete_restart}
	Let $\sigma > 0$ $\alpha_N=\sigma^2 2^{-N} $. Then for all $i\in \Z$ and  $x>0$,  
	\begin{equation}\label{eq4}
\Pp^\sigma\big(\big[\shdif^{N}_{\dir_i}\big]^{\tau^{N}_{\dir_i}}\in \aabullet\,\,\big|\,\tau^{N}_{\dir_i}=x\big)= \mathcal{D}^\sigma_{\alpha_N}(\aabullet) .
	\end{equation}
\end{corollary}
\begin{proof}
	From the definition of the finite-dimensional marginals of the stationary horizon (Proposition \ref{prop:SH_cons}\ref{itm:SH_dist}), we have 
 \be \label{Gsigdef}
 (G^\sigma_{\dir_{i - 1}}, G^{\sigma}_{\dir_i}) \deq (Z^1,D(Z^2,Z^1)),
 \ee
 where $Z^1,Z^2$ are independent two-sided Brownian motions, each with diffusivity $\sigma$ and drifts $\sigma^2 \dir_{i - 1},\sigma^2 \dir_i$, respectively. Then, $Y := Z^2 - Z^1$ is a Brownian motion with diffusivity $\sqrt{2\sigma^2}$ and drift $\sigma^2(\dir_i - \dir_{i - 1}) = \sigma^2 2^{-N} $. Set $\wt \shdif^N(y) = D(Z^2,Z^1)(y) - Z^1(y)$. By Lemma \ref{DRcont}, 
 \be \label{Jnrep}
 \begin{aligned}
\wt \shdif^N(y) = \Bigl(\sup_{0 \le x \le y}\{Y(x)\} - \sup_{-\infty < x \le 0}\{Y(x)\} \Bigr)^+ 
 = \Bigl[\sup_{0 \le x \le y} W +  Y(x)\Bigr],
 \end{aligned}
 \ee
 where $W = - \sup_{-\infty < x \le 0}\{Y(x)\}$ is an almost surely negative random variable, independent of $\{Y(x):x \ge 0\}$.
Define
	\begin{equation}\label{thetaN}
	\theta^N = \inf\{x > 0: \wt J^{N}(x) > 0 \}= \inf\{x > 0:  Y(x) \ge 0 \}.
	\end{equation}
	Hence, now $(\shdif^{N}_{\dir_i},\tau^{N}_{\dir_i})\deq (\wt{J}^{N},\theta^N)$, and the corollary now follows from Lemma \ref{lem:WBSM}.
\end{proof}

\noindent For $\dir\in \R$ let  
\begin{equation}\label{eq9}
\tau_\dir=\inf\{x\ge 0:\shdif_\dir(x) > 0\}. 
\end{equation}
The connection with the discrete counterpart in \eqref{taui} is 
\be\label{tau45}  
\tau^N_{\dir_i}= \min\{ \tau_\dir:  \dir\in(\dir_{i-1}, \dir_i]\}.  
\ee
On the space  $\R_{\ge0}\times \R$ define the random  point measure and its  mean measure  
\begin{equation}\label{SHpp}  
	\SHpp=\sum_{(\tau_\dir,\dir):\tau_{\dir}<\infty} \delta_{(\tau_\dir,\dir)}
	\quad\text{ and } \quad 
	\lambda_{\SHpp}^\sigma(\aabullet):=
	\E^\sigma[ \tspb{\SHpp}(\aabullet)\tspb]. 
\end{equation}
The point process $\SHpp$ records the jump  directions $\dir$  and the points $\tau_\dir$ where $G_\dir$ and  $G_{\dir-}$ separate on $\R_{\ge0}$.   Theorem \ref{thm:SH_sticky_thm} ensures that $\SHpp$ and $\lambda_{\SHpp}$ are locally finite, $\Pp^\sigma$-almost surely. 
It will cause no confusion to use the same  symbol $\Gamma$ to denote the random set  \[ 
	\SHpp=\{(\tau_\dir,\dir): \dir\in\R, \tau_{\dir}<\infty\} . 
\] 
Then also $\lambda^\sigma_{\SHpp}(\aabullet)=
	\E^\sigma(\tspb|\SHpp\cap \aabullet|\tspb)$ where $| \aabullet |$ denotes cardinality.  
The counterparts for the approximating process are 
\begin{equation} \label{SHppdis}
\SHpp^{(N)}=\{(\tau^{N}_{\dir_i},\dir_i):i\in \Z, \tau^{N}_{\dir_i}<\infty\}
\quad\text{and}\quad 
\lambda_{\SHpp}^{(N)}(\aabullet):=\mathbb{E}^\sigma(|\SHpp^{(N)}\cap  \aabullet|),
\end{equation}
where, at the risk of slightly abusing notation, we have dropped the dependence of $\sigma$ in the notation for $\lambda_{\SHpp}^{(N)}$.

The dyadic partition in \eqref{Hd} imposes a certain monotonicity as $N$ increases: $\tau_\xi$ values can be added but not removed. The $\xi$-coordinates that are not dyadic rationals  move as the partition refines.  So we have 
\be\label{tau67} 
\{  \tau^{N}_{\dir_i} : (\tau^{N}_{\dir_i},\dir_i) \in \SHpp^{(N)}\} 
\subset 
\{  \tau^{N+1}_{\dir_i} : (\tau^{N+1}_{\dir_i},\dir_i) \in \SHpp^{(N+1)}\}
\subset 
\{  \tau_{\dir} : (\tau_{\dir},\dir) \in \SHpp\}.  
\ee


\begin{lemma}\label{lm:ac}
	For $\sigma > 0$, the measure $\lambda_{\SHpp}^\sigma$ and Lebesgue measure $m$ are mutually absolutely continuous on $\R_{>0}\times\R$.  
	The Radon-Nikodym derivative is given by 
	\be\label{RN128} 
	\f{d\lambda_{\SHpp}^\sigma}{dm}(\tau,\dir) = \f{\sigma}{\sqrt{\pi \tau}} \qquad \text{for } \ (\tau,\dir)\in\R_{>0}\times\R .
	\ee
\end{lemma}
\begin{proof}
Let $\dir \in \R,y > 0$, and $\ve > 0$. Theorem \ref{thm:SH_jump_process} states that 
\[
\Ee^\sigma[\# \{\eta \in (\dir,\dir + \ve]: G_{\eta- }(y) < G_{\eta +}(y)\}] = 2 \ve \sqrt{\f{\sigma^2 y}{\pi}}.
\]

Further, observe that, for $\dir \in \R,\tau > 0$, and $\delta,\ve > 0$. 
\begin{multline*}
 \; \lambda^\sigma_{\SHpp}\big((\tau,\tau+\delta]\times (\dir,\dir+\ve]\big) \\
= \Ee^\sigma[\# \{\eta \in (\dir,\dir + \ve]: G_{\eta- }(\tau + \delta) < G_{\eta +}(\tau + \delta)\}] \\- \Ee^\sigma[\# \{\eta \in (\dir,\dir + \ve]: G_{\eta- }(\tau) < G_{\eta +}(\tau)\}]
\end{multline*}
because $(\tau_\eta,\eta) \in \SHpp \cap(\tau,\tau+\delta]\times (\dir,\dir+\ve]$ if and only if $\eta \in \cap (\dir,\dir + \ve]$, $\tau_\eta < \infty$, and the splitting of $G_{\eta-}$ and $G_{\eta}$ occurred in $(0,\tau + \delta]$, but not in $[0,\tau]$. 
Hence, 
\[ 
\lambda^\sigma_{\SHpp}\big((\tau,\tau+\delta]\times (\dir,\dir+\ve]\big) = \f{2 \sigma \ve}{\sqrt \pi}(\sqrt{\tau + \delta} - \sqrt{\tau}) = \int_{\dir}^{\dir + \ve} \int_{\tau}^{\tau + \delta} \f{\sigma}{\sqrt{\pi x}}\,dx \,d\alpha.  \qedhere
 \] 
\end{proof}


By \eqref{RN128},  $\lambda_{\SHpp}$ does not have a finite marginal on the $\dir$-component, as expected since the set $\XiSH$ is dense, $\Pp^\sigma$ almost surely. Below we do Palm conditioning on the pair $(\tau_{\dir},\dir)\in\R_{>0}\times\XiSH$ and not on the jump directions $\dir\in\XiSH$ alone.

\begin{lemma}\label{lm:SH5}  Let  $A \subseteq C(\R_{\ge0})$ be a Borel set. Then for any open rectangle $R=(a,b)\times(c,d)\subseteq \R_{\ge0}\times\R$, 
	\begin{equation}\label{Omr2}
	\mathbb{E^\sigma}\Bigl[ \; \sum_{(\tau,\,\dir)\,\in\,  \SHpp}\ind_{A}([\shdif_\dir]^\tau)\tspb\ind_{R}(\tau,\dir)\Bigr] =\lambda^\sigma_{\SHpp}(R)\tspb\mathcal{D}^\sigma(A). 
	\end{equation}
\end{lemma}
\begin{proof}  It suffices to prove \eqref{Omr2} for  continuity sets $A$ of the distribution $\mathcal{D}^\sigma$ of the type $ A = \{f\in C(\R_{\ge0}): f|_{[0,k]} \in A_k\}$
for $k > 0$ and Borel $A_k \subseteq C[0,k]$.  
Such sets form a $\pi$-system that generates the Borel $\sigma$-algebra of $C(\R_{\ge0})$.

We prove a discrete counterpart of \eqref{Omr2} for $\shdif^{N}$.	Condition on $\tau^N_{\dir_i}$ and  use \eqref{eq4}:  
	\begin{equation}\label{Omr3}
	\begin{aligned}
	&\E^\sigma\Big(\sum_{(\tau^N_{\dir_i}\!,\,\dir_i)\in R\cap\SHpp^{(N)}} \ind_A\big([\shdif^{N}_{\dir_i}]^{\tau^{N}_{\dir_i}}\big)\Big)
=\E^\sigma\Big(\sum_{\dir_i\in(c,d)} \ind_A\big([\shdif^{N}_{\dir_i}]^{\tau^{N}_{\dir_i}}\big)\ind_{(a,b)}\big(\tau^{N}_{\dir_i}\big)\Big)\\
	&=\sum_{\dir_i\in(c,d)} \E^\sigma\bigg(\ind_{(a,b)}\big(\tau^{N}_{\dir_i}\big)\, \E^\sigma\Big[\ind_A\big([\shdif^{N}_{\dir_i}]^{\tau^{N}_{\dir_i}}\big)\tspb\Big|\tspb\tau^N_{\dir_i}\Big]\bigg) \\&
	\stackrel{\eqref{eq4}}{=}\sum_{\dir_i\in(c,d)}\Pp\big(\tau^{N}_{\dir_i}\in (a,b)\big)\mathcal{D}_{\alpha_N}^\sigma(A)
 =\mathcal{D}_{\alpha_N}^\sigma(A)\tspb \lambda_{\SHpp}^{(N)}(R).
	\end{aligned}
	\end{equation}
	
	To conclude the proof, we check that \eqref{Omr2} arises as we  let $N\to\infty$ in the first and last member of the string of equalities above.  $\mathcal{D}^\sigma_{\alpha_N}(A)\to\mathcal{D}^\sigma(A)$ by the  continuity of $\alpha\mapsto \mathcal{D}^\sigma_\alpha$ in the weak topology of measures on $C(\R)$ with respect to uniform convergence on compact sets, and the assumption that $A$ is a continuity set.

	As an intermediate step, we verify that  $\forall k > 0$, 
	$\ind_{\mathcal{U}_N^k}\to 1$ almost surely 
	for the events 
 \newpage 
	\begin{align}\label{Uset}  
	    \mathcal{U}_N^k&=\bigl\{\,|\SHpp^{(N)}\cap R| = |\SHpp\cap R| \text{ \,and for every $(\tau,\dir)\in \SHpp\cap R$  there is a unique} \\
	    &\qquad 
	    \text{  $(\tau^N_{\dir_i},\dir_i)\in \SHpp^{(N)}\cap R$
	    such that  $[\shdif^{N}_{\dir_i}]^{\tau^{N}_{\dir_i}}\big|_{[0,k]}=[\shdif_{\dir}]^{\tau_{\dir}}\big|_{[0,k]}$}\bigr\}. \nonumber
	\end{align}
	$\Pp^\sigma$ almost surely,  $\SHpp\cap R$ is finite, and none of its points lie on the boundary of $R$ (Proposition \ref{prop:SH_cons}\ref{itm:SH_cont} and Theorem \ref{thm:SH_jump_process}).  For any such realization, the condition \eqref{Uset} holds when (i) all points $(\tau_\xi, \xi)\in\SHpp\cap R$ lie in distinct rectangles $(a,b)\times(\xi_{i-1}, \xi_i]\subset(a,b)\times(c,d)$, (ii) when no point $(\tau^N_{\dir_i},\dir_i)\in \SHpp^{(N)}\cap R$ is generated by a point $(\tau_\xi, \xi)\in\SHpp$ outside $R$, and (iii) when $N$ is large enough so that for the unique $i$ with $\dir_i < \dir \le \dir_{i + 1}$, $G_{\dir-}(x) = G_{\dir_i}(x)$ and $G_{\dir+}(x) = G_{\dir_{i + 1}}(x)$ for all $x \in [0,\tau_\dir + k]$. By Theorem \ref{thm:SH_sticky_thm}\ref{itm:SH_all_jump}, this happens for all the finitely many $(\tau,\dir) \in \Gamma \cap R$ when the mesh $2^{-N}$ is fine enough.  Thus, for each  $k > 0$, almost every realization lies eventually in $\mathcal{U}_N^k$.

\smallskip

We prove that 
	$\lambda_{\SHpp}^{(N)}(R)\to 		\lambda_{\SHpp}(R)$.   The paragraph above gave $|\SHpp^{(N)}\cap R| \to  |\SHpp\cap R|$ almost surely. We also have  $|\SHpp^{(N)}\cap R| \le   |\SHpp\cap ((a,b)\times(c-1,d))|$ because \eqref{tau45} shows that each  point $(\tau^N_{\dir_i},\dir_i)$ that is not matched to a unique point $(\tau_\dir, \dir)\in\SHpp\cap  R$ must be generated by some point $(\tau_\dir, \dir)\in \SHpp\cap  ((a,b)\times(c-1,d))$.  The limit $\lambda_{\SHpp}^{(N)}(R)\to 		\lambda^\sigma_{\SHpp}(R)$ comes now from dominated convergence.

	\smallskip
	
	It remains  to show that 
\[ 
	\E^\sigma\Big(\sum_{(\tau^N_{\dir_i}\!,\,\dir_i)\in R\cap\SHpp^{(N)}} \ind_{ A}([\shdif^{N}_{\dir_i}]^{\tau^{N}_{\dir_i}})\Big)
	\underset{N\to\infty}\longrightarrow
	\E^\sigma\Big(\sum_{(\tau_{\dir},\dir)\in R\cap\SHpp} \ind_{ A}([\shdif_\dir]^{\tau_\dir})\Big).
\] 
This follows by choosing $k > 0$ so that $A$ depends only on the domain $[0,k]$. Then, the difference in absolute values in the display below vanishes on $\mathcal{U}_N^k$.
\[
\begin{aligned}
	    &\lim_{N\to\infty} \E^\sigma\bigg[ \,\Big|\sum_{(\tau^N_{\dir_i}\!,\,\dir_i)\in R\cap\SHpp^{(N)}} \ind_{ A}\bigl([\shdif^{N}_{\dir_i}]^{\tau^{N}_{\dir_i}}\bigr)-\sum_{(\tau_{\dir},\dir)\in R\cap\SHpp} \ind_{ A}\bigl([\shdif_{\dir}]^{\tau_{\dir}}\bigr)\Big|
	    \cdot (\ind_{\mathcal{U}_N^k}+\ind_{(\mathcal{U}_N^k)^c})\bigg]    \\
&\qquad \qquad \le  \lim_{N\to\infty}  
2\tspb \E^\sigma\bigl[ \tspb |\SHpp\cap ((a,b)\times(c-1,d))| \cdot \ind_{(\mathcal{U}^k_N)^c} \bigr] =0,	 
	    \end{aligned}
	    \]
     and the last equality follows by dominated convergence.
\end{proof}

\smallskip 

To describe the distribution of $[\shdif_\xi]^{\tau_\xi}$, we augment the point measure  $\SHpp$ of \eqref{SHpp} to a point measure on    the space  $\R_{\ge0}\times \R\times C(\R_{\ge0})$:   
\begin{equation}\label{SHHpp}  
	\SHHpp=\sum_{(\tau_\dir,\,\dir)\, \in \,\SHpp} \delta_{(\tau_\dir,\,\dir,\, [\shdif_\xi]^{{\scaleobj{1.5}{\tau}}_{\!\!\xi}})}. 
\end{equation}
The \textit{Palm kernel} of $[\shdif_\xi]^{\tau_\xi}$  with respect to $\SHpp$ is the stochastic kernel $Q^\sigma$  from  $\R_{\ge0}\times \R$ into  $C(\R_{\ge0})$ that satisfies the following identity:  for every bounded Borel function $\Psi$ on $\R_{\ge0}\times \R\times C(\R_{\ge0})$ that is supported on  $B\times C(\R_{\ge0})$ for some bounded Borel set $B\subset \R_{\ge0}\times \R$,  
\begin{equation}\label{Opalm}\begin{aligned} 
\E^\sigma \sum_{(\tau_\dir,\,\dir)\, \in \, B\tspa \cap\tspa\SHpp} \Psi\bigl(\tau_\dir,\,\dir,\, [\shdif_\xi]^{\tau_\xi}\bigr)     
&\; = \; \E^\sigma\!\!\int\limits_{\R_{\ge0}\times \R\times C(\R_{\ge0})} \!\! \Psi(\tau, \xi, h) \,\SHHpp(d\tau, d\xi, dh) \\
&= \;  \int\limits_{\R_{\ge0}\times\mathbb{R}} \lambda^\sigma_{\SHpp}(d\tau, d\xi)
\int\limits_{C(\R_{\ge0})}  Q^\sigma(\tau, \xi, dh)\, \Psi(\tau, \xi, h)  .
\end{aligned} \end{equation}
For more background on the general theory of Palm conditioning, we refer the reader to \cite{Kallenberg-book}. The first equality above is a restatement of the definition of $\SHHpp$ and included to give greater clarity to the next theorem.  
The key result of this section is this  characterization of $Q^\sigma$.

\begin{theorem}\label{thm:Lac}
For 
Lebesgue-almost every $(\tau, \xi)$,  $Q^\sigma(\tau, \xi, \aabullet)=\mathcal{D^\sigma}(\aabullet)$,  the distribution of the running maximum of a Brownian motion with diffusivity $\sqrt{2\sigma^2}$. 
 \end{theorem}
 
 \begin{proof} This  comes from Lemma \ref{lm:SH5}: 
  take  $\Psi(\tau, \xi, h)=\ind_R(\tau, \xi)\ind_A(h)$  in \eqref{Opalm} and note that the left-hand side of \eqref{Omr2} is exactly the left-hand side of \eqref{Opalm}. Lemma \ref{lm:ac} turns   $\lambda_{\SHpp}$-almost everywhere into Lebesgue-almost everywhere. 
 \end{proof}

\noindent Denote the set of directions $\dir$ for which $G_\dir$ and $G_{\dir-}$ separate on $\R_{\ge0}$   by  
\[
	\begin{aligned}
		\XiSH'&=\{\dir\in\R:\tau_\dir < \infty\}.
	\end{aligned}
\]
We show in Corollary \ref{cor:dcLR} that $\Pp^\sigma(\XiSH' = \XiSH) = 1$ (recall the definition of $\XiSH$ in  \eqref{XiSHdef}), but it is not immediately obvious that these should be the same set. The key to proving equality is the following theorem.

\begin{theorem}\label{thm:Lac5}
Let $A\subseteq C(\R_{\ge0})$ be a Borel  set such that  $\mathcal{D}(A)=0$. Then
\begin{equation}\label{Oeq}
    \Pp^\sigma\big(\exists \dir\in \XiSH' :  [\shdif_\dir]^{\tau_\dir}\in A\big)=0.
\end{equation}
\end{theorem}
\begin{proof}  Let   $R_N=(0,N)\times(-N,N)$.  Since $\xi\in\XiSH'$ means that $\tau_\xi<\infty$, we have 
 \begin{align*} 
    &\Pp^\sigma\big(\exists \dir\in \XiSH :  [\shdif_\dir]^{\tau_\dir}\in A\big)
    =  \lim_{N\to\infty}  \Pp^\sigma\big(\exists \dir\in \XiSH : (\tau_\dir,\dir)\in R_N,  [\shdif_\dir]^{\tau_\dir}\in A\big) \\
& \qquad\quad    \leq  \lim_{N\to\infty}  \mathbb{E^\sigma}\sum_{(\tau,\,\dir)\in  \SHpp}\ind_{A}([\shdif_\dir]^\tau) \tspb\ind_{R_N}(\tau, \xi)   
\overset{\eqref{Omr2}}=  \lim_{N\to\infty}  \lambda_{\SHpp}(R_N) \tspb \mathcal{D}^\sigma(A) =0.
\qquad \ \qedhere  \end{align*} 
\end{proof}

\begin{corollary} \label{cor:dcLR}
For all $\sigma > 0$, $\Pp^\sigma(\XiSH' = \XiSH) = 1$. Furthermore, 
\[
\Pp^\sigma(\forall \dir \in \XiSH\;\; \lim_{x \to \pm \infty} \shdif_\dir(x) = \pm \infty) = 1.
\]
\end{corollary}
\begin{proof}
By Theorem \ref{thm:Lac5} and the associated fact for the running max of a Brownian motion, 
\be \label{XIshto+inf}
\Pp^\sigma(\forall \dir \in \XiSH', \lim_{x \to +\infty} \shdif_\dir(x) = +\infty) = 1.
\ee
Recall the notation $[J_\dir]^m(x) = J_\dir(m + x) - J_\dir(m)$ \eqref{shiftnot}. By definition, 
\[
\XiSH' = \{\dir \in \R: \shdif_\dir(x) \neq 0 \text{ for some }x >0\} \subseteq \{\dir \in \R: \shdif_\dir(x) \neq 0 \text{ for some }x \in \R\} = \XiSH.
\]
 We show that, $\Pp^\sigma$-almost surely, if $\shdif_\dir(x) \neq 0$ for some $x < 0$, then $\shdif_\dir(x) \neq 0$ for some $x > 0$. If not, then there exist $\dir \in \R$ and $m \in \Z_{<0}$ such that $[\shdif_\dir]^{m}|_{[0,\infty)} \neq 0$, but $[\shdif_\dir]^m|_{[-m,\infty)}$ is constant. In particular, $[\shdif_\dir]^{m}|_{[0,\infty)}$ is bounded. Let $\tau^m_\dir = \inf \{x > 0: [\shdif_\dir]^m(x) > 0\}$.  Then, $[\shdif_\dir]^{m}|_{[0,\infty)} \neq 0$ iff  $\tau^m_\dir<\infty$, and we have
\be \label{XiLR}
\begin{aligned}
    & \Pp^\sigma\bigl(\XiSH \neq \XiSH'\bigr) 
    \le \sum_{m \in \Z_{<0}}\Pp^\sigma\bigl(\exists \dir \in \R:  \; \tau^m_\dir<\infty  
    \text{ but }  [\shdif_\dir]^m|_{[0,\infty)} \text{ is bounded} \bigr) = 0. 
\end{aligned}
\ee
The probability equals zero by \eqref{XIshto+inf} because by shift invariance \eqref{H65}, $[\shdif]^m\deq \shdif$.
To finish, \eqref{XIshto+inf} proves the limits for $x \to +\infty$. The limits as $x \to -\infty$ then follow from \eqref{XIshto+inf} and the reflection invariance of Theorem \ref{thm:SH_dist_invar}\ref{itm:SH_reflinv}.
\end{proof}

Let $\nu_f$ denote the Lebesgue-Stieltjes measure of a non-decreasing function $f$ on $\R$.   Denote the  support of $\nu_f$ by $\text{supp}(\nu_f)$. The Hausdorff dimension of a set $A$ is denoted by $\dim_H(A)$ (Recall Section \ref{sec:notat}).

\begin{corollary} \label{cor:SHHaus1/2} Consider the Lebesgue-Stieltjes measure $\nu_{\shdif_\dir}$ for $\dir\in \XiSH$ on the entire real line.  Then, we have 
\begin{equation}\label{eq8}
	\Pp^\sigma\big(\forall \dir\in \XiSH :  \dim_H \big(\text{\rm supp}(\nu_{\shdif_\dir})\big)=1/2\big)=1.
\end{equation}
\end{corollary}
\begin{proof}
	First, note that 
	\[ 
		\big\{\exists \dir\in \XiSH : \text{dim}_H\big(\text{supp}(\nu_{\shdif_\dir})\big)\neq \tfrac12\big\}\subseteq\!\bigcup_{m\in\Z_{\le0}}\!\!\big\{\exists \dir\in \XiSH : \text{dim}_H\big(\text{supp}(\nu_{\shdif_\dir})\cap[m,\infty)\big)\neq \tfrac12\big\}.
	\]  
By shift invariance \eqref{H65}, it is enough to take $m=0$ and show that 
 \[
 \Pp^\sigma\big(\exists \dir\in \XiSH :\text{dim}_H\big(\text{supp}(\nu_{\shdif_\dir})\cap[0,\infty)\big)\neq 1/2\big)=0.
 \]
 This last claim follows from  Theorem \ref{thm:Lac5} because the event in question has zero probability for the running maximum of Brownian motion 
 \cite[Theorem 4.24 and Exercise 4.12]{morters_peres_2010}. 
\end{proof}

\begin{remark}
Representation of the  difference of  Busemann functions   as the running maximum of random walk 
goes back to \cite{bala-busa-sepp-20}. It was used  in \cite{busa-ferr-20} to capture the local universality  of geodesics. 
The representation of the difference profile as the running maximum of  Brownian motion in the point-to-point setup emerges from the Pitman transform   \cite{Ganguly-Hegde-2021,Dauvergne-22}.  Theorem 1 and Corollary 2 in \cite{Ganguly-Hegde-2021} are   point-to-point analogues of  our Theorem \ref{thm:Lac} and Corollary \ref{cor:SHHaus1/2}. Their proof is different from ours. Although an analogue of the Pitman transform exists in the stationary case \cite[Section 3]{Busani-2021}, comparing the running maximum of a Brownian motion to the profile requires different tools in the two settings.  
\end{remark}

\subsection{Decoupling} \label{sec:decoup}  	By Corollary \ref{cor:dcLR},  whenever $\dir$ is a jump direction, the difference profile $x \mapsto J_\dir(x)$ for both positive and negative $x$ are nontrivial.  
We extend   Theorem \ref{thm:Lac} to show that    these two  difference profiles  are independent and equal in distribution.  
We spell out only the   modifications needed in  the arguments of the previous section.    For the difference profile on the left, define for $x\ge0$
\[ 
		\bck{\shdif}_{\xi}(x):=-\shdif_{\xi}(-x) 
\quad\text{and}\quad 
		\bck{\tau_\xi}:=\inf \{x >  0:\bck{\shdif}_{\xi}(x) > 0\}.
\] 
For $N\in \Z_{>0}$ and  $\xi_i$ as in \eqref{Hd}, the discrete approximations are 
\[ 
	\bck{\shdif}^N_{\xi_i}(x):=-\shdif^N_{\xi_i}(-x)  
	\quad\text{and}\quad 
	\bck{\tau_{\xi_i}}^N:=\inf\{x >  0:\bck{\shdif}^N_{\xi_i}(x) > 0\}.
\] 
The measures $\bck{\SHpp}$, $\lambda_{\bck{\SHpp}}$,  $\bck{\SHpp}^{(N)}$, and $\lambda_{\bck{\SHpp}}^{(N)}$ are defined  as in \eqref{SHpp} and \eqref{SHppdis}, but now with   $(\bck{\tau_\dir},\dir)$ and $(\bck{\tau_{\dir_i}},\dir_i)$. 
	Extend the measure $\SHHpp$ of \eqref{SHHpp} with a component for the left profile: 
\[
\SHHpp'=\sum_{(\bck{\tau_\dir},\,\dir)\, \in \,\bck{\SHpp}} \delta_{(\bck{\tau_\dir},\,\dir,\; [\shdif_\xi]^{{\scaleobj{1.5}{\tau}}_{\!\!\dir}}, \; [\bck{\shdif}_\xi]^{\bck{\scaleobj{1.5}{\tau}}_{\!\!\dir}}) }. 
\]
Since  $\tau_\dir < \infty$ if and only if $\bck{\tau_\dir} < \infty$ (Corollary \ref{cor:dcLR}),  it is immaterial whether we sum over  $({\tau}_\dir,\dir)$ or  $(\bck{\tau_\dir},\dir)$.  The latter is more convenient for the next calculations.

	 The \textit{Palm kernel} of $\big([\shdif_\xi]^{\tau_\xi},[\bck{\shdif}_\xi]^{\bck{\tau_\dir}}\big)$  with respect to $\bck{\SHpp}$ is the stochastic kernel $Q_2^\sigma$  from  $\R_{\ge0}\times \R$ into  $C(\R_{\ge0})\times C(\R_{\ge0})$ that satisfies the following identity:  for every bounded Borel function $\Psi$ on $\R_{\ge0}\times \R\times C(\R_{\ge0})\times C(\R_{\ge0})$ that is supported on  $B\times C(\R_{\ge0})\times C(\R_{\ge0})$ for some bounded Borel set $B\subset \R_{\ge0}\times \R$,  
	\begin{equation}\label{Opalm2}
	\begin{aligned} 
	&\quad \,\E^\sigma \Bigl[ \; \sum_{(\bck{\tau_\dir},\,\dir)\, \in \, B\tspa \cap\tspa\bck{\SHpp}} \Psi\bigl(\bck{\tau_\dir},\,\dir,\, [\shdif_\xi]^{\tau_\xi},[\bck{\shdif}_\xi]^{\bck{\tau_\dir}}\bigr)  \Bigr]  \\ 
			&= \int\limits_{\R_{\ge 0} \times \R} \lambda_{\bck{\SHpp}}(d\bck{\tau}, d\xi)
			\int\limits_{C(\R_{\ge0})\times C(\R_{\ge 0})}  Q_2^\sigma(\bck{\tau}, \xi, dh^1,dh^2)\, \Psi(\bck{\tau}, \xi, h^1,h^2).  
	\end{aligned} 
	\end{equation}

	\begin{theorem} \label{thm:indep_loc}
	For 
	Lebesgue-almost every $(\tau, \xi)$,  $Q_2^\sigma(\tau, \xi, \aabullet)=(\mathcal{D}^\sigma \otimes \mathcal{D}^\sigma)(\aabullet)$,   the product of the distribution of the running maximum of a Brownian motion with diffusivity $\sqrt{2\sigma^2}$. In particular, for any Borel set $A \subseteq C(\R_{\ge 0}) \times C(\R_{\ge 0})$ such that $(\D^\sigma \otimes \D^\sigma)(A) = 0$,
	\[
        \Pp^\sigma\bigl( \exists \dir \in \XiSH: \bigl([\shdif_\xi]^{\tau_\xi},[\bck{\shdif}_\xi]^{\bck{\tau_\dir}}\bigr)   \in A\bigr) = 0.	
	\]
\end{theorem}
\begin{proof}
As in \eqref{Gsigdef} and \eqref{Jnrep},  as functions in $C(\R)$,
\be \label{supdif}
\shdif_{\dir_i}^N(y) \deq \wt \shdif_{\dir_i}^N(y) :=  \sup_{-\infty < x \le y}\{Y(x)\} - \sup_{-\infty < x \le 0}\{Y(x)\},
\ee
where $Y$ is a two-sided Brownian motion with drift $\alpha_N = \sigma^2 2^{-N}$ and diffusivity $\sqrt{2\sigma^2}$. Let $(\Omega,\F,P)$ be a probability space on which $Y$ is defined. 
Define two independent sub $\sigma$-algebras of $\F$:
\[
\F_- = \sigma(Y(x): x \le 0),\qquad\text{and}\qquad \F_+ = \sigma(Y(x): x \ge 0).
\]

\noindent When $y > 0$, \eqref{Jnrep} states that
\be \label{shdif2}
\wt \shdif_{\dir_i}^N(y) = \Bigl[W + \sup_{0 \le x \le y} Y(x)\Bigr]^+,
\ee
where $W = -\sup_{-\infty < x \le 0}\{Y(x)\} \in \F_-$, and $\sup_{0 \le x \le y} Y(x) \in \F_+$. Then, conditional on $\F_-$, $W$ is constant while the law of $Y(x)$ for $x \ge 0$ is unchanged. Let
\[
\wt \tau^N_{\dir_i} = \inf\{x > 0: \wt J^N_{\dir_i}(x) > 0\},\qquad\text{and}\qquad\bck{\wt \tau_{\dir_i}^N} = \inf\{x > 0: -\wt J^N_{\dir_i}(-x) > 0\}
\]
Then, by \eqref{shdif2} and Equation \eqref{uncondD} of Lemma \ref{lem:WBSM} in the special case where $W$ is constant (using the exact same reasoning as in the proof of Corollary \ref{cor:discrete_restart}),
\be \label{F-cond}
P(\big[ \wt \shdif_{\dir_i}^{N}\big]^{\wt \tau_{\dir_i}^{N}} \in \aabullet\,|\,\F_-)= P(\big[ \wt \shdif_{\dir_i}^{N}\big]^{\wt \tau_{\dir_i}^{N}} \in \aabullet\,|\,W) = \D^\sigma_{\alpha_N}(\aabullet).
\ee

 For a fixed $i$,  $\bck{\shdif}^N_{\xi_i}$ and  $\shdif^N_{\xi_i}$ have the same distribution as functions on $\R$. This comes  by   first applying the reflection invariance in Theorem \ref{thm:SH_dist_invar}\ref{itm:SH_reflinv} and then the shift invariance in \eqref{H65}, shifting the directions by $\xi_{i-1}+\xi_i$: 
\begin{align*}
 \bck{\shdif}^N_{\xi_i}(x) &= -{\shdif}^N_{\xi_i}(-x) = - G^\sigma_{\xi_i}(-x) +   G^\sigma_{\xi_{i-1}}(-x) 
 \deq - G^\sigma_{-\xi_i}(x) +   G^\sigma_{-\xi_{i-1}}(x)\\
 &\deq  -  G^\sigma_{\xi_{i-1}}(x) + G^\sigma_{\xi_i}(x)  
 ={\shdif}^N_{\xi_i}(x) . 
\end{align*}
	By \eqref{supdif}, $(\bck{ \wt  \shdif^N},\bck{\wt \tau^N}) \in \F_-$.  We mimic the calculation in \eqref{Omr3}, for two Borel sets $A_1,A_2\subseteq C(\R_{\ge 0})$ and an open rectangle $R=(a,b)\times(c,d)\subseteq \R_{\ge0}\times\R$: 
		\begin{align}\label{Omr4}
			&\quad\;\E^\sigma\Big(\sum_{(\bck{\tau}^N_{\!\!\dir_i},\;\dir_i)\in R\cap\bck{\SHpp}^{(N)}} \ind_{A_1}\big([\shdif^{N}_{\dir_i}]^{\tau^{N}_{\dir_i}}\big)\ind_{A_2}\big([\bck{\shdif}^{N}_{\dir_i}]^{\bck{\tau}^{N}_{\dir_i}}\big)\Big) \\
			&= \sum_{\dir_i\in(c,d)} E\bigg(\ind_{A_2}\big([\bck{ \wt \shdif}^{N}_{\dir_i}]^{\bck{ \wt \tau}^{N}_{\dir_i}}\big)\;\ind_{(a,b)}\big(\bck{\wt \tau^{N}_{\dir_i}}\big) E\Big[\big(\ind_{A_1}\big([\wt \shdif^{N}_{\dir_i}]^{\wt \tau^{N}_{\dir_i}}\big)\tspb\Big|\F_-\Big]\bigg) \nonumber\\
			&\stackrel{\eqref{F-cond}}{=} \sum_{\dir_i\in(c,d)}  E\bigg( E\Big[\ind_{A_2}\big([\bck{\wt \shdif}^{N}_{\dir_i}]^{\bck{\wt \tau}^{N}_{\dir_i}}\big)\;\ind_{(a,b)}\big(\bck{\wt \tau}^{N}_{\dir_i}\big)\tspb\Big|\bck{\wt \tau}_{\dir_i}^N\Big] \bigg)\D^\sigma_{\alpha_N}(A_1) \nonumber\\
			&= \sum_{\dir_i\in(c,d)}  E\bigg(\ind_{(a,b)}\big(\bck{\wt \tau}^{N}_{\dir_i}\big) E\Big[\ind_{A_2}\big([\bck{\wt \shdif}^{N}_{\dir_i}]^{\bck{\wt \tau}^{N}_{\dir_i}}\big)\;\tspb\Big|\bck{\wt \tau}_{\dir_i}^N\Big] \bigg)\D^\sigma_{\alpha_N}(A_1) \nonumber \\
			&\stackrel{\eqref{eq4}}{=} \sum_{\dir_i\in(c,d)} P\big(\bck{\wt \tau}^{N}_{\dir_i}\in (a,b)\big)\D^\sigma_{\alpha_N}(A_1)\D^\sigma_{\alpha_N}(A_2) \nonumber
			\\
   &= \sum_{\dir_i\in(c,d)} \Pp^\sigma\big(\bck{ \tau}^{N}_{\dir_i}\in (a,b)\big)\D^\sigma_{\alpha_N}(A_1)\D^\sigma_{\alpha_N}(A_2) \nonumber
					\\&
			=\mathcal{D}^\sigma_{\alpha_N}(A_1)\mathcal{D}^\sigma_{\alpha_N}(A_2)\tspb \lambda_{\bck{\SHpp}}^{(N)}(R).\nonumber
	\end{align}
As in the proof of Lemma \ref{lm:SH5}, we derive from  the above that 
\begin{equation}\label{Q2}
	\E^\sigma\Big(\sum_{(\bck{\tau}_{\dir},\dir)\in R\cap\bck{\SHpp}} \ind_{A_1}\big([\shdif_{\dir}]^{\tau_{\dir}}\big)\ind_{A_2}\big([\bck{\shdif}_{\dir}]^{\bck{\tau}_{\dir}}\big)\Big)=\mathcal{D}^\sigma(A_1)\mathcal{D}^\sigma(A_2)\tspb \lambda_{\bck{\SHpp}}(R), 
\end{equation}
through the convergence of line \eqref{Omr4} to the left-hand side of \eqref{Q2}. Instead of the events $\mathcal{U}_N^k$ in \eqref{Uset}, consider 
\begin{align*}  
	\mathcal{\wt U}_N^k&=\bigl\{\,|\bck{\SHpp}^{(N)}\cap R| = |\bck{\SHpp}\cap R|, \text{ \,and   $\forall \tspb(\bck{\tau},\dir)\in \bck{\SHpp}\cap R$, \  $\exists$  unique $(\tau^N_{\dir_i},\dir_i)\in \bck{\SHpp}^{(N)}\cap R$
		} \\
		&\qquad\qquad\text{ such that  $[\shdif^{N}_{\dir_i}]^{\tau^{N}_{\dir_i}}\big|_{[0,k]}=[\shdif_{\dir}]^{\tau_{\dir}}\big|_{[0,k]}$ and $[\bck{\shdif}^{N}_{\dir_i}]^{\bck{\tau}^{N}_{\dir_i}}\big|_{[0,k]}=[\bck{\shdif}_{\dir}]^{\bck{\tau}_{\dir}}\big|_{[0,k]}$}\bigr\}. 
\end{align*}
For each $k>0$,   $\ind_{\mathcal{\wt U}_N^k}\to 1$ almost surely, as it did for \eqref{Uset}.   Indeed, there are finitely many pairs $(\bck{\tau},\dir) \in \bck{\SHpp} \cap R$, and each has a finite forward splitting time $\tau$. All these can be confined  in a common compact rectangle.
From here, the proof 
continues as for  Lemma \ref{lm:SH5} and Theorem \ref{thm:Lac5}. 
\end{proof}

\chapter{The stationary horizon as the unique coupled invariant measure for the KPZ fixed point} \label{chap:invar}
\section{Introduction}

In this chapter, we prove that the stationary horizon is the unique invariant coupled distribution and an attractor for the KPZ fixed point (Theorem \ref{thm:invariance_of_SH}). This theorem was originally proved in \cite{Busa-Sepp-Sore-22arXiv} using a limit transition from exponential last-passage percolation. In this chapter, we present an alternate proof using a model called Brownian last-passage percolation (BLPP). There are three main ingredients: invariance of the SH for BLPP (proved in Section \ref{sec:BLPP_stat}), the convergence of BLPP to the directed landscape (DL) from \cite{Directed_Landscape,Dauvergne-Virag-21} (recorded here as Theorem \ref{thm:BLPPtoDL}), and exit point bounds for BLPP from its stationary initial condition. For this, we adapt techniques from \cite{Emrah-Janjigian-Seppalainen-21} to prove what we call the EJS-Rains identity (short for Emrah-Jenjigian-Sepp\"ai\"ainen-Rains) for BLPP (Theorem \ref{thm:BLPP-EJSR}). A detailed discussion of the hisotry of the EJS-Rains technique and related results comes in Section \ref{sec:ExitEJS}. The proof of Theorem \ref{thm:invariance_of_SH} comes in Section \ref{sec:invarattract_proof}.

\section{Brownian last-passage percolation, the directed landscape, and the KPZ fixed point}

\subsection{Brownian last-passage percolation}

On an appropriate probability space, let $\mathbf B = \{B_r\}_{r \in \Z}$ be a field of independent, two-sided standard Brownian motions. For $x \le y$ and $m \le n$, define
\be \label{BLPPdef}
L(x,m;y,n) = \sup\Bigl\{\sum_{r = m}^n B_r(x_{r - 1},x_r): x = x_{m - 1} \le x_m \le \cdots \le x_{n - 1} \le x_n = y \Bigr\}
\ee
We call this model \textit{Brownian last-passage percolation} (BLPP). Observe that for $n = m$, $L(x,m;y,m) = B_m(x,y)$. 

\begin{figure}[ht!]
    \centering
    \begin{tikzpicture}[scale = 0.9]
\draw[gray,thin] (0.5,0) -- (15.5,0);
\draw[gray,thin] (0.5,0.5) --(15.5,0.5);
\draw[gray, thin] (0.5,1)--(15.5,1);
\draw[gray,thin] (0.5,1.5)--(15.5,1.5);
\draw[gray,thin] (0.5,2)--(15.5,2);
\draw[black,thick] plot coordinates {(1.5,-0.1)(1.7,-0.2)(1.9,0.1)(2.1,0.4)(2.3,-0.3)(2.5,0.2)(2.7,0.1)(2.9,-0.3)(3.1,-0.4)(3.3,0.2)(3.5,0.3)(3.7,-0.4)(3.9,0.1)(4.1,-0.2)(4.3,-0.3)(4.45,0.2)};
\draw[red, ultra thick] plot coordinates {(4.5,0.4)(4.7,0.7)(4.9,0.3)(5.1,0.1)(5.4,0.6)(5.7,0.3)(5.7,0.6)(5.9,0.8)(6.1,0.4)(6.3,0.3)(6.6,0.4)(6.95,0.8)};
\draw[blue,thick] plot coordinates {(7,0.8)(7.3,1.2)(7.7,1.1)(8,0.6)(8.2,0.8)(8.5,1.3)(8.7,1.1)(9,0.7)(9.2,1.1)(9.45,1.3)};
\draw[green,thick] plot coordinates
{(9.5,1.1)(9.7,1.3)(9.9,1.7)(10.2,1.1)(10.4,1.3)(10.6,1.9)(10.8,1.4)
(11.1,1.2)(11.3,1.6)(11.5,1.9)(11.8,1.4)(12,1.1)(12.2,1.8)(12.4,1.3)(12.6,1.2)(12.8,1.1)(12.95,1.7)};
\draw[brown,thick] plot coordinates {(13,2.3)(13.3,1.7)(13.5,2.4)(13.7,2.3)(13.9,2.1)(14.1,1.9)(14.3,1.7)(14.4,2.1)(14.6,1.8)(14.8,1.6)(15,2.1)};
\node at (1.5,-0.5) {$x$};
\node at (4.5,-0.5) {$x_0$};
\node at (7,-0.5) {$x_1$};
\node at (9.5,-0.5) {$x_2$};
\node at (13,-0.5) {$x_3$};
\node at (15,-0.5) {$y$};
\node at (0,0) {$0$};
\node at (0,0.5) {$1$};
\node at (0,1) {$2$};
\node at (0,1.5) {$3$};
\node at (0,2) {$4$};
\end{tikzpicture}
    \caption{\small The Brownian increments $B_r(x_{r - 1},x_r)$ for $r=0,\dotsc,4$ in \eqref{BLPPdef} that make up the weight of the path depicted in Figure \ref{fig:BLPP_geodesic}. }
    \label{fig:BLPP maximizing path}
\end{figure}
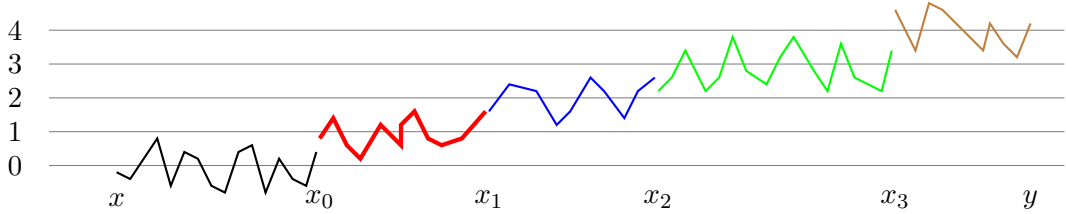
Figure \ref{fig:BLPP maximizing path} depicts the Brownian motions on each horizontal level. Maximizing sequences $x_{m - 1} \le x_m \le \cdots \le x_{n - 1} \le x_n$ exist in \eqref{BLPPdef} by continuity of the Brownian motions. For such a sequence, we associate an up-right path consisting of horizontal segments $[x_{r - 1},x_r] \times \{r\}$ and vertical segments $\{x_r\} \times [r,r + 1]$, as depicted in Figure \ref{fig:BLPP_geodesic}. 
Despite the fact that this is a maximal instead of minimal path, by convention, we call this path a \textit{geodesic} between the points.

\begin{figure}[ht]
\begin{tikzpicture}[scale = 0.9]
\draw[gray,thin] (0.5,0) -- (15.5,0);
\draw[gray,thin] (0.5,0.5) --(15.5,0.5);
\draw[gray, thin] (0.5,1)--(15.5,1);
\draw[gray,thin] (0.5,1.5)--(15.5,1.5);
\draw[gray,thin] (0.5,2)--(15.5,2);
\draw[red, ultra thick] (1.5,0)--(4.5,0)--(4.5,0.5)--(7,0.5)--(7,1)--(9.5,1)--(9.5,1.5)--(13,1.5)--(13,2)--(15,2);
\filldraw[black] (1.5,0) circle (2pt) node[anchor = north] {$(x,0)$};
\filldraw[black] (15,2) circle (2pt) node[anchor = south] {$(y,4)$};
\node at (4.5,-0.5) {$x_0$};
\node at (7,-0.5) {$x_1$};
\node at (9.5,-0.5) {$x_2$};
\node at (13,-0.5) {$x_3$};
\node at (0,0) {$0$};
\node at (0,0.5) {$1$};
\node at (0,1) {$2$};
\node at (0,1.5) {$3$};
\node at (0,2) {$4$};
\end{tikzpicture}
\caption{\small Example of a planar path from $(x,0)$ to $(y,4)$, represented by the sequence $(x=x_{-1} \le x_0 \le x_1 \le x_2\le  x_3 \le x_4=y)$.}
\label{fig:BLPP_geodesic}
\end{figure}
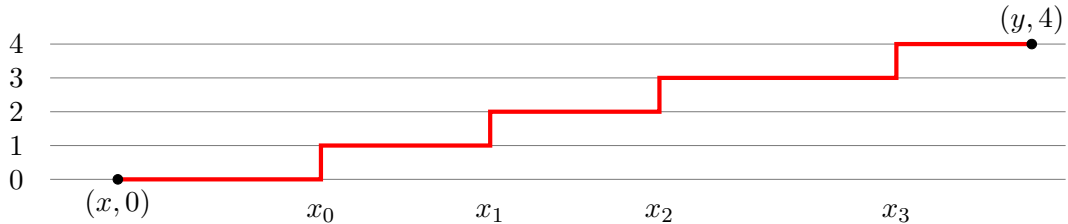

We often consider BLPP with a boundary condition. Let $f:\R \to \R$ be a function satisfying,
$
\liminf_{x \to -\infty} \f{f(x)}{x} > 0.
$
For $m < n$ and $y \in \R$, define
\be \label{BLPP_bdy}
H_L(n,y;m,f) = \sup_{-\infty < x \le y}\{f(x) + L(x,m + 1;y,n)\}.
\ee
When $m = 0$, we simply write $H_L(n,y;f)$. 

BLPP was first introduced in \cite{glynn1991}. We call it a semi-discrete model because there is one discrete and one continuous coordinate. As such, it is often used as an intermediate object for studying the directed landscape. A remarkable fact about BLPP (while we do not use this in the present dissertation) is that, for all $n \ge 1$, $L(1,0;n,1)$ has the same distribution as the largest eigenvalue of an $n \times n$ GUE matrix \cite{Baryshnikov,Gravner,rep_non_colliding}.

\subsection{The directed landscape} \label{sec:DL_geod}
  The directed landscape (DL), originally constructed in \cite{Directed_Landscape}, is a random continuous function $\Ll:\Rup \to \R$ that arises as the scaling limit of a large class of models in the KPZ universality class, and is expected to be a universal limit of such models. We  summarize some key points from \cite{Directed_Landscape} here. The directed landscape satisfies the metric composition law: for $(x,s;y,u) \in \Rup$ and $t \in (s,u)$,
\be \label{eqn:metric_comp}
\Ll(x,s;y,u) = \sup_{z \in \R}\{\Ll(x,s;z,t) + \Ll(z,t;y,u)\}.
\ee
This implies the reverse triangle inequality:  for $s < t < u$ and $(x,y,z) \in \R^3$, $\Ll(x,s;z,t) + \Ll(z,t;y,u) \le \Ll(x,s;y,u)$. 
Furthermore, over disjoint time intervals $(s_i,t_i)$, $1 \le i \le n$, the processes $(x,y) \mapsto \Ll(x,s_i;y,t_i)$ are independent. 

 Under the directed landscape, the length of  a continuous path $g:[s,t] \to \R$  is
\[
\Ll(g) = \inf_{k \in \Z_{>0}} \; \inf_{s = t_0 < t_1 < \cdots < t_k = t} \sum_{i = 1}^k \Ll(g(t_{i - 1}),t_{i - 1};g(t_i),t_i),
\]
where the second infimum is over all partitions $s = t_0 < t_1 < \cdots < t_k < t$.
By the reverse triangle inequality, $\Ll(g) \le \Ll(g(s),s;g(t),t)$. We call $g$ a \textit{geodesic} if equality holds. When this occurs, every  partition $s = t_0 < t_1 < \cdots < t_k = t$ satisfies 
\[
\Ll(g(s),s;g(t),t) = \sum_{i = 1}^k \Ll(g(t_{i - 1}),t_{i - 1};g(t_i),t_i).
\]
  For  fixed $(x,s;y,t) \in \Rup$, there exists a unique geodesic between $
(x,s)$ and $(y,t)$, almost surely \cite[Sect.~12--13]{Directed_Landscape}. Across all points, there exist leftmost and rightmost geodesics. The leftmost geodesic $g$ is such that, for each $u \in (t,s)$, $g(u)$ is the leftmost maximizer of $\Ll(x,s;z,u) + \Ll(z,u;y,t)$ over $z \in \R$.  The analogous fact holds for the rightmost geodesic. Geodesics in the directed landscape   have H\"older regularity  $\f{2}{3} - \ve$  but not  $\f{2}{3}$ \cite{Directed_Landscape,Dauvergne-Sarkar-Virag-2020}.

Perhaps the most straightforward way to understand the DL is through the limit of last-passage models. BLPP is the first model for which convergence to the DL was shown. While this convergence was first shown in \cite{Directed_Landscape}, the convergence for an arbitrary direction $\rho > 0$ was later recorded in \cite{Dauvergne-Virag-21}. 

\begin{theorem} \cite[Theorems 11.1,13.1]{Directed_Landscape} \cite[Theorems 1.7-1.8]{Dauvergne-Virag-21} \label{thm:BLPPtoDL}
Fix $\rho \in (0,\infty)$, and let $\chi,\alpha,\beta,\tau$ be defined by the relations
\be \label{eqn:BLPPchitau}
\chi^2 = \rho,\;\; \alpha = 2\sqrt \rho, \;\; \beta = \f{1}{\sqrt \rho}, \;\; \chi/\tau^2 = \f{1}{4\rho^{3/2}}.
\ee
Then, there exists a coupling $(\{L^N\}_{N \ge 1},\Ll)$ of copies $L^N$ of $L$ \eqref{BLPPdef} and $\Ll$ so that 
\begin{align*}
    &\chi N^{-1/3} \Bigl(L^N(N\rho s +N^{2/3}\tau x, \lfloor N  s \rfloor ;N \rho t + N^{2/3} \tau y, \lfloor N t \rfloor  ) - \alpha N(t - s) - \beta \tau N^{2/3}(y - x)\Bigr) \\
    &= (\Ll + o_N)(x,s;y,t),
\end{align*}
where the random function $o_\sigma$ is small in the sense that, for any compact set $K \subseteq \Rup$, $\sup_K |o_N| \to 0$ almost surely, and there is $c > 0$ such that 
\[
\E \exp\bigl(c \sup_K [(o_N^-)^{3/4} + (o_N^+)^{3/2}]\bigr) \to 1.
\]
Furthermore, under this coupling, let $\pi_N$ be the image of an arbitrary geodesic for $L^N$ under the linear map satisfying $(N\rho,N) \mapsto (0,1)$ and $N^{2/3}(\tau,0) \to (1,0)$. Then, as $N \to \infty$, if the endpoints of $\pi_N$ converge to points $p,q$ with $(p;q) \in \Rup$, then $\pi_N$ is precompact in the Hausdorff topology. All subsequential limits are geodesics from $p$ to $q$ for $\Ll$. 
\end{theorem}
We remark here that our formulation of BLPP differs from that in \cite{Directed_Landscape,Dauvergne-Virag-21}. there, the index of the discrete level increases as we move downward. To compensate for that, the version of Theorem \ref{thm:BLPPtoDL} we cite has been modified so that the discrete coordinates are $\lfloor Ns \rfloor$ and $\lfloor Nt \rfloor$ instead of $-\lfloor Ns \rfloor$ and $-\lfloor Nt \rfloor$.

\subsection{The KPZ fixed point} \label{sec:KPZ_fixed}
The KPZ fixed point $h(t,\aabullet;\mathfrak h)$ started from initial state $\h$  is a Markov process on the space of upper semi-continuous functions. More precisely, its state space is defined as 
\be \label{UCdef}
\begin{aligned}
\UC &= \{\text{ upper semi-continuous functions }\h:\R \to \R \cup \{-\infty\}: \\\  &\quad \text{ there exist }a,b > 0 \text{ such that } \quad \h(x) \le a + b|x| \text{ for all }x \in \R, \\ &\quad \text{ and }\h(x) > -\infty \text{ for some }x \in \R\}.
\end{aligned}
\ee
The topology on this space is that of local Hausdorff convergence of hypographs. When restricted to continuous functions, this convergence is equivalent to uniform convergence on compact sets. See Section 3.1 in \cite{KPZfixed} for more on the topology on the space $\UC$. This subspace of continuous functions is preserved under the KPZ fixed point (see \cite{KPZfixed}, or alternatively, Lemma \ref{lem:max_restrict}). The process $\{h_\Ll(t,\aabullet;\h)\}_{t > s}$, started from initial data $\h$ at time $s$ can be  represented as \cite{reflected_KPZfixed}  
\be \label{eqn:KPZ_DL_rep}
h_\Ll(t,y;s,\h) = \sup_{x \in \R}\{\h(x) + \Ll(x,s;y,t)\}, \quad t > s, y\in\R, 
\ee
where $\Ll$ is the directed landscape. Note here the analogy with \eqref{BLPP_bdy}. Here, we add the subscript $\Ll$ to emphasize that a realization of $\Ll$ drives the growth of $h$. When $s = 0$, we simply write $h_\Ll(t,y;\h)$. This formulation allows us to couple the KPZ fixed point from different initial data, but with the same random noise.  
If $\h$ is a two-sided Brownian motion with diffusivity  $\sqrt 2$ and arbitrary drift, then $ h_\Ll(t,\aabullet;\h) - h_\Ll(t,0;\h) \deq \h(\aabullet)$  for each $t > 0$  \cite{KPZfixed,Pimentel-21a,Pimentel-21b}. 

\subsection{Main result}

The main result of this chapter is the unique  invariance and  attractiveness of SH  under the KPZ fixed point. This generalizes the invariance of a single Brownian motion with drift and provides a new uniqueness statement  (Remark \ref{rmk:k = 1 uniqueness} below). Attractiveness is proved under these assumptions on the asymptotic drift $\dir \in \R$ of the initial function  $\h \in \UC$:
\begin{equation} \label{eqn:drift_assumptions}
    \begin{aligned}
    &\text{If } \dir = 0, \quad &\limsup_{x \to +\infty} \f{\h(x)}{x} \in [-\infty,0] \qquad &\text{and}\quad &\liminf_{x \to -\infty} \f{\h(x)}{x} \in [0,+\infty], \\
    &\text{if } \dir > 0,\quad &\lim_{x \to +\infty} \f{\h(x)}{x} = 2\dir\qquad&\text{and}\quad &\liminf_{x \to -\infty} \f{\h(x)}{x} \in (-2\dir,+\infty], \\
    &\text{and if } \dir < 0,\quad &\lim_{x \to -\infty} \f{\h(x)}{x} = 2 \dir\qquad&\text{and}\quad &\limsup_{x \to +\infty} \f{\h(x)}{x} \in [-\infty, -2\dir).
    \end{aligned}
\end{equation}
As spelled out in the theorem below, these conditions describe the basins of attraction for the KPZ fixed point. When $\dir > 0$ and $x>0$ is large, this condition forces $\h(x)$  to be approximated by $2\dir x$. The directed landscape $\Ll(x,s;y,t)$ can be approximated by $-\f{(x - y)^2}{t - s}$ (Lemma \ref{lem:Landscape_global_bound}), so that $\h(x) + \Ll(x,0;y,t) \approx 2\dir x - \f{(y - x)^2}{t}$, which has its maximum at $x = y + \dir t$. Once we can control the maximizers, Lemma \ref{lem:DL_crossing_facts} allows us to compare the KPZ fixed point from different initial conditions. This, of course, must be made precise. In the $\dir > 0$ case of the proof of Lemma \ref{lem:unq}, the $\liminf$ condition as $x \to -\infty$ forces the maximizer to be positive, and an analogous statement holds for $\dir < 0$.  These drift conditions are analogous to the conditions on the drift studied in \cite{Bakhtin-Cator-Konstantin-2014,Bakhtin-Li-19,Bakhtin-16,Bakhtin-16chapter,Bakhtin-2013} for stationary solutions of the Burgers' equation with random Poisson and kick forcing.

We recall that $G^{\sqrt 2}$ is the parameterization of the SH where $G^{\sqrt 2}_\dir$ is a Brownian motion with diffusivity $\sqrt 2$ and drift $2\dir$. 

\newpage
\begin{theorem} \label{thm:invariance_of_SH}
Let $(\Omega,\F,\Pp)$ be a  probability space on which the stationary horizon $G^{\sqrt 2}=\{G^{\sqrt 2}_\dir\}_{\dir \in \R}$ and directed landscape $\Ll$ are defined, and such that  the processes \\$\{\Ll(x,0;y,t):x,y \in \R, t > 0\}$ and $G^{\sqrt 2}$ are independent. For each $\dir \in \R$, let $G^{\sqrt 2}_\dir$ evolve under the KPZ fixed point in the same environment $\Ll$, i.e., for each $\dir \in \R$,
\be \label{KPZSH2}
h_\Ll(t,y;G^{\sqrt 2}_\dir) = \sup_{x \in \R}\{G^{\sqrt 2}_\dir(x) + \Ll(x,0;y,t)\},\qquad\text{for all } y\in\R \text{ and } t > 0.
\ee

{\rm(Invariance)} 
For each $t > 0$,   the equality in distribution  \[\{h_\Ll(t,\aabullet;G^{\sqrt 2}_\dir) - h_\Ll(t,0;G^{\sqrt 2}_\dir)\}_{\dir \in \R} \deq G^{\sqrt 2} \] holds between  random elements of $D(\R,C(\R))$.

\smallskip 

{\rm(Attractiveness)}
Let $k \in \Z_{>0}$ and $\dir_1 < \cdots < \dir_k$ in $\R$. Let $(\h_1,\ldots,\h_k)$ be a $k$-tuple of functions in $\UC$,  coupled with $(G^{\sqrt 2},\Ll)$ {\rm arbitrarily},  and that almost surely satisfy \eqref{eqn:drift_assumptions} for $(\h, \dir) = (\h_i, \dir_i)$  for each  $i\in\{1,\dotsc,k\}$.  Then, for any $a > 0$, 
\begin{multline*}
\lim_{t \to \infty} \Pp\bigl\{   h_\Ll(t,x;\h_i)-h_\Ll(t,0;\h_i)= h_\Ll(t,x;G^{\sqrt 2}_{\xi_i})-h_\Ll(t,0;G^{\sqrt 2}_{\xi_i}) \\ \forall x \in [-a,a],1 \le i \le k\bigr\} = 1.
\end{multline*}
Consequently,  as $t \to \infty$, the distributional limit  
\[
\bigl(h_\Ll(t,\aabullet;\h_1) -h_\Ll(t,0;\h_1) ,\ldots,h_\Ll(t,\aabullet;\h_k) - h_\Ll(t,0;\h_k)\bigr)
\Longrightarrow
\bigl(G^{\sqrt 2}_{\dir_1}(\aabullet),\ldots,G^{\sqrt 2}_{\dir_k}(\aabullet)\bigr)
\]
holds in $\UC^k$ (or in $\C(\R)^k$ if the $\h_i$ are continuous). 

\smallskip 

{\rm(Uniqueness)}
In particular, on the space $\UC^k$,  $\bigl(G^{\sqrt 2}_{\dir_1}, \dotsc, G^{\sqrt 2}_{\dir_k})$ is the unique invariant distribution of the KPZ fixed point  such that for each  $i\in\{1,\dotsc,k\}$  the condition \eqref{eqn:drift_assumptions} holds for $(\h, \dir) = (\h_i, \dir_i)$. 
\end{theorem}
\begin{remark}
Theorem \ref{thm:DL_Buse_summ}\ref{itm:global_attract} in  Section \ref{sec:Buse_geod_results} states that the Busemann process   is a global attractor of the backward KPZ fixed point. Namely, start the KPZ fixed point at time $t$ with initial data $\h$ satisfying \eqref{eqn:drift_assumptions} and run it backward in time to a fixed final time $s$.  Then, in a given a compact set,  for large enough $t$  the increments of the backwards KPZ fixed point at time $s$, started from initial data $\h$ at time $t$, match those of the Busemann function in direction $\dir$.   
\end{remark}
\begin{remark}
 The process $t \mapsto \{h_\Ll(t,\aabullet;\h_\dir) - h_\Ll(t,0;\h_\dir) \}_{\dir \in \R}$ is a well-defined Markov process on a state space which is a  Borel subset of  $\{\h_\dir\}_{\dir \in \R} = D(\R,C(\R))$ (Lemma \ref{lem:KPZ_preserve_Y}). By the uniqueness result for finite-dimensional distributions, $G^{\sqrt 2}$ is the unique invariant distribution on this space of $C(\R)$-valued cadlag paths. 
\end{remark}
\begin{remark} \label{rmk:k = 1 uniqueness}
In the above strength, the attractiveness result was previously unknown even in the case $k = 1$ (a single initial function).  \cite{Pimentel-21a,Pimentel-21b} proved attractiveness for $k =1$ and $\dir = 0$  under the following condition on the initial data $\h$: there exist $\gamma_0 > 0$ and $\psi(r)$ such that for all $\gamma > \gamma_0$ and $r \ge 1$, 
\be \label{eqn:ergodicity_assumption}
\Pp(\gamma^{-1}\h(\gamma^2 x) \le r|x| \, \forall x \ge 1) \ge 1 - \psi(r), \  \text{ where } \lim_{r \to \infty} \psi(r) = 0.
\ee
\end{remark}

\section{Invariance of the SH for BLPP}
\begin{lemma} \label{lem:BLPP_inductive}
Recall the map $D$ \eqref{Ddef}. Let $f \in \CRpin$ be a function satisfying
\[
\liminf_{x \to -\infty}\f{f(x)}{x} > 0.
\]
Fix $m \in \Z$. Then, almost surely, for all $y \in \R$ and $n \ge m$,
\be \label{HLind}
H_L(n + 1,y;m,f) = \sup_{-\infty < w \le y}\{H_L(n,w;m,f) + B_{n + 1}(w,y)\} < \infty,
\ee
where we define $H_L(m,w;m,f) = f$. 
Furthermore, for $n \ge m$, define \[f_n(y) = H_L(n,y;m,f) - H_L(n;0,m,f).\] Then, for $n \ge m$, 
$
f_{n +1} = D(f_{n},B_{n + 1}).
$
\end{lemma}
\begin{proof}
BLPP satisfies the dynamic programming principle: for $m  <  r \le n$,
\[
L(x,m + 1;y,n + 1) = \sup_{x \le w \le y}\{L(x,m + 1;w,r) + L(w,r + 1;y,n + 1)\}.
\]
In the case $r = n$, this becomes
\[
L(x,m + 1;y,n + 1) = \sup_{x \le w \le y}\{L(x,m;w,n) + B_{n + 1}(w,y)\}.
\]
Below, if $m = n$, we define $L(x,m + 1;w,n) = 0$. Applying dynamic programming,
\begin{align*}
  &\qquad \; H_L(n + 1,y;m,f)\\ &= \sup_{-\infty < x \le y} \{f(x) + L(x,m + 1;y,n + 1)\} \\
  &=\sup_{-\infty < x \le y}\sup_{x \le w \le y}\{f(x) + L(x,m + 1;w,n) + B_{n + 1}(w,y)\} \\
  &= \sup_{-\infty < w \le y}\{H_L(n,w;m,f) + B_{n + 1}(w,y)\},
\end{align*}
this proving the equality \eqref{HLind} (we prove finiteness a the end). Next, we have 
\begin{align*}
    &\quad \; f_{n + 1}(y) =  H_L(n + 1,y;m,f) -  H_L(n + 1,0;m,f) \\
    &=\sup_{-\infty < w \le y}\{H_L(n,w;m,f) + B_{n + 1}(w,y)\} - \sup_{-\infty < w \le 0}\{H_L(n,w;m,f) + B_{n + 1}(w,0)\} \\
    &= B_{n + 1}(y) + \sup_{-\infty < w \le y}\{H_L(n,w;m,f) -H_L(n,0;m,f)  - B_{n + 1}(w)\} \\
    &\qquad- \sup_{-\infty < w \le 0}\{H_L(n,w;m,f) -H_L(n,0;m,f)  - B_{n + 1}(w)\} 
    = D(f_{n},B_{n + 1})(y),
\end{align*}
as desired. To show finiteness in \eqref{HLind}, we use the almost sure limits $\lim_{x \to -\infty} \f{B_{n}(x)}{x} = 0$. By inductively applying Lemma \ref{liminflem}, for each $n \ge m$, $\liminf_{x \to -\infty} \f{f_n(x)}{x}  > 0$. Hence, 
\[
\limsup_{x \to -\infty}[f_n(x) - B_{n + 1}(x)] - \infty,
\]
so $f_{n + 1}(y) = D(f_n,B_{n + 1})(y)$ is finite for all $y \in \R$. 
\end{proof}

\begin{lemma} \label{lem:BLPP_invar}
Let $\{B_r\}_{r \ge \in\Z}$ be a sequence of i.i.d. two-sided Brownian motions with diffusivity $1$ and zero drift that define BLPP times $L$ \eqref{BLPPdef}. Assume $\{B_r\}_{r \ge \in\Z}$  is coupled with the SH $G^1 = \{G^1_\dir\}_{\dir \in \R}$ so that $G^1$ and $\{B_r\}_{r > m}$ are independent.  Recall the notation $H_L$ \eqref{BLPP_bdy}. Let $0 < \dir_1 < \cdots < \dir_k$. Then, for each $n \ge m$, as processes on $C(\R,\R^k)$,
\[
\{H_L(n,\abullet;m,G^1_{\dir_i}) - H_L(n,0;m,G^1_{\dir_i})\}_{1 \le i \le k} \deq \{G^1_{\dir_i}\}_{1 \le i \le k} .
\]
\end{lemma}
\begin{proof}
For $n \ge m$ and $1 \le i \le k$, set 
\[
f_n^i(y) =  H_L(n,\abullet;m,G^1_{\dir_i}) - H_L(n,0;m,G^1_{\dir_i}).
\]
By Lemma \ref{lem:BLPP_inductive}, $f_n^i = D(f_{n - 1}^i,B_n)$, so as a process in $n \ge m$ with state space $C(\R,\R^k)$, $(f_n^1,\ldots,f_n^k)$ evolves as the Markov chain \eqref{Busemann Markov chain} started from initial data $(G^1_{\dir_1},\ldots,G^1_{\dir_k})$. From the construction in Proposition \ref{prop:SH_cons}\ref{itm:SH_dist}, $(G^1_{\dir_1},\ldots,G^1_{\dir_k}) \sim \mu_1^{(\dir_1,\ldots,\dir_k)}$, and Theorem \ref{existence of an invariant measure for Busemann MC} states that this measure is invariant for the Markov chain.
\end{proof}

The SH is in fact the unique coupled invariant measure in the sense of Lemma \ref{lem:BLPP_invar}. To see this, we first cite the following result. 

\begin{theorem}\cite[Theorem 3]{Cator-Lopez-Pimentel-2019} \label{thm:CLP}
Let $B_\dir$ be a two-sided Brownian motion with diffusivity $1$ and drift $\dir \in (0,\infty)$, independent of the field of Brownian motions $\{B_r\}_{r \in \Z}$ defining $L$. Let $f$ be a random continuous function $\R \to \R$, independent of the environment defining $L$, so that $(f,B_\dir)$ has jointly stationary and ergodic increments and such that, almost surely,
\be \label{fBLPPcond}
\liminf_{x \to -\infty} \f{f(x)}{x} \ge \dir,\qquad\text{and}\qquad \limsup_{x \to +\infty} \f{f(x)}{x} \le \dir. 
\ee
Let $m \in \Z$. Then, for all compact sets $K \subseteq \R$ and $\ve > 0$,
\begin{multline*}
\lim_{n \to \infty}\Pp( \sup_{y \in K} \bigl|\bigl (H_L(n,y;m,f) - H_L(n,0;m,f)) \\ - (H_L(n,y;m,B_\dir) - H_L(n,0;m,B_\dir))\bigr| > \ve) = 0.
\end{multline*}
\end{theorem}
\newpage
Since each $G^1_{\dir_i}$ is a two-sided Brownian motion with diffusivity $1$ and drift $\dir_i$, we get the following extension to coupled initial data from a simple union bound. 
\begin{corollary}
Let $G^1$ be the SH, independent of $\{B_r\}_{r \ge m}$. Let $0 < \dir_1 < \cdots < \dir_k$, and assume that $(f_1,\ldots,f_k)$ is a random function $\R \to \R^k$ independent of $\{B_r\}_{r \ge m}$ so that each $(f_i,G^1_{\dir_i})$ has jointly stationary and ergodic increments and satisfies \eqref{fBLPPcond} with $\dir = \dir_i$.  Then,
\begin{multline*}
\lim_{n \to \infty} \Pp(\sup_{\substack{y \in K \\ 1 \le i \le k}} \bigl|\bigl (H_L(n,y;m,f_i) - H_L(n,0;m,f_i))  \\
 - (H_L(n,y;m,G^1_{\dir_i}) - H_L(n,0;m,G^1_{\dir_i}))\bigr| > \ve) = 0.
\end{multline*}
In particular, $(G^1_{\dir_1},\ldots,G^1_{\dir_k})$ is the unique stationary distribution in the sense of Lemma \ref{lem:BLPP_invar} such that, for each $1 \le i \le k$, $G^1_{\dir_i}$ has stationary and ergodic increments and satisfies \eqref{fBLPPcond} with $\dir = \dir_i$.
\end{corollary}

\section{Exit point bounds for stationary BLPP and the EJS-Rains identity} \label{sec:ExitEJS}
\subsection{Main theorem for exit point bounds}
We have seen in Lemma \ref{lem:BLPP_invar} that the SH is a stationary initial condition for BLPP with coupled boundary conditions. In this section, we focus on a single stationary initial condition. 
For a continuous initial condition $f:\R \to \R$ with $\liminf_{x \to -\infty}\f{f(x)}{x} > 0$, we set 
\be \label{BLPPZgen}
Z_L^f(n,y) = \sup\{x \le y: f(x) + L(x,1;y,n) = H_L(n,y;f)\}.
\ee
That is, $Z_L^f$ is the largest maximizer of $x \mapsto f(x) + L(x,1;y,n)$ over $x \le y$.

The goal of Section \ref{sec:ExitEJS} is to prove the following.

\begin{theorem} \label{thm:BLPP_exit_pts}
Let $t,\rho > 0$, and  $y,\dir \in \R$.  For $N \ge 1$, let $\lambda_N = \rho^{-1/2} + \dir N^{-1/3}$. For each $N$, let $B_{\lambda_N}$ be a two-sided Brownian motion with diffusivity $1$ and zero drift, independent of the Brownian motions $\{B_r\}_{r \ge 0}$ defining $L$. Let $(\Omega^N,\Ff^N,\Pp^N)$ be probability spaces on which these are defined.  Then, there exists a constant $C = C(y,\dir,\rho,t)$ so that for all sufficiently large $M \ge 1$,
\begin{align}
&\limsup_{N \to \infty} \Pp^N(Z_L^{B_{\lambda_N}}(\lfloor Nt \rfloor,\rho t N + y N^{2/3}) > M N^{2/3}) \le e^{-CM^3} \qquad \text{and} \label{BLPPub1} \\
&\limsup_{N \to \infty} \Pp^N(Z_L^{B_{\lambda_N}}(\lfloor Nt\rfloor ,\rho t N + yN^{2/3}) < -M  N^{2/3}) \le e^{-CM^3}. \label{BLPPub2}
\end{align}
\end{theorem}

Theorem \ref{thm:BLPP_exit_pts} states that the stationary BLPP model obeys the KPZ wandering exponent $\f{2}{3}$. In analogy with the results of \cite{Emrah-Janjigian-Seppalainen-21,BasuSarkarSly_Coalescence,Bhatia-2020} for exponential LPP, the bound of $e^{-CM^3}$ should be of optimal order. However, the optimality of this bound is not needed in the present dissertation, so we do not prove it here. To my knowledge, such bounds have not appeared in the literature for BLPP. Exit point bounds have, however, appeared for a positive-temperature variant of this model known as the O'Connell-Yor polymer \cite{Sepp_and_Valko, Noack-Sosoe-2020,Landon-Sosoe-22b}. 

The study of exit point bounds has a rich history and has been done using both integrable and non-integrable techniques. Using what is known as the \textit{coupling technique} developed in \cite{Cator-Groeneboom-2005}, exit point bounds with $CM^{-3}$ instead of $e^{-CM^3}$ were proved for Poisson last-passage percolation in \cite{Cator-Groeneboom-06} and then in exponential last-passage percolation in \cite{Balazs-Cator-Seppalainen-2006}. Since then, there has been a large amount of research centered around the stationary exponential LPP model. Refinements to the results of \cite{Balazs-Cator-Seppalainen-2006} have come in \cite{Pimentel-18,Pimentel-21a,Sepp_lecture_notes}. Closely related to the study of exit points is that of the coalescence times of semi-infinite geodesics, as shown by the duality in \cite{pimentel2016}, and studied further in \cite{Seppalainen-Shen-2020}.  

On the integrable side of things, improvements to the exit point upper bounds came in \cite{Ferrari-Occelli-18, Ferrari-Ghosal-Nejjar-19}, which gave exit point bounds of order $e^{-CM^2}$. Fluctuation bounds for geodesics in the bulk (as opposed to the exit point) appeared in \cite{Ferrari-Occelli-18, BasuSarkarSly_Coalescence,Bhatia-2020,Martin-Sly-Zhang-21} using integrable methods. These methods allow one to derive the optimal-order exit point bounds with the appropriate inputs (see for example \cite[Theorem 2.5]{Bhatia-2020} and \cite[Lemma 2.5]{Ferrari-Occelli-18}).

A major breakthrough came using the coupling technique in the work of Emrah, Janigian, and Sepp\"al\"ainen \cite{Emrah-Janjigian-Seppalainen-21}. Adapting a moment generating function identity of Rains \cite{Rains-2000} (now called the EJS-Rains identity), they obtained the optimal $e^{-CM^3}$ order bounds for exit points in exponential LPP from an arbitrary down-right path. Shortly after \cite{Bhatia-2020} proved these same bounds using integrable methods for exit points from the axes. Since then, Busani and Ferrari \cite{busa-ferr-20} used the results of \cite{Emrah-Janjigian-Seppalainen-21} to develop optimal-order bounds for fluctuations of geodesics in the bulk. The EJS-Rains identity has been adapted for interacting diffusions and polymers in \cite{Landon-Sosoe-Noack-2020, Landon-Sosoe-22a,Landon-Sosoe-22b}. 

In the present chapter, we prove Theorem \ref{thm:BLPP_exit_pts} using the coupling technique. We first develop some facts about stationary BLPP from \cite{brownian_queues} in Section \ref{sec:BLPP_stat}. In Section \ref{sec:EJS-Rains-BLPP}, we derive the analogue of the EJS-Rains identity for BLPP (Theorem \ref{thm:BLPP-EJSR}) and use it to prove the bounds of Theorem \ref{thm:BLPP_exit_pts}. 

Before moving to the next section, we make the following observation, which follows from a standard paths-crossing argument (See, for example, \cite{blpp_utah,Basu-Ganguly-Hammond-21,Ganguly-Hegde-2021,Pimentel-21b})
\begin{lemma} \label{lem:Zfmont}
Let $f:\R \to \R$ be continuous with $\liminf_{x \to -\infty}\f{f(x)}{x} > 0$. Then, if $x < y$,
\[
Z_L^f(n,x) \le Z_L^f(n,y). 
\]
\end{lemma}
\begin{proof}
    Suppose, to the contrary, that $z_x := Z_L^f(n,x) > Z_L^f(n,y) =: z_y$. By planarity, the BLPP geodesic from $(z_x,1)$ to $(x,n)$ must cross the BLPP geodesic from $(z_y,1)$ to $(y,n)$ at some point $(w,r)$. Since $z_x$ is maximal for $f(x) + L(x,1;y,n)$ and $(w,r)$ lies on the geodesic from $(z_x,1)$ to $(x,n)$,
    \begin{align*}
    &\quad\; f(z_x) + L(z_x,1;w,r) + L(w,r;x,n) \\
    &= f(z_x) + L(z_x,1;x,n) \\ 
    &\ge f(z_y) + L(z_y,1;w,r)  + L(w,r;x,n),
    \end{align*}
    which implies $f(z_x) + L(z_x,1;w,r) \ge f(z_y) + L(z_y,1;w,r)$.
    Next, since $(w,r)$ lies on the geodesic from $(z_y,1)$ to $(y,n)$,
    \begin{align*}
    &\quad \; f(z_x) + L(z_x,1;y,n) \\
    &\ge f(z_x) + L(z_x,1;w,r) + L(w,r;y,n) \\
    &\ge f(z_y) + L(z_y,1;w,r) + L(w,r;y,n) \\
    &= f(z_y) + L(z_y,1;y,n) = \sup_{x \in \R}\{f(x) + L(x,1;y,n)\}.
    \end{align*}
    Hence, $z_x$ is a maximizer for $f(x) + L(x,1;y,n)$. This contradicts the definition of $z_y$ as the largest such maximizer. 
\end{proof}

\subsection{Stationary and near-stationary BLPP} \label{sec:BLPP_stat}
In this section, we study a coupling of initial conditions that is \textit{different} from the SH. This coupling is not jointly stationary under the BLPP dynamics, but will allow us to derive valuable information about a single initial condition. Let $B$ be a two-sided Brownian motion with diffusivity $1$ and zero drift, independent of the field  of Brownian motions $\{B_r\}_{r \ge 0}$ defining $L$ \eqref{BLPPdef}. We let $(\Omega,\Ff,\Pp)$ be the probability space on which these are defined. Henceforth, we consider BLPP from an initial boundary condition on level $0$. For $n \in \Z_{> 0}$, $y \ge 0$, and $\rho,\lambda > 0$, set 
\begin{align*}
L^{\rho,\lambda}(n,y) &= \sup_{-\infty < x \le 0}\{B(x) + \rho x + L(x,1;y,n)\}\vee \sup_{0\le x \le y}\{B(x) + \lambda x + L(x,1;y,n)\} \\
&=\sup_{-\infty < x \le y}\{f^{\rho,\lambda}(x) + L(x,1;y,n)\}
\end{align*}
where $f^{\rho,\lambda}(x) = B(x) + \rho(x)$ for $x < 0$ and $B(x) + \lambda x$ for $x > 0$.
When $\lambda = \rho$, we use the shorthand notation $L^\lambda$. 
We adopt the notation
\[
Z^{\rho,\lambda}(n,y) = Z_L^{f^{\rho,\lambda}}(n,y) = \sup\{x \le y: f^{\rho,\lambda}(x) + L(x,1;y,n) = L^{\rho,\lambda}(n,y)\},
\]
where we denote $Z^\lambda(n,y) = Z^{\lambda,\lambda}(n,y)$. We record the following facts.
\begin{lemma} \label{lem:BLPPsigmamont}
Let $n\in \Z_{>0}$ and  $y \ge 0$. If $\rho'\ge \rho > 0$ and $\lambda'\ge \lambda > 0$, then $Z^{\rho',\lambda'}(n,y) \ge Z^{\rho,\lambda}(n,y)$.
\end{lemma}

\begin{proof}
By Lemma \ref{lemma:max_monotonicity}, it suffices to prove that $f^{\rho,\lambda} \li f^{\rho',\lambda'}$, meaning $f^{\rho,\lambda}(x,y) \le f^{\rho',\lambda'}(x,y)$ for all $x < y$. If $x < y \le 0$ or $0 \le x \le y$, the result is immediate. If $x < 0 < y$, then
\[
f^{\rho,\lambda}(x,y)  = B(x,y) + \lambda y - \rho x \le B(x,y) +\lambda' y - \rho 'x = f^{\rho',\lambda'}(x,y). \qedhere
\]
\end{proof}

\begin{lemma} \label{lem:BLPPLrl=}
Let $n\in \Z_{>0}$, $y \ge 0$, and $\rho,\alpha,\lambda > 0$. If $Z^{\rho,\alpha}(n,y) \le 0$ and $Z^{\rho,\lambda}(n,y) \le 0$, then $L^{\rho,\alpha}(n,y) = L^{\rho,\lambda}(n,y)$. If $Z^{\alpha,\lambda}(n,y) \ge 0$ and $Z^{\rho,\lambda}(n,y) \ge 0$, then $L^{\alpha,\lambda}(n,y) = L^{\rho,\lambda}(n,y)$.
\end{lemma}
\begin{proof}
If $Z^{\rho,\alpha}(n,y) \le 0$ and $Z^{\rho,\lambda}(n,y) \le 0$, then
\[
L^{\rho,\alpha}(n,y) = \sup_{-\infty \le x \le 0}\{B(x) + \rho x + L(x,1;y,n)\} = L^{\rho,\lambda}(n,y).
\]
The other statement is proved analogously. 
\end{proof}
\begin{lemma} \label{lem:BLPPshinvar}
BLPP satisfies the following shift invariance: for $r \in \Z$ and $z \in \R$,
\[
 \{L(x,m;y,n):x \le y, m \le n\} \deq \{L(x + z,m + r;y + z,n + r):x \le y, m \le n\}
\]
\end{lemma}
\begin{proof}
The invariance under $r$-shifts follows because $\{B_r\}_{r \in \Z}$ is i.i.d. For $z \in \R$, the following distributional equality holds on the process level:
    \begin{align*}
    &\quad \; L(x,m;y,n) \\&= \sup\Bigl\{\sum_{r = m}^n B_r(x_{r - 1},x_r) : x = x_{m - 1} \le x_m \le \cdots \le x_{n - 1} \le x_n = y\Bigr\} \\
    &\deq \sup\Bigl\{\sum_{r = m}^n B_r(x_{r - 1} + z,x_r + z) : x = x_{m - 1} \le x_m \le \cdots \le x_{n - 1} \le x_n = y\Bigr\} \\
    &= \sup\Bigl\{\sum_{r = m}^n B_r(x_{r - 1} ,x_r ) : x + z = x_{m - 1} \le x_m \le \cdots \le x_{n - 1} \le x_n = y + z\Bigr\} \\
    &= L(x + z,m;y + z,n). \qquad\qquad\qquad\qquad\qquad\qquad \qquad\qquad\qquad\qquad\qquad\qquad \qedhere
    \end{align*}
\end{proof}
For $n \in \Z_{>0}$ and $y \ge 0$, define
\be \label{Erl}
\Erl(n,y) = L^{\rho,\lambda}(n,y) - B(y) - \lambda y.
\ee
Again, when $\rho = \lambda$, we write a single superscript.
We now cite the following result. 

\begin{lemma}\cite[Theorem 2 and Equation (24)]{brownian_queues} \label{lem:Wrldist}
The random variable 
\[
\sup_{-\infty < x \le 0}\{B(x) + \rho x + L(x,1;0,n)\}
\]
has the distribution of the sum of $n$ i.i.d. exponential random variables with rate $\rho$. 
\end{lemma}



\begin{corollary} \label{cor:east_bdy}
    For $n \in \Z_{>0}$ and $y \ge 0$, the random variable $\mathcal E^\rho(n,y)$ has the distribution of a sum of $n$ i.i.d. exponential random variables with rate $\rho$. 
\end{corollary}
\begin{proof}
    Using shift invariance of $B$ as well as the shift invariance of $L$ from Lemma \ref{lem:BLPPshinvar},
    \begin{align*}
    \mathcal E^\rho(n,y) &= \sup_{-\infty < x \le y}\{B(x) + \rho x + L(x,1;y,n)\} - B y - \lambda y \\
    &= \sup_{-\infty < x \le y}\{B(y,x) + \rho(x - y) + L(x,1;y,n)\} \\
    &\deq \sup_{-\infty < x \le y}\{B(x - y) + \rho(x - y) + L(x - y,1;0,n)\} \\
    &= \sup_{-\infty < x \le 0}\{B(x) + \rho x + L(x,1;0,n)\}, 
    \end{align*}
    and the result now follows from Lemma \ref{lem:Wrldist}.
\end{proof}

Corollary \ref{cor:east_bdy} illustrates one way in which Brownian motion with drift is a stationary initial condition for BLPP. Much richer information can be derived about the stationarity in this model, to which we refer the reader to \cite{brownian_queues,blpp_utah}, and Appendix C of \cite{Seppalainen-Sorensen-21a}.  
\subsection{The EJS-Rains identity for BLPP and Proof of Theorem \ref{thm:BLPP_exit_pts}} \label{sec:EJS-Rains-BLPP} 
Recalling the definition of $\mathcal E^\rho(n,y)$ \eqref{Erl}, Corollary \ref{cor:east_bdy} implies that 
\[
\Ee[L^\rho(n,y)] = \Ee[\mathcal E^\rho(n,y)] + \Ee[B(y) + \rho y] = \f{n}{\rho} + \rho y.
\]
In light of this calculation, we make the following definitions.
For $n \in \Z_{>0}$ and $y \ge 0$,
\[
\begin{aligned}
M^\rho(n,y) &= \f{n}{\rho} + \rho y \qquad R^{\rho,\lambda}(n,y) = \int_{\lambda}^\rho M^\alpha(n,y)\,d\alpha = \f{\rho^2 - \lambda^2}{2}y + n \log \rho - n \log \lambda, \\
\zeta(n,y) &= \arg \inf_{\rho > 0} M^\rho(n,y) = \sqrt{\f{n}{y}}, \qquad \gamma(n,y) = \inf_{\rho > 0} M^{\rho}(n,y) = \f{n}{\zeta} + \zeta y = 2\sqrt{ny} 
\end{aligned}
\]
We make a simple observation: For $\ve  \in (0,1)$, there exists a constant $c = c(\ve) > 0$ so that for all $n \in \N$ and $n\ve < y < n\ve^{-1}$, 
\be \label{n+ynybd}
c(n + y) \le \gamma(n,y) = 2\sqrt{ny} \le n + y.
\ee
The upper bound holds by comparing geometric and arithmetic means, and the lower bound holds because $\gamma(n,y) \ge n\sqrt \ve$, and $n + y \le n(1 + \ve^{-1})$.

\begin{theorem}[The 
EJS-Rains identity for BLPP] \label{thm:BLPP-EJSR}
Let $y \ge 0$ and $n \in \Z_{>0}$. For $\rho,\lambda > 0$, we have 
\[
\Ee\Bigl[\exp\Bigl((\rho - \lambda) L^{\rho,\lambda}(n,y)\Bigr)\Bigr] = \exp\Bigl(R^{\rho,\lambda}(n,y)\Bigr).
\]
\end{theorem}
\begin{proof}
Recall
\[
L^{\rho,\lambda}(n,y) = \sup_{-\infty < x \le 0}\{B(x) + \rho x + L(x,1;y,n)\} \vee \sup_{0 \le x \le y}\{B(x) + \lambda x + L(x,1;y,n)\},
\]
and this is a function of three independent processes, namely 
\[
\{B(x) + \rho x: x \le 0\},\qquad\{B(x) + \lambda x: x \in [0,y]\},\qquad \text{and}\qquad L.
\]
Now, we observe that 
\begin{align*}
    &\quad\;\Ee\Bigl[\exp\Bigl((\rho - \lambda) L^{\rho,\lambda}(n,y)\Bigr)\Bigr] \\
    &= \Ee\Bigl[e^{(\rho - \lambda)(B(y) + \lambda y)}\exp\Bigl((\rho - \lambda) (L^{\rho,\lambda}(n,y) - B(y) - \lambda y)\Bigr)\Bigr] \\
    &= \exp\Bigl(\f{\rho^2 - \lambda^2}{2}y\Bigr)\Ee\Bigl[\exp\Bigl((\rho - \lambda)(L^{\rho}(n,y) - B(y) - \rho y) \Bigr)\Bigr] 
    \end{align*}
    Here, the last step came from the Cameron-Martin-Girsanov Theorem (Theorem \ref{cameron-martin}): we changed the process $\{B(x) + \lambda x: [0,y]\}$ to a Brownian motion with drift $\rho$, thus producing $L^\rho$. Now, we recall \eqref{Erl} that $\mathcal E^\rho(n,y) = L^{\rho}(n,y) - B(y) - \rho y$, which has the distribution of the sum of $n$ i.i.d. exponential random variables with rate $\rho$ (Corollary \ref{cor:east_bdy}). Hence, we can compute the expectation above:
    \begin{align*}
    &\quad \;\exp\Bigl(\f{\rho^2 - \lambda^2}{2}y\Bigr)\Ee\Bigl[\exp\Bigl((\rho - \lambda)\mathcal E^{\rho}(n,y) \Bigr)\Bigr] \\
    &= \exp\Bigl(\f{\rho^2 - \lambda^2}{2}y\Bigr) \int_0^\infty e^{(\rho - \lambda)x}\f{\rho^n}{\Gamma(n)} x^{n - 1}e^{-\rho x}\,dx \\
    &= \exp\Bigl(\f{\rho^2 - \lambda^2}{2}y + n \log(\rho) - n \log(\lambda)\Bigr),
\end{align*}
and this is $\exp(R^{\rho,\lambda}(n,y))$, as desired. 
\end{proof}
\noindent Before proving Theorem \ref{thm:BLPP_exit_pts}, we first prove the following technical lemmas.  
\begin{lemma} \label{lem:MaBLPPbd}
    For all $n \in \Z_{>0}$, $y \ge 0$ and $\rho > 0$,
    \[
    M^\rho(n,y) - \gamma(n,y) - \f{\gamma(n,y)}{2\zeta(n,y)^2}(\rho - \zeta)^2 = - \f{\gamma(n,y)}{2 \rho \zeta^2} (\rho - \zeta)^3.
    \]
    Consequently, for each $\ve \in (0,1)$, there exists a constant $C = C(\ve)$ so that whenever $\ve < \f{n}{y} < \ve^{-1}$ and $\ve < \rho < \ve^{-1}$, for $\zeta = \zeta(n,y)$ and $\gamma = \gamma(n,y)$,
    \[
    \Bigl| M^\rho(n,y) - \gamma - \f{\gamma}{2\zeta^2} (\rho - \zeta)^2\Bigr| \le C(n + y) |\rho - \zeta|^3.
    \]
\end{lemma}
\begin{proof}
The ``consequently" part comes because of the assumption that $\rho$ and $\zeta$ are bounded away from $0$ and $\infty$ and \eqref{n+ynybd}. Now, observe that 
\begin{align*}
M^\rho(n,y) -\gamma &= \f{n}{\rho} + \rho y - \f{n}{\zeta} - \zeta y = (\rho - \zeta)y + \f{n(\zeta - \rho)}{\rho \zeta} = (\rho - \zeta)\f{\rho \zeta y - n}{\rho \zeta} \\
&=(\rho - \zeta) \Bigl(\f{\rho \sqrt{\f{n}{y}} y - n}{\rho \zeta}\Bigr) = \f{\sqrt{ny}}{\rho \zeta}(\rho - \zeta)^2 = \f{\gamma}{2\rho \zeta}(\rho - \zeta)^2.
\end{align*}
Then,
\[
M^\rho(n,y) -\gamma - \f{\gamma}{2\zeta^2}(\rho - \zeta)^2 = (\rho - \zeta)^2 \Bigl(\f{\gamma}{2\rho \zeta} - \f{\gamma}{2\zeta^2}\Bigr) = - \f{\gamma}{2\rho \zeta^2}(\rho - \zeta)^3. \qedhere
\]
\end{proof}

\begin{lemma} \label{lem:RrlBLPPbd}
For $\ve \in (0,1)$, there exists a constant $C = C(\ve)$ so that for $\rho,\lambda, \f{n}{y} \in (\ve,\ve^{-1})$ and $\zeta = \zeta(n,y) = \sqrt{\f{n}{y}}$,
\[
\Bigl|R^{\rho,\lambda}(n,y) - \gamma(\rho - \lambda) - \f{\gamma}{6\zeta^2}((\rho - \zeta)^3 - (\lambda - \zeta)^3)\Bigr| \le C(n + y) ((\rho - \zeta)^4 + (\lambda - \zeta)^4).
\]
\end{lemma}
\begin{proof}
Since $R^{\rho,\lambda}(n,y)= \int_\lambda^\rho M^\alpha(n,y)\,d\alpha$, the proof follows from integrating inside the absolute value on the left in Lemma \ref{lem:MaBLPPbd}. 
\end{proof}

\begin{lemma} \label{lem:BLPPexitbdmain}
For each $\ve \in (0,1)$, there exists a  constant $C = C(\ve) > 0$ so that for $\zeta =\zeta(n,y) = \sqrt{\f{n}{y}}$, whenever $\ve < \rho < \zeta  < \ve^{-1}$,
\be \label{EJRbd1}
\Pp(Z^\rho(n,y) \ge 0) \le \exp(-C(n + y)(\zeta - \rho)^3),
\ee
and whenever $\ve < \zeta < \rho < \ve^{-1}$, 
\be \label{EJRbd2}
\Pp(Z^\rho(n,y) \le 0) \le \exp(-C(n + y)(\rho - \zeta)^3).
\ee
\end{lemma}
\begin{proof}
We prove \eqref{EJRbd1}, and \eqref{EJRbd2} follows by a symmetric argument. Let $\ve < \rho < \zeta < \ve^{-1}$, and let $\lambda = (\zeta - \rho)/4$ so that $\rho = \zeta - 4\lambda$. Lemma \ref{lem:BLPPsigmamont} implies that, on the event $\{Z^\rho \ge 0\}$  (dropping the $(n,y)$ argument for simplicity), we also have 
\[
Z^{\zeta,\zeta - 2\lambda} = Z^{\rho + 4\lambda,\rho + 2\lambda} \ge 0, \qquad \text{and} \qquad Z^{\zeta - 4\lambda,\zeta - 2\lambda} = Z^{\rho,\rho + 2\lambda} \ge 0.
\]
Hence, by Lemma \ref{lem:BLPPLrl=}, $L^{\zeta,\zeta - 2\lambda} = L^{\zeta - 4\lambda,\zeta - 2\lambda}$ on the event $\{Z^\rho \ge 0\}$.
Then,
\begin{align*}
\Pp(Z^\rho \ge 0) &= \Ee[\exp(\lambda L^{\zeta, \zeta - 2\lambda}- \lambda L^{\zeta - 4\lambda,\zeta - 2\lambda}) \ind(Z^\rho \ge 0)]\\
&\le \Ee[\exp(\lambda L^{\zeta,\zeta - 2\lambda} - \lambda L^{\zeta - 4\lambda,\zeta - 2\lambda})] \\
&\le  \Ee[\exp(2\lambda L^{\zeta,\zeta - 2\lambda} )]^{1/2} \Ee[\exp(-2\lambda L^{\zeta - 4\lambda,\zeta - 2\lambda})]^{1/2} \\
&= \exp\Bigl(\f{1}{2}R^{\zeta,\zeta - 2\lambda} + \f{1}{2}R^{\zeta - 4\lambda,\zeta - 2\lambda}\Bigr) = \exp\Bigl(\f{1}{2}R^{\zeta,\zeta - 2\lambda} - \f{1}{2}R^{\zeta - 2\lambda,\zeta - 4\lambda}\Bigr).
\end{align*}
In the second inequality, we used H\"older's inequality, and in the last equality, we used $R^{\rho,\lambda} = -R^{\lambda,\rho}$. Then, using the bounds of Lemma \ref{lem:RrlBLPPbd}, we obtain (changing the constant $C$ from line to line),
\begin{align*}
    &\quad \; R^{\zeta,\zeta - 2\lambda} - R^{\zeta - 2\lambda,\zeta - 4\lambda} \\
    &\le \gamma(2\lambda) + \f{\gamma}{6\zeta^2}(-(-2\lambda)^3) + C(n + y)((4\lambda)^4) \\
    &\qquad\qquad\qquad\qquad - \Bigl(\gamma(2\lambda)  + \f{\gamma}{6\zeta^2}((-2\lambda)^3 - (-4\lambda)^3) - C(n + y)((2\lambda^4) + (4\lambda)^4  \Bigr)) \\
    &=  -\f{8\gamma}{\zeta^2} \lambda^3 + C(n + y)\lambda^4 \le -C(n + y)\lambda^3,
\end{align*}
thus proving \eqref{EJRbd1}.
Here, we used \eqref{n+ynybd} and the assumption that $\zeta^2 = \f{n}{y} \in (\ve ,\ve^{-1})$.  
\end{proof}

\noindent We now complete this section by proving Theorem \ref{thm:BLPP_exit_pts}.
\begin{proof}[Proof of Theorem \ref{thm:BLPP_exit_pts}]
We prove this via the coupling introduced in Section \ref{sec:BLPP_stat}. Here, the probability space  is $(\Omega,\Ff,\Pp)$. Let $\lambda_N = \rho^{-1/2} + \dir N^{-1/3}$.  From shift-invariance of BLPP (Lemma \ref{lem:BLPPshinvar}),
\[
\Pp(Z^{\lambda_N}(\lfloor Nt\rfloor ,\rho t N + y N^{2/3}) > M N^{2/3}) = \Pp(Z^{\lambda_N}(\lfloor Nt \rfloor , \rho tN + (y - M) N^{2/3})> 0).
\]
Note that 
\begin{align*}
\zeta_N &:= \zeta(Nt, \rho tN + (y - M)N^{2/3}) = \sqrt{\f{Nt}{\rho tN + (y - M)N^{2/3}}} - \rho^{-1/2} \\
&= \f{1}{\sqrt{\rho + (y-M)N^{-1/3}/t}} - \f{1}{\sqrt \rho} = \f{(M - y)N^{-1/3}/t}{\sqrt \rho \sqrt{\rho + (y-M)N^{-1/3}/t} (\sqrt \rho +\sqrt{\rho + (y-M)N^{-1/3}/t} )} \\
    &\ge \f{(M - y)N^{-1/3}/t}{\sqrt \rho \sqrt{\rho + y N^{-1/3}/t} (\sqrt \rho +\sqrt{\rho + y N^{-1/3}/t} )} \ge C(M - y)N^{-1/3}.
\end{align*}
Comparing to $\lambda_N$, we have $\zeta_N > \lambda_N$ for sufficiently large $M$ (independent of $N$). Choose such an $M$ and let $ve \in (0,1)$. Let $N$ be sufficiently large (depending on the choice of $M$) so that $\ve < \lambda_N < \zeta_N < \ve^{-1}$.
By a Taylor approximation, 
\begin{align*}
&\quad \; \zeta_N 
&= \rho^{-1/2} \sqrt{\f{1}{1 +\rho^{-1}(y - M) N^{-1/3}}} = \rho^{-1/2} - \f{1}{2}(y- M) \rho^{-3/2} N^{-1/3} + O(N^{-2/3}),
\end{align*}
where the $O(N^{-2/3})$ term may depend on $M$. However, this presents no substantive issue because the term will vanish in the limit with $N$. Indeed, by Lemma \ref{lem:BLPPexitbdmain}, there exists a constant $C = C(t,\dir,\rho,y) > 0$ so that
\begin{align*}
    &\quad \; \limsup_{N \to \infty}\Pp(Z^{\lambda_N}(Nt, \rho tN + (y - M) N^{2/3})>  0) \\
    &\le \limsup_{N \to \infty}\exp\Biggl[-C\Bigl(Nt + \rho tN + (y - M)N^{2/3}\Bigr)((M-y)N^{-1/3} + O(N^{-2/3}))^3\Biggr] \\
    &= \exp(-C(t + \rho t) (M-y)^3).
\end{align*}
This proves \eqref{BLPPub1} for $M$ large enough. The proof of \eqref{BLPPub2} is  analogous. 
\end{proof}

\section{Maximizers and bounds for the KPZ fixed point}
We now prove the following technical lemmas for the DL and the KPZ fixed point.  
\begin{lemma}\label{lem:unq}
	 Fix $\xi\in \R$ and $a>0$. Consider the KPZ fixed point starting at time $s$ from a function $\h \in \UC$. 
For $t > s$, let $Z_\h^{a,s,t}\in\R$ denote the set of exit points from the time horizon $\Hh_s$ of the geodesics associated with $\h$ and that terminate in $\{t\}\times [-a,a]$. That is,
	\be \label{exitpt}
		Z_\h^{a,s,t}=\bigcup_{y\in [-a,a]}\argmax_{x\in\R} \{\h(x)+\Ll(x,s;y,t)\}.
	\ee
	Then, on the full probability event of Lemma \ref{lem:Landscape_global_bound}, whenever $\h \in \UC$ satisfies condition \eqref{eqn:drift_assumptions}, and when  $\ve>0$, $a > 0$, and $s \in \R$, there exists  a random $t_0 = t_0(\ve,a,s) > s \vee 0$ such that for any $t> t_0$,
	\be \label{upexit}
	Z_\h^{a,s,t}\subset \big[(\xi-\ve)t,(\xi +\ve)t\big].
	\ee
	In particular, if $\h$ is a random function almost surely satisfying condition \eqref{eqn:drift_assumptions}, then this random $t_0$ exists almost surely, and
\[
	\lim_{t \to \infty} \Pp\Big(Z_\h^{a,s,t}\subset \big[(\xi-\ve)t,(\xi +\ve)t\big]\Big) = 1.
\]  
Furthermore, an analogous statement holds on the same full-probability event if $t$ is held fixed and $s \to -\infty$. That is, there exists a random $s_0 = s_0(\ve,a,t)< t \wedge 0 $ such that for any $s < s_0$,
\be \label{downexit}
Z_\h^{a,s,t}\subset \big[-(\xi-\ve)s,-(\xi +\ve)s\big]
\ee
\end{lemma}
\begin{proof}
We show \eqref{upexit}, and \eqref{downexit} follows by an analogous proof. The idea of the proof is that  $\h(x)+\Ll(x,0;y,t)$ is a noisy version of $2\xi x-\frac{x^2}{t}$ and  $\argmax_{x\in\R} (2\xi x-\frac{x^2}{t})=\xi t$, but the noise  cannot change the exit point by much when $t$ is large. The drift conditions \eqref{eqn:drift_assumptions} are used in the proof to ensure that for $t$ large enough, all maximizers are positive for $\dir > 0$ and negative for $\dir < 0$. Below, we prove the result for $\xi>0$, and the proof for $\xi<0$ follows by symmetry. The case $\xi=0$ will be proven separately. Fix $\ve>0$. Suppressing the dependence on $a,s$, set
\begin{equation} \label{F}
	F(x;t)=C_{\text{DL}}(t-s)^{1/3}\log^2\big(2\sqrt{a^2+x^2+s^2 + t^2}+4\big),
\end{equation}
where $C_{\text{DL}}$ is the random positive constant from Lemma \ref{lem:Landscape_global_bound}. 

\newpage
\begin{lemma} \label{lem:dFbd}
For $\ve > 0,a > 0,$ and $s < t \in \R$, there exists $t_1 = t_1(\ve,a,s) > s$ such that for all $t > t_0$,  $|F'(x;t)| < \ve$ uniformly for all $x \in \R$. In addition, for $t > t_0$, $F'(x;t) > 0$ for $x > 0$ while $F'(x;t) < 0$ for $x < 0$.
\end{lemma}
\begin{proof}
A quick computation shows that
\[
	F'(x;t)=\f{2 C_\text{DL}(t-s)^{1/3} x\log\big(2\sqrt{a^2+x^2+t^2 + s^2}+4 \big)}{(\sqrt{a^2+x^2+t^2 + s^2}+2)\sqrt{a^2+x^2+t^2 + s^2}}.
\]
We note that
$
|x|(a^2 + x^2 +t^2 + s^2)^{1/2} \le 1,
$
and that $\log(x)/x$ is decreasing for $x > e$. 
Hence, for all $a > 0,s < t,\ve > 0$, and $x \in \R$,
\[
|F'(x;t)| \le 2C(t - s)^{1/3} \f{\log(2|t| +4)}{|t| + 2} \overset{t \to \infty}{\longrightarrow} 0. \qedhere
\]
\end{proof}

Back to the main proof,  from the drift assumption \eqref{eqn:drift_assumptions}, for each $\ve > 0$, there exists $R_\ve > 0$ so that for all $x \ge R_\ve$, $|\f{\h(x)}{x} - \dir| \le \f{\ve}{2}$. Let $C_\ve = \sup_{0 \le R_\ve} \h(x)$. By the global bound $\h(x) \le a + b|x|$ (assumed in the definition of $\UC$), $C_\ve < \infty$. Further, observe that $-\f{(x - y)^2}{t - s} = -\f{x^2}{t - s} + \f{2xy}{t - s} - \f{y^2}{t - s}$, and $|\f{2xy}{t - s} - \f{y^2}{t - s}| \le \ve + \f{\ve}{2}|x|$ for large $t$, uniformly for $y \in [-a,a]$. Using this and the bounds of Lemma \ref{lem:Landscape_global_bound}, for $t  > s + 1$ sufficiently large (depending on $\ve,a$), for all $y \in [-a,a]$ and $x \ge 0$,
	\begin{equation}\label{ub1'}
		\begin{aligned}
		 &\h(x)+\Ll(x,s;y,t)\leq M_{U}(x;t):=  C_\ve + 2\xi x + \ve x-\frac{x^2}{t - s}+ \ve + F(x;t),
	\end{aligned}
	\end{equation}
 Further, for all $x\ge R_\ve$, and sufficiently large $t > s + 1$,
\be \label{ub2'}
 	 \h(x)+\Ll(x,s;y,t)\geq M_{L}(x;t):= 2\xi x - \ve x -\frac{x^2}{t - s} - \ve -F(x;t). 
\ee

By the assumption \eqref{eqn:drift_assumptions}, we may choose $\gamma$ so that 
$-2\dir < 2\gamma < \liminf_{x \to -\infty} \f{\h(x)}{x} \le + \infty$. Then, applying a similar procedure as before and adjusting the constant $C_\ve$ if needed, for all $y \in [-a,a]$ and $x \le 0$,
	\be \label{ub3'}
	\h(x) + \Ll(x,s;y,t) \le  M_U^-(x;t) := C_\ve  + 2\gamma x - \f{x^2}{t - s} + F(x;t).
	\ee
We start by using these bounds to show that when $t > s$ is sufficiently large,
\be \label{max_right}
\sup_{x \ge 0}\{\h(x) + \Ll(x,s;y,t)\} > \sup_{x \le 0}\{\h(x) + \Ll(x,s;y,t)\}, \qquad\forall y \in [-a,a].
\ee
so that all maximizers of $\h(x) + \Ll(x,s;y,t)$ over $x \in \R$ are nonnegative. First, we observe that for $t$ large enough so that $\dir(t - s) \ge R_\ve$, for all $y \in [-a,a]$,
\begin{equation}\label{Mlm}
\begin{aligned}
	\quad \;\sup_{x \ge 0}\{\h(x) + \Ll(x,s;y,t)\}\ge M_L(\xi (t-s),t) 
	= (\dir^2 - \dir \ve)(t - s) + o(t).
	\end{aligned}
\end{equation}

Next, using \eqref{ub3'}, we obtain
\begin{align*}
&\quad \;\sup_{x \le 0}\{\h(x) + \Ll(x,s;y,t) \} \\ &\le \sup_{x \le 0}\{2\gamma x - \f{x^2}{t-s}  + C_\ve + F(x;t)\} \\
&\le\sup_{x \le 0}\{2\gamma x - 2\ve x - \f{x^2}{t-s}\}+ \sup_{x \le 0}\{2\ve x + F(x;t) \} + C_\ve. =  (\gamma - \ve)^2(t - s) + o(t).
\end{align*}
To justify the last equality, we note that the first supremum on the RHS above is equal to  $(\gamma - \ve)^2(t - s)$, while for large enough $t$, Lemma \ref{lem:dFbd} implies that the function inside the second supremum is increasing, so the maximum is achieved at $x = 0$. Note that $F(0;t) = o(t)$. Since $\gamma > -\dir$, by choosing $\ve > 0$ small enough, a comparison with \eqref{Mlm} verifies \eqref{max_right} for sufficiently large $t$.

Next, we find (approximately) where the maximizers of $M_{U}$ on $x\geq 0$ are. The function $M_U(x;t)$ has leading order $-(x - y)^2/(t-s)$, so maximizers exist on $x \ge 0$. A quick computation shows that $M_U(0;t) = o(t)$, so by \eqref{Mlm}, all maximizers are strictly positive for sufficiently large $t$. Hence, for any maximizer $x$, $M_U'(x;t) = 0$. 
First, note that
\begin{equation}
	M'_{U}(x;t)=2(\xi+\ve)-\frac{2x}{t - s}+F'(x;t).
\end{equation}
By Lemma \ref{lem:dFbd}, for sufficiently large $t > s$, $0 < F'(x;t) < \ve$ for all $x > 0$. For such $t$, and $y \in [-a,a]$, 
\begin{equation}
	 \{x \ge 0 :M'_U(x;t)=0\}\subseteq \big((\xi + \ve) (t - s), (\xi+2\ve)(t - s)\big).
\end{equation}
and that
\begin{equation}\label{d}
	\begin{aligned}
		&M'_U(x;t)<0 \quad \forall x\geq (\xi + 2\ve)(t - s)\\
		& M'_U(x;t)>0 \quad \forall x\leq   (\xi + \ve) (t - s).
	\end{aligned}
\end{equation}
Next we consider the supremum of $x\mapsto M_U$ outside the interval 
\[
I_{\ve}:=[(\xi-2\sqrt{\xi \ve})(t - s),(\xi+2\sqrt{\xi\ve})(t - s)].
\]
By choosing $\ve > 0$ small enough, then for $t$ large enough (depending on $\ve,a$), 
\begin{equation}\label{sub}
 \big((\xi - \ve) (t - s),(\xi+2\ve)(t - s)\big) \subseteq I_{\ve}.
\end{equation}
 From \eqref{d} and \eqref{sub}, we see that to determine the supremum of $M_U$ outside $I_{\ve}$ it is enough to take the maximum of $M_U$ at the endpoints of $I_{\ve}$. Plugging the end points of the interval on the right-hand side of  \eqref{sub} in $M_U$, for small enough $\ve$,
\begin{equation}\label{sup1}
	\begin{aligned}
	&M_U((\xi-2\sqrt{\xi\ve})(t - s);t)=\big[\xi^2 -3\xi\ve - 2\dir^{1/2}\ve^{3/2}\big](t - s) + o(t), \qquad\text{and}\\
	&M_U( (\xi+2\sqrt{\xi}\ve^{1/2})(t -s);t)=\big[\xi^2-3\xi\ve+2\xi^{1/2}\ve^{3/2}\Big](t - s) + o(t).
	\end{aligned}
\end{equation}
 It follows that for $\ve$ small enough and $t>t_0(\ve,a,s,C_{\text{DL}},C_\ve)$,
\begin{equation}\label{Mum}
	\sup_{x\notin I_{\ve}} M_U(x;t)\leq \max\{M_U((\xi-2\sqrt{\xi\ve})(t-s)),M_U((\xi+2\sqrt{\xi\ve})(t -s))\}\leq (\xi^2-2\xi\ve)(t-s).
\end{equation}
From \eqref{ub1'}, \eqref{max_right}, and \eqref{sub}, for sufficiently large $t$,
\begin{align}\label{sub2}
		&\big\{\sup_{x\notin I_{\ve}}  M_U(x;t)<\sup_{x\in I_{\ve}} M_L(x;t)\big\} \\
	&\subseteq\big\{\sup_{x\notin I_{\ve}}  \h(x)+\Ll(x,s;y,t)<\sup_{x\in I_{\ve}} \h(x)+\Ll(x,s;y,t)\quad \forall y\in[-a,a]\big\}\subseteq\{Z_\h^{a,s,t}\subset I_{\ve}\}.\nonumber 
	\end{align}
\eqref{Mlm} and \eqref{Mum}, imply that, almost surely, there is a random $t_0 = t_0(\ve,a,s) > 0$ so that for $t > t_0$,
\begin{equation}
	\sup_{x\notin I_{\ve}}  M_U(x;t)<\sup_{x\in I_{\ve}} M_L(x;t),
\end{equation}
so by replacing $\ve$ with $\ve^2/(4\dir)$, the inclusion \eqref{sub2} completes the proof in the case $\dir > 0$.

Now we prove the separate $\xi = 0$ case. This time, we set  $I_{\ve}=[2\sqrt{\ve}\tspb (t-s), 2\sqrt{\ve}\tspb (t-s)]$. Fix a point $x^\star \in \R$ so that $\h(x^\star) > -\infty$. Then, for $t$ large enough $x^\star \in I_{\ve}$, and
\be \label{ub4'}
\sup_{x \in I_{\ve}} \h(x) + \Ll(x,s;y,t) \ge \h(x^\star) + \Ll(x^\star,s;y,t)\ge \h(x^\star) - \f{(x^\star - y)^2}{t - s} +F(x^\star,t) = o(t).
\ee
 
 By the assumption \eqref{eqn:drift_assumptions} and upper semi-continuity, following a similar argument as in the previous case, for all $y \in [-a,a]$ and $x \in \R$,
\[
\begin{aligned}
\h(x)+\Ll(x,s;y,t)\leq M_{U}(x;t) &:= -\frac{x^2}{t-s}+\ve x+C_\ve+F(x;t)
\end{aligned}
\]
A similar proof as before shows that 
\[
\sup_{x \notin I_{\ve}} M_U(x;t) \le \max\{M_U(-2\sqrt \ve(t - s)),M_U(2\sqrt \ve (t - s))\} = -3\ve (t-s) + o(t),
\]
and comparison with \eqref{ub4'} completes the proof. 
\end{proof}

We believe the following Lemma is well-known, but we do not have a precise reference. In particular, \cite{KPZfixed} states that the KPZ fixed point preserves the space of linearly bounded continuous functions 
and gives regularity estimates for the KPZ fixed point.

\begin{lemma} \label{lem:max_restrict}
Let $\h \in \UC$ be initial data for the KPZ fixed point sampled at time $s \in \R$. Let $h_{\Ll}(t,y;s,\h)$ be defined as in \eqref{992}. 
Then, on the full-probability event of Lemma \ref{lem:Landscape_global_bound}, the following hold
\begin{enumerate}[label=\rm(\roman{*}), ref=\rm(\roman{*})]  \itemsep=3pt
    \item \label{itm:KPZcont} If $\h$ is continuous, then $(t,y) \mapsto h_\Ll(t,y;s,\h)$ is continuous on $(s,\infty) \times \R$. 
    \item \label{itm:KPZ_unif_line} For each compact set $K \subseteq (s,\infty)$, there exist constants $A = A(a,b,K)$ and $B = B(a,b,K)$ such that for all $t \in K$ and all $y \in \R$, $h_\Ll(t,y;s,\h) \le A + B|y|$. If we assume that $\h(x) \ge -a - b|x|$ for some constants $a,b > 0$, then we also obtain the bound $h_\Ll(t,y;s,\h) \ge -A - B|y|$ for all $t \in K$ and $y \in \R$ \rm{(}the upper bound $\h(x) \le a + b|x|$ is assumed in the definition of $\UC$\rm{)}.
    \item \label{itm:KPZrestrict}
    If there exists $a,b > 0$ so that $|\h(x)| \le a + b|x|$ for all $x$, then for any $t > s$, $\delta > 0$, there exists $Y = Y(t,\delta) > 0$ so that when $|y| \ge Y$, all maximizers of $\h(x) + \Ll(x,s;y,t)$ over $x \in \R$ lie in the interval $(y - |y|^{1/2 + \delta},y + |y|^{1/2+ \delta})$
\end{enumerate} 
\end{lemma}
\begin{proof}
 \noindent \textbf{Item \ref{itm:KPZcont}}
By definition of $\UC$, there exists constants $a,b > 0$ so that $\h(x) \le a + b|x|$ for all $x \in \R$. Combined with the bounds on the directed landscape in Lemma \ref{lem:Landscape_global_bound}, this implies that when $(y,t)$ varies over a compact set of $(s,\infty) \times \R$, the supremum in \eqref{992} can be taken uniformly over a  common compact set. Then, continuity of $\h$ and $\Ll$ gives the continuity of $h$.

 \noindent \textbf{Item \ref{itm:KPZ_unif_line}:}
By Lemma \ref{lem:Landscape_global_bound}, for each $x \in \R$,
\be \label{414}
\h(x) + \Ll(x,s;y,t) \le  a + b|x| -\tspc \f{(x - y)^2}{t - s} + C(t - s)^{1/3} \log^{2}\Bigl(\f{2\sqrt{x^2 + y^2 + s^2 + t^2} + 4}{(t - s)\wedge 1}\Bigr).
\ee
 The $\log$ term in \eqref{414}.  can be bounded by an affine function uniformly for $t \in K$ and $x,y \in \R$. Then, for constants $a_1 = a_1(a,b,K), b_1 = b_1(a,b,K)$, and $b_2 = b_2(a,b,K)$,
\begin{align*}
    h_\Ll(t,y;s,\h)
    &\le \sup_{x \in \R}\Bigl\{-\f{(x - y)^2}{t - s} + a_1 + b_1|x| + b_2|y| \Bigr\}  \\
    &\le \sup_{x \in \R}\Bigl\{-\f{(x - y)^2}{t - s} + a_1 + b_1x + b_2|y| \Bigr\}  \vee \Bigl\{-\f{(x - y)^2}{t - s} + a_1 - b_1 x + b_2|y| \Bigr\}\\
    &\le a_1 + b_2|y| + \Bigl(b_1 y +\f{b_1^2 t}{4}\Bigr) \vee \Bigl(-b_1y + \f{b_1^2t}{4}\Bigr),
\end{align*}
giving a linear bound, uniformly for $t \in K$. The lower bound is simpler: By Lemma \ref{lem:Landscape_global_bound} and the assumption $\h(x) \ge -a - b|x|$ for all $x \in \R$,
\be \label{413}
\begin{aligned}
h_\Ll(t,y;s,\h) &= \sup_{x \in \R}\{\h(x) + \Ll(x,s;y,t)\} \\ &\ge \h(y) + \Ll(y,s;y,t) \ge -a - b|y| - C(t - s)^{1/3}\log^2\Bigl(\f{2\sqrt{2y^2 + t^2 + s^2} + 4}{(t - s)\wedge 1}\Bigr),
\end{aligned}
\ee
and again, the $\log$  term can be bounded by an affine function, uniformly for $t \in K$ and $y \in \R$.

 \noindent \textbf{Item \ref{itm:KPZrestrict}:} By comparing \eqref{414} to \eqref{413}, when $|y|$ is sufficiently large, for $x \notin (y - |y|^{1/2 + \delta},y + |y|^{1/2 + \delta})$, \eqref{414} is strictly less than $h_\Ll(t,y;s,\h)$, so maximizers cannot lie outside the interval $(y - |y|^{1/2 + \delta},y + |y|^{1/2 + \delta})$.
\end{proof}

\begin{lemma} \label{lem:KPZ_preserve_lim}
The following holds simultaneously for all initial data and all $t > s$ on the event of probability one from Lemma \ref{lem:Landscape_global_bound}. Let $\h \in \UC$ be initial data for the KPZ fixed point, sampled at time $s$. 
For $t > s$, let $h_\Ll$ be defined as in \eqref{992}.
Then, simultaneously for all $t > s$,
\be \label{hliminfbd1}
\liminf_{x \to +\infty} \f{h_\Ll(t,x;s,\h)}{x} \ge \liminf_{x \to + \infty} \f{\h(x)}{x},\qquad\text{and}\qquad \limsup_{x \to - \infty} \f{h_\Ll(t,x;s,\h)}{x} \le \limsup_{x \to -\infty} \f{\h(x)}{x}.
\ee
Furthermore, assuming that $\h:\R \to \R$ is \textbf{continuous} and satisfies
\be \label{liminfsupfinite}
\liminf_{x \to \pm \infty} \f{\h(x)}{x} > -\infty \qquad\text{and}\qquad \limsup_{x \to \pm \infty} \f{\h(x)}{x} < +\infty,
\ee
then also
\be \label{hliminfbd2}
\limsup_{x \to + \infty} \f{h_\Ll(t,x;s,\h)}{x} \le \limsup_{x \to +\infty} \f{\h(x)}{x},\qquad\text{and}\qquad \liminf_{x \to -\infty} \f{h_\Ll(t,x;s,\h)}{x} \ge \liminf_{x \to - \infty} \f{\h(x)}{x}.
\ee

In particular, for \textbf{continuous} initial data $\h$ satisfying \eqref{liminfsupfinite}, if either \rm{(}or both\rm{)} of the limits
$
\lim_{x \to \pm \infty} \f{\h(x)}{x}
$
exist \rm{(}potentially with different limits on each side\rm{)}, then for $t > s$,
\[
\lim_{x \to \pm \infty} \f{h_\Ll(t,x;s,\h)}{x} = \lim_{x \to \pm \infty} \f{\h(x)}{x}.
\]
\end{lemma}
\begin{proof}
We start with \eqref{hliminfbd1} by proving the first inequality, and the other is analogous.
If $\liminf_{x \to +\infty} \f{\h(x)}{x} = -\infty$, there is nothing to show. Otherwise, let $\dir_1 \in \R$ be an arbitrary number smaller than $\liminf_{x \to +\infty} \f{\h(x)}{x}$. Let $y$ be sufficiently large and positive so that $\h(y)\ge \dir_1 y$. Then, using Lemma \ref{lem:Landscape_global_bound}, for such sufficiently large positive $y$,
\be \label{105}
\begin{aligned}
&\quad \;\sup_{x \in \R}\{\h(x) + \Ll(x,s;y,t)\} \\
&\ge \h(y) + \Ll(y,s;y,t) \ge \dir_1 y - C(t-s)^{1/3} \log^{2}\Bigl(\f{2\sqrt{2y^2 + t^2 + s^2} + 4}{(t - s)\wedge 1}\Bigr),
\end{aligned}
\ee
where $C$ is a constant. Therefore,
$
\liminf_{y \to \infty} \f{h_\Ll(t,y;s,\h)}{y} \ge \dir_1,
$
but this is true for all $\dir_1 < \liminf_{x \to +\infty} \f{\h(x)}{x}$, so 
$
\liminf_{y \to \infty} \f{h_\Ll(t,y;s,\h)}{y} \ge \liminf_{x \to +\infty} \f{\h(x)}{x}.
$

Next, we turn to proving \eqref{hliminfbd2}. Again, we prove the first inequality and the second follows analogously. Set $\dir_2 = \limsup_{x \to +\infty} \f{\h(x)}{x}$ and let $\ve > 0$. By continuity,  the assumption \eqref{liminfsupfinite} on the asymptotics of $\h$ implies there exist constants $a,b > 0$ so that $|\h(x)| \le a + b|x|$ for all $x \in \R$.  Lemma \ref{lem:max_restrict}\ref{itm:KPZrestrict} implies that for $\ve  > 0$ and sufficiently large $y > 0$, 
\begin{align*}
&\quad\; \sup_{x \in \R}\{\h(x) + \Ll(x,s;y,t)\} \\
&\le \sup_{x \in (y - y^{2/3},y + y^{2/3})}\Big\{(\dir_2 + \ve)x -\f{(x - y)^2}{t - s} + C(t - s)^{1/3} \log^{2}\Bigl(\f{2\sqrt{x^2 + y^2 + t^2 + s^2} + 4}{(t - s)\wedge 1}\Bigr) \Big\} \\
&\le \sup_{x \in (y - y^{2/3},y + y^{2/3})}\Big\{(\dir_2 + \ve)x-\f{(x - y)^2}{t - s } + \ve(x + y)  \Big\} \\
&= (\dir_2 + 3\ve)y + C(\ve,s,t,\dir_2),
\end{align*}

and so
\[
\limsup_{y \to \infty} \f{h_\Ll(t,y;s,\h)}{y} \le \dir_2 + 3\ve. \qedhere
\]
\end{proof}

\section{Proof of invariance and attractiveness (Theorem \ref{thm:invariance_of_SH})} \label{sec:invarattract_proof}

On a probability space $(\Omega,\F,\Pp)$, let $G^{\sqrt 2} = \{G^{\sqrt 2}_\dir\}_{\dir \in \R}$ be the stationary horizon with $\sigma = \sqrt 2$, independent of $\{\Ll(x,0;y,t):x,y \in \R, t > 0\}$. For $\dir \in \R$, define $h_\Ll(t,y;G_\dir^{\sqrt 2})$ as in \eqref{KPZSH2}:
\[
h_\Ll(t,y;G_\dir) = \sup_{x \in \R}\{G_\dir^{\sqrt 2}(x) + \Ll(x,0;y,t)\},\qquad\text{for all }y\in\R \text{ and } t > 0.
\]
Define the following state space: 
\be\label{Y}
\begin{aligned}
\Y &:= \bigl\{\{\h_\dir\}_{\dir \in \R} \in D(\R,C(\R)): \h_{\dir_1} \li \h_{\dir_2} \text{ for }\dir_1 < \dir_2,\\
&\qquad\qquad \text{ and  for all }\dir \in \R, \; \h_\dir(0) = 0 \text{ and } \h_\dir \text{ satisfies condition \eqref{eqn:drift_assumptions}} \\
&\qquad\qquad\qquad  \text{ with all $\limsup$ and $\liminf$ terms finite}\bigr\}. 
\end{aligned}
\ee
\begin{lemma} \label{lem:KPZ_preserve_Y}
 The space $\Y$   defined in \eqref{Y}  is a measurable subset of $D(\R,C(\R))$. Let  $\Ll$ be  the directed landscape, $\{\h_\dir\}_{\dir \in \R} \in \Y$, $h_0(\aabullet;\h_\dir)=\h_\dir$, and 
 \[
h_\Ll(t,y;\h_\dir) = \sup_{x \in \R}\{\h_\dir(x) + \Ll(x,0,y;t)\} \qquad \text{for } \ t>0, \  y\in\R \text{ and } \dir \in \R. 
\]
Then 
$t \mapsto \{h_\Ll(t,\aabullet;\h_\dir) - h_\Ll(t,0;\h_\dir) \}_{\dir \in \R}$ is a Markov process on $\Y$. Specifically, on the event of full probability from Lemma \ref{lem:Landscape_global_bound},    $\{h_\Ll(t,\aabullet;\h_\dir) - h_\Ll(t,0;\h_\dir) \}_{\dir \in \R} \in \Y$ for each $t > 0$.
\end{lemma}
\begin{proof}  
Recall that we defined the $\sigma$-algebra on $D(\R,C(\R))$ be the smallest $\sigma$-algebra such that $h_\dir(x)$ is measurable for each $x,\dir \in \R$. We claim that
\begin{align*} \Y &= \bigcap_{\alpha_1 < \alpha_2 \in \Q}\Bigl\{ \{\h_\dir\}_{\dir \in \R} \in D(\R,C(\R)):\h_{\alpha_1}(q_1,q_2) \le \h_{\alpha_2}(q_1,q_2) \text{ for $q_1 < q_2 \in \Q$,  } \\&\qquad\qquad\qquad\qquad \h_{\alpha_1}(0) = 0, \text{ and }\h_{\alpha_1} \text{ satisfies condition \eqref{eqn:drift_assumptions} with }\dir = \alpha_1 \\
&\qquad\qquad\qquad\qquad\qquad\qquad\qquad  \text{and} \limsup_{|x| \to \infty} \f{|h_{\alpha_1}(x)|}{|x|} < \infty \Bigr\}.\end{align*}
In particular, for $\{h_{\dir}\}_{\dir \in \R}$ lying in the set on the right, $\h_{\alpha_1} \li \h_{\alpha_2}$ for $\alpha_1 \le \alpha_2 \in \Q$ by continuity of each function. Then, $\h_{\dir_1} \li \h_{\dir_2}$ extends to all $\dir_1 < \dir_2 \in \R$ by taking limits of rational numbers $\alpha_1^n \searrow \dir_1$ and $\alpha_2^n \searrow \dir_2$. Since $\h_\dir(0) = 0$ and $\h_{\dir_1} \le \h_{\dir_2}$ for $\dir_1 \le \dir_2$, $\h_{\dir_1}(x) \le \h_{\dir_2}(x)$ for  $x > 0$, and the inequality flips for $x < 0$. Hence, if Condition \eqref{eqn:drift_assumptions} holds $\alpha \in \Q$ and all $\limsup,\liminf$ terms are finite, this extends to all $\dir \in \R$. To finish the proof of measurability of $\Y$, it remains to show that the four quantities $\liminf_{x \to \pm \infty}\f{\h_\dir(x)}{x}$ and $\limsup_{x \to \pm \infty}\f{\h_\dir(x)}{x}$ are measurable. This was demonstrated previously in Section \ref{sec:queue_intro}.

We now  show that $\{h_\Ll(t,\aabullet;\h_\dir)-h_\Ll(t,0;\h_\dir)\}_{\dir \in \R} \in \Y$ for all $t > 0$. Lemma \ref{lem:DL_crossing_facts}\ref{itm:KPZ_crossing_lemma} shows the preservation of the ordering of functions, Lemma \ref{lem:KPZ_preserve_lim} shows the preservation of limits, and Lemma \ref{lem:max_restrict}\ref{itm:KPZcont} shows that $h_\Ll(t,\abullet;\h_\dir) \in C(\R)$ for all $\dir$. It remains to show that $\{h_\Ll(t,\aabullet;\h)\}_{\dir \in \R} \in D(\R,C(\R))$ for each $t > 0$. Since $\h_{\dir_1} \li \h_{\dir_2}$,   Lemma \ref{lemma:max_monotonicity} and the global bounds of Lemma \ref{lem:Landscape_global_bound} imply  that, for each compact $K \subseteq \R$ and $\dir \in \R$, there exists a random $M = M(\dir,t,K) > 0$ such that for all $y \in K$, $\alpha \in (\dir - 1,\dir + 1)$,
\[
\sup_{x \in \R}\{\h_\alpha(x) + \Ll(x,0;y,t)\} = \sup_{x \in [-M,M]}\{\h_\alpha(x) + \Ll(x,0;y,t)\}.
\]
Then, it follows that $\{h_\Ll(t,\aabullet;\h_\dir)\}_{\dir \in \R}$, as an $\R \to C(\R)$ function of $\dir$, is right-continuous with left limits because this is true of $\{\h_\dir\}_{\dir \in \R}$. 
By the metric composition \eqref{eqn:metric_comp}  of the directed landscape $\Ll$,  for $0 < s < t$,
\begin{align*} 
&\quad \; h_\Ll(t,y;\h_\dir) - h_\Ll(t,0;\h_\dir) \\
&= \sup_{x \in \R, z \in \R}\{\h_\dir(x) + \Ll(x,0,z,s) + \Ll(z,s;y,t)\} \\
&\qquad\qquad\qquad\qquad - \sup_{x \in \R, z \in \R}\{\h_\dir(x) + \Ll(x,0,z,s) + \Ll(z,s;0,t)\} \\
&= \sup_{z \in \R}\{h_\Ll(s,z;\h_\dir) + \Ll(z,s;y,t)\} - \sup_{z \in \R}\{h_\Ll(s,z;\h_\dir) + \Ll(z,s;0,t)\} \\
&=\sup_{z \in \R}\{h_\Ll(s,z;\h_\dir) -h_\Ll(s,0;\h_\dir)+ \Ll(z,s;y,t)\}\\
&\qquad\qquad\qquad\qquad 
- \sup_{z \in \R}\{h_\Ll(s,z;\h_\dir)- h_\Ll(s,0;\h_\dir) + \Ll(z,s;0,t)\}.
\end{align*}  
Then, $t \mapsto \{h_\Ll(t,\aabullet;\h_\dir) - h_\Ll(s,0;\h_\dir)\}_{\dir \in \R}$ is Markov because $\Ll$ has i.i.d. time increments. 
\end{proof}

\begin{proof}[Proof of Theorem \ref{thm:invariance_of_SH}]
\textbf{Invariance:} Since the $\sigma$-algebra on $D(\R,C(\R))$ is generated by the projection maps and since the subspace $\Y \subseteq D(\R,C(\R))$ is preserved under $\Ll$ (Lemma \ref{lem:KPZ_preserve_Y}), it suffices to show that for  $-\infty < \dir_1 < \cdots < \dir_k < \infty$ and $t > 0$, 
\be \label{606}
\bigl(h_\Ll(t,\aabullet;G^{\sqrt 2}_{\dir_1})-h_\Ll(t,0;G^{\sqrt 2}_{\dir_1}),\ldots, h_\Ll(t,\aabullet;G^{\sqrt 2}_{\dir_k})-h_\Ll(t,0;G^{\sqrt 2}_{\dir_k})\bigr) \deq (G^{\sqrt 2}_{\dir_1},\ldots,G^{\sqrt 2}_{\dir_k}).
\ee

We prove this by using the invariance of $G^1$ for BLPP (Lemma \ref{lem:BLPP_invar}), the convergence of BLPP to $\Ll$ (Theorem \ref{thm:BLPPtoDL}), and the exit point bounds for BLPP from its stationary initial condition proved in Theorem \ref{thm:BLPP_exit_pts}. For this, choose $\rho > 0$, and let $\alpha,\beta,\tau$ be defined as in \eqref{eqn:BLPPchitau}:
\[
\chi^2 = \rho,\;\; \alpha = 2\sqrt \rho, \;\; \beta = \f{1}{\sqrt \rho}, \;\; \chi/\tau^2 = \f{1}{4\rho^{3/2}}.
\]
We observe here that
\be \label{eqn:diff2}
\f{\tau}{\chi^2} = \f{\tau}{\rho} = \f{\sqrt{\chi 4\rho^{3/2}}}{p} = \f{2\rho}{\rho} = 2
\ee

We will assume that our probability space $(\Omega,\F,\Pp)$ contains the coupling of $\{L^N\}_{N \ge 1}$ and $\Ll$ from Theorem \ref{thm:BLPPtoDL} and an independent realization of the SH $G^1$  with $\sigma = 1$. 
By convention, we say that $L^N(x,m;y,n) = -\infty$ if $m > n$ or $x > y$. For $1 \le i \le k, N \ge 1, x,y \in \R$, and $t > s$,  define
\begin{align*}
\lambda_i^N &= \beta + \f{\dir_i}{\chi N^{1/3}}, \\
H_i^N(t,y) &= \f{1}{\chi N^{1/3}} \sup_{x \in \R}\Bigl( \Bigl\{G_{\lambda_i^N}^1(x) + L^N(x,1;N\rho t + N^{2/3}\tau y, \lfloor Nt \rfloor)  \Bigr\} \\
&\qquad\qquad\qquad\qquad\qquad\qquad\qquad\qquad\qquad\qquad- \alpha N t - \beta \tau N^{2/3}y\Bigr), \\
Z_i^N(t,y) &= Z_{L^N}^{G_{\lambda_i^N}}(\lfloor Nt \rfloor, \rho t N + yN^{2/3} y) \qquad \qquad \text{(Recall \eqref{BLPPZgen})}, \\
G_i^N(x) &= \f{1}{\chi N^{1/3}}\Bigl(G^1_{\lambda_i^N}(N^{2/3} \tau x) - \beta \tau N^{2/3} x\Bigr), \qquad\text{and} \\
\Ll_N(x,s;y,t) &= \f{1}{\chi N^{1/3}} \Bigl(L^N(N\rho s + N^{2/3}\tau x, \lfloor Ns \rfloor ;N \rho t + N^{2/3} \tau y, \lfloor N t \rfloor  ) \\
&\qquad\qquad\qquad\qquad\qquad\qquad\qquad\qquad\qquad\qquad- \alpha N - \beta \tau N^{2/3}(y - x)\Bigr). 
\end{align*}

By the invariance of $G^1$ for $L^N$ (Lemma \ref{lem:BLPP_invar}) followed by shift invariance of the SH (Theorem \ref{thm:SH_dist_invar}\ref{itm:shinv}), the scaling relations of the SH (Theorem \ref{thm:SH_dist_invar}\ref{itm:SHscale} with \\$b = \f{1}{\chi N^{1/3}}$, $c = N^{1/3}\sqrt \tau$, and $\nu = \f{\beta \chi}{\tau} N^{1/3}$), and lastly \eqref{eqn:diff2}, for $t > 0$ and $Nt \ge 1$,
\be \label{Hiinvar}
\begin{aligned}
&\quad \; \{H_i^N(t,\abullet) - H_i^N(t,0)\}_{1 \le i \le k} \\
&\deq  \Bigl\{\f{1}{\chi N^{1/3}}\bigl(G_{\lambda_i^N}^1(N\rho t + N^{2/3} \tau \abullet) - G_{\dir_i}^1(N\rho t) - \beta \tau N^{2/3} \abullet\bigr)\Bigr\}_{1 \le i \le k} \\
&\deq \Bigl\{\f{1}{\chi N^{1/3}} (G_{\lambda_i^N}^1(N^{2/3} \tau \abullet) - \beta \tau N^{2/3} \abullet)\Bigr\}_{1 \le i \le k} \\
&= \Bigl\{\f{1}{\chi N^{1/3}} G_{\lambda_i^N}^1(N^{2/3} \tau \abullet) - \f{\tau}{\chi^2} \beta \chi N^{1/3} \abullet)\Bigr\}_{1 \le i \le k} \\
&\deq \bigl\{G^{\sqrt \tau/\chi}_{(\lambda_i^N - \beta)\chi N^{1/3}}\bigr\}_{1 \le i \le k} = \{G^{\sqrt 2}_{\dir_i}\}_{1 \le i \le k}. 
\end{aligned}
\ee
Therefore, to show
$
\{h_\Ll(t,\aabullet;G^{\sqrt 2}_\dir) - h_\Ll(t,0;G^{\sqrt 2}_\dir)\}_{\dir \in \R} \deq G^{\sqrt 2}
$
it suffices to show that, for each $t > 0$, the  following distributional convergence holds as processes on $C(\R,\R^k)$:
\be \label{BLPPKPZconv}
\{H_i^N(t,\abullet)\}_{1 \le i \le k} \overset{N \to \infty}{\Longrightarrow} \{h_{\Ll}(t,\abullet;G_{\dir_i}^{\sqrt 2})\}_{1 \le i \le k}.
\ee
From definition of $H_i^N$, we observe the following:
\be \label{Hicomp}
\begin{aligned}
    H_i^N(t,y) &= \f{1}{\chi N^{1/3}} \sup_{x \in \R}\Bigl( \Bigl\{G_{\lambda_i^N}^1(x) + L^N(x,1;N\rho t + N^{2/3}\tau y, \lfloor Nt \rfloor)  \Bigr\} \\
&\qquad\qquad\qquad\qquad\qquad\qquad\qquad\qquad\qquad\qquad- \alpha N t - \beta \tau N^{2/3}y\Bigr) \\
&= \f{1}{\chi N^{1/3}} \sup_{x \in \R}\Bigl\{G_{\lambda_i^N}^1(N^{2/3}\tau x) - \beta \tau N^{1/3} x \\
&\qquad\qquad\qquad + L^N(N^{2/3} \tau x,1;N\rho t + N^{2/3}\tau y, \lfloor Nt \rfloor)    
- \alpha N t - \beta \tau N^{2/3}(y - x) \Bigr\} \\
&= \sup_{x \in \R}\{G_i^N(x) + \Ll_N(x - \rho \tau^{-1} N^{-2/3},N^{-1};y,t)\} - \f{\alpha- \beta \rho}{\chi N^{1/3}}, 
\end{aligned}
\ee
where the $\f{\alpha - \beta \rho}{\chi N^{-1/3}}$ is a correction term that appears simply because we start $L^N$ along level $1$ instead of level $0$. We observed in \eqref{Hiinvar} that the distribution of $G_i^N$ does not depend on $N$; specifically, $\{G_i^N\}_{1 \le i \le k} \deq \{G^{\sqrt 2}_{\dir_i}\}_{1 \le i \le k}$. We define the process $\{G^{\sqrt 2}_{\dir_i}\}_{1 \le i \le k}$   on our probability space by setting $G^{\sqrt 2}_{\dir}(x) = G^1_\dir(2x)$ (Theorem \ref{thm:SH_dist_invar}\ref{itm:SHscale}). For $1 \le i \le k$, define
\[
\wt H_i^N(t,y) = \sup_{x \in \R} \{G^{\sqrt 2}_{\dir_i}(x) + \Ll_N(x - \rho \tau^{-1} N^{-2/3}; N^{-1};y,t)\}, 
\]
and let $\wt Z_i^N(t,y)$ be the largest maximizer of the above. Then, by \eqref{Hicomp},
\be \label{Hitilde}
(\{\wt H_i^N - \f{\alpha - \beta \rho}{\chi N^{1/3}}\}_{1 \le i \le k}, \{ N^{2/3} \tau \wt Z_i^N \}_{1 \le i \le k}) \deq (\{H_i^N\}_{1 \le i \le k},\{Z_i^N\}_{1 \le i \le k}),
\ee
noting that we made the change of variable $x \mapsto N^{2/3} \tau x$ in \eqref{Hicomp}. 
To show \eqref{BLPPKPZconv}, it therefore suffices to show that, for any $t,a,\ve > 0$, 
\[
\lim_{N \to \infty} \Pp\Bigl(\sup_{\substack{1 \le i \le k \\ y \in [-a,a]}} \Bigl|\wt H_i^N(t,y) - h_\Ll(t,y;G_{\dir_i})   \Bigr| > \ve   \Bigr) = 0.
\]
Before estimating this probability, we make one last shorthand definition:
\[
Z_{\Ll}^{i}(t,y) = \max \argmax \{G_{\dir_i}(x) + \Ll(x,0;y,t)\}.
\]
Observe that $Z_i^N(t,x) \le Z_i^N(t,y)$ for $x < y$ (Lemma \ref{lem:Zfmont}). By a similar argument, $Z_{\Ll}^i(t,x) \le Z_{\Ll}^i(t,y)$ for $x < y$. Then, we observe that  for any choice of $M > 0$,
\begin{align}
&\quad \; \Pp\Bigl(\sup_{\substack{1 \le i \le k \\ y \in [-a,a]}} \Bigl|\wt H_i^N(t,y) - h_\Ll(t,y;G_{\dir_i})   \Bigr| > \ve   \Bigr) \nonumber \\
&= \Pp\Bigl(\sup_{\substack{1 \le i \le k \\ y \in [-a,a]}} \Bigl|\sup_{x \in \R} \{G^{\sqrt 2}_{\dir_i}(x) + \Ll_N(x - \rho \tau^{-1} N^{-2/3}; N^{-1};y,t)\} \nonumber \\
&\qquad\qquad\qquad\qquad\qquad\qquad\qquad- \sup_{x \in \R}\{G^{\sqrt 2}_{\dir_i}(x) + \Ll(x,0;y,t)\}   \Bigr| > \ve   \Bigr)  \nonumber \\
&\le \sum_{i = 1}^k \Pp\Bigl(\sup_{y \in [-a,a]}\Bigl|\sup_{x \in [-M,M]} \{G^{\sqrt 2}_{\dir_i}(x) + \Ll_N(x - \rho \tau^{-1} N^{-2/3}; N^{-1};y,t)\} \nonumber \\
&\qquad\qquad\qquad\qquad\qquad\qquad\qquad - \sup_{x \in [-M,M]}\{G^{\sqrt 2}_{\dir_i}(x) + \Ll(x,0;y,t)\}   \Bigr| >\ve \Bigr)\label{606a} \\
&+ \Pp(Z_\Ll^i(t,-a) < -M) + \Pp(Z_\Ll^i(t,a) > M) \label{606b} \\
&+\sum_{i = 1}^k\bigl[ \Pp(\wt Z_i^N(t,-a) < -M) + \Pp(\wt Z_i^N(t,a) > M)\bigr]. \label{606d}
\end{align}

The term \eqref{606a} goes to $0$ as $N \to \infty$ by the uniform convergence on compact sets $\Ll_N \to \Ll$ (Theorem \ref{thm:BLPPtoDL}) and continuity of $\Ll$, the term \eqref{606b} goes to $0$ as $M \to \infty$ because, by Lemma \ref{lem:Landscape_global_bound}, $\Ll(x,0;y,t) \sim - \f{(x - y)^2}{t}$ and $G_{\dir_i}^{\sqrt 2}$ is a Brownian motion with drift, so  $x \mapsto G_{\dir_i}^{\sqrt 2}(x) + \Ll(x,0;y,t)$ almost surely has a finite maximizer. Then, using \eqref{Hitilde},\eqref{606d}  and the exit point bounds of Theorem \ref{thm:BLPP_exit_pts},
\begin{align*}
&\quad \; \limsup_{N \to \infty} \Pp\Bigl(\sup_{\substack{1 \le i \le k \\ y \in [-a,a]}} \Bigl|\wt H_i^N(t,y) - h_\Ll(t,y;G_{\dir_i})   \Bigr| > \ve   \Bigr) \\
&\le \sum_{i = 1}^k \limsup_{M \to \infty} \limsup_{N \to \infty}[ \Pp(\wt Z_i^N(t,-a) < -M) + \Pp(\wt Z_i^N(t,a) > M)\bigr]\\
&= \sum_{i = 1}^k \limsup_{M \to \infty} \limsup_{N \to \infty}[ \Pp( Z_i^N(t,-a) < -M\tau N^{2/3}) + \Pp(Z_i^N(t,a) > M \tau N^{2/3})\bigr] = 0.  
\end{align*}
This concludes the proof of invariance.

\noindent \textbf{Attractiveness and uniqueness:} The proof idea is similar to that of Theorem 3.3 in \cite{Bakhtin-Cator-Konstantin-2014}. 
Let $k\in\N$ and let $\bar{\dir} = (\dir_1,\ldots,\dir_k) \in \R^k$ be a strictly increasing vector. Let $\bar{\h}=(\h_1,...,\h_k) \in \UC^k$ satisfy \eqref{eqn:drift_assumptions} with $\h = \h_i$ and $\dir = \dir_i$ for $1 \le i \le k$.  
	Let $\ve > 0$. By the $\sigma = \sqrt 2$ case of Theorem \ref{thm:SH_sticky_thm}\ref{itm:SH_stick}, almost surely, there exists $\delta > 0$ such that $G^{\sqrt 2}_{\dir_i \pm \delta}(x) = G^{\sqrt 2}_{\dir_i}(x) \ \text{ for all } x \in [-a,a], \, 1 \le i \le k$. Hence, we may choose $\delta > 0$ small so that 
	\[
	\Pp\bigl\{G^{\sqrt 2}_{\dir_i \pm \delta}(x) = G^{\sqrt 2}_{\dir_i}(x) \ \,\forall x \in [-a,a], \, 1 \le i \le k\bigr\} \ge 1 - {\ve}/{2}.
	\]
	Then, by invariance of $G^{\sqrt 2}$ under $h_{\Ll}$, for all $t > 0$, 
	\be\label{lb1}\begin{aligned}
	\Pp\bigl\{h_\Ll(t,x;G^{\sqrt 2}_{\xi_i\pm \delta})-h_\Ll(t,0;G^{\sqrt 2}_{\xi_i\pm\delta})&= h_\Ll(t,x;G^{\sqrt 2}_{\xi_i})-h_\Ll(t,0;G^{\sqrt 2}_{\xi_i}) \\ &\qquad   \forall  x\in[-a,a], \, 1 \le i \le k\bigr\} \ge 1 - {\ve}/{2}.
\end{aligned}\ee
Recall the  sets $Z^{a,0,t}_f$ of exit points from \eqref{exitpt}.  Because $G^{\sqrt 2}_{\xi_i\pm\delta}$ is a Brownian motion with diffusivity $\sqrt 2$ and drift $2(\dir_i \pm \delta)$ (Proposition \ref{prop:SH_cons}\ref{itm:SHBM}), it  satisfies \eqref{eqn:drift_assumptions} for $\xi_i\pm\delta$. By 
the temporal reflection symmetry of Lemma \ref{lm:landscape_symm},  Lemma \ref{lem:unq} implies that, almost surely for all $t$ sufficiently large, 
\[
\begin{aligned}
Z^{a,0,t}_{G^{\sqrt 2}_{\xi_i-\delta}} &\subseteq [(\dir_i - 5\delta/4)t,(\dir_i - 3\delta/4)t],\;\; Z^{a,0,t}_{\h_i} \subseteq [(\dir_i - \delta/4)t,(\dir + \delta/4)t], \\
\;\;\text{and}\;\; Z^{a,0,t}_{G^{\sqrt 2}_{\xi_i+\delta}} &\subseteq [(\dir_i + 3\delta/4)t,(\dir_i + 5\delta/4)t]. 
\end{aligned}
\]
Hence, for sufficiently large $t$, 
\be \label{largp}
\Pp\bigl(Z^{a,0,t}_{G^{\sqrt 2}_{\xi_i-\delta}} \leq Z^{a,0,t}_{\h_i}\leq Z^{a,0,t}_{G^{\sqrt 2}_{\xi_i+\delta}} \ \, \forall \tspa 1 \le i \le k\bigr) > 1 - {\ve}/{2},
\ee
where    for $A,B\subseteq \R$ we say$A\leq  B$ if $\sup A\leq \inf B$.  By Lemma \ref{lem:DL_crossing_facts}\ref{itm:KPZ_crossing_lemma},   on the event in \eqref{largp} the following holds for all $x \in [0,a]$ and $1 \le i \le k$: 
\be\label{lb2}
\begin{aligned}
	&h_\Ll(t,x;G^{\sqrt 2}_{\xi_i-\delta})-h_\Ll(t,0;G^{\sqrt 2}_{\xi_i-\delta}) \\ &\qquad\qquad\qquad\qquad\qquad\leq h_\Ll(t,x;\h_i)-h_\Ll(t,0;\h_i) \\ &\qquad\qquad\qquad\qquad\qquad\qquad\qquad\qquad\qquad\qquad\leq h_\Ll(t,x;G^{\sqrt 2}_{\xi_i+\delta})-h_\Ll(t,0;G^{\sqrt 2}_{\xi_i+\delta}).  
 \end{aligned}
\ee
The reverse inequalities hold for $x \in [-a,0]$. Combining \eqref{lb1}--\eqref{lb2}, we have that for sufficiently large $t$,
\[
	\Pp\bigl\{h_\Ll(t,x;G^{\sqrt 2}_{\xi_i})-h_\Ll(t,0;G^{\sqrt 2}_{\xi_i})= h_\Ll(t,x;\h_i)-h_\Ll(t,0;\h_i)\ \, \forall  x\in[-a,a],1 \le i \le k\bigr\} \ge 1 - \ve. 
\]
The proof of Theorem \ref{thm:invariance_of_SH} is complete.  
\end{proof}

\chapter{Busemann process and the global structure of semi-infinite geodesics in the directed landscape} \label{chap:Buse}
\section{Introduction}
\subsection{Semi-infinite geodesics in the DL}
Recall the notion of geodesics for the directed landscape (DL) discussed in Section \ref{sec:DL_geod}. A \textit{semi-infinite geodesic} starting  from $(x,s) \in \R^2$ is a continuous path $g:[s,\infty) \to \R$ such that $g(s) = x$ and   the restriction of $g$ to each compact interval $[s,t]\subseteq[s,\infty)$ is a geodesic between $(x,s)$ and $(g(t),t)$. Such an infinite path $g$ has {\it direction} $\dir \in \R$ if $\lim_{t \to \infty} g(t)/t=\dir$.  Two semi-infinite geodesics $g_1$ and $g_2$ \textit{coalesce} if there exists $t$ such that $g_1(u) = g_2(u)$ for all $u\ge t$. If $t$ is the minimal such time, then  $(g_1(t),t)$ is the \textit{coalescence point}.
Two semi-infinite geodesics $g_1,g_2:[s,\infty) \to \R$ are \textit{distinct} if $g_1(t) \neq g_2(t)$ for at least some  $t\in(s,\infty)$ and \textit{disjoint} if $g_1(t) \neq g_2(t)$ for all $t\in(s,\infty)$.

In this  chapter, we give a detailed study of semi-infinite geodesics in the DL, as has appeared in the author's joint work with Busani and Sepp\"al\"ainen \cite{Busa-Sepp-Sore-22arXiv}. A significant consequence of Theorem \ref{thm:invariance_of_SH} is that the stationary horizon characterizes the distribution of the  Busemann process of the directed landscape (Theorem \ref{thm:Buse_dist_intro}). The Busemann process is a key tool that allows us to construct semi-infinite geodesics in every direction and from every initial point.    Many of the tools necessary for this study were developed in the BLPP context in the author's work with Sepp\"al\"ainen \cite{Seppalainen-Sorensen-21a,Seppalainen-Sorensen-21b}, to which we refer the reader for a detailed study of infinite geodesics in BLPP. 

It is not immediately obvious that semi-infinite geodesics should exist, nor is it obvious that they should have directions. The first to study semi-infinite geodesics in the DL was Rahman and Vir\'ag \cite{Rahman-Virag-21}, and we give a summary of their results in Section \ref{sec:RV_summ}. For a fixed direction, they established the existence of semi-infinite geodesics from all points and showed the coalescence of these geodesics. Closely tied to the study of geodesics is that of  \textit{Busemann functions}, a tool originating from differential geometry \cite{Busemann-55}. The work of \cite{Rahman-Virag-21} showed the existence of Busemann functions for a fixed direction. They also showed that, for a fixed point, there exists semi-infinite geodesics in each direction, with an at most countable set of directions for which the geodesic is not unique.

Starting from the definition in \cite{Rahman-Virag-21}, we construct the full  Busemann process across all directions. Through the properties of this process, we establish a classification of uniqueness and coalescence of semi-infinite geodesics in  the directed landscape. Similar constructions of the Busemann process and classifications for discrete and semi-discrete models have previously been achieved  \cite{Sepp_lecture_notes,Janjigian-Rassoul-2020b,Janjigian-Rassoul-Seppalainen-19,Seppalainen-Sorensen-21b}, but the procedure in the directed landscape is more delicate. One reason is that the  space is fully continuous.  Another difficulty is that Busemann functions in DL possess monotonicity only in horizontal directions, while    discrete and semi-discrete models  exhibit   monotonicity   in both  horizontal and vertical directions. A new perspective is needed to construct the Busemann process for arbitrary initial points. 

The full Busemann process is necessary for a complete understanding of the geometry of semi-infinite geodesics.  In particular, countable dense sets of initial points or directions cannot capture non-uniqueness of geodesics or the singularities of the Busemann process.

After the first version of \cite{Busa-Sepp-Sore-22arXiv} was posted, Ganguly and Zhang \cite{Ganguly-Zhang-2022a} gave an independent construction of a Busemann function and semi-infinite geodesics, again for a fixed direction. They   defined a notion of ``geodesic local time'' which was key to understanding the global fractal geometry of geodesics in DL. Later in \cite{Ganguly-Zhang-2022b}, the same authors showed that the discrete analogue of geodesic local time in exponential LPP converges to geodesic local time for the DL.

Matching the conjectures emanating from the nonexistence of bi-infinite geodesics in discrete models, it was recently shown by Bhatia \cite{Bhatia-23} that, with probability one, bi-infinite geodesics do not exist for the DL.

\subsection{Non-uniqueness of geodesics and random fractals} Among the key questions   is the uniqueness of semi-infinite geodesics in the directed landscape. We show the existence of a countably infinite, dense  random  set $\DLBusedc$ of directions $\dir$ such that, from each initial point in $\R^2$,    two semi-infinite geodesics in direction $\dir$ emanate,   separate immediately or after some time, and never return back together. It is interesting to relate this result and its proof to earlier work on disjoint finite geodesics.  

The set of exceptional pairs of points between which there is a non-unique geodesic in DL was studied in \cite{Bates-Ganguly-Hammond-22}. Their approach relied on \cite{Basu-Ganguly-Hammond-21} which studied the random nondecreasing function $z \mapsto \Ll(y,s;z,t) - \Ll(x,s;z,t)$ for fixed $x < y$ and $s < t$. This process is locally constant except on an exceptional set of Hausdorff dimension $\f{1}{2}$. From here \cite{Bates-Ganguly-Hammond-22} showed that for fixed $s < t$ and $x < y$, the set of $z \in \R$ such that there exist disjoint geodesics from $(x,s)$ to $(z,t)$ and from $(y,s)$ to $(z,t)$ is exactly the set of local variation of the function $z \mapsto \Ll(x,s;z,t) - \Ll(y,s;z,t)$, and therefore has Hausdorff dimension $\f{1}{2}$. Going further, they showed that for fixed $s < t$, the set of pairs $(x,y) \in \R^2$ such that there exist two disjoint geodesics from $(x,s)$ to $(y,t)$ also has Hausdorff dimension $\f{1}{2}$, almost surely. Later, this exceptional set in the time direction was studied in \cite{Ganguly-Zhang-2022a}, and was shown to have Hausdorff dimension $2/3$. Across the entire plane, this set has Hausdorff dimension $\f{5}{3}$.
In a similar spirit, Dauvergne \cite{Dauvergne-23} recently posted a paper detailing all the possible configurations of non-unique point-to-point geodesics, along with the Hausdorff dimensions--with respect to a particular metric--of the sets of points with those configurations.

Our focus is on the  limit  of the measure studied in \cite{Basu-Ganguly-Hammond-21},  namely, the nondecreasing function $\dir \mapsto \W_{\dir}(y,s;x,s)= \lim_{t \to \infty}[\Ll(y,s;t\dir,t) - \Ll(x,s;t\dir,t)]$, which is exactly the Busemann function in direction $\dir$. 
  The support of its Lebesgue-Stieltjes measure corresponds to the existence of disjoint geodesics (Theorem \ref{thm:Buse_pm_equiv}), but in contrast to \cite{Bates-Ganguly-Hammond-22},  the measure is supported on a countable discrete set instead of on a set of Hausdorff dimension $\f{1}{2}$ (Theorem \ref{thm:DLBusedc_description}\ref{itm:DL_Buse_no_limit_pts} and Remark \ref{rmk:shock_measure}). 

We encounter a Hausdorff dimension $\f{1}{2}$ set if we look along a fixed time level  $s$ for those space-time points $(x,s)$ out of  which there are disjoint semi-infinite geodesics in a {\it random, exceptional} direction (Theorem \ref{thm:Split_pts}\ref{itm:Hasudorff1/2}). Up to the removal of an at-most countable set, this Hausdorff dimension $\f{1}{2}$ set is the support of the random measure defined by the function
$
x \mapsto f_{s,\dir}(x) = \W_{\dir +}(x,s;0,s) - \W_{\dir -}(x,s;0,s),
$
where $\W_{\dir\pm}$ 
are the right and left-continuous Busemann processes (Theorem \ref{thm:random_supp}). This is a semi-infinite analogue of the result in \cite{Bates-Ganguly-Hammond-22}.  

The distribution of $f_{s,\dir}$ is  delicate.  The set of directions $\dir$ such that $\W_{\dir -} \neq \W_{\dir +}$, or equivalently such that 
$\tau_\xi=\inf\{x>0: f_{s,\dir}(x)>0\}<\infty $,
is the set   $\DLBusedc$ mentioned above. A fixed direction $\dir$ lies in $\DLBusedc$ with probability $0$. Because the SH describes the Busemann process for the DL,  Theorem \ref{thm:BusePalm} shows that the law of $f_{s,\dir}(\tau_\xi+\aabullet)$ on $\R_{\ge0}$, conditioned on $\dir \in \DLBusedc$ in the appropriate Palm sense,  is exactly that of the running maximum  of a Brownian motion, or equivalently, that of Brownian local time. This complements the fact that the function $z \mapsto \Ll(y,s;z,t) - \Ll(x,s;z,t)$ is locally absolutely continuous with respect to Brownian local time \cite{Ganguly-Hegde-2021}.  

Since the first version of \cite{Busa-Sepp-Sore-22arXiv} appeared, Bhatia \cite{Bhatia-22,Bhatia-23} posted two papers that use the results as inputs. The first, \cite{Bhatia-22} studies the Hausdorff dimension of the set of splitting points of geodesics along a geodesic itself. The second, \cite{Bhatia-23} answers an open problem presented in \cite{Busa-Sepp-Sore-22arXiv}. Namely, for all points in the set $\NU^{\dir \sig}$ defined in \eqref{NU0}, the geodesics split immediately from the initial point, and for a fixed direction $\dir$ the set $\NU^\dir$ almost surely has Hausdorff dimension $\f{4}{3}$. 

\section{Main results} \label{sec:BuseSIGmain}
In this chapter, the Busemann  process is used to construct a special class of semi-infinite geodesics called \textit{Busemann geodesics} simultaneously from all initial points and in all directions  (Theorem \ref{thm:DL_SIG_cons_intro}).  The definition of Busemann geodesics, along with a detailed study, comes in Section \ref{sec:Buse_geod_results}. 

The first theorem  states our conclusions for general semi-infinite geodesics.  The random countably infinite dense set $\DLBusedc$ of directions  is later characterized in \eqref{eqn:DLBuseDC_def} as the discontinuity set of the Busemann process, and its properties stated in  Theorem \ref{thm:DLBusedc_description}.

 We assume the probability space $(\Omega,\F,\Pp)$ of  the directed landscape $\Ll$  complete. All statements about semi-infinite geodesics are with respect to $\Ll$.



\begin{theorem} \label{thm:DLSIG_main} The following statements hold on a single event of full probability. There exists a random countably infinite dense subset $\DLBusedc$, of $\R$, such that parts \ref{itm:good_dir_coal}--\ref{itm:bad_dir_split} below hold. 
\begin{enumerate} [label=\rm(\roman{*}), ref=\rm(\roman{*})]  \itemsep=3pt
    \item \label{itm:all_dir} Every semi-infinite geodesic has a direction $\dir \in \R$. From each initial point $p \in \R^2$ and in each direction $\dir \in \R$, there exists at least one semi-infinite geodesic from $p$ in direction $\dir$.
    \item \label{itm:good_dir_coal} When $\dir \notin \DLBusedc$, all semi-infinite geodesics in direction $\dir$ coalesce. There exists a random set of initial points, of zero planar Lebesgue measure, outside of which the semi-infinite geodesic in each direction $\dir \notin \DLBusedc$ is unique. 
    \item \label{itm:bad_dir_split} When $\dir \in \DLBusedc$, there exist at least two families of semi-infinite geodesics in direction $\dir$, called the $\dir -$ and $\dir +$ geodesics. From every initial point $p \in \R^2$ there exists both a $\dir -$ geodesic and a $\dir +$ geodesic which  eventually separate and never come back together. All $\dir -$ geodesics coalesce, and all $\dir +$ geodesics coalesce.
    \end{enumerate}
    \end{theorem}

 
    \begin{remark}[Busemann geodesics and general geodesics] 
Theorem \ref{thm:DLSIG_main} is proved   by controlling all semi-infinite geodesics with  Busemann geodesics. Namely,  from each initial point $p$ and in each direction $\dir$, all  semi-infinite geodesics lie between  the  leftmost and rightmost  Busemann geodesics (Theorem \ref{thm:all_SIG_thm_intro}\ref{itm:DL_LRmost_SIG}). Furthermore, for  all   $p$ outside a random set of Lebesgue measure zero and all   $\dir \notin \DLBusedc$, the two extreme  Busemann geodesics coincide and thereby imply the uniqueness of the  semi-infinite geodesic from $p$ in direction $\dir$ (Theorem \ref{thm:DLSIG_main}\ref{itm:good_dir_coal}). 
   Even more generally, whenever $\dir \notin \DLBusedc$,      all semi-infinite geodesics in direction $\dir$ are Busemann geodesics (Theorem \ref{thm:DL_good_dir_classification}\ref{itm:DL_allBuse}).  This is presently unknown for $\dir \in \DLBusedc$, but may be expected   by virtue of what is known about  exponential LPP  \cite{Janjigian-Rassoul-Seppalainen-19}. 
   
   This work therefore gives a nearly complete  description of the global behavior of semi-infinite geodesics in the directed landscape. The conjecture that all semi-infinite geodesics are Busemann geodesics is equivalent to the following statement: In Item \ref{itm:bad_dir_split}, for $\dir \in \DLBusedc$, there are \textit{exactly} two families of coalescing semi-infinite geodesics in direction $\dir$. That is,  each $\dir$-directed semi-infinite geodesic  coalesces either with the $\dir-$ geodesics or the $\dir +$ geodesics.
\end{remark}

\begin{remark}[Non-uniqueness of geodesics]
The non-uniqueness of geodesics from initial points in a Lebesgue null set  in Theorem \ref{thm:DLSIG_main}\ref{itm:good_dir_coal} is temporary in the sense that these geodesics eventually coalesce. This forms a  ``bubble."  The first point of intersection after the split is the coalescence point (Theorem \ref{thm:DL_all_coal}\ref{itm:DL_split_return}). Hence, these particular geodesics form at most one bubble. This contrasts with  the non-uniqueness of  Theorem \ref{thm:DLSIG_main}\ref{itm:bad_dir_split}, where geodesics do not return together (Figure \ref{fig:non_unique_comp}).  Non-uniqueness is discussed in detail in Section \ref{sec:LR_sig}.
\end{remark}

\begin{remark}
The authors of \cite{Rahman-Virag-21} alluded to  non-uniqueness of geodesics. They showed that for a fixed initial point, with probability one, there are at most countably many directions with a non-unique geodesic. On page 23 of \cite{Rahman-Virag-21}, they note that the set of directions with a non-unique geodesic ``should be dense over the real line.'' Our result is that this set is   dense and, furthermore, it is the set $\DLBusedc$ of discontinuities of the Busemann process. 
\end{remark}

    \begin{figure}[t]
    \centering
    \includegraphics[height = 1.5in]{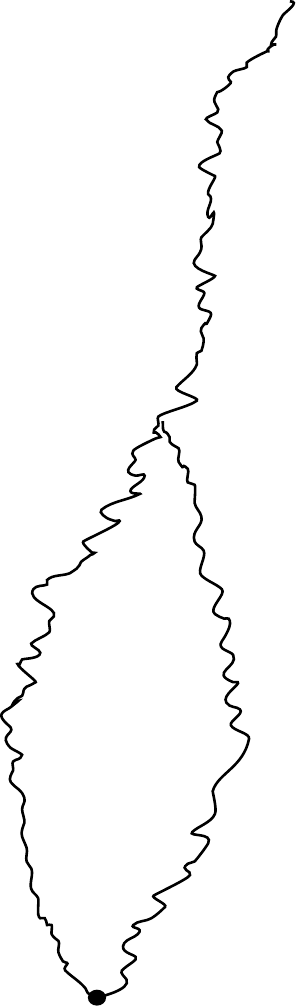} \qquad\qquad  
    \includegraphics[height = 1.5in]{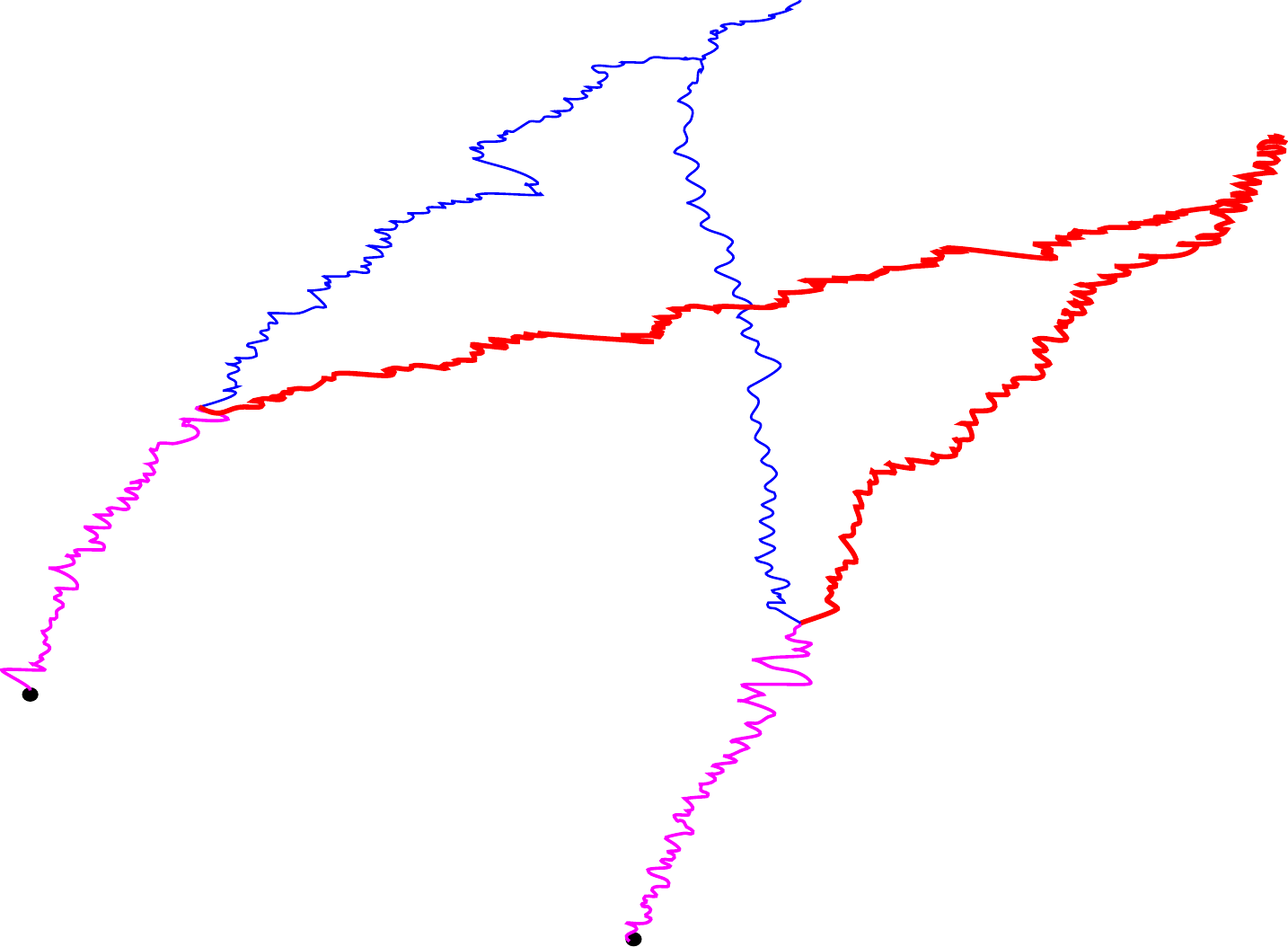}
    \caption{\small On the left, a depiction of the non-uniqueness in Theorem \ref{thm:DLSIG_main}\ref{itm:good_dir_coal}: geodesics separate and coalesce back together, forming a bubble. Due to work of Bhatia \cite{Bhatia-23} and Dauvergne \cite{Dauvergne-23}, this is the only possible configuration for this type of non-uniqueness--that is, geodesics which split and later coalesce can only split at the initial point. On the right, $\dir\in \DLBusedc$. The  blue/thin paths depict the $\dir -$ geodesics, while the red/thick paths  depict the $\dir +$ geodesics. From each point, the  $\dir -$ and $\dir +$ geodesics separate at points of $\Split$. The $\dir -$ and $\dir +$ families each have a coalescing structure.}
    \label{fig:non_unique_comp}
\end{figure}

\newpage 
    The second theorem of this section describes   the  set of initial  points with disjoint geodesics in the same direction. By disjoint, we mean that the geodesics only common point in the plane is the initial point. Let $\DLBusedc$ be the random set from Theorem \ref{thm:DLSIG_main} (precisely characterized in \eqref{eqn:DLBuseDC_def}).  Define the following random sets of splitting points.
    \begin{align}
    \Split_{s,\dir} &:= \{x \in \R:  \exists \text{ 
    \textbf{disjoint}}   \text{   }\text{semi-infinite  geodesics from  }(x,s) \text{ in direction }\dir\} \label{Split_sdir} \\
    \Split &:= \bigcup_{s \tspb\in\tspb \R, \, \dir\tspb \in\tspb \DLBusedc} \Split_{s,\dir} \times \{s\}. \label{eqn:gen_split_set}
    \end{align}
    \begin{remark}
    From Theorem \ref{thm:DLSIG_main}\ref{itm:good_dir_coal}, $\Split_{s,\dir} = \varnothing$  whenever $\dir \notin \DLBusedc$.
    \end{remark}
    
    \begin{theorem} \label{thm:Split_pts} The following hold.
    \begin{enumerate} [label=\rm(\roman{*}), ref=\rm(\roman{*})]  \itemsep=3pt 
    \item \label{itm:split_dense}
    On a single event of full probability, the set $\Split$ is dense in $\R^2$.
    \item \label{itm:splitp0} For each fixed $p \in \R^2$, $\Pp(p \in \Split) = 0$.
    \item \label{itm:Hasudorff1/2} For each $s \in \R$, on an $s$-dependent full-probability event, for every $\dir \in \DLBusedc$, the set $\Split_{s,\dir}$ has Hausdorff dimension $\f{1}{2}$.
    \item \label{itm:nonempty} On a single event of full probability, simultaneously for every $s \in \R$ and $\dir \in \DLBusedc$, the set $\Split_{s,\dir}$ is nonempty and unbounded in both directions.
\end{enumerate}
\end{theorem}

\begin{remark} \label{rmk:supports}
For each $s \in \R$ and $\dir \in \DLBusedc$, the set 
$\Split_{s,\dir}$
has an interpretation as the support of a random measure, up to the removal of a countable set. Thus, since $\DLBusedc$ is countable, for each $s \in \R$, the set $\{x \in \R: (x,s) \in \Split\}$ is the countable union of supports of random measures, up to the removal of an at most countable set. By Item \ref{itm:Hasudorff1/2}, this set also has Hausdorff dimension $\f{1}{2}$. 
Conditioning in the appropriate Palm sense on $\dir \in \DLBusedc$, the random measure whose support is ``almost'' $\Split_{s,\dir}$ is   equal to the local time of a Brownian motion  (Theorems \ref{thm:random_supp}, \ref{thm:BusePalm},and \ref{thm:indep_loc}). 
We expect that, simultaneously for all $s\in\R$, the set $\Split_{s,\dir}$ has Hausdorff dimension $\f{1}{2}$, but currently lack  a global result stronger than Item \ref{itm:nonempty}.  
\end{remark}

\subsection{Organization of the chapter}
In the following section, we cite results from \cite{Rahman-Virag-21}. The remainder of the chapter covers finer results on the Busemann process and semi-infinite geodesics. Sections \ref{sec:Buse_geod_results}--\ref{sec:geometry_sec} each start with several theorems that are then proved later in the section. The theorems can be read independently of the proofs.  Each section depends on the sections that came before. Section \ref{sec:Buse_geod_results} describes the construction of the Busemann process and infinite geodesics in all directions. Section \ref{sec:LR_sig} gives a detailed discussion of non-uniqueness of geodesics. Section \ref{sec:geometry_sec} is concerned with coalescence and connects the regularity of the Busemann process to the geometry of geodesics. The proofs of Theorems \ref{thm:DLSIG_main} and \ref{thm:Split_pts} come in Section \ref{sec:Busegeod_finalproofs}.

\section{Summary of the Rahman--Vir\'ag results}
\label{sec:RV_summ}
The paper \cite{Rahman-Virag-21} shows existence of the Busemann function for a fixed direction.
Below is a summary of their results that we use.

\begin{theorem}[\cite{Rahman-Virag-21}]\label{thm:RV-SIG-thm}
The following hold.
\begin{enumerate}  [label=\rm(\roman{*}), ref=\rm(\roman{*})]  \itemsep=3pt
    \item \label{itm:p_fixed} For fixed initial point $p$, there exist almost surely leftmost and rightmost semi-infinite geodesics $g_p^{\dir,\ell}$ and $g_p^{\dir,r}$ from $p$ in every direction $\dir$ simultaneously. There are at most countably many directions $\dir$ such that $g_p^{\dir,\ell}\neq g_p^{\dir,r}$ 
    \item \label{itm:d_fixed} For fixed direction $\dir$, there exist almost surely leftmost and rightmost geodesics $g_p^{\dir,\ell}$ and $g_p^{\dir,r}$ in direction $\dir$  from every initial point $p$.
    \item \label{itm:pd_fixed} For fixed $p =(x,s) \in \R^2$ and $\dir \in \R$, $g := g_p^{\dir,\ell}= g_p^{\dir,r}$ with probability one. 
    \item \label{itm:fixed_coal} Given  $\dir\in\R$, all semi-infinite geodesics in direction $\dir$ coalesce with probability one.
\end{enumerate}
\end{theorem}
\begin{remark}

Article  \cite{Rahman-Virag-21} used $-$ and $+$ in place of the superscripts $\ell$ and $r$ used above.  We replaced    $-/+$   with $\ell/r$ to avoid confusion with our $\pm$ notation that links with the left- and right-continuous Busemann processes.  As demonstrated in Section \ref{sec:LR_sig}, non-uniqueness of geodesics is properly characterized by two parameters $\sigg \in \{-,+\}$ and $S \in \{L,R\}$. 
\end{remark}

For fixed direction $\dir$, \cite{Rahman-Virag-21}  defines $\kappa^\dir(p,q)$ as the coalescence point of the rightmost geodesics in direction $\dir$ from initial points  $p$ and $q$. Then, they define the Busemann function
\be\label{RVW-def}
\W_\dir(p;q) = \Ll(p;\kappa^\dir(p,q)) - \Ll(q;\kappa^\dir(p,q)).
\ee


\begin{theorem}[\cite{Rahman-Virag-21}, Corollary 3.3, Theorem 3.5, Remark 3.1] \label{thm:RV-Buse}
$ $ 
\begin{enumerate}[label=\rm(\roman{*}), ref=\rm(\roman{*})]  \itemsep=3pt
    \item \label{itm:DL_Buse_BM} For each $t \in \R$, the process $x \mapsto \W_\dir(x,t;0,t)$ is a two-sided Brownian motion with diffusivity $\sqrt 2$ and drift $2\xi$.
\end{enumerate} 
Given a direction $\dir$,  the following hold on a $\dir$-dependent event of  probability one.
\begin{enumerate} [resume, label=\rm(\roman{*}), ref=\rm(\roman{*})]  \itemsep=3pt
\item \label{itm:fixed_additive} 
Additivity: $\W_\dir(p;q) + \W_\dir(q;r) = \W_\dir(p;r)$ for all $p,q,r \in \R^2$. 

    \item \label{itm:DL_Buse_var} For all $s < t$ and $x,y \in \R$,
    \[
    \W_\dir(x,s;y,t) = \sup_{z \in \R}\{\Ll(x,s;z,t) + \W_\dir(z,t;y,t)\}.
    \]
    The supremum is attained exactly at those $z$ such that $(z,t)$ lies on a semi-infinite geodesic from $(x,s)$ in direction $\dir$. 
    \item \label{itm:DL_Buse_cont} The function $\W_\dir:\R^4 \to \R$ is continuous. 
\end{enumerate}
Moreover:  
\begin{enumerate} [resume, label=\rm(\roman{*}), ref=\rm(\roman{*})]  \itemsep=3pt
\item \label{itm:DL_Buse_mont} For a pair of fixed directions $\dir_1 < \dir_2$, with probability one, for every  $t \in \R$ and 
      $x < y$, $\W_{\dir_1}(y,t;x,t) \le \W_{\dir_2}(y,t;x,t)$.
      \end{enumerate}
\end{theorem}

\section{Busemann process and Busemann geodesics} \label{sec:Buse_geod_results}
In this section, we first present a list of theorems regarding the Busemann process in Section \ref{sec:Buse_results}. Section \ref{sec:DL_SIG_intro} defines Busemann geodesics and states their main properties.  The proofs are found in Section \ref{sec:DL_Buse_cons}, except for the proofs of Theorem \ref{thm:DL_Buse_summ}\ref{itm:BuseLim1}-\ref{itm:global_attract} and the mixing in Theorem \ref{thm:Buse_dist_intro}\ref{itm:stationarity}, which are proved in Section \ref{sec:last_proofs1}, and Theorem \ref{thm:DLBusedc_description}\ref{itm:Busedc_t}, which is proved in Section \ref{sec:last_proofs2}.
\subsection{The Busemann process} \label{sec:Buse_results}

The Busemann process $\{\W_{\dir \sig}(p;q)\}$ is indexed by points $p,q \in \R^2$, a direction $\dir \in \R$, and a sign $\sigg \in \{-,+\}$. 
The following theorems describe this  process. The parameter $\sigg \in \{-,+\}$ denotes the left- and right-continuous versions of this process as a function of $\dir$. 

\begin{theorem} \label{thm:DL_Buse_summ}
 On $(\Omega,\F,\Pp)$, there exists a process
\[
\{\W_{\dir \sig}(p;q): \dir \in \R, \,  \sigg \in \{-,+\}, \, p,q \in \R^2\}
\]
satisfying the following properties.  All the  properties below hold on a single event of probability one, simultaneously for all directions $\dir \in \R$, signs $\sigg \in \{-,+\}$, and points $p,q \in \R^2$, unless otherwise specified. Below, for $p,q \in \R^2$, we define the sets
\be \label{eqn:DLBuseDC_def}
\DLBusedc(p;q) = \{\dir \in \R: \W_{\dir -}(p;q) \neq \W_{\dir +}(p;q)\}\qquad\text{and}\qquad\DLBusedc = \textstyle\bigcup_{p,q \,\in\, \R^2} \DLBusedc(p;q).
\ee
\begin{enumerate} [label=\rm(\roman{*}), ref=\rm(\roman{*})]  \itemsep=3pt
\item{\rm(Continuity)} \label{itm:general_cts}  As an $\R^4 \to \R$ function,  $(x,s;y,t) \mapsto \W_{\dir \sig}(x,s;y,t)$ is  continuous. 
 \item {\rm(Additivity)} \label{itm:DL_Buse_add} For all $p,q,r \in \R^2$, 
    $\W_{\dir \sig}(p;q) + \W_{\dir \sig}(q;r) = \W_{\dir \sig}(p;r)$.   In particular, $\W_{\dir \sig}(p;q) = -\W_{\dir \sig}(q;p)$ and $\W_{\dir \sig}(p;p) = 0$.
    \item {\rm(Monotonicity along a horizontal line)}
    \label{itm:DL_Buse_gen_mont} Whenever $\dir_1< \dir_2$, $x < y$, and $t \in \R$,
    \[
    \W_{\dir_1 -}(y,t;x,t) \le \W_{\dir_1 +}(y,t;x,t) \le \W_{\dir_2 -}(y,t;x,t) \le \W_{\dir_2 +}(y,t;x,t).
    \]
    \item {\rm(Backwards evolution as the KPZ fixed point)}\label{itm:Buse_KPZ_description} For 
    all $x,y \in \R$ and $s < t$,
    \be\label{W_var}
    \W_{\dir \sig}(x,s;y,t) = \sup_{z \in \R}\{\Ll(x,s;z,t) + \W_{\dir \sig}(z,t;y,t)\}.
    \ee
    \item {\rm(Regularity in the direction parameter)}
    \label{itm:DL_unif_Buse_stick}
    The process $\dir\mapsto\W_{\dir +}$ is right-continuous in the sense of uniform convergence on compact sets of functions $\R^4 \to \R$, and $\dir\mapsto\W_{\dir -}$ is left-continuous in the same sense. The restrictions to compact sets are locally constant in the parameter $\dir$:  for each $\dir \in \R$ and compact set $K \subseteq \R^4$ there exists a random $\ve =\ve(\dir,K)>0$ such that, whenever $\dir - \ve < \alpha < \dir < \beta < \dir + \ve$ and $\sigg \in \{-,+\}$,   we have these  equalities  for all $(x,s;y,t) \in K$: 
    \be \label{208}
    \W_{\alpha \sig}(x,s;y,t) = \W_{\dir -}(x,s;y,t)\qquad\text{and}\qquad\W_{\beta \sig}(x,s;y,t) = \W_{\dir +}(x,s;y,t). 
    \ee
   
    \item \rm{(Busemann limits I)} \label{itm:BuseLim1}  If $\dir \notin \DLBusedc$, then, for any compact set $K \subseteq \R^2$ and any net $r_t= (z_t,u_t)_{t \in \R_{\ge 0}}$ with $u_t \to \infty$ and $z_t/u_t \to \dir$ as $t \to \infty$, there exists $R \in \R_{>0}$ such that, for all $p,q \in K$ and $t \ge R$, 
    \[
    \W_{\dir}(p;q) = \Ll(p;r_t) - \Ll(q;r_t).
    \]
    \item {\rm(Busemann limits II)} \label{itm:BuseLim2} For all $\dir \in \R$, $s \in \R$, $x < y \in \R$, and any net $(z_t,u_t)_{t \in \R_{\ge 0}}$ in $\R^2$ such that   $u_t \to \infty$ and $z_t/u_t \to \dir$ as $t \to \infty$,
    \begin{align*}
    \W_{\dir -}(y,s;x,s) &\le \liminf_{t \to \infty} \Ll(y,s;z_t,u_t) - \Ll(x,s;z_t,u_t) \\  &\le \limsup_{t \to \infty} \Ll(y,s;z_t,u_t) - \Ll(x,s;z_t,u_t) \le \W_{\dir +}(y,s;x,s).
    \end{align*}
    \item \rm{(Global attractiveness)} \label{itm:global_attract} Assume that $\dir \notin \DLBusedc$, and let $\h \in \UC$ satisfy condition \eqref{eqn:drift_assumptions} for the parameter $\dir$. For $s < t$, let 
    \[
        h_{s,t}(x;\h) = \sup_{z \in \R}\{\Ll(x,s;z,t) + \h(z)\}.   
    \]
    Then, for any $s \in \R$ and $a > 0$, there exists a random $t_0  = t_0(a,\dir,s)<\infty$ such that for all $t > t_0$ and $x \in [-a,a]$, $h_{s,t}(x;\h) - h_{s,t}(0;\h) = \W_{\dir}(x,s;0,s)$.
    \end{enumerate}
\end{theorem}   
\begin{remark}
 Item \ref{itm:BuseLim1} is novel in   that it shows the limits simultaneously for all $\dir \notin \DLBusedc$, uniformly over  compact subsets of $\R^2$.  The existence of Busemann limits in fixed directions
 is shown in \cite{Rahman-Virag-21} and \cite{Ganguly-Zhang-2022a}.
Item \ref{itm:global_attract} is analogous to  Theorem 3.3 in \cite{Bakhtin-Cator-Konstantin-2014} and Theorem 3.3 in \cite{Bakhtin-Li-19} on the global solutions of the Burgers' equation with random forcing.  
When comparing with  \cite{Bakhtin-Cator-Konstantin-2014,Bakhtin-Li-19},  note  that our geodesics travel north while theirs head south. 

\end{remark}

\noindent We describe the distribution of this process. The key to Item \ref{itm:SH_Buse_process} is Theorem \ref{thm:invariance_of_SH}.
\begin{theorem} \label{thm:Buse_dist_intro}
The following hold.
\begin{enumerate} [label=\rm(\roman{*}), ref=\rm(\roman{*})] \itemsep=3pt
\item {\rm(Independence)} \label{itm:indep_of_landscape} For each $T \in \R$, these processes are independent: 
    \begin{align*}
    &\{\W_{\dir \sig}(x,s;y,t): \dir \in \R, \,\sigg \in \{-,+\}, \, x,y \in \R, \, s,t \ge T \} \\[4pt] 
    &\qquad\qquad\qquad \text{and } \ \ 
    \{\Ll(x,s;y,t): x,y \in \R,\, s < t \le T\}. 
    \end{align*}
    
 \item {\rm(Stationarity and mixing)} \label{itm:stationarity} The process
 \be \label{eqn:stat}
 \{\Ll(v),\W_{\dir \sig}(p;q):v \in \Rup, \, p,q \in \R^2, \,\dir \in \R, \,\sigg \in \{-,+\} \}
 \ee
 is stationary and mixing under shifts in any space-time direction. More precisely,  let $a,b \in \R$, not both $0$, and $z > 0$. Set $r_z = (az,bz)$.  Then, the process \eqref{eqn:stat} is stationary and mixing {\rm(}for fixed $a,b$ as $z \to +\infty${\rm)} under the transformation
 \begin{align*}
 \bigl\{\Ll(v), \W_{\dir \sig}(p;q )\bigr\} \mapsto T_{z;a,b}\{\Ll,\W \} := \{\Ll(v + (r_z;r_z)), \W_{\dir \sig}(p + r_z;q +r_z)\},
 \end{align*}
 where the process on each side is a function of $(v,(p,q))\in\Rup \times \R^4$.  Mixing means that, for all $k\in\Z_{>0}$, $\dir_1,\dotsc,\dir_k\in\R$,  and Borel subsets $A,B \subseteq C(\Rup,\R)\times C(\R^4,\R)^k$, if we denote $\W_{\dir_{1:k}}=(\W_{\dir_{1}},\dotsc,\W_{\dir_{k}})\in C(\R^4,\R)^k$, then 
 \begin{align*}
&\lim_{z \to \infty}\Pp\Bigl(\{\Ll, \W_{\dir_{1:k}}\} \in A, \{T_{z;a,b} \Ll, T_{z;a,b}\W_{\dir_{1:k}}\} \in B\Bigr)  \\
     &\qquad\qquad\qquad =\Pp\bigl( \{\Ll, \W_{\dir_{1:k}}\} \in A\bigr) \Pp\bigl(\{\Ll, \W_{\dir_{1:k}}\} \in B \bigr) .
 \end{align*}
    \item  {\rm(Distribution along a time level)}\label{itm:SH_Buse_process} For each $t \in \R$,  the following equality in distribution holds between random elements of the Skorokhod space $D(\R,C(\R))$:
\[
\{\W_{\dir +}(\aabullet,t;0,t)\}_{\dir \in \R} \deq \bigl\{G^{\sqrt 2}_{\dir}(\aabullet) \bigr\}_{\dir \in \R},
\]
where $G^{\sqrt 2}$ is the stationary horizon with diffusivity $\sqrt 2$ and drifts $2\xi$.
\end{enumerate}
\end{theorem}
\begin{remark}
 Combining  Items   \ref{itm:indep_of_landscape} and \ref{itm:SH_Buse_process} with  Theorem \ref{thm:DL_Buse_summ}\ref{itm:Buse_KPZ_description} gives a description of the   Busemann process on the full plane $\R^2$. 
\end{remark}

We describe the random sets of Busemann  discontinuities defined in \eqref{eqn:DLBuseDC_def}. Item\ref{itm:Busedc_t} below states that the discontinuities of the Busemann process are present along each horizontal line. Since the Busemann process along each line is described by the SH (Theorem \ref{thm:Buse_dist_intro}\ref{itm:SH_Buse_process}), the distributional invariances for $\XiSH$ proved in Theorem \ref{thm:SH_sticky_thm} also hold for $\DLBusedc(\abullet,t;\abullet,t)$.
  
\begin{theorem} \label{thm:DLBusedc_description}
The following hold on a single event of probability one.
    \begin{enumerate} [label=\rm(\roman{*}), ref=\rm(\roman{*})]  \itemsep=3pt
    \item \label{itm:Busedc_horiz_mont} For each $t \in \R$, the set    $\DLBusedc(x,t;-x,t)$ is nondecreasing as a function of $x \in \R_{\ge 0}$.
    \item \label{itm:Busedc_t} For $s,\dir \in \R$, define the function
    \be \label{fsdir}
x \mapsto f_{s,\dir}(x) := \W_{\dir +}(x,s;0,s) - \W_{\dir -}(x,s;0,s).
\ee
Then, $\dir \in \DLBusedc$ if and only if, for all $s \in \R$,
\be \label{bad_ub}
\lim_{x \to \pm \infty} f_{s,\dir}(x) = \pm \infty.
\ee
In particular, simultaneously for all $s,x \in \R$ and all sequences  $|x_k|\to\infty$,
    \be \label{eqn:dcset_union1}
    \DLBusedc = \bigcup_k \tspb\DLBusedc(x_k,s;x,s).
    \ee
    \item \label{itm:DL_dc_set_count} The set $\DLBusedc$ is countably infinite and dense in $\R$, while for each fixed $\dir \in \R$, \\ $\Pp(\dir \in \DLBusedc) = 0$. In particular, the full-probability event of the theorem can be chosen so that $\DLBusedc$ contains no directions $\dir \in \Q$.
    \item \label{itm:DL_Buse_no_limit_pts} For each $p \neq q$ in $\R^2$, the set $\DLBusedc(p;q)$ is discrete, that is, has no limit points in $\R$. The function $\dir \mapsto \W_{\dir -}(p;q) = \W_{\dir +}(p;q)$ is constant on each open interval $I \subseteq (\R \setminus \DLBusedc(p;q))$. For $t \in \R$, on a $t$-dependent full-probability event, for all $x < y$, $\DLBusedc(y,t;x,t)$ is infinite and unbounded, for both positive and negative $\dir$.    
\end{enumerate}
\end{theorem}
\begin{remark} \label{rmk:shock_measure}
Item \ref{itm:Busedc_t} states that all discontinuities of the Busemann process are present on each  horizontal ray. By Item \ref{itm:DL_Buse_no_limit_pts}  $\dir \mapsto \W_{\dir \pm}(p;q)$ are the left- and right-continuous versions of a jump process. This function defines a random signed measure  supported on a discrete set. When $p$ and $q$ lie on the same horizontal line, this function is monotone  (Theorem \ref{thm:DL_Buse_summ}\ref{itm:DL_Buse_gen_mont}) and the support of the measure is exactly the set of directions   at which a properly chosen coalescence point of semi-infinite geodesics  jumps (see Definition \ref{def:coal_pt} and Theorems \ref{thm:DL_eq_Buse_cpt_paths}--\ref{thm:Buse_pm_equiv}).
\end{remark}

\subsection{Busemann geodesics} \label{sec:DL_SIG_intro}
The study of Busemann geodesics starts with this definition. 
\begin{definition} \label{def:LR_maxes}
For $\dir  \in \R$, $\sigg \in \{-,+\}$,  $(x,s) \in \R^2$ and $t\in[s,\infty)$, let $g_{(x,s)}^{\dir \sig,L}(t)$ and $g_{(x,s)}^{\dir \sig,R}(t)$ denote, respectively, the leftmost and rightmost maximizer of $\Ll(x,s;y,t) + \W_{\dir \sig}(y,t;0,t)$ over $y \in \R$.
\end{definition}
\begin{remark}
The modulus of continuity bounds of the directed landscape recorded in Lemma \ref{lem:Landscape_global_bound}, along with continuity of $W_{\dir \sig}$, imply that $\lim_{t \searrow s} g_{(x,s)}^{\dir \sig,L/R}(t) = x$, so we define $g_{(x,s)}^{\dir \sig,L/R}(s) = x$.
\end{remark}

As noted earlier, Rahman and Vir\'ag \cite{Rahman-Virag-21} showed the existence of semi-infinite  geodesics, almost surely for a fixed initial point  across all directions and almost surely for a fixed direction across all initial points. We  extend this  simultaneously across both all initial points and directions. 
Theorem \ref{thm:RV-Buse}\ref{itm:DL_Buse_var}, quoted from \cite{Rahman-Virag-21}, states that for a {\it fixed} direction $\dir$, with probability one at times $t > s$, 
the maximizers $z$ of the function $\Ll(x,s;z,t) + \W_{\dir}(z,t;0,t)$ are exactly the points on semi-infinite $\xi$-directed geodesics from $(x,s)$.
Theorem \ref{thm:DL_SIG_cons_intro} clarifies this on a global scale: across all directions, initial points and signs, one can construct semi-infinite geodesics from the Busemann process.   
Furthermore,  $g_{(x,s)}^{\dir \sig,L}$ and $g_{(x,s)}^{\dir \sig,R}$ both define  semi-infinite geodesics in direction $\dir$  
and give  the leftmost (or rightmost) geodesic between any two of their points. We use this fact heavily in the present chapter. 

\begin{theorem} \label{thm:DL_SIG_cons_intro}
 The following hold on a single event of probability one across all initial points $(x,s) \in \R^2$, times $t>s$, directions $\dir  \in \R$, and signs $\sigg \in \{-,+\}$.
 \begin{enumerate} [label=\rm(\roman{*}), ref=\rm(\roman{*})]  \itemsep=3pt
 \item \label{itm:intro_SIG_bd} 
 All maximizers of $z\mapsto\Ll(x,s;z,t) + \W_{\dir \sig}(z,t;0,t)$ are finite. Furthermore, as  $x,s,t$ vary over a compact set $K\subseteq \R$ with $s \le t$, the set of all maximizers is bounded.
    \item \label{itm:arb_geod_cons} 
    Let $s = t_0 < t_1 < t_2 < \cdots$ be an arbitrary increasing sequence with $t_n \to \infty$. Set $g(t_0) = x$, and for each $i \ge 1$, let $g(t_i)$ be \textit{any} maximizer of $\Ll(g(t_{i - 1}),t_{i - 1};z,t_i) + W_{\dir \sig}(z,t_i;0,t_i)$ over $z \in \R$. Then, pick \textit{any} geodesic of $\Ll$ from $(g(t_{i - 1}),t_{i - 1})$ to $(g(t_i),t_i)$, and for $t_{i - 1} < t < t_i$, let $g(t)$ be the location of this geodesic at time $t$. Then, regardless of the choices made at each step, the following hold.
    \begin{enumerate} [label=\rm(\alph{*}), ref=\rm(\alph{*})]
        \item \label{itm:g_is_geod} The path $g:[s,\infty)\to \R$ is a semi-infinite geodesic.
        \item \label{itm:weight_of_geod} For all  $ t < u$ in $[s,\infty)$,
    \be \label{eqn:SIG_weight}
    \Ll(g(t),t;g(u),u) = \W_{\dir \sig}(g(t),t;g(u),u).
    \ee
    \item \label{itm:maxes} For all  $ t < u$ in $[s,\infty)$, $g(u)$ maximizes $\Ll(g(t),t;z,u) + \W_{\dir \sig}(z,u;0,u)$ over $z \in \R$. 
    \item \label{itm:geo_dir} The geodesic $g$ has direction $\dir$, i.e., $g(t)/t \to \dir$ as $t \to \infty$. 
    \end{enumerate}
    \item \label{itm:DL_all_SIG} For 
    $S \in \{L,R\}$, $g_{(x,s)}^{\dir \sig,S}:[s,\infty) \to \R$ is a semi-infinite geodesic from $(x,s)$ in direction $\dir$. Moreover, for any $s \le t < u$, we have that 
    \[
    \Ll\bigl(g_{(x,s)}^{\dir \sig,S}(t),t;g_{(x,s)}^{\dir \sig,S}(u),u\bigr) = \W_{\dir \sig}\bigl(g_{(x,s)}^{\dir \sig,S}(t),t;g_{(x,s)}^{\dir \sig,S}(u),u\bigr),
    \]
    and $g_{(x,s)}^{\dir \sig,S}(u)$ is the leftmost/rightmost {\rm(}depending on $S${\rm)} maximizer of \\$\Ll(g_{(x,s)}^{\dir \sig,S}(t),t;z,u) + \W_{\dir \sig}(z,u;0,u)$ over $z \in \R$.
    \item \label{itm:DL_LRmost_geod} 
    The path $g_{(x,s)}^{\dir \sig,L}$ is the leftmost geodesic between any two of its points, and $g_{(x,s)}^{\dir \sig,R}$ is the rightmost geodesic between any two of its points.
    \end{enumerate}
    \end{theorem}
\begin{definition}
We refer to the  geodesics constructed  in Theorem \ref{thm:DL_SIG_cons_intro}\ref{itm:arb_geod_cons} as $\dir \sig$ \textit{Busemann geodesics}, or simply {\it $\dir \sig$ geodesics}. 
\end{definition} 
\begin{remark}
The geodesics  $g_{(x,s)}^{\dir \sig,L}$ and $g_{(x,s)}^{\dir \sig,R}$ are special  Busemann geodesics. By   Theorem \ref{thm:DL_SIG_cons_intro}\ref{itm:DL_all_SIG}--\ref{itm:DL_LRmost_geod}, for any sequence $s=t_0 < t_1< t_2 < \cdots$ with $t_n \to \infty$, the path $g = g_{(x,s)}^{\dir \sig,L}$ can be constructed by choosing $g(t_i)$ as the leftmost maximizer of $\Ll(g(t_{i - 1}),t_{i - 1};z,t_i) + \W_{\dir \sig}(z,t_i;0,t_i)$ over $z \in \R$, and   for $t \in (t_{i - 1},t_i)$, taking $g(t)$ to be     the leftmost geodesic from $(g(t_{i - 1}),t_{i - 1})$ to $(g(t_i),t_i)$. The analogous statement holds for $L$ replaced with $R$ and ``leftmost'' replaced with ``rightmost''.
\end{remark}

\subsection{Construction and proofs for the Busemann process and Busemann geodesics} \label{sec:DL_Buse_cons}
This section proves the results of Sections \ref{sec:Buse_results} and \ref{sec:DL_SIG_intro}. The order in which the items are proved is somewhat delicate, so we outline that here. After proving some lemmas, we prove Theorem \ref{thm:DL_Buse_summ}\ref{itm:general_cts}--\ref{itm:Buse_KPZ_description} and Theorem \ref{thm:Buse_dist_intro}. We then skip ahead to constructing the semi-infinite geodesics, culminating in the proof of Theorem \ref{thm:DL_SIG_cons_intro}. Afterward, we turn to the proof of the regularity in Theorem \ref{thm:DL_Buse_summ}\ref{itm:DL_unif_Buse_stick}, then prove Theorem \ref{thm:DLBusedc_description}, except for Item \ref{itm:Busedc_t}, which is proved in Section \ref{sec:last_proofs2}.

We construct a full-probability event $\Omega_1$.  Later in \eqref{omega2} and \eqref{omega3} follow full-probability events $\Omega_3 \subseteq \Omega_2 \subseteq \Omega_1$. For the rest of the proofs, we work  almost exclusively on these events. Once the events are constructed and shown to have full probability, the remaining proofs are deterministic statements that hold on those events.
\be \label{omega1}
\text{We define $\Omega_1 \subseteq \Omega$ to be the event of probability one on which the following hold.}
\ee
\begin{enumerate} [label=\rm(\roman{*}), ref=\rm(\roman{*})]  \itemsep=3pt
    \item \label{om1lrf} Simultaneously for all $(x,s;y,t) \in \Rup$ there exist leftmost and rightmost geodesics {\rm(}possibly in agreement{\rm)} between $(x,s)$ and $(y,t)$ (see Section \ref{sec:DL_geod}).
    \item \label{om1rsi} For each rational direction $\dir \in \Q$ and each point $p \in \R^2$, there exist leftmost and rightmost semi-infinite geodesics {\rm(}possibly in agreement{\rm)} from $p$ in direction $\dir$, and all semi-infinite geodesics in direction $\dir$ coalesce {\rm(}see Theorem \ref{thm:RV-SIG-thm}, Items \ref{itm:d_fixed} and \ref{itm:fixed_coal}{\rm)}.
    \item \label{om1pdf} For each rational direction $\dir \in \Q$ and each rational point $p \in \Q^2$, there is a unique semi-infinite geodesic from $p$ in direction $\dir$ (see Theorem \ref{thm:RV-SIG-thm}\ref{itm:pd_fixed}).
    \item \label{om1dirrat} For each rational direction $\dir\in \Q$, the Busemann process defined by \eqref{RVW-def} satisfies conditions \ref{itm:fixed_additive}--\ref{itm:DL_Buse_cont} of Theorem \ref{thm:RV-Buse}. For any pair $\dir_1 < \dir_2$ or rational directions, Item \ref{itm:DL_Buse_mont} of Theorem \ref{thm:RV-Buse} holds.
    \item \label{om1agree} For each $(x,t,y,\dir) \in \Q^4$,
    $
        \lim_{\Q \ni \alpha \to \dir} \W_\alpha(y,t;x,t) = \W_{\dir}(y,t;x,t).    
    $
    \item \label{om1asym} For every rational time $t \in \Q$ and rational direction $\dir \in \Q$, 
    \be \label{eqn:rat_asymp}
        \lim_{x \to \pm \infty} x^{-1}  {\W_\dir(x,t;0,t)} = 2\dir. 
    \ee
    This holds with probability one by properties of Brownian motion and Theorem \ref{thm:RV-Buse}\ref{itm:DL_Buse_BM}.
    \item \label{om1appB} The conclusions of Lemmas \ref{lem:Landscape_global_bound}, \ref{lm:BGH_disj}, and \ref{lem:geod_pp} hold for  $\Ll$. Note that then  Lemma \ref{lem:Landscape_global_bound} holds also for the reflected version  
    $
    \{\Ll(y;-t,x;-s):(x,s;y,t) \in \Rup\}. 
    $   
\end{enumerate}

To justify $\Pp(\Omega_1)=1$, it remains to check Item \ref{om1agree}. By   Theorem \ref{thm:RV-Buse}\ref{itm:DL_Buse_mont},  for $y \ge x$,
\be \label{lrlimit}
\lim_{\Q \ni \alpha \nearrow \dir} \W_{\alpha}(y,t;x,t) \le \W_{\dir}(y,t;x,t) \le \lim_{\Q \ni \alpha \searrow \dir} \W_{\alpha}(y,t;x,t).
\ee
By Theorem \ref{thm:RV-Buse}\ref{itm:DL_Buse_BM}, $\W_{\alpha}(y,t;x,t) \sim \Nor(2\alpha(y - x),2(y - x))$. Hence, 
all terms in \eqref{lrlimit} have the same distribution and are almost surely equal.

Now, on the full-probability event $\Omega_1$, we have defined the process
 \be \label{rat_process}
 \{\W_{\alpha}(p;q):p,q \in \R^2,\alpha \in \Q \}.
 \ee
 On this event, for an arbitrary direction $\dir$, and $t,x,y \in \R$, define
\be \label{eqn:Buse_def}
\W_{\dir -}(y,t;x,t) = \lim_{\Q \ni \alpha \nearrow \dir}\W_{\alpha }(y,t;x,t)\;\;\text{and}\;\; \W_{\dir+}(y,t;x,t) = \lim_{\Q \ni \alpha \searrow \dir} \W_{\alpha }(y,t;x,t).
\ee
By Theorem \ref{thm:RV-Buse}\ref{itm:DL_Buse_mont}, these limits exist for all $t \in \R$. Complete the definition by setting,  
\be \label{eqn:gen_Buse_var}\begin{aligned} 
\text{  for $s < t$, } \ 
\W_{\dir \sig}(x,s;y,t) &= \sup_{z \in \R}\{\Ll(x,s;z,t) + \W_{\dir \sig}(z,t;y,t)\},\\
\text{ and finally for $s > t$, } \ \W_{\dir \sig}(x,s;y,t) &= -\W_{\dir \sig}(y,t;x,s).
\end{aligned} \ee
With this construction in place, we prove an intermediate lemma.

\begin{lemma} \label{lem:DL_horiz_Buse}
The following hold on the event $\Omega_1$, across all points, directions and signs.
\begin{enumerate} [label=\rm(\roman{*}), ref=\rm(\roman{*})]  \itemsep=3pt
\item \label{itm:DL_agree_horiz} For all $x,y,t \in \R$, and $\dir \in \Q$, $\W_{\dir -}(y,t;x,t) = \W_{\dir +}(y,t;x,t) = \W_\dir(y,t;x,t)$, where $W_\dir$ is the originally defined Busemann function from \eqref{rat_process}. 
    \item \label{itm:DL_h_add} Horizontal Busemann functions are additive: $\forall\,x,y,z,t \in \R$, $\dir \in \R$, and $\sigg \in \{-,+\}$,
    \[
    \W_{\dir \sig}(x,t;y,t) + \W_{\dir \sig}(y,t;z,t) = \W_{\dir \sig}(x,t;z,t).
    \]
    \item \label{itm:DL_h_unif_conv} For every $t,\dir \in \R$, the limits   \eqref{eqn:Buse_def} hold uniformly over $(x,y)$ on compact sets. Further, for each $t,\dir \in \R$ and $\sigg \in \{-,+\}$,
    these limits hold in the same sense: 
    \be \label{itm:horiz_lim}
    \lim_{\alpha \nearrow \dir} \W_{\alpha \sig}(y,t;x,t) = \W_{\dir -}(y,t;x,t)\ \text{ and }\  \lim_{\alpha \searrow \dir}\W_{\alpha \sig}(y,t;x,t) = \W_{\dir +}(y,t;x,t). 
    \ee
    
    \item \label{itm:DL_lim} For every $\dir \in \R$, $\sigg \in \{-,+\}$, $(p,q) \mapsto \W_{\dir \sig}(p;q)$ is continuous, and for each $t \in \R$,
    \be \label{eqn:horiz_asymptotics}
    \lim_{x \to \pm \infty}  x^{-1} {\W_{\dir \sig}(x,t;0,t)} = 2\dir.
    \ee
    \end{enumerate}

\end{lemma}
\begin{proof}
We prove Item \ref{itm:DL_agree_horiz} last.

\noindent
\textbf{Item \ref{itm:DL_h_add}} follows from the same property in rational directions (Theorem \ref{thm:RV-Buse}\ref{itm:fixed_additive}).

\noindent \textbf{Item \ref{itm:DL_h_unif_conv}:} The monotonicity of the horizontal Busemann process from Theorem \ref{thm:RV-Buse}\ref{itm:DL_Buse_mont} extends to all directions by limits. That is, for any two rational directions $\dir_1 < \dir_2$ and any real $x < y$, and $t$,
    \be \label{eqn:Buse_mont}
    \W_{\dir_1 -}(y,t;x,t) \le \W_{\dir_1}(y,t;x,t) \le  \W_{\dir_1 +}(y,t;x,t) \le \W_{\dir_2 -}(y,t;x,t), 
    \ee
    and when $\dir_1 \notin \Q$, the same monotonicity holds, removing the middle term that does not distinguish between $\pm$. 
Hence, the limits as $\alpha \nearrow \dir$ and $\alpha \searrow \dir$ exist and agree with the limits from rational directions (without the $\sigg$). Without loss of generality, we take the compact set to be $[a,b]^2$. Then, by \eqref{eqn:Buse_mont} and Lemma \ref{lem:ext_mont}, for $\alpha < \dir$, $\sigg \in \{-,+\}$, and $a \le x \le y \le b$,
\be \label{156}
0 \le \W_{\dir-}(y,t;x,t)- \W_{\alpha \sig }(y,t;x,t)   \le   \W_{\dir-}(b,t;a,t)- \W_{\alpha \sig }(b,t;a,t),
\ee
and for general $(x,y) \in [a,b]^2$, 
\[
|\W_{\dir-}(y,t;x,t)- \W_{\alpha \sig }(y,t;x,t)|   \le   |\W_{\dir-}(b,t;a,t)- \W_{\alpha \sig }(b,t;a,t)|,
\]
so the limit as $\alpha \nearrow \dir$ is uniform on compacts. An analogous argument applies to   $\alpha \searrow\dir$.

\noindent \textbf{Item \ref{itm:DL_lim}:} 
For $t,\dir \in \R$ and $\sigg \in \{-,+\}$,  the continuity of  
$
(x,y) \mapsto \W_{\dir \sig}(y,t;x,t)
$
follows from Item \ref{itm:DL_h_unif_conv} and the continuity for rational $\xi$ in Theorem \ref{thm:RV-Buse}\ref{itm:DL_Buse_cont}. Before showing the general continuity, we show   the limits \eqref{eqn:horiz_asymptotics}. For   $\dir, t \in \Q$, \eqref{eqn:rat_asymp} holds by definition of $\Omega_1$.
Keeping $\dir \in \Q$, let $s \in \R$, and let $t > s$ be rational. By Theorem \ref{thm:RV-Buse}\ref{itm:fixed_additive}--\ref{itm:DL_Buse_var}, 
\begin{align*}
\W_{\dir}(x,s;0,s) 
&= \W_{\dir}(x,s;0,t) + \W_{\dir}(0,t;0,s) \\
&= \sup_{z \in \R}\{\Ll(x,s;z,t) + \W_{\dir}(z,t;0,t)\}+ \W_{\dir}(0,t;0,s).
\end{align*}
 Then, by Lemma \ref{lem:KPZ_preserve_lim} (for the temporally reflected $\Ll$),
 $
 \ddd\lim_{x \to \pm \infty} x^{-1}  {\W_{\dir}(x,s;0,s)} = 2\dir.
 $
 Now, let $\dir \in \R$, $\sigg \in \{-,+\}$, and $t \in \R$ be arbitrary. Then, the monotonicity of \eqref{eqn:Buse_mont} implies that for $\alpha < \dir < \beta$ with $\alpha,\beta \in \Q$,
 \[
 \alpha \le \liminf_{x \to \infty} x^{-1}{\W_{\dir \sig}(x,t;0,t)} \le \limsup_{x \to \infty} x^{-1} {\W_{\dir \sig}(x,t;0,t)}  \le \beta.
 \]
 Sending $\Q \ni \alpha \nearrow \dir$ and $\Q \ni \beta \searrow \dir$ implies \eqref{eqn:horiz_asymptotics} for $+\infty$. The case    $x \to -\infty$ follows a symmetric argument.

Lastly, the continuity of $(x,y) \mapsto \W_{\dir \sig}(y,t;x,t)$ and \eqref{eqn:horiz_asymptotics} imply that $\W_{\dir \sig}(x,t;0,t) \le a + b|x|$ for some constants $a,b$. The general continuity follows from \eqref{eqn:gen_Buse_var} and  Lemma \ref{lem:max_restrict}\ref{itm:KPZcont}.

\noindent \textbf{Item \ref{itm:DL_agree_horiz}:} 
The statement holds for all $x,y,t,\dir \in \Q$ by Item \ref{om1agree} of $\Omega_1$. The continuity proved in Item \ref{itm:DL_lim} extends this to all $x,y,t \in \R$. 
\end{proof}

\noindent Recall Definition \ref{def:LR_maxes} of  the extreme maximizers $g_{(x,s)}^{\dir \sig,L/R}(t)$.

\begin{lemma} \label{lem:bounded_maxes}
For each $\omega \in \Omega_1$, $(x,s;y,t) \in \Rup$, $\dir \in \R$, and $\sigg \in \{-,+\}$,
    \be \label{eqn:zto_pminfty}
    \lim_{z \to \pm \infty} \Ll(x,s;z,t) + \W_{\dir \sig}(z,t;y,t) = -\infty
    \ee
    so that   $g_{(x,s)}^{\dir \sig,L/R}$ are well-defined.
    Let $K \subseteq \R$ be a compact set,   $\dir \in \R$ and $\sigg \in \{-,+\}$.  Then, there exists a random $Z = Z(\dir \sig,K)  \in (0,\infty)$ such that for all $x,s,t \in K$ with $s < t$ and $S \in \{L,R\}$, $|g_{(x,s)}^{\dir \sig,S}(t)| \le Z$.
\end{lemma}
\newpage
\begin{proof}
By the continuity and asymptotics of Lemma \ref{lem:DL_horiz_Buse}\ref{itm:DL_lim},  for all $t \in \R$, there  exists $ a,b > 0$ such that $|\W_{\dir \sig}(x,t;0,t)| \le a + b|x|$ for all $x \in \R$.   Lemma \ref{lem:Landscape_global_bound} implies  $\Ll(x,s;z,t) \sim -\f{(z - x)^2}{t - s}$,  which gives   \eqref{eqn:zto_pminfty}. 
Next we observe that
\be \label{107}
\begin{aligned}
&\inf_{x,s,t \in K, s < t}\sup_{z \in \R}\{\Ll(x,s;z,t) + \W_{\dir \sig}(z,t;0,t)\} \\[-3pt] 
 &\qquad\qquad\qquad \ge \inf_{x,s,t \in K, s < t} \Ll(x,s;x,t) + \W_{\dir \sig}(x,t;0,t) > -\infty.
\end{aligned}
\ee
The last inequality is justified as follows. Since $\W_{\dir \sig}(x,t;0,t)$ evolves backwards in time as the KPZ fixed point \eqref{eqn:gen_Buse_var}, Lemma \ref{lem:max_restrict}\ref{itm:KPZ_unif_line} implies that   $a$ and $b$ can be chosen uniformly for $t \in K$.   Lemma \ref{lem:Landscape_global_bound} states that $\forall x,s,t \in \R$ with $s < t$, there is a constant $C$ such that 
\[
\Ll(x,s;x,t) \ge - C(t - s)^{1/3}\log^2\bigl(\tfrac{2\sqrt{2x^2 + s^2 + t^2} + 4}{(t - s)\wedge 1}\bigr).
\]
Taking the infimum over $x,s,t \in K$ with $s < t$ yields the last inequality in \eqref{107}.  

To  contradict the last statement of the lemma, assume there exist   maximizers $z_n$ of $\Ll(x_n,s_n;z,t_n) + \W_{\dir \sig}(z,t_n;0,t_n)$ over $z \in \R$ such that  $x_n,s_n,t_n \in K$  but $|z_n| \to \infty$.
Then, by \eqref{107},
\be \label{9}
\liminf_{n \to \infty} \Ll(x_n,s_n;z_n,t_n) + \W_{\dir \sig}(z_n,t_n;0,t_n) > -\infty,
\ee
but since $z_n \to \infty$ and $x_n,s_n,t_n \in K$ for all $n$, $\Ll(x_n,s_n;z_n;t_n) \sim -\f{(z_n - x_n)^2}{t_n - s_n}$ by Lemma \ref{lem:Landscape_global_bound}. By the bound $|\W_{\dir \sig}(x,t;0,t)| \le a + b|x|$ that holds uniformly for $t \in K$ and $x \in \R$, the inequality \eqref{9} cannot hold.
\end{proof}

\begin{proof}[Proof of Theorem \ref{thm:DL_Buse_summ}, Items  \ref{itm:general_cts}--\ref{itm:Buse_KPZ_description}]
The full-probability event of these items is $\Omega_1$. The remaining items are proved later.

\noindent
\textbf{Item \ref{itm:general_cts} (Continuity):} This was proved in Lemma \ref{lem:DL_horiz_Buse}\ref{itm:DL_lim}. 

\noindent \textbf{Item \ref{itm:DL_Buse_add} (Additivity):} 
First, we show that on $\Omega_1$ for $s < t$, $x \in \R$, $\dir_1 < \dir_2$, and $S \in \{L,R\}$,
    \be \label{eqn:mont_maxes}
    -\infty < g_{(x,s)}^{\dir_1 -,S}(t) \le g_{(x,s)}^{\dir_1 +,S}(t) \le g_{(x,s)}^{\dir_2 -,S}(t) \le g_{(x,s)}^{\dir_2 +,S}(t)  < \infty. 
    \ee
    The finiteness of the maximizers comes from Lemma \ref{lem:bounded_maxes}. The rest of \eqref{eqn:mont_maxes} follows from the monotonicity of \eqref{eqn:Buse_mont} and Lemma \ref{lemma:max_monotonicity}. Next, we show that for $(x,s;y,t) \in \R^4$ and $\dir \in \R$, $\W_{\alpha }(x,s;y,t)$ converges pointwise to $\W_{\dir -}(x,s;y,t)$ as $\Q \ni \alpha \nearrow \dir$. The same holds for limits from the right, with $\dir -$ replaced by $\dir +$ (Later we prove that the convergence is locally uniform).   By \eqref{eqn:gen_Buse_var}, it suffices to assume $s < t$. By \eqref{eqn:mont_maxes} and the additivity of Lemma \ref{lem:DL_horiz_Buse}\ref{itm:DL_h_add} when $s = t$, for all $\alpha \in [\dir - 1,\dir + 1] \cap \Q$ and $\sigg \in \{-,+\}$,
\begin{align*}
\W_{\alpha }(x,s;y,t) &= \sup_{z \in \R}\{\Ll(x,s;z,t) + \W_{\alpha }(z,t;y,t)\} \\
&= \sup_{z \in \R}\{\Ll(x,s;z,t) + \W_{\alpha}(z,t;0,t)\} + \W_{\alpha}(0,t;y,t) \\
&= \sup_{z \in [g_{(x,s)}^{(\dir - 1)-,L}(t),g_{(x,s)}^{(\dir + 1)+,R}(t)]}\{\Ll(x,s;z,t) + \W_\alpha(z,t;0,t)\} + \W_{\alpha}(0,t;y,t).
\end{align*}
By Lemma \ref{lem:DL_horiz_Buse}\ref{itm:DL_h_unif_conv}, $\W_{\alpha}(z,t;y,t)$ converges uniformly on compact sets to $\W_{\dir -}(x,t;y,t)$ as $\Q \ni \alpha \nearrow \dir$ and to $\W_{\dir +}(x,t;y,t)$ as $\Q \ni \alpha \searrow \dir$. This implies the desired pointwise convergence. The additivity follows from the additivity  for rational $\dir$  (Theorem \ref{thm:RV-Buse}\ref{itm:fixed_additive}). 

\noindent \textbf{Item \ref{itm:DL_Buse_gen_mont} (Monotonicity along a horizontal line):} This was previously proven as Equation \eqref{eqn:Buse_mont}.

\noindent \textbf{Item \ref{itm:Buse_KPZ_description} (Backwards evolution as the KPZ fixed point):} This follows directly from the construction \eqref{eqn:gen_Buse_var}.

\noindent   Item \ref{itm:DL_unif_Buse_stick} is proved after the proof of Theorem \ref{thm:Buse_dist_intro}, and Items \ref{itm:BuseLim2}--\ref{itm:global_attract} are proved after the proof of Theorem \ref{thm:DL_good_dir_classification}. No subsequent results  depend on Items \ref{itm:BuseLim2}--\ref{itm:global_attract}, except  the mixing in Theorem \ref{thm:Buse_dist_intro}\ref{itm:stationarity}, proven later.
\end{proof}

\begin{proof}[Proof of Theorem \ref{thm:Buse_dist_intro} (Distributional properties)]
\noindent \textbf{Item \ref{itm:indep_of_landscape} (Independence):} We \\ know that, for $T \in \R$, $\{\Ll(x,s;y,t):s,y \in \R, s < t \le T\}$ is independent of $\{\Ll(x,s;y,t):s,y \in \R, T \le s < t\}$ . From the definition of the Busemann process and \eqref{eqn:Buse_def}--\eqref{eqn:gen_Buse_var}, the process
\[
\{\W_{\dir \sig}(x,s;y,t): \dir \in \R, \,\sigg \in \{-,+\}, \, x,y \in \R, \, s,t \ge T \} 
\]
is a function of 
$\{\Ll(x,s;y,t):s,y \in \R, T \le s < t\}$, and independence follows.

\newpage 
\noindent \textbf{Item \ref{itm:stationarity} (Stationarity):}
Similarly as the previous item, the stationarity of the process follows from the stationarity of the directed landscape from Lemma \ref{lm:landscape_symm}\ref{itm:time_stat}. The mixing properties  will be proven in Section \ref{sec:last_proofs1}, along with Items \ref{itm:BuseLim2}--\ref{itm:global_attract} of Theorem \ref{thm:DL_Buse_summ}.

\noindent \textbf{Item \ref{itm:SH_Buse_process} (Distribution along a time level):} By the additivity of Theorem \ref{thm:DL_Buse_summ}\ref{itm:DL_Buse_add} and the variational definition \eqref{eqn:gen_Buse_var}, for $x \in \R$, $s < t$, and $\sigg \in \{-,+\}$, on the full-probability event $\Omega_1$,
\begin{align*}
&\W_{\dir \sig}(x,s;0,s) = \W_{\dir \sig}(x,s;0,t) - \W_{\dir \sig}(0,s;0,t) \\
&\qquad 
=\sup_{y \in \R}\{\Ll(x,s;y,t) + \W_{\dir \sig}(y,t;0,t)\} - \sup_{y \in \R} \{\Ll(0,s;y,t) + \W_{\dir \sig}(y,t;0,t)\}.
\end{align*}
By Item \ref{itm:indep_of_landscape}, Theorem \ref{thm:DL_Buse_summ}\ref{itm:DL_Buse_gen_mont},  and  Items \ref{itm:DL_h_unif_conv} and \ref{itm:DL_lim} of Lemma \ref{lem:DL_horiz_Buse}, $\{\W_{\dir +}(\abullet,t;0,t):\dir \in \R\}_{t \in \R}$ is a reverse-time Markov process that almost surely lies in the state space $\Y$ defined in \eqref{Y}. By the stationarity of Item \ref{itm:stationarity}, the law of $\{\W_{\dir +}(\abullet,t;0,t):\dir \in \R\}$ must be invariant for this process.  By the temporal reflection invariance of the directed landscape  (Lemma \ref{lm:landscape_symm}\ref{itm:DL_reflect}),  $\{\W_{\dir +}(\abullet,t;0,t):\dir \in \R\}$ is also invariant for the KPZ fixed point, forward in time. The uniqueness part of Theorem \ref{thm:invariance_of_SH} completes the proof. 
\end{proof}

\begin{lemma} \label{lem:L_and_Buse_ineq}
For every $\omega \in \Omega_1$ and $(x,s;y,t) \in \Rup$, $\Ll(x,s;y,t) \le \W_{\dir \sig}(x,s;y,t)$, and equality occurs if and only if $y$ maximizes $\Ll(x,s;z,t) + \W_{\dir \sig}(z,t;0,t)$ over $z \in \R$. 
\end{lemma}
\begin{proof}
For $s < t$, Theorem \ref{thm:DL_Buse_summ}\ref{itm:DL_Buse_add},\ref{itm:Buse_KPZ_description} gives
\begin{align} \label{106}
\W_{\dir \sig}(x,s;y,t) &= \sup_{z \in \R}\{\Ll(x,s;z,t) + \W_{\dir \sig}(z,t;y,t)\} \nonumber
\\ &= \sup_{z \in \R}\{\Ll(x,s;z,t) + \W_{\dir \sig}(z,t;0,t)\} + \W_{\dir \sig}(0,t;y,t).
\end{align}
Setting $z = y$ on the right-hand side of \eqref{106}, it follows that $\W_{\dir \sig}(x,s;y,t) \ge \Ll(x,s;y,t)$, and equality holds if and only if $y$ is a maximizer.
\end{proof}
\newpage
\begin{proof}[Proof of Theorem \ref{thm:DL_SIG_cons_intro} (Construction of the Busemann geodesics)]
The \\ full-probability event of this theorem is $\Omega_1$ \eqref{omega1}.

\noindent
\textbf{Item \ref{itm:intro_SIG_bd} (Finiteness of the maximizers):} This follows immediately from Lemma \ref{lem:bounded_maxes}.

\noindent 
 We prove \textbf{Items \ref{itm:arb_geod_cons}--\ref{itm:DL_LRmost_geod}} together. By Lemma \ref{lem:L_and_Buse_ineq}, for any such construction of a path from the sequence of times $s = t_0 < t_1 < \cdots$ and any $i \ge 1$,
\[
\Ll(g(t_{i - 1}),t_{i - 1};g(t_i),t_i) = \W_{\dir \sig}(g(t_{i - 1}),t_{i - 1};g(t_i),t_i).
\]
Furthermore, for any $t_{i - 1} \le t < u \le t_i$, it must hold that  
\[
\Ll(g(t),t;g(u),u) = \W_{\dir \sig}(g(t),t;g(u),u),
\]
for otherwise, by additivity of the Busemann functions (Theorem \ref{thm:DL_Buse_summ}\ref{itm:DL_Buse_add}),
\begin{align*}
&\quad \Ll(g(t_{i - 1}),t_{i - 1};g(t_i),t_i)\\
&= \Ll(g(t_{i - 1}),t_{i - 1};g(t),t) +\Ll(g(t),t;g(u),u) + \Ll(g(u),u;g(t_i),t_i)\\
&<  \W_{\dir \sig}(g(t_{i - 1}),t_{i - 1};g(t),t) +\W_{\dir \sig}(g(t),t;g(u),u) + \W_{\dir \sig}(g(u),u;g(t_i),t_i) \\
&= \W_{\dir \sig}(g(t_{i - 1}),t_{i - 1};g(t_i),t_i),
\end{align*}
a contradiction. Additivity extends \eqref{eqn:SIG_weight} to all $s \le t < u$. Therefore, the path is a semi-infinite geodesic because the weight of the path in between any two points is optimal by Lemma \ref{lem:L_and_Buse_ineq}. From the equality \eqref{eqn:SIG_weight} and Lemma \ref{lem:L_and_Buse_ineq}, for \textit{every} $t \ge s$, $g(t)$ maximizes $\Ll(x,s;z,t) + \W_{\dir \sig}(z,t;0,t)$ over $z \in \R$.

\begin{figure}
    \centering
    \includegraphics[height = 2in]{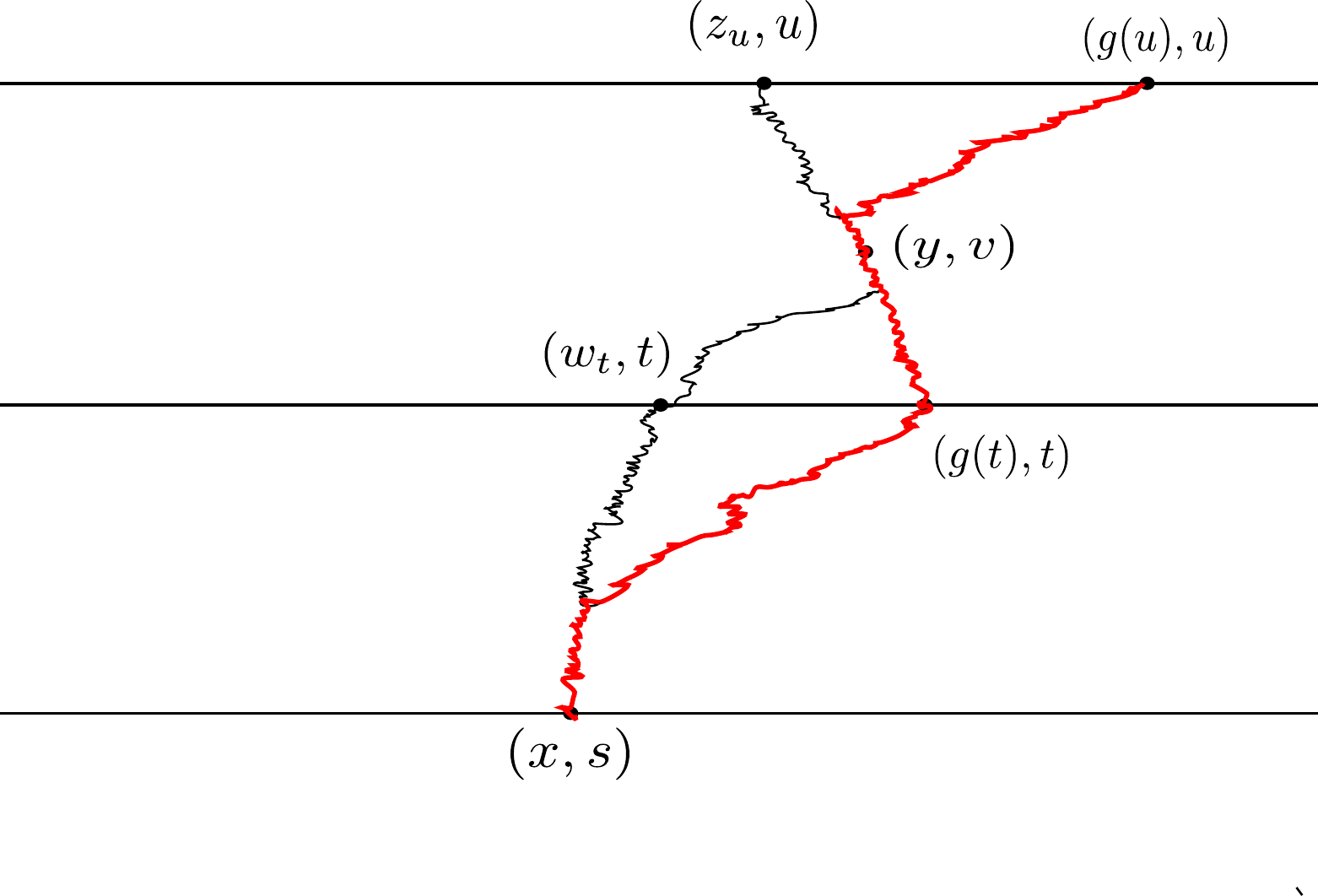}
    \caption{\small Illustration of the proof of Lemma \ref{lem:RM_geod_SIG}. Here, the red/thick path denotes the path $\hat \gamma$ in the case $w_t < g(t)$, which is to the right of the rightmost geodesic between $(x,s)$ and $(g(u),u)$, which passes through $(w_t,t)$ by assumption. This gives the contradiction. }
    \label{fig:RM_geod}
\end{figure}

Before global directedness of all  geodesics, we show that $g_{(x,s)}^{\dir \sig,S}$ are semi-infinite geodesics and the leftmost/rightmost geodesics between any two of their points. Take $S = R$, and the result for $S = L$ follows similarly. Omit $x,s,\dir$, and $\sigg$ from the notation temporarily, and write $g(t) = g_{(x,s)}^{\dir \sig,R}(t)$. 
By what was just proved, it is sufficient to prove the following lemma.
\begin{lemma} \label{lem:RM_geod_SIG}
Let $g$ be as defined above. For $s < t < u$, let $z_u$ be the rightmost maximizer of $\Ll(g(t),t;z,u) + \W_{\dir \sig}(z,u;0,u)$ over $z \in \R$, and let $w_t$ be the rightmost maximizer of $\Ll(x,s;w,t) + \Ll(w,t;g(u),u)$ over $w \in \R$ {\rm(}Equivalently, $(w_t,t)$ is the point at level $t$ on the rightmost geodesic between $(x,s)$ and $(g(u),u)${\rm)}. Then, $g(t) = w_t$ and $g(u) = z_u$.
\end{lemma}
\begin{proof}
 By Lemma \ref{lem:L_and_Buse_ineq} and Items \ref{itm:arb_geod_cons}\ref{itm:weight_of_geod}--\ref{itm:maxes}, $w_t$ maximizes $\Ll(x,s;z,t)+ \W_{\dir \sig}(z,t;0,t)$ over $z \in \R$, and $z_u$ maximizes $\Ll(x,s;z,u) + \W_{\dir \sig}(z,u;0,u)$ over $z \in \R$. By definition of $g(u)$ and $g(t)$ as the rightmost maximizers, we have $w_t \le g(t)$ and $z_u \le g(u)$ in general. Assume, to the contrary, that $g(t) \neq w_t$ or $g(u) \neq z_u$. We first prove a contradiction in the case $w_t < g(t)$. For the proof, refer to Figure \ref{fig:RM_geod} for clarity. Let $\gamma_1:[s,u]\to \R$ be the rightmost geodesic from $(x,s)$ to $(g(u),u)$ (which passes through $(w_t,t)$), and let $\gamma_2$ be the concatenation of the rightmost geodesic from $(x,s)$ to $(g(t),t)$ followed by the rightmost geodesic from $(g(t),t)$ to $(z_u,u)$.  By Item \ref{itm:arb_geod_cons}\ref{itm:weight_of_geod} for $i = 1,2$, the weight of the portion of any part of $\gamma_i$ is equal to the Busemann function between the points. Since $w_t < g(t)$ and $z_u \le g(u)$, $\gamma_1$ and $\gamma_2$ must split before time $t$, and then meet again before or at time $u$. Let $(y,v)$ be a crossing point, where $t < v \le u$. Let $\hat \gamma: [s,u] \to \R$ be defined by $\hat \gamma(r) = \gamma_2(r)$ for $r \in [s,v]$ and $\hat \gamma(r) = \gamma_1(r)$ from $(y,v)$ to $(g(u),u)$. Then, by the additivity of Busemann functions, the weight $\Ll$ of any portion of the path $\hat \gamma$ is equal to the Busemann function between the two points. By Lemma \ref{lem:L_and_Buse_ineq}, $\wh \gamma$ is then a geodesic between $(x,s)$ and $(g(u),u)$, which is to the right of $\gamma_1$, which was defined to be the rightmost geodesic between the points, a contradiction.

Now, we consider the case $z_u < g(u)$. Define $\gamma_1$ and $\gamma_2$ as in the previous case. Since $z_u < g(u)$, there is some point $(y,v)$ with $t \le v < u$ such that $\gamma_1$ splits from or crosses  $\gamma_2$ at $(y,v)$. Then, define $\hat \gamma$ as in the previous case. Again, the weight   $\Ll$ of any portion of the path $\hat \gamma$ is equal to the Busemann function between the two points. Specifically, $\Ll(g(t),t;g(u),u) = \W_{\dir \sig}(g(t),t;g(u),u)$, and by Item \ref{lem:L_and_Buse_ineq}, $g(u)$ maximizes $\Ll(g(t),t;z,u) + \W_{\dir \sig}(z,u;0,u)$ over $z \in \R$. This contradicts the definition of $z_u$ as the rightmost such maximizer.
\end{proof}
Returning to the proof of Theorem \ref{thm:DL_SIG_cons_intro}, we show the global directedness of all Busemann geodesics constructed in the manner described in Item \ref{itm:arb_geod_cons}. By \eqref{eqn:mont_maxes}, for $t \ge s$ and $\alpha < \dir < \beta$ with $\alpha,\beta \in \Q$,
\be \label{sig_sand}
g_{(x,s)}^{\alpha,L}(t) \le g_{(x,s)}^{\dir \sig,L}(t) \le g(t) \le g_{(x,s)}^{\dir \sig,R}(t) \le g_{(x,s)}^{\beta,R}(t).
\ee
Note that on $\Omega_1$ the $\pm$ distinction is absent for $\alpha,\beta \in \Q$ (Lemma \ref{lem:DL_horiz_Buse}\ref{itm:DL_agree_horiz}).
By definition \eqref{omega1} of the event $\Omega_1$ and Theorem \ref{thm:RV-Buse}\ref{itm:DL_Buse_var}, $\forall\alpha\in\Q$, the maximizers of $\Ll(x,s;z,t) + W_{\alpha}(z,t;0,t)$ over $z \in \R$ are exactly the locations $z$ where an $\alpha$-directed geodesic goes through $(z,t)$. Therefore, $g_{(x,s)}^{\alpha,L}(t)/t \to \alpha$ and $g_{(x,s)}^{\beta,R}(t)/t \to \beta$ when $\alpha, \beta\in\Q$. By \eqref{sig_sand},
\[
\alpha  \le \liminf_{t \to \infty} t^{-1}{g(t)} \le \limsup_{t \to \infty} t^{-1}{g(t)} \le \beta.
\]
Sending $\Q \ni \alpha \nearrow \dir$ and $\Q \ni \beta \searrow \dir$ completes the proof of   Theorem \ref{thm:DL_SIG_cons_intro}.
\end{proof}

We now define the next   full-probability event. 
\be \label{omega2}
\text{Let $\Omega_2$ be the subset of $\Omega_1$ on which the following hold.} 
\ee
\begin{enumerate} [label=\rm(\roman{*}), ref=\rm(\roman{*})] \itemsep=3pt
    \item \label{itm:stickT} For each integer $T \in \Z$ and each compact set $K \subseteq \R^2$, there exists $\ve =\ve(\dir,T,K) > 0$ such that for $\dir - \ve < \alpha < \dir < \beta < \dir + \ve$ and $(x,y) \in K$,
    \be \label{465}
\W_{\alpha \sig}(y,T;x,T) = \W_{\dir -}(y,T;x,T)\ \ \text{ and }\ \  \W_{\beta \sig}(y,T;x,T) = \W_{\dir +}(y,T;x,T).
\ee
\item For each integer $T \in \Z$, the set
\be \label{infinit_dense}
\{\dir \in \R: \W_{\dir -}(x,T;0,T) \neq \W_{\dir +}(x,T;0,T) \text{ for some }x \in \R\}
\ee
is countably infinite and dense in $\R$.
\item For each $s < t \in \R$, $x,\dir \in \R$, $\sigg \in \{-,+\}$, and $S \in \{L,R\}$,
\be \label{eqn:dirtoinf}
    \lim_{\xi \to \pm \infty} g_{(x,s)}^{\xi \sig,S}(t) = \pm \infty.
\ee
\end{enumerate}

\begin{lemma}
$\Pp(\Omega_2) = 1.$
\end{lemma}
\begin{proof}
Recall the distributional equality
$ 
\{W_{\dir+}(\abullet,T;0,T)\}_{\dir \in \R} \deq \{G^{\sqrt 2}_\dir\}_{\dir \in \R}
$ 
from Theorem \ref{thm:Buse_dist_intro}\ref{itm:SH_Buse_process}.
The fact that \ref{itm:stickT} holds with probability one is a direct consequence of the regularity of the SH from \ref{thm:SH_sticky_thm}\ref{itm:SH_stick}. The set~\eqref{infinit_dense} is countably infinite and dense for all $T \in \Z$  by the corresponding properties of $G^{\sqrt 2}$ from Theorem \ref{thm:SH_sticky_thm}\ref{itm:SH_all_jump},\ref{itm:SH_Xi_dense}.

Now, we prove that \eqref{eqn:dirtoinf} holds with probability one. By the monotonicity of \eqref{eqn:mont_maxes}, the limits $\lim_{\xi \to \infty} g_{(x,s)}^{\xi \sig,S}(t)$ and $\lim_{\xi \to -\infty} g_{(x,s)}^{\xi \sig,S}(t)$ exist in $\R \cup \{-\infty,\infty\}$. Furthermore, by this monotonicity, it is sufficient to show that 
\be \label{678}
\lim_{\xi \to \infty} g_{(x,s)}^{\xi -,L}(t) = \sup_{\dir \in \R}g_{(x,s)}^{\xi -,L}(t) = \infty\ \ \text{ and }\ \  \lim_{\dir \to -\infty} g_{(x,s)}^{\dir +,R}(t) = \inf_{\dir \in \R} g_{(x,s)}^{\dir +,R}(t) = -\infty.
\ee
First, we show that \eqref{678} holds with probability one for a fixed initial point $(x,s)$ and fixed $t > s$. It is therefore sufficient to take $(x,s) = (0,0)$ and then $t > 0$. By the monotonicity, it suffices to take limits over $\dir \in \Q$ so that by Theorem \ref{thm:RV-SIG-thm}\ref{itm:pd_fixed}, the $\pm$ and $L/R$ distinctions are unnecessary.   $\W_{\xi \sig}(z,t;0,t)$ is a two-sided Brownian motion with drift $2\xi$ and diffusivity $\sqrt 2$, independent of the random function $(x,y) \mapsto \Ll(x,0;y,t)$ (Theorem \ref{thm:Buse_dist_intro}\ref{itm:indep_of_landscape}). Let $B$ be a standard Brownian motion, independent of $\Ll$. Using skew stationarity with $c = -\dir$ in the third equality below  and time stationarity in the fifth equality (Lemma \ref{lm:landscape_symm}), we obtain, for $\dir \in \Q$,
\begin{align*}
   g_{(x,s)}^\dir(t) &= \argmax_{z \in \R} \{\Ll(x,s;z,t)  + \W_{\xi}(z,t;0,t) \} \\
   &\deq \argmax_{z \in \R} \{\Ll(x,s;z,t) + \sqrt 2 B(z) + 2\xi z\} \\
   &\deq \argmax_{z \in \R} \{\Ll(x - \dir s,s;z - \dir t,t) + 2\dir(x -z) + (t - s)\dir^2  + \sqrt 2 B(z) +2\dir z\}\\
   &= \argmax_{z \in \R} \{\Ll(x - \dir s,s;z - \dir t,t) + \sqrt 2(B(z) - B(\dir (t-s))) \} \\
    &\deq  \argmax_{z \in \R} \{\Ll(x,s;z - \dir (t-s),t) + \sqrt 2 B(z - \dir (t - s)) \} \\
    &= \argmax_{z \in \R}\{\Ll(x,s;z,t) + \sqrt 2 B(z) \} + \dir (t-s)  \deq  g_{(x,s)}^0(t) + \dir(t-s).
\end{align*}
Therefore, $\forall\dir \in \Q$, the distribution of $g_{(x,s)}^{\dir}(t)$ is that of a fixed, almost surely finite, random variable plus $\dir (t-s)$. Since we know  $\lim_{\Q \ni \dir \to \pm \infty} g_{(x,s)}^{\dir}(t)$ exists, the limit must be $\pm \infty$ a.s.

Now, consider the intersection of $\Omega_1$ with event of probability one on which for each triple $(w,q_1,q_2) \in \Q^3$ with $q_1 < q_2$, 
\be \label{692}
\lim_{\dir \to +\infty} g_{(w,q_1)}^{\dir-,L}(q_2) = + \infty\qquad\text{and}\qquad \lim_{\dir \to -\infty} g_{(w,q_1)}^{\dir+,R}(q_2) = - \infty.
\ee
On this event, let $(x,s,t) \in \R^3$ with $s < t$ be arbitrary. Assume, by way of contradiction, that 
\be \label{792}
z := \sup_{\dir \in \R} g_{(x,s)}^{\dir-,L}(t) < \infty,
\ee
and let $g:[s,t]$ denote the leftmost geodesic from $(x,s)$ to $(z,t)$. For this proof, refer to Figure \ref{fig:fanning_proof} for clarity. By the assumption \eqref{792} and the fact that $g_{(x,s)}^{\dir-,L}$ is the leftmost geodesic between any two of its points (Theorem \ref{thm:DL_SIG_cons_intro}\ref{itm:DL_LRmost_geod}),  $g_{(x,s)}^{\dir-,L}(t) \le g(t)$ for all $\dir \in \R$ and $t > s$.
Let $q_1 \in (s,t)$ be rational. Choose  $w \in \Q$ such that $w < g(q_1)$. By continuity of geodesics, we may choose $q_2 \in (q_1,t) \cap \Q$ to be sufficiently close to $t$ so that $|g(q_2) - z| < 1$. Next, by \eqref{692}, we may choose positive $\dir$ sufficiently large so that 
\be \label{892}
g_{(w,q_1)}^{\dir-,L}(q_2) > z + 1 > g(q_2) \ge g_{(x,s)}^{\dir -,L}(q_2).
\ee
Since $w < g(q_1)$,  $g_{(w,q_1)}^{\dir-,L}$ and $g_{(x,s)}^{\dir -,L}$ cross at some  $(\hat z,\hat t)$ with $\hat t \in (q_1,q_2)$.  By Theorem \ref{thm:DL_SIG_cons_intro}\ref{itm:DL_all_SIG}, both $g_{(w,q_1)}^{\dir-,L}(q_2)$ and $g_{(x,s)}^{\dir -,L}(q_2)$ are the leftmost maximizer of $\Ll(\hat z,\hat t; y,q_2) + \W_{\dir -}(y,q_2;0,q_2)$ over $y \in \R$, contradicting \eqref{892}. The proof for   $\dir \to -\infty$ is analogous.
\end{proof}
\begin{figure}[t]
    \centering
    \includegraphics[height = 2 in]{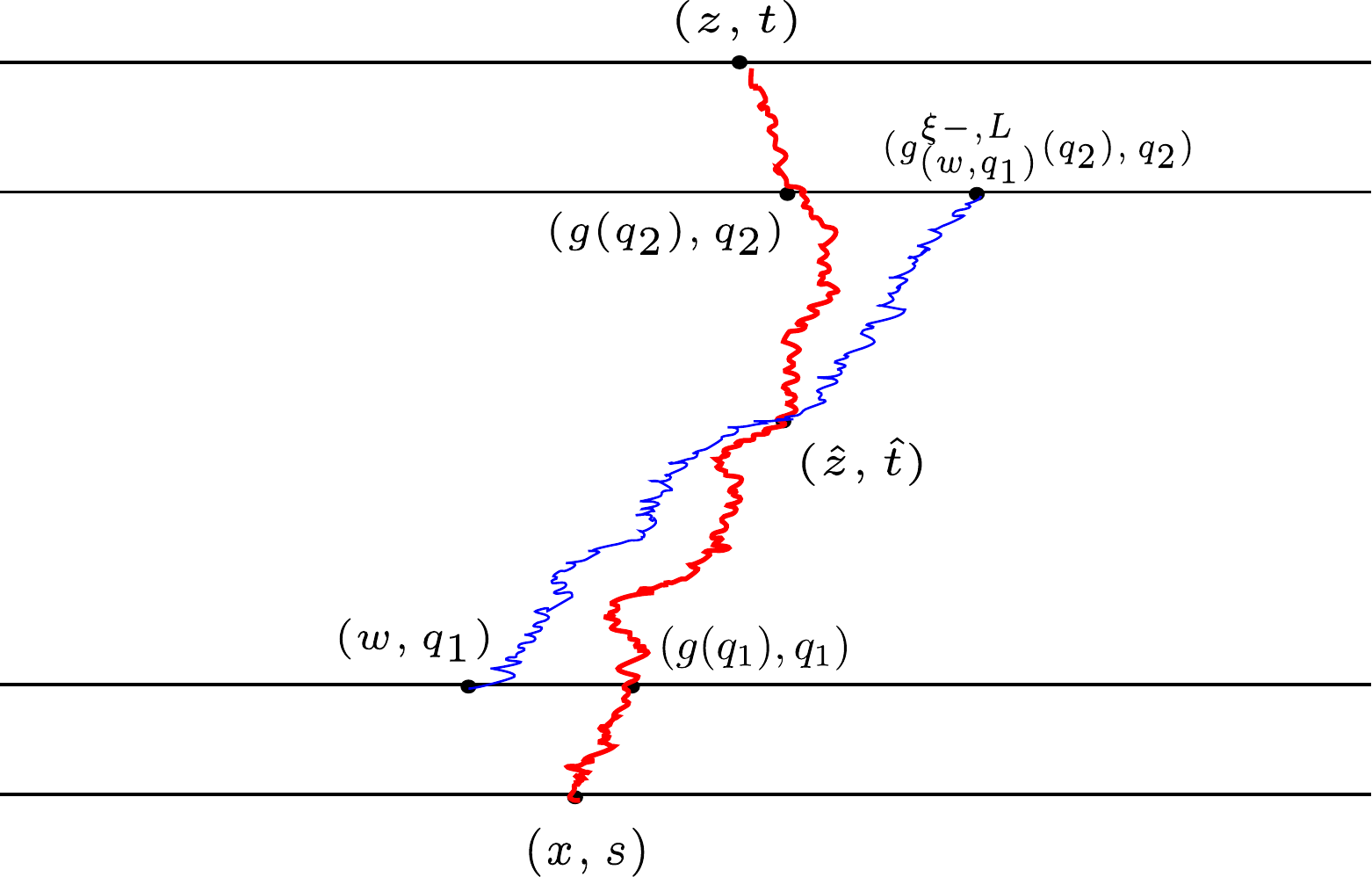}
    \caption{\small The blue/thin path represents $g_{(w,q_1)}^{\dir -,L}$ and the red/thick path represents $g$.}
    \label{fig:fanning_proof}
\end{figure}

\begin{proof}[Proof of Theorem \ref{thm:DL_Buse_summ}\ref{itm:DL_unif_Buse_stick} (Regularity of the Busemann process)]

 By definition of the \\event $\Omega_2$ \eqref{omega2}, for each $\dir \in \R$, each integer $T$ and compact set $K\subseteq \R^2$, there is a $\ve > 0$ so that \eqref{465} holds for all $(x,y) \in K$. Now, let $\dir \in \R$, let $K$ be a compact subset of $\R^4$, and let $T$ be an integer greater than $\sup\{t \vee s: (x,s;y,t) \in K\}$. Let 
\begin{align*}
A &:= \inf\{g_{(x,s)}^{(\dir - 1)-,L}(T)\wedge g_{(y,t)}^{(\dir - 1)-,L}(T) : (x,s;y,t) \in K\}, 
\qquad\text{and} \\
B &:= \sup\{g_{(x,s)}^{(\dir + 1)+,R}(T)\vee g_{(y,t)}^{(\dir + 1)+,R}(T) : (x,s;y,t) \in K\}.
\end{align*}
By \eqref{eqn:mont_maxes} and Lemma \ref{lem:bounded_maxes}, $-\infty < A < B < \infty$.
By  additivity (Theorem \ref{thm:DL_Buse_summ}\ref{itm:DL_Buse_add}) and \eqref{eqn:mont_maxes}, for all $(x,s;y,t) \in K$ and $\alpha \in (\dir - 1,\dir + 1)$, 
\be \label{Tdiff}
\begin{aligned}
&\W_{\alpha \sig}(x,s;y,t) = \W_{\alpha \sig}(x,s;0,T) - \W_{\alpha \sig}(y,t;0,T) \\
&
= \sup_{z \in \R}\{\Ll(x,s;z,T) + \W_{\alpha\sig}(z,T;0,T)\} - \sup_{z \in \R}\{\Ll(y,t;z,T) + \W_{\alpha \sig}(z,T;0,T)\} \\
&= \sup_{z \in [A,B]}\{\Ll(x,s;z,T) + \W_{\alpha \sig}(z,T;0,T)\} - \sup_{z \in [A,B]}\{\Ll(y,t;z,T) + \W_{\alpha \sig}(z,T;0,T)\}.
\end{aligned}
\ee
By \eqref{465}, the conclusion follows. 
\end{proof}

\begin{proof}[Proof of Theorem \ref{thm:DLBusedc_description} (Description of the discontinuity set)]
The full \\ probability event of this theorem is $\Omega_2$, except for Item \ref{itm:Busedc_t} whose proof is postponed until Section \ref{sec:last_proofs2}. The only result that relies on this Item is Theorem \ref{thm:Split_pts}, which is proved afterward. 

\noindent \textbf{Item \ref{itm:Busedc_horiz_mont} (Monotonicity):}
 By the monotonicity of Theorem \ref{thm:DL_Buse_summ}\ref{itm:DL_Buse_mont}, and by Lemma \ref{lem:ext_mont}, for $a \le x \le y \le b$,
 \be \label{801}
 0 \le \W_{\dir +}(y,t;x,t) - \W_{\dir -}(y,t;x,t) \le \W_{\dir +}(b,t;a,t) - \W_{\dir -}(b,t;a,t).
 \ee
Thus, discontinuities of $\dir \mapsto \W_{\dir \sig}(y,t;x,t)$ are also discontinuities for $\dir \mapsto \W_{\dir \sig}(b,t;a,t)$.

\noindent \textbf{Item \ref{itm:DL_dc_set_count} ($\DLBusedc$ is a countable dense set):}
Similarly as in \eqref{Tdiff}, if $(x,s;y,t) \in \R^4$, then for $\dir \in \R$, $\sigg \in \{-,+\}$, and any integer $T > s\vee t$,
\be \label{880}\begin{aligned} 
\W_{\dir \sig}(x,s;y,t) &= \sup_{z \in \R}\{\Ll(x,s;z,T) + \W_{\dir \sig}(z,T;0,T)\}\\
&\qquad - \sup_{z \in \R} \{\Ll(y,t;z,T) + \W_{\dir \sig}(z,T;0,T)\}.
\end{aligned} \ee
So if $\W_{\dir -}(z,T;0,T) = \W_{\dir+}(z,T;0,T) \; \forall z \in \R$, then  $\W_{\dir -}(x,s;y,t) = \W_{\dir +}(x,s;y,t)$, and 
\be \label{881}
\DLBusedc = \bigcup_{T \in \Z} \{\dir \in \R: \W_{\dir - }(x,T;0,T) \neq \W_{\dir +}(x,T;0,T) \text{ for some }x \in \R\}. 
\ee
On $\Omega_2$,  $\DLBusedc$   is countably infinite and dense  by  \eqref{omega2}.  Lemma \ref{lem:DL_horiz_Buse}\ref{itm:DL_agree_horiz}, along with \eqref{881} imply that $\DLBusedc$ contains no rational directions $\dir$. For an arbitrary $\dir \in \R$, $\W_{\dir -}(\abullet,T;0,T)$ and $\W_{\dir +}(\abullet,T;0,T)$ are both Brownian motions with the same diffusivity and drift, and $\W_{\dir -}(y,T;x,T) \le \W_{\dir +}(y,T;x,T)$ for $x < y$ by Theorem \ref{thm:DL_Buse_summ}\ref{itm:DL_Buse_gen_mont}. By \eqref{881} and continuity,
\[
\Pp(\dir \in \DLBusedc) \le \sum_{T \in \Z, x \in \Q} \Pp(\W_{\dir - }(x,T;0,T) \neq \W_{\dir +}(x,T;0,T)) = 0,
\]
where $\Pp(\W_{\dir - }(x,T;0,T) \neq \W_{\dir +}(x,T;0,T)) = 0$ because the two random variables have the same law and are ordered.

\noindent \textbf{Item \ref{itm:DL_Buse_no_limit_pts} ($\DLBusedc(p;q)$ is discrete):} The discreteness of the jumps is a direct consequence of the regularity of 
 the Busemann process from Theorem \ref{thm:DL_Buse_summ}\ref{itm:DL_unif_Buse_stick}.  The fact that $\DLBusedc(y,t;x,t)$ is countably infinite and unbounded on a $t$-dependent full probability event follows from Theorem \ref{thm:SH_sticky_thm}\ref{itm:SH_all_jump} 
\end{proof}

\section{Non-uniqueness of semi-infinite geodesics} \label{sec:LR_sig} 
Theorem \ref{thm:DL_SIG_cons_intro} established global existence of semi-infinite geodesics from each initial point and into each direction. 
We know from Theorem 3.3 of \cite{Rahman-Virag-21}, recorded earlier in  Theorem \ref{thm:RV-SIG-thm}\ref{itm:pd_fixed},  that for a fixed initial point and a fixed direction, there almost surely is a unique semi-infinite geodesic.  However, this uniqueness does not extend globally to all initial points and directions simultaneously.

 In fact, two qualitatively different types of non-uniqueness of    Busemann  geodesics from a given point into a given direction arise.  One is denoted by the $L/R$ distinction and the other by the  $\pm$ distinction. 
 All semi-infinite geodesics from $p$ in direction $\dir$ lie between the leftmost Busemann geodesic $g_{p}^{\dir -,L}$ and the rightmost Busemann geodesic   $g_p^{\dir +,R}$.  See Theorem \ref{thm:all_SIG_thm_intro}\ref{itm:DL_LRmost_SIG}. We refer the reader back to Figure \ref{fig:non_unique_comp} for the two types of non-uniqueness. The $L/R$ uniqueness is depicted on the left, where geodesics split and return to coalesce, while the $\pm$ non-uniqueness is depicted on the right in the figure, where geodesics split and stay apart, all the way to $\infty$.

The $L/R$ non-uniqueness is a feature of continuous space.  Only the $\pm$ non-uniqueness appears in the discrete corner growth model with exponential weights, while both $L/R$ and $\pm$  non-uniqueness are   present  in semi-discrete  BLPP \cite{Seppalainen-Sorensen-21a,Seppalainen-Sorensen-21b}.  

To capture $L/R$ non-uniqueness, we  introduce  the following random set of initial points.   For $\dir \in \R$ and $\sigg \in \{-,+\}$, let $\NU^{\dir \sig}$ be the set of points $p \in \R^2$ such that the $\dir \sigg$ geodesic from $p$ is not unique. In notational terms, 
\begin{align}
\NU^{\dir \sig} &= \{(x,s) \in \R^2: g_{(x,s)}^{\dir \sig,L}(t) < g_{(x,s)}^{\dir \sig,R}(t) \text{ for some } t > s\}, \label{NU0}.
\end{align}
Next, let
\be \label{NU0_global}
\NU = \textstyle\bigcup_{\dir \tspb\in\tspb  \R,\,\sig \tspb\in\tspb \{-,+\}} \NU^{\dir \sig}.
\ee

\begin{figure}[t]
    \centering
    \includegraphics[height = 2in]{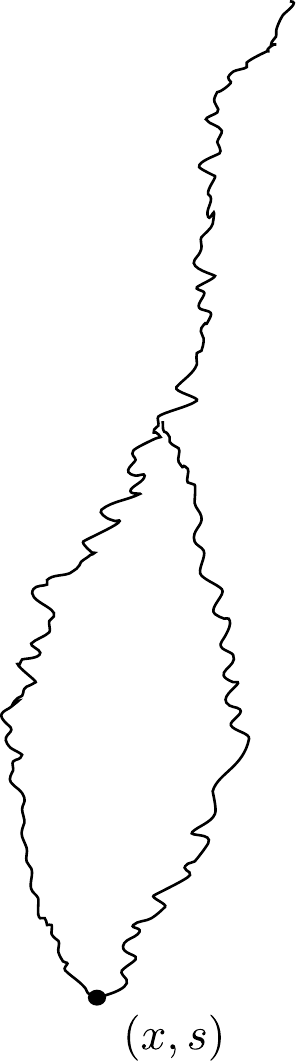} \qquad\qquad\qquad 
    \includegraphics[height = 2in]{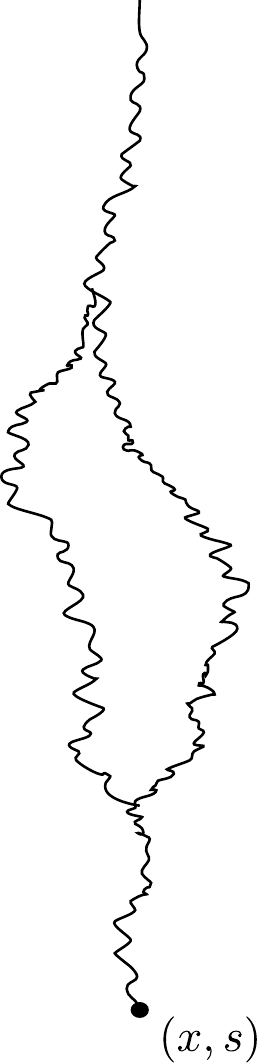}
    \caption{\small Given the definition of $\NU$ and the fact that the geodesics coalesce,  this figure shows two possible configurations for points $(x,s) \in \NU$. However, due to a recent result of Bhatia \cite{Bhatia-23} and Dauvergne \cite{Dauvergne-23} (recorded as Theorem \ref{thm:Bhatia} below), the configuration depicted on the right does not appear in the DL. That, is the splitting of geodesics for points in $\NU$ must occur immediately from the initial point. }  
    \label{fig:NU}
\end{figure}
Figure \ref{fig:NU} illustrates the set  $\NU$. Theorem \ref{thm:DLNU}\ref{itm:DL_NU_count} establishes that,   with probability one, for each $\dir \in \R$ and $\sigg \in \{-,+\}$, the restriction of $\NU^{\dir\sig}$ to each  time level $s$ is countably infinite. By Theorem \ref{thm:DL_all_coal}\ref{itm:DL_allsigns_coal},  on a single event of probability one,  for each direction $\dir$ and sign $\sigg \in \{-,+\}$, all $\dir \sig$ geodesics coalesce. Therefore, from each $p \in \NU^{\dir \sig}$, two $\dir \sig$ geodesics separate but eventually   come back together. This coalescence result was first proven in the author's joint work with Busani and Sepp\"al\"ainen \cite{Busa-Sepp-Sore-22arXiv}. Afterward, Bhatia \cite{Bhatia-23} gave greater clarity to the set $\NU$. In particular, for a point $(x,s) \in \NU^{\dir \sig}$, the splitting of geodesics must occur immediately from the initial point. This follows from the coalescence we proved in \cite{Busa-Sepp-Sore-22arXiv} and the following theorem, proved first Bhatia \cite{Bhatia-23} and also independently by Dauvergne \cite{Dauvergne-23}.

\begin{theorem}[\cite{Bhatia-23}, Theorem 1 and \cite{Dauvergne-23}, Lemma 3.3] \label{thm:Bhatia}
With probability one, there is no point $(x,s;y,t) \in \Rup$ and pairs of distinct geodesics $g_1,g_2$ from $(x,s)$ to $(y,t)$ satisfying, for some $\ve > 0$, $g_1(u) = g_2(u)$ for all $u \in (s,s + \ve)\cup (t - \ve,t)$. In words, geodesics cannot move together, then form bubbles, then move together again. See Figure \ref{fig:NU}. 
\end{theorem} 

Since   $\NU^{\dir -} \cup \NU^{\dir +}$  captures only the $L/R$ distinction and not the $\pm$ distinction, it does  \textit{not} in general contain all  the   initial points from which  the $\dir$-directed  semi-infinite geodesic  is not unique. However, when the $\dir\pm$ distinction is absent, Theorem \ref{thm:all_SIG_thm_intro}\ref{itm:DL_LRmost_SIG} implies that $\NU^{\dir}=\NU^{\dir\pm}$ is exactly the set of points $p \in \R^2$ such that the semi-infinite geodesic from $p$ in direction $\dir$ is not unique. This happens under two scenarios: when $\dir \notin \DLBusedc$, and when we restrict attention to the $\dir$-dependent event of full probability on which $g_{p}^{\dir -,S} = g_p^{\dir +,S}$ for all $p \in \R^2$ and $S \in \{L,R\}$.

The failure to capture the $\pm$ non-uniqueness  is also evident from the size of $\NU$. Whenever $\dir \in \DLBusedc$, there are at least two semi-infinite geodesics with direction $\dir$ from {\it every} initial point. But along a fixed time level $\NU$ is  countable, and thereby  a strict subset of $\R^2$ (Theorem \ref{thm:DLNU}\ref{itm:DL_NU_count} below).

Recall that $\Hh_s=\{(x,s): x \in \R\}$ is the set of  space-time points at time level $s$. Theorem \ref{thm:DLBusedc_description}\ref{itm:DL_dc_set_count} states that on a single event of full probability, $\DLBusedc \subseteq \R \setminus \Q$, so for $\dir \in \Q$,  we can drop the $\pm$ distinction and write $\NU^\dir =\NU^{\dir -} = \NU^{\dir +}$.
\begin{theorem} \label{thm:DLNU}
On a single event of probability one, the set $\NU$ satisfies  
    \be \label{109}
    \NU = \textstyle\bigcup_{\dir \in \Q}\NU^{\dir}.
    \ee In particular, the following hold.
\begin{enumerate} [label=\rm(\roman{*}), ref=\rm(\roman{*})]  \itemsep=3pt
    \item \label{itm:DL_NU_p0} 
     For each $p \in \R^2$, $\Pp(p \in \NU) = 0$, and the full-probability event of the theorem can be chosen so that $\NU$ contains no points of $\Q^2$.
    \item \label{itm:DL_NU_count} On a single event of full probability, simultaneously for every $s \in \R$, $\dir \in \R$ and $\sigg \in \{-,+\}$, the set $\NU^{\dir \sig} \cap\, \Hh_s$ is countably infinite and unbounded in both directions. Specifically, for each $s \in \R$, there exist  sequences $x_n \to -\infty$ and $y_n \to +\infty$ such that $(x_n,s),(y_n,s) \in \NU^{\dir \sig}$. 
    By \eqref{109}, $\NU \cap\, \Hh_s$ is also countably infinite.
\end{enumerate}
\end{theorem}
\begin{remark}
The set $\Q$ can be replaced by any countable dense subset of $\R$, by adjusting the full-probability event. In all applications in this chapter, we use the set $\Q$. 
\end{remark}
\begin{remark}
As Item \ref{itm:DL_NU_count} states, the set $\NU$ is countably infinite along each horizontal line. Hence, the set $\NU$ is uncountable. Bhatia \cite{Bhatia-23} recently proved that the Hausdorff dimension of $\NU$ is $\f{4}{3}$. 
\end{remark}
The next theorem states properties of Busemann geodesics that involve the $L/R$ and $\pm$ distinctions.
\begin{theorem} \label{thm:g_basic_prop}
The following hold on a single event of full probability. 
\begin{enumerate} [label=\rm(\roman{*}), ref=\rm(\roman{*})]  \itemsep=3pt
    \item \label{itm:DL_mont_dir} For $s < t$, $x \in \R$, $\dir_1 < \dir_2$, and $S \in \{L,R\}$,
    \[
    g_{(x,s)}^{\dir_1 -,S}(t) \le g_{(x,s)}^{\dir_1 +,S}(t) \le g_{(x,s)}^{\dir_2 -,S}(t) \le g_{(x,s)}^{\dir_2 +,S}(t).  
    \]
     \item \label{itm:DL_SIG_unif} Let $\dir \in \R$, let $K \subseteq \R$ be a compact set, and let $T > \max K$. Then, there exists a random $\ve = \ve(\dir,T,K)>0$ such that, whenever $\dir - \ve < \alpha < \dir < \beta < \dir + \ve$, $\sigg \in \{-,+\}$, $S \in \{L,R\}$, and $x,s \in K$,
    \[
    g_{(x,s)}^{\alpha \sig,S}(t) = g_{(x,s)}^{\dir -,S}(t)\qquad\text{and}\qquad g_{(x,s)}^{\beta \sig,S}(t) = g_{(x,s)}^{\dir+,S}(t)\qquad\text{for all }t \in [s,T].
    \]
    \item \label{itm:limits_to_inf} For each $(x,s) \in \R^2$, $t > s$, $\sigg \in \{-,+\}$, and $S \in \{L,R\}$,
    $
    \lim_{\xi \to \pm \infty} g_{(x,s)}^{\xi \sig,S}(t) = \pm \infty.
    $
    \item \label{itm:DL_SIG_mont_x} For all $\dir \in \R$, $\sigg \in \{-,+\}$, $s < t$ and $x < y$, $g_{(x,s)}^{\dir \sig,R}(t) \le g_{(y,s)}^{\dir \sig,L}(t)$. More generally, if $x < y$, $s \in \R$, and  $g_1$ is a $\dir \sig$ geodesic from $(x,s)$ and $g_2$ is a $\dir \sig$ geodesic from $(y,s)$ such that $g_1(t) = g_2(t)$ for some $t > s$, then $g_1(u) = g_2(u)$ for all $u > t$. In other words, if   $g_1$ and $g_2$ intersect, they coalesce at their first point of intersection.
    \item \label{itm:DL_SIG_conv_x} For all $\dir \in \R$, $\sigg \in \{-,+\}$, $S \in \{L,R\}$, $x \in \R$, and $s < t$,
    \be \label{371}
    \lim_{w \nearrow x} g_{(w,s)}^{\dir \sig,S}(t) = g_{(x,s)}^{\dir \sig,L}(t),\qquad\text{and}\qquad \lim_{y \searrow x} g_{(y,s)}^{\dir \sig,S}(t) = g_{(x,s)}^{\dir \sig,R}(t),
    \ee
    and if \;  $g_{(x,s)}^{\dir \sig,L}(t) = g_{(x,s)}^{\dir \sig,R}(t) =: g_{(x,s)}^{\dir \sig}(t)$, then for $S \in \{L,R\}$,
    \be \label{372}
    \lim_{(w,u) \rightarrow (x,s)} g_{(w,u)}^{\dir \sig,S}(t) = g_{(x,s)}^{\dir \sig}(t).
    \ee
    Furthermore,
    \be \label{373}
    \lim_{x \to \pm \infty} g_{(x,s)}^{\xi \sig,S}(t) = \pm \infty.
    \ee
\end{enumerate}
\end{theorem}
\begin{remark} \label{rmk:mixing_LR_pm}
 In general, Theorem \ref{thm:g_basic_prop}\ref{itm:DL_mont_dir} cannot be extended to mix $L$ with $R$. Pick a point $(x,s) \in \NU$, where $\NU$ is defined as in \eqref{NU0_global}. Then, on the full-probability event of Theorem \ref{thm:DLNU}, there exists a rational direction $\dir$ and $t > s$ such that 
 \[
 g_{(x,s)}^{\dir -,L}(t) = g_{(x,s)}^{\dir +,L}(t) < g_{(x,s)}^{\dir -,R}(t) = g_{(x,s)}^{\dir +,R}(t).
 \]
 By Theorem \ref{thm:g_basic_prop}\ref{itm:DL_SIG_unif}, we may choose $\dir_1 < \dir < \dir_2$ sufficiently close to $\dir$ such that 
\[
g_{(x,s)}^{\dir_2 -,L}(t) = g_{(x,s)}^{\dir_2+,L}(t) = g_{(x,s)}^{\dir -,L}(t) < g_{(x,s)}^{\dir +,R}(t) = g_{(x,s)}^{\dir_1 -,R}(t) = g_{(x,s)}^{\dir_1+,R}(t).
\]

Item \ref{itm:DL_SIG_mont_x} is an extension of Item 2 of Theorem 3.4 in \cite{Rahman-Virag-21} to all directions and all pairs of initial points on the same horizontal level.
It is not true that for all $\dir \in \R$, $s < t$, and $x < y$, $g_{(x,s)}^{\xi +,R}(t) \le g_{(y,s)}^{\xi -,L}(t)$. This is discussed later in Remark \ref{rmk:split_from_all_p}.
\end{remark}

The  next theorem    controls   all  semi-infinite geodesics with Busemann geodesics.
    \begin{theorem} \label{thm:all_SIG_thm_intro}
    The following hold on a single event of probability one.  Let \\ $(x_r,t_r)_{r \in \R_{\ge 0}}$ be any net such that $t_r \to \infty$ and $x_r/t_r \to \dir$. 
    \begin{enumerate} [label=\rm(\roman{*}), ref=\rm(\roman{*})]  \itemsep=3pt
    \item  \label{itm:DL_LRmost_SIG} 
    Let $(x,s) \in \R^2$ and $\dir \in \R$. For each $r$ large enough so that $t_r > s$, let $g_r:[s,t_r] \to \R$ be a geodesic from $(x,s)$ to $(x_r,t_r)$.  Then, for each $t \ge s$, 
    \be \label{987}
    g_{(x,s)}^{\dir -,L}(t) \le \liminf_{r \to \infty} g_r(t) \le \limsup_{r \to \infty} g_r(t) \le g_{(x,s)}^{\dir +,R}(t).
    \ee
    In particular, $g_{(x,s)}^{\dir-,L}$ is the leftmost and $g_{(x,s)}^{\dir+,R}$ the rightmost among \textbf{all} semi-infinite geodesics from $(x,s)$ in direction $\dir$. 
   
    
    \item \label{itm:finite_geod_stick} Let $K \subseteq \R^2$ be compact.  Suppose that there is a level $t$ after which all semi-infinite geodesics from $(x,s) \in K$ in direction $\dir$ have coalesced. For $u \ge t$, let $g(u)$ be  this geodesic. Then, given $T > t$, there exists $R \in \R_{>0}$ such that for $r \ge R$ and all $(x,s) \in K$, if $g_r:[s,t_r]\to\R$ is a geodesic from $(x,s)$ to $(x_r,t_r)$, then
    \[
    g_r(u) = g(u)  \qquad\text{for all }u \in [t,T].
    \]
    In particular, suppose there is a unique semi-infinite geodesic from $(x,s)$ in direction $\dir$, denoted by  $g_{(x,s)}^\dir$. Then given  $T > s$, for sufficiently large $r$, we have 
    \[
   g_r(u) = g_{(x,s)}^\dir(u)   \qquad\text{for all }u \in [s,T].
    \]
\end{enumerate}
\end{theorem}
\begin{remark}
Theorem \ref{thm:DL_all_coal}\ref{itm:DL_allsigns_coal} below states that the assumed coalescence in Item \ref{itm:finite_geod_stick} occurs whenever $\dir \notin \DLBusedc$. The second statement of Item \ref{itm:finite_geod_stick} is in Corollary 3.1 in \cite{Rahman-Virag-21}. We provide a different proof that uses the regularity of the Busemann process.  
\end{remark}

\subsection{Proofs}
 In this section, we prove Theorems \ref{thm:DLNU}, \ref{thm:g_basic_prop}, and \ref{thm:all_SIG_thm_intro}. In each of these, the full-probability event is $\Omega_2$ \eqref{omega2}.
We start by proving parts of Theorem \ref{thm:g_basic_prop}, then go to the proof of Theorem \ref{thm:DLNU}.

\begin{proof}[Proof of Theorem \ref{thm:g_basic_prop}, Items \ref{itm:DL_mont_dir}--\ref{itm:limits_to_inf}]

\noindent 
\textbf{Item \ref{itm:DL_mont_dir} (monotonicity of geodesics in the direction parameter)} was already proven as Equation \eqref{eqn:mont_maxes}. In fact, this item holds on $\Omega_1$.

\noindent \textbf{Item \ref{itm:DL_SIG_unif} (geodesics agree locally for close directions): } This follows a similar proof as the proof of Theorem \ref{thm:DL_Buse_summ}\ref{itm:DL_unif_Buse_stick}. Let $K$ be a compact subset of $\R$, and let $T$ be an integer greater than $\max K$. Set 
\[
A = \inf\{g_{(x,s)}^{(\dir - 1)-,L}(T):x,s \in K\},\qquad \text{and}\qquad B = \sup\{g_{(x,s)}^{(\dir + 1)+,R}(T):x,s \in K\}.
\]
By Lemma \ref{lem:bounded_maxes} and Item \ref{itm:DL_mont_dir}, $-\infty < A < B < \infty$. Then, for all $0 < \ve < 1$ sufficiently small, all $\dir- \ve < \alpha < \dir$, and all $x,s \in K$, the functions $z \mapsto \Ll(x,s;z,T) + \W_{\alpha \sig}(z,T;0,T)$ and  $z \mapsto \Ll(x,s;z,t) + \W_{\dir-}(z,T;0,T)$ agree on the set $[A,B]$, which contains all maximizers. Hence, for such $\alpha$ and $\sigg \in \{-,+\}$, and $S \in \{L,R\}$, $g_{(x,s)}^{\alpha \sig ,S}(T) = g_{(x,s)}^{\dir -,S}(T)$. 
Since $g_{(x,s)}^{\alpha \sig,L}:[s,\infty) \to \R$ and $g_{(x,s)}^{\alpha \sig,R}:[s,\infty) \to \R$ define semi-infinite geodesics that are, respectively, the leftmost and rightmost geodesics between any of their points (Theorem \ref{thm:DL_SIG_cons_intro}\ref{itm:DL_all_SIG}-\ref{itm:DL_LRmost_geod}), it must also hold that for $S \in \{L,R\}$ and $t \in [t,T]$,
$g_{(x,s)}^{\alpha \sig,S}(t) = g_{(x,s)}^{\dir -,S}(t)
$.
Otherwise, taking $S = L$ without loss of generality, there would exist two distinct leftmost geodesics from $(x,s)$ to $(g_{(x,s)}^{\dir -,L}(T),T)$, a contradiction. The proof for the $\dir +$ geodesics where $\beta$ is sufficiently close to $\dir$ from the right is analogous. 

\noindent \textbf{Item \ref{itm:limits_to_inf} (limit of geodesics as direction goes to $\pm \infty)$:} This holds on $\Omega_2$ by definition \eqref{omega2}. 

\noindent We postpone the proof of Items \ref{itm:DL_SIG_mont_x} and \ref{itm:DL_SIG_conv_x} until after the following proof.
\end{proof}

\begin{proof}[Proof of Theorem \ref{thm:DLNU} (Description of the set $\NU$)]

By Theorem \ref{thm:DLBusedc_description}\ref{itm:DL_NU_count}, on the event $\Omega_2$, $\alpha \notin \DLBusedc$ for all $\alpha \in \Q$, so we omit the $\pm$ distinction in this case. 
We first prove \eqref{109}. 
If $(x,s) \in \NU^{\dir \sig}$ then 
$g_{(x,s)}^{\dir \sig,L}(t) < g_{(x,s)}^{\dir \sig,R}(t)$ for some $t > s$. By  Theorem \ref{thm:g_basic_prop}\ref{itm:DL_SIG_unif}, there exists a rational direction $\alpha$ (greater than $\dir$ if $\sigg = +$ and less than $\dir$ if $\sigg = -$) such that
\[
g_{(x,s)}^{\alpha,L}(t) = g_{(x,s)}^{\dir \sig,L}(t) < g_{(x,s)}^{\dir \sig,R}(t) =g_{(x,s)}^{\alpha,R}(t).
\]
Hence, $(x,s) \in \NU^\alpha$.

\noindent \textbf{Item \ref{itm:DL_NU_p0}:} By
Theorem \ref{thm:RV-SIG-thm}\ref{itm:pd_fixed}, for fixed direction $\dir$ and fixed initial point $p$, there is a unique semi-infinite geodesic from $p$ in direction $\dir$, implying $(x,s) \notin \NU^\dir$. The result now follows directly from \eqref{109} and a union bound. In particular, by definition of the event $\Omega_1\supset \Omega_2$ \eqref{omega1}, for each $(q,r) \in \Q^2$ and $\dir \in \Q$, $(q,r) \notin \NU^\dir$. Then, by \eqref{109}, on the event $\Omega_2$, $\NU \subseteq \R^2 \setminus \Q^2$.

\noindent 
We postpone the proof of Item \ref{itm:DL_NU_count} until the end of this subsection.
\end{proof}


\begin{proof}[Remaining proofs of Theorem \ref{thm:g_basic_prop}]

\noindent 
 \textbf{Item \ref{itm:DL_SIG_mont_x} (Spatial monotonicity of geodesics):} We first prove a weaker result. Namely, for $s \in \R$, $x < y$, $\dir \in \R$, $\sigg \in \{-,+\}$, and $S \in \{L,R\}$, 
\be \label{110}
g_{(x,s)}^{\dir \sig,S}(t) \le g_{(y,s)}^{\dir \sig,S}(t)\qquad\text{for all }t \ge s.
\ee
By continuity of geodesics, it suffices to assume that $z := g_{(x,s)}^{\dir \sig,L}(t) = g_{(y,s)}^{\dir \sig,L}(t)$ for some $t > s$, and then show that $g_{(x,s)}^{\dir \sig,L}(u) = g_{(y,s)}^{\dir \sig,L}(u)$ for all $u >  t$. 
By Theorem \ref{thm:DL_SIG_cons_intro}\ref{itm:DL_all_SIG}, if $z := g_{(x,s)}^{\dir \sig,S}(t) = g_{(y,s)}^{\dir \sig,S}(t)$, then for $u > t$, both $g_{(x,s)}^{\dir \sig,L}(u)$ and $g_{(y,s)}^{\dir \sig,L}(u)$ are the leftmost maximizer of $\Ll(z,t;w,u) + \W_{\dir \sig}(w,u;0,u)$ over $w \in \R$, so they are equal.

Now, to prove the stated result, we follow a similar argument as Item 2 of Theorem 3.4 in \cite{Rahman-Virag-21}, adapted to give a global result across all direction, signs, and pairs of points along the same horizontal line. Let $g_1$ be a $\dir \sig$ geodesic from $(x,s)$ and let $g_2$ be a $\dir \sig$ geodesic from $(y,s)$, and assume that $g_1(t) = g_2(t)$ for some $t > s$. By continuity of geodesics, we may take $t$ to be the minimal such time. Choose $r \in (s,t)\cap \Q$ and then choose $q \in (g_1(r),g_2(r)) \cap \Q$. See Figure \ref{fig:choose_rational}. By Theorem \ref{thm:DLNU}\ref{itm:DL_NU_p0}, on the event $\Omega_2$, there is a unique $\dir \sig$ Busemann geodesic from $(q,r)$, which we shall call $g = g_{(q,r)}^{\dir \sig,L} = g_{(q,r)}^{\dir \sig,R}$.
For $u \ge r$, 
\be \label{210}
 g_1(u)\le g_{(x,s)}^{\dir \sig,R}(u) \le g(u) \le  g_{(y,s)}^{\dir \sig,L}(u) \le g_2(u). 
\ee
The two middle inequalities come from \eqref{110}.  The two outer inequalities come from the definition of $g_{(x,s)}^{\dir \sig,L/R}(u)$ as the left and rightmost maximizers.  

By assumption and \eqref{210}, $z := g_1(t) = g(t) = g_2(t)$. By Theorem \ref{thm:DL_SIG_cons_intro}\ref{itm:arb_geod_cons}\ref{itm:maxes}, for $u > t$, $g_1(u),g_2(u)$, and $g(u)$ are all maximizers of $\Ll(z,t;w,u) + \W_{\dir \sig}(w,u;0,u)$ over $w \in \R$. However, since there is a unique $\dir \sig$ geodesic from $(q,r)$, there can be only one such maximizer, so the inequalities in \eqref{210} are equalities for $u \ge t$.  
\begin{figure}[t]
    \centering
    \includegraphics[height = 1.5in]{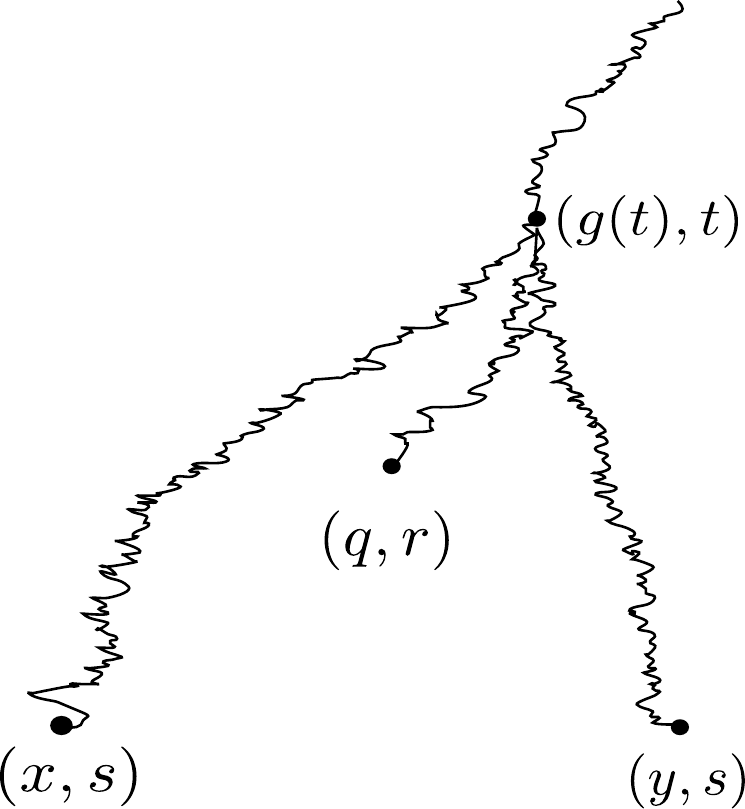}
    \caption{\small Choosing a point $(q,r) \in \Q^2$ whose $\dir \tiny{\boxempty}$ geodesic is unique}
    \label{fig:choose_rational}
\end{figure}

\noindent \textbf{Item \ref{itm:DL_SIG_conv_x} (limits of geodesics in the spatial parameter):} We start by proving \eqref{371}. We prove the statement for the limits as $w \nearrow x$, and the limits as $w \searrow x$ follow analogously. By Item \ref{itm:DL_SIG_mont_x}, $z := \lim_{w \nearrow x} g_{(w,s)}^{\dir \sig,S}(t)$ exists and is less than or equal to $g_{(x,s)}^{\dir \sig,L}(t)$. Further, by the same monotonicity, for all $w \in [x - 1,x]$, all maximizers of $\Ll(w,s;y,t) + \W_{\dir \sig}(y,t;0,t)$ over $y \in \R$ lie in the common compact set $[g_{(x - 1,s)}^{\dir \sig,L}(t),g_{(x,s)}^{\dir \sig,R}(t)]$. By continuity of the directed landscape (Lemma \ref{lem:Landscape_global_bound}), as $w \nearrow x$, the function $y \mapsto \Ll(w,s;y,t) + \W_{\dir \sig}(y,t;0,t)$ converges uniformly on compact sets to the function $y \mapsto \Ll(x,s;y,t) + \W_{\dir \sig}(y,t;0,t)$. Hence, Lemma \ref{lemma:convergence of maximizers from converging functions} implies that $z$ is a maximizer of $\Ll(x,s;y,t) + \W_{\dir \sig}(y,t;0,t)$ over $y \in \R$. Since $z \le g_{(x,s)}^{\dir \sig,L}(t)$, and $g_{(x,s)}^{\dir \sig,L}(t)$ is the leftmost such maximizer, equality holds. 

The proof of \eqref{372} is similar: in this case, Lemma \ref{lem:bounded_maxes} implies that for all $(w,u)$ sufficiently close to $(x,s)$, the maximizers of $y \mapsto\Ll(w,u;y,t) + \W_{\dir \sig}(y,t;0,t)$ lie in a common compact set. Then, by Lemma \ref{lemma:convergence of maximizers from converging functions}, every subsequential limit of $g_{(w,u)}^{\dir \sig,S}(t)$ as $(w,u) \to (x,s)$ is a maximizer of $y \mapsto\Ll(x,s;y,t) + \W_{\dir \sig}(y,t;0,t)$. By assumption, there is only one such maximizer, so the desired convergence holds.  

Lastly, to show \eqref{373}, we recall that the Busemann process evolves as the KPZ fixed point (Theorem \ref{thm:DL_Buse_summ}\ref{itm:Buse_KPZ_description}).  The Busemann functions are continuous and satisfy the asymptotics prescribed in Lemma \ref{lem:DL_horiz_Buse}\ref{itm:DL_lim}. Therefore, for each $t,\dir$, and $\sigg$, there exists constants $a,b > 0$ so that $|\W_{\dir \sig}(x,t;0,t)| \le a + b|x|$.  Lemma \ref{lem:max_restrict}\ref{itm:KPZrestrict} applied to the temporally reflected version of $\Ll$ states that for sufficiently large $|x|$, $g_{(x,s)}^{\dir \sig,S}(t) \in (x - |x|^{2/3},x + |x|^{2/3})$. 
\end{proof}

\begin{proof}[Proof of Theorem \ref{thm:all_SIG_thm_intro}]

We remind the reader that this theorem controls arbitrary\\ geodesics via the Busemann geodesics. 
  
\noindent \textbf{Item \ref{itm:DL_LRmost_SIG}:} Let $\alpha < \dir < \beta$. By directedness of Busemann geodesics (Theorem \ref{thm:DL_SIG_cons_intro}\ref{itm:DL_all_SIG}) and the assumption $x_r/r_r \to \dir$, for all sufficiently large $r$,
\[
g_{(x,s)}^{\alpha -,L}(t_r) < x_r < g_{(x,s)}^{\beta +,R}(t_r).
\]
Since $g_{(x,s)}^{\alpha -,L}$ is the leftmost geodesic between any of its points and $g_{(x,s)}^{\beta +,R}$ is the rightmost (Theorem \ref{thm:DL_SIG_cons_intro}\ref{itm:DL_LRmost_geod}), it follows that for $u \in [s,t_r]$,
\be \label{504}
g_{(x,s)}^{\alpha -,L}(u) \le g_r(u) \le g_{(x,s)}^{\beta +,R}(u).
\ee
Hence, for all $t \ge s$,
\[
g_{(x,s)}^{\alpha -,L}(t) \le \liminf_{r \to \infty} g_r(t) \le \limsup_{r \to \infty} g_r(t) \le g_{(x,s)}^{\beta +,R}(t).
\]
By Theorem \ref{thm:g_basic_prop}\ref{itm:DL_SIG_unif}, taking limits as $\alpha \nearrow \dir$ and $\beta \searrow \dir$ completes the proof. 

\noindent \textbf{Item \ref{itm:finite_geod_stick}:} Assume that all geodesics in direction $\dir$, starting from a point in the compact set $K$, have coalesced by time $t$, and for $u \ge t$, let $g(u)$ be the spatial location of this common geodesic. By Item \ref{itm:DL_LRmost_SIG}, for all $p \in K$ and $u \ge t$,
\[
g(u) = g_p^{\dir -,L}(u) = g_p^{\dir +,R}(u).
\]
Let $T > t$ be arbitrary. By Theorem \ref{thm:g_basic_prop}\ref{itm:DL_SIG_unif}, we may choose $\alpha < \dir < \beta$ such that, for all $p \in K$ and $u \in [t,T]$,
\be \label{301}
g_{(g(t),t)}^{\alpha-,L}(u) = g_p^{\alpha -,L}(u) = g(u) = g_p^{\beta +,R}(u) = g_{(g(t),t)}^{\beta +,R}(u).
\ee
The outer equalities hold because the geodesics pass through $(g(t),t)$. With this choice of $\alpha,\beta$, by the directedness of Theorem \ref{thm:DL_SIG_cons_intro}\ref{itm:DL_all_SIG} and since $x_r/t_r \to \dir$, we may choose $r$ large enough so that $t_r \ge T$ and
$
g_{(g(t),t)}^{\alpha -,L}(t_r) < x_r < g_{(g(t),t)}^{\beta +,R}(t_r).
$
 Then,  as in the proof of Item \ref{itm:DL_LRmost_SIG}, for all $u \in [t,t_r]$,
\[
g_{(g(t),t)}^{\alpha -,L}(u) \le g_r(u) \le g_{(g(t),t)}^{\beta +,R}(u).
\]
Combining this with \eqref{301} completes the proof. 
\end{proof}

\noindent It remains to prove Theorem \ref{thm:DLNU}\ref{itm:DL_NU_count}. We first prove a lemma.

\begin{lemma} \label{lem:NU_line}
Let $\omega \in \Omega_2$, $\dir \in \R$, $\sigg \in \{-,+\}$, $\Q \ni s < t \in \R$,  and assume that there is a nonempty interval $I = (a,b) \subseteq \R$ such that for all $x \in \Q$, $g_{(x,s)}^{\dir \sig}(t) \notin I$ {\rm(}By Theorem \ref{thm:DLNU}\ref{itm:DL_NU_p0}, we may ignore the $L/R$ distinction when $(x,s) \in \Q^2${\rm)}. Then, there exists $\hat x \in \R$ such that 
\be \label{794}
g_{(\hat x,s)}^{\dir \sig,L}(t) \le a < b \le g_{(\hat x,s)}^{\dir \sig,R}(t). 
\ee
\end{lemma}
\begin{proof}
Choose some $y \in (a,b)$, and let 
\[
\hat x = \sup\{x \in \Q : g_{(x,s)}^{\dir \sig}(t) < y\}.
\]
By Equation \eqref{373} of Theorem \ref{thm:g_basic_prop}\ref{itm:DL_SIG_conv_x}, $\hat x \in \R$. By the monotonicity of Theorem \ref{thm:g_basic_prop}\ref{itm:DL_SIG_mont_x}, for all $\Q \ni x < \hat x$, $g_{(x,s)}^{\dir \sig}(t) < y$, while for all $\Q \ni x > \hat x$, $g_{(x,s)}^{\dir \sig}(t) \ge y$. By assumption of the lemma,
this further implies that for $\Q \ni x < \hat x$, $g_{(x,s)}^{\dir \sig}(t) \le a$ while for $\Q \ni x > \hat x$, $g_{(x,s)}^{\dir \sig}(t) \ge b$.  By taking limits via Equation \eqref{371} of Theorem \ref{thm:g_basic_prop}\ref{itm:DL_SIG_conv_x}, we obtain \eqref{794}.
\end{proof}

\begin{proof}[Proof of Theorem \ref{thm:DLNU}\ref{itm:DL_NU_count} ($\NU^{\dir \sig} \cap \Hh_s$ is countably infinite and unbounded)] 
We prove \\the statement in three steps. First, we show that on $\Omega_2$, for all $s \in \Q$, $\dir \in \R$, $\sigg \in \{-,+\}$, the set $\NU^{\dir \sig} \cap\, \Hh_s$ is infinite and unbounded in both directions. Next, we show that, on $\Omega_2$, $\NU^{\dir \sig} \cap\, \Hh_s$ is in fact infinite and unbounded in both directions for all $s \in \R$. Lastly, we show that the set $\NU \cap\, \Hh_s$ (the union over all directions and signs) is countable.

 For the first step, Theorem \ref{thm:DLNU}\ref{itm:DL_NU_p0} states that, on the event $\Omega_2$, for each $(x,s) \in \Q^2$, $\dir \in \R$, and $\sigg \in \{-,+\}$, there is a unique $\dir \sig$ geodesic $g_{(x,s)}^{\dir \sig}$, and therefore this geodesic is both the leftmost and rightmost $\dir \sig$ geodesic from $(x,s)$. Since leftmost (resp. rightmost) Busemann geodesics are leftmost (rightmost) geodesics between any two of their points (Theorem \ref{thm:DL_SIG_cons_intro}\ref{itm:DL_LRmost_geod}), it follows that  $g_{(x,s)}^{\dir \sig}$, restricted to times $t \in [s,s+2]$, is the unique geodesic from $(x,s)$ to $(g_{(x,s)}^{\dir \sig}(s + 2),s +2)$. By Lemma \ref{lem:bounded_maxes}, for each compact set $K$, the set 
 \[
 \{g_{(x,s)}^{\dir \sig}(s + 1): x \in \Q \cap K\} 
 \]
is contained in some compact set $K'$.
Then, we have the following inclusion of sets:
\begin{align} 
&\quad \; \{g_{(x,s)}^{\dir \sig}(s + 1): x \in \Q \cap K \} \label{873}
\subseteq  \bigcup_{g \in \mathcal A_{K,K'}}\{g(s + 1) \} 
\end{align}
where 
\[
\mathcal A_{K,K'} = \{g: \text{$g$  is the unique geodesic from }(x,s) \text{ to }(y,s+2) \text{ for some } x\in K,y \in K'\}. 
\]
By Lemma \ref{lem:geod_pp}, the set in the RHS of \eqref{873} is finite, so the set on the LHS  is finite as well. Therefore, the set
\be \label{875}
\{g_{(x,s)}^{\dir \sig}(s + 1): x \in \Q \} = \bigcap_{k \in \Z_{>0}}\{g_{(x,s)}^{\dir \sig}(s + 1): x \in \Q \cap [-k,k] \}
\ee
is a union of finite nested sets. Further, by the ordering of geodesics from Theorem \ref{thm:g_basic_prop}\ref{itm:DL_SIG_mont_x}, for each $k$, the difference 
\[
\{g_{(x,s)}^{\dir \sig}(s + 1): x \in \Q \cap [-(k + 1),k + 1] \} \setminus \{g_{(x,s)}^{\dir \sig}(s + 1): x \in \Q \cap [-k,k] \}
\]
lies entirely in the union of intervals
\[
\Bigl(-\infty, \inf \bigl\{g_{(x,s)}^{\dir \sig}(s + 1): x \in \Q \cap [-k,k] \bigr\}\Bigr] \cup \Bigl[\sup \bigl\{g_{(x,s)}^{\dir \sig}(s + 1): x \in \Q \cap [-k,k] \bigr\},\infty\Bigr).
\]
Therefore, the set \eqref{875} has no limit points. 
Further, by Equation \eqref{373} of Theorem \ref{thm:g_basic_prop}\ref{itm:DL_SIG_conv_x}, the set \eqref{875} is unbounded in both directions. These two facts imply that there exist infinitely many disjoint nonempty intervals whose intersection with the set \eqref{875} is empty, and the set of endpoints of such intervals is unbounded. By Lemma \ref{lem:NU_line}, for each $k > 0$, there exists $(x,s) \in \NU^{\dir \sig}$ such that $g_{(x,s)}^{\dir \sig,R}(s + 1) \ge k$, and there exists $(x,s) \in \NU^{\dir \sig}$ such that $g_{(x,s)}^{\dir \sig,L}(s + 1) \le -k$. Next, assume, by way of contradiction, that the set $\{x \in \R: (x,0) \in \NU^{\dir \sig}\}$ has an upper bound $b$. Then, by the monotonicity of Theorem \ref{thm:g_basic_prop}\ref{itm:DL_SIG_mont_x}, for all $x \in \R$ with $(x,s) \in \NU^{\dir \sig}$, $g_{(x,s)}^{\dir \sig,R}(s + 1) \le g_{(b,s)}^{\dir \sig,R}(s +1)$. But this contradicts the fact we showed that $\{g_{(x,s)}^{\dir \sig,R}(s + 1): x \in \R\}$ is not bounded above. Hence, there exists a sequence $y_n \to \infty$ such that $(y_n,s) \in \NU^{\dir \sig}$ for all $n$. By a similar argument, there exists a sequence $x_n \to -\infty$ such that $(x_n,s) \in \NU^{\dir \sig}$ for all $n$.

Now, for arbitrary $s \in \R$, pick a rational number $T > s$. Pick $(z,T) \in \NU^{\dir \sig}$, and let
\begin{align*}
x_1 = \sup\{  x \in \R:   g_{(x,s)}^{\dir\sig, L}(T) \le  z  \}, \qquad \text{and}\qquad
x_2 = \inf\{  x \in \R:   g_{(x,s)}^{\dir\sig, R}(T) \ge  z  \}.
\end{align*}
By the limits in Equation \eqref{373} of Theorem \ref{thm:g_basic_prop}\ref{itm:DL_SIG_conv_x}, 
$x_1$ and $x_2$ lie in $\R$.

We first show that $x_2 \le x_1$. If not, then choose $x \in (x_1,x_2)$. Then, $g_{(x,s)}^{\dir\sig, R}(T) <   z <    g_{(x,s)}^{\dir\sig, L}(T)$, contradicting the meaning of L and R.  Hence  $x_2 \le x_1$. For any $x > x_2$, $g_{(x,s)}^{\dir \sig,R}(T) \ge z$, and by the limit in Equation \eqref{371} of Theorem \ref{thm:g_basic_prop}\ref{itm:DL_SIG_conv_x}, $g_{(x_2,s)}^{\dir \sig,R}(T) \ge z$ as well. By an analogous argument, for $x < x_1$, $g_{(x,s)}^{\dir \sig,L}(T) \le z$, and the inequality $g_{(x_1,s)}^{\dir \sig,L}(T) \le z$ holds by the same argument. Hence, for $x \in [x_2,x_1]$,
\[
g_{(x,s)}^{\dir \sig,L}(T) \le z,\qquad\text{and}\qquad g_{(x,s)}^{\dir \sig,R}(T) \ge z.
\]
Then, by the monotonicity of Theorem \ref{thm:g_basic_prop}\ref{itm:DL_SIG_mont_x}, for $t \ge T$,
\be \label{1000}
g_{(x,s)}^{\dir \sig,L}(t) \le g_{(z,T)}^{\dir \sig,L}(t) \le g_{(z,T)}^{\dir \sig,R}(t) \le g_{(x,s)}^{\dir \sig,R}(t).
\ee
By assumption that $(z,T) \in \NU^{\dir \sig}$, there exists $t > T$ such that the middle inequality in \eqref{1000} is strict, so $(x,s) \in \NU^{\dir \sig}$. Furthermore, by assumption, the set $\{z \in \R: (z,T) \in \NU\}$ has neither an upper or lower bound. Then, by the $t = T$ case of \eqref{1000} and a similar argument as for the $s = 0$ case, the set $\{x \in \R: (x,s) \in \NU\}$ also has  neither an upper nor lower bound.

We lastly show countability of the sets. By \eqref{109}, it suffices to show that for each $\dir \in \Q$ and $s \in \R$, $\NU^{\dir} \cap\, \Hh_s$ is countable. The proof is that of Theorem 3.4, Item 3 in \cite{Rahman-Virag-21}, adapted to all horizontal lines simultaneously. For each $(x,s) \in \NU^\dir$, there exists $t > s$ such that $g_{(x,s)}^{\dir ,L}(t) < g_{(x,s)}^{\dir ,R}(t)$. By continuity of geodesics, the space between the two geodesics contains an open subset of $\R^2$. By the monotonicity of Theorem \ref{thm:g_basic_prop}\ref{itm:DL_SIG_mont_x}, for $x < y$, $g_{(x,s)}^{\dir ,R}(t) \le g_{(y,s)}^{\dir ,L}(t)$ for all $t \ge s$. Hence, for $x < y$, with $(x,s),(y,s) \in \NU^\dir$, the associated open sets in $\R^2$ are disjoint, and $\NU^\dir \cap\, \Hh_s$ is at most countably infinite. 
\end{proof}

\section{Coalescence and the global geometry of geodesics} \label{sec:geometry_sec}
We can now describe the global structure of the semi-infinite geodesics, beginning with coalescence. 
\begin{theorem} \label{thm:DL_all_coal}
On a single event of full probability, the following hold across all directions $\dir \in \R$ and signs $\sigg \in \{-,+\}$. 
\begin{enumerate} [label=\rm(\roman{*}), ref=\rm(\roman{*})]  \itemsep=3pt
    \item \label{itm:DL_allsigns_coal} For all 
    $p,q \in \R^2$, if $g_1$ and $g_2$ are $\dir \sig$ Busemann geodesics from $p$ and $q$, respectively, then $g_1$ and $g_2$ coalesce. If the first point of intersection of the two geodesics is not $p$ or $q$, then the first point of intersection is the coalescence point of the two geodesics. 
    \item \label{itm:DL_split_return} 
    Let $g_1$ and $g_2$ be two distinct  $\dir \sig$ Busemann geodesics  from an initial point $(x,s)\in \NU^{\dir \sig}$.  Then, the set $\{t > s: g_1(t) \neq g_2(t)\}$ is a bounded open interval. That is, after the geodesics split, they coalesce exactly when they meet again. 
    \item \label{itm:unif_coal} For each 
    compact set $K \subseteq \R^2$, there exists a random $T = T(K,\dir,\sigg)<\infty$ such that for any two $\dir \sig$ geodesics $g_1$ and $g_2$ whose starting points lie in $K$, $g_1(t) = g_2(t)$ for all $t \ge T$. That is, there is a time level $T$ after which all semi-infinite geodesics started from points in $K$ have coalesced into a single path.
\end{enumerate} 
\end{theorem}
\begin{remark}
Theorem 1 of \cite{Bhatia-23} implies the following refinements of the results in this section. In Theorem \ref{thm:DL_all_coal}\ref{itm:DL_split_return}, $\{t > s: g_1(t) \neq g_2(t)\}=(s, r)$ for some $r\in(s,\infty)$.  Under Condition \ref{itm:DL_good_dir} of Theorem \ref{thm:DL_good_dir_classification} below, the entire collection of semi-infinite geodesics in direction $\dir$ is a tree. 
\end{remark}

The following gives a full classification of the directions in which geodesics coalesce. We refer the reader to Theorems \ref{thm:DL_eq_Buse_cpt_paths} and \ref{thm:Buse_pm_equiv} below for the connection between coalescence and the regularity of the Busemann process.
 
\begin{theorem} \label{thm:DL_good_dir_classification}
On  a single event of probability one, the following are equivalent. 
\begin{enumerate} [label=\rm(\roman{*}), ref=\rm(\roman{*})]  \itemsep=3pt
    \item \label{itm:DL_good_dir} $\dir \notin \DLBusedc$.
    \item \label{itm:DL_LR_all_agree} $g_{p}^{\dir -,S} = g_{p}^{\dir +,S}$ for all $p \in \R^2$ and $S \in \{L,R\}$.
    \item \label{itm:DL_good_dir_coal} All semi-infinite geodesics in direction $\dir$ coalesce {\rm(}whether Busemann geodesics or not{\rm)}.
    \item \label{itm:DL_good_dir_unique_geod} For all $p \in \R^2 \setminus \NU$, there is a unique geodesic starting from $p$ with direction $\dir$.
    \item \label{itm:DL_good_dir_pt_unique}   There is a unique $\dir$-directed semi-infinite geodesic from some $p\in \R^2$.
    \item \label{itm:DL_good_dir_L_unique} There exists $p \in \R^2$ such that $g_{p}^{\dir -,L} = g_{p}^{\dir +,L}$.
    \item \label{itm:DL_good_dir_R_unique} There exists $p \in \R^2$ such that $g_{p}^{\dir -,R} = g_{p}^{\dir +,R}$
\end{enumerate}
Under these equivalent conditions, the following also holds.
\begin{enumerate}[resume, label=\rm(\roman{*}), ref=\rm(\roman{*})]  \itemsep=3pt
    \item \label{itm:DL_allBuse}  From any $p \in \R^2$, all semi-infinite geodesics in direction $\dir$ are Busemann geodesics. 
\end{enumerate}
\end{theorem}
\begin{remark} \label{rmk:split_from_all_p}
The equivalence \ref{itm:DL_good_dir}$\Leftrightarrow$\ref{itm:DL_good_dir_L_unique} implies that $\forall\dir \in \DLBusedc$ and $p \in \R^2$,  geodesics $g_{p}^{\dir -,L}$ and $g_{p}^{\dir +,L}$ are distinct. The same is true when $L$ is replaced with $R$. Since $g_{p}^{\dir -,L}$ and $g_p^{\dir +,L}$ are both leftmost geodesics between any two of their points (Theorem \ref{thm:DL_SIG_cons_intro}\ref{itm:DL_LRmost_geod}) 
then if $\dir\in\DLBusedc$,
these two geodesics must separate at some time $t\ge s$, and they cannot ever come back together. For each $\dir \in \DLBusedc$, there are two coalescing families of geodesics, namely the $\dir-$ and $\dir +$ geodesics. (See again Figure \ref{fig:non_unique_comp}.)  In particular, 
whenever $\xi \in \DLBusedc$, $s \in \R$,  and $x < y$, $g_{(x,s)}^{\xi +,L}(t) > g_{(y,s)}^{\xi -,R}(t)$ for sufficiently large $t$, as alluded to in Remark \ref{rmk:mixing_LR_pm}.
\end{remark}

\subsection{Proofs}
In each of these theorems, the full-probability event is $\Omega_2$ \eqref{omega2}. We start by proving some lemmas that allow us to prove Theorem \ref{thm:DL_all_coal}. The proof of Theorem \ref{thm:DL_good_dir_classification} comes at the very end of this subsection. Section \ref{sec:last_proofs1} proves Theorem \ref{thm:DLSIG_main} as well as lingering results from Section \ref{sec:Buse_geod_results}.
\begin{lemma} \label{lem:Buse_equality_coal}
Let $\omega \in \Omega_1$, $s \in \R$ and $x < y \in \R$. Assume, for some $\alpha < \dir$ and $\sigg_1,\sigg_2 \in \{-,+\}$, that $\W_{\alpha \sig_1}(y,s;x,s) = \W_{\dir \sig_2}(y,s;x,s)$. We also allow $\alpha = \dir$ if $\sigg_1 = -$ and $\sigg_2 = +$. If  $t > s$ and $g_{(x,s)}^{\dir \sig_2,R}(t) \le g_{(y,s)}^{\alpha \sig_1,L}(t)$, then for all $u \in [s,t]$,
\be \label{111}
g_{(x,s)}^{\alpha \sig_1,R}(u) = g_{(x,s)}^{\dir \sig_2,R}(u)\qquad\text{and}\qquad g_{(y,s)}^{\alpha \sig_1,L}(u) = g_{(y,s)}^{\dir \sig_2,L}(u).
\ee
\end{lemma}
\begin{proof}
By assumption, whenever $w < z$ and $t \in \R$, Theorem \ref{thm:DL_Buse_summ}\ref{itm:DL_Buse_gen_mont} gives 
\be \label{100}
\W_{\alpha \sig_1}(z,t;w,t) \le W_{\dir \sig_2}(z,t;w,t).
\ee
For the rest of the proof, we  suppress the $\sigg_1,\sigg_2$ notation.
By Theorem \ref{thm:DL_Buse_summ}\ref{itm:DL_Buse_add},\ref{itm:Buse_KPZ_description},
\begin{align}
\W_\dir(y,s;x,s) &= \W_\dir(y,s;0,t) - \W_\dir(x,s;0,t) \nonumber \\
&= \sup_{z \in \R}\{\Ll(y,s;z,t) + \W_\dir(z,t;0,t)\} - \sup_{z \in \R}\{\Ll(x,s;z,t) + \W_\dir(z,t;0,t)\}, \label{eqn:W_queue}
\end{align}
and the same with $\dir$ replaced by $\alpha$. Recall that $g_{(x,s)}^{\dir \sig,L}(t)$ and $g_{(x,s)}^{\dir \sig,R}(t)$ are, respectively, the leftmost and rightmost maximizers of $\Ll(x,s;z,t) + \W_{\dir \sig}(z,t;0,t)$ over $z \in \R$. 
Understanding that these quantities depend on $s$ and $t$, we use the shorthand notation $g_x^{\dir,R} = g_{(x,s)}^{\dir \sig_1,R}(t)$, and similarly with the other quantities. Then, we have
\begin{align}
    \Ll(x,s;g_x^{\dir,R},t) + W_\dir(g_x^{\dir,R},t;0,t) &- (\Ll(x,s;g_x^{\dir,R},t) + W_\alpha(g_x^{\dir,R},t;0,t)) \nonumber \\
    \ge \sup_{z \in \R}\{\Ll(x,s;z,t) + \W_\dir(z,t;0,t)\} &- \sup_{z \in \R}\{\Ll(x,s;z,t) + \W_\alpha(z,t;0,t)\} \label{101}\\
    =\sup_{z \in \R}\{\Ll(y,s;z,t) + \W_\dir(z,t;0,t)\} &- \sup_{z \in \R}\{\Ll(y,s;z,t) + \W_\alpha(z,t;0,t)\} \nonumber \\
    \ge \Ll(y,s;g_y^{\alpha,L},t) + W_\dir(g_y^{\alpha,L},t;0,t) &- (\Ll(y,s;g_y^{\alpha,L},t) + W_\alpha(g_y^{\alpha,L},t;0,t)), \label{102} 
\end{align}
where the middle equality came from the assumption that $\W_{\dir}(y,s;x,s) = \W_{\alpha}(y,s;x,s)$ and Equation \eqref{eqn:W_queue} applied to both $\dir$ and $\alpha$.
Rearranging the first and last lines yields 
\[
\W_\dir(g_y^{\alpha,L},t;g_{x}^{\dir,R},t) \le \W_\alpha(g_y^{\alpha,L},t;g_{x}^{\dir,R},t).
\]
However, the assumption $g_x^{\dir,R} \le g_{y}^{\alpha,L}$
combined with \eqref{100} implies that this inequality is an equality. Hence, inequalities \eqref{101} and \eqref{102} are also equalities. From the equality \eqref{101},  
\[
\Ll(x,s;g_x^{\dir,R},t) + W_\alpha(g_x^{\dir,R},t;0,t) = \sup_{z \in \R}\{\Ll(x,s;z,t) + \W_\alpha(z,t;0,t)\},
\]
so $z=g_x^{\dir,R}$ is a maximizer of $\Ll(x,s;z,t) + \W_\alpha(z,t;0,t)$. By definition, $g_{x}^{\alpha,R}$ is the rightmost maximizer, and by geodesic ordering (Theorem \ref{thm:g_basic_prop}\ref{itm:DL_mont_dir}), $g_x^{\dir,R} \ge g_x^{\alpha,R}$, so  $g_x^{\dir,R} = g_x^{\alpha,R}$. An analogous argument applied to \eqref{102} implies  $g_y^{\alpha,L} = g_y^{\dir,L}$. We have shown that 
\[
g_{(x,s)}^{\alpha \sig_1,R}(t) = g_{(x,s)}^{\dir \sig_2,R}(t),\qquad\text{and}\qquad g_{(y,s)}^{\alpha \sig_1,L}(t) = g_{(y,s)}^{\dir \sig_2,L}(t).
\]
Since $g_{(x,s)}^{\alpha \sig_1,R}$ and $g_{(x,s)}^{\dir \sig_2,R}$ are both the rightmost geodesics between any two of their points and similarly with the leftmost geodesics from $(y,s)$ (Theorem \ref{thm:DL_SIG_cons_intro}\ref{itm:DL_LRmost_geod}), Equation \eqref{111} holds for all $u \in [s,t]$, as desired. 
\end{proof}

\begin{lemma} \label{lem:DL_LR_coal}
Let $\omega \in \Omega_2$, $s \in \R$, and $x < y$. If, for some $\alpha < \dir$ and $\sigg_1,\sigg_2 \in \{-,+\}$ we have that $\W_{\alpha \sig_1}(y,s;x,s) = \W_{\dir \sig_2}(y,s;x,s)$, then $g_{(x,s)}^{\alpha \sig_1,R}$ coalesces with $g_{(y,s)}^{\alpha \sig_1,L}$, $g_{(x,s)}^{\dir \sig_2,R}$ coalesces with $g_{(y,s)}^{\dir \sig_2,L}$, and the coalescence points of the two pairs of geodesics are the same. 
\end{lemma}
\begin{proof}
By Theorem \ref{thm:DL_SIG_cons_intro}\ref{itm:DL_all_SIG}, $g_{(x,s)}^{\dir \sig_2,R}(t)/t \to \dir$ while $g_{(y,s)}^{\alpha \sig_1,L}(t)/t \to \alpha$ as $t \to \infty$. By this and continuity of geodesics, there exists a minimal time $t > s$ such that  \\
$z := g_{(x,s)}^{\dir \sig_2,R}(t) = g_{(y,s)}^{\alpha \sig_1,L}(t).
$
By Lemma \ref{lem:Buse_equality_coal}, 
\[
g_{(x,s)}^{\alpha \sig_1,R}(u) = g_{(x,s)}^{\dir \sig_2,R}(u)\qquad\text{and}\qquad g_{(y,s)}^{\alpha \sig_1,L}(u) = g_{(y,s)}^{\dir \sig_2,L}(u) \qquad \text{for all }u \in [s,t].
\]
Since $t$ was chosen to be minimal, Theorem \ref{thm:g_basic_prop}\ref{itm:DL_SIG_mont_x} implies that the pair $g_{(x,s)}^{\alpha \sig_1,R}$, $g_{(y,s)}^{\alpha \sig_1,L}$ and the pair $g_{(x,s)}^{\dir \sig_2,R}$,  $g_{(y,s)}^{\dir \sig_2,L}$ both coalesce at $(z,t)$.
\end{proof}

\begin{proof}[Proof of Theorem \ref{thm:DL_all_coal}]

\noindent
\textbf{Item \ref{itm:DL_allsigns_coal} (Coalescence):} Let $g_1$ and $g_2$ be $\dir \sigg$ Busemann \\ geodesics from $(x,s)$ and $(y,t)$, respectively, and take $s \le t$ without loss of generality. Let $a = (g_1(t) \wedge y) - 1$ and $b = (g_1(t) \vee y) + 1$. By Theorem \ref{thm:g_basic_prop}\ref{itm:DL_SIG_mont_x}, for all $u \ge t$,
\be \label{112}
g_{(a,t)}^{\dir \sig,R}(u) \le g_1(u) \wedge g_2(u) \le g_1(u)\vee g_2(u) \le g_{(b,t)}^{\dir \sig,L}(u).
\ee
By Theorem \ref{thm:DL_Buse_summ}\ref{itm:DL_unif_Buse_stick}, there exists $\alpha$, sufficiently close to $\dir$, (from the left for $\sigg = -$ and from the right for $\sigg = +$) such that $\W_{\dir  \sig}(b,t;a,t) = \W_{\alpha \sig}(b,t;a,t)$. By Lemma \ref{lem:DL_LR_coal}, $g_{(a,t)}^{\dir \sig,R}$ coalesces with $g_{(b,t)}^{\dir \sig,L}$. Then, for $u$ large enough, all inequalities in \eqref{112} are equalities, and $g_1$ and $g_2$ coalesce.  

If the first point of intersection is not $(y,t)$, then $g_1(t) \neq y$, and the coalescence point of $g_1$ and $g_2$ is the first point of intersection by Theorem \ref{thm:g_basic_prop}\ref{itm:DL_SIG_mont_x}.

\noindent \textbf{Item \ref{itm:DL_split_return} (Geodesics coalesce when they meet):} Let $(x,s) \in \NU^{\dir \sig}$, and let $g_1$ and $g_2$ be two distinct $\dir \sig$ Busemann geodesics from $(x,s)$. The set $\text{GNEQ} := \{t>s:g_1(t) \neq g_2(t)\}$ is therefore nonempty and infinite by continuity of $g_1$ and $g_2$. Assume, by way of contradiction, that $\text{GNEQ}$ is not an open interval. By continuity of geodesics, $\text{GNEQ}$ cannot be a closed or half-closed interval, so $\text{GNEQ}$ is not path connected. Thus, there exists $t_1 < t_2 < t_3$ so that 
\[
g_1(t_1) \neq g_2(t_1),\quad g_1(t_2) = g_2(t_2),\quad \text{ and}\quad  g_1(t_3) \neq g_2(t_3).
\]
 The geodesics $g_1|_{[t_1,\infty)}$ and $g_2|_{[t_1,\infty)}$ started from $(g_1(t_1),t_1)$ and $(g_2(t_1),t_1)$, respectively, are both Busemann geodesics by their construction in Theorem \ref{thm:DL_SIG_cons_intro}. Since the geodesics $g_1|_{[t_1,\infty)}$ and $g_2|_{[t_1,\infty)}$  start at different spatial locations (namely $g_1(t_1)$ and $g_2(t_1)$) along the same time level $t_1$, they cannot intersect at either of their starting points.  By Item \ref{itm:DL_allsigns_coal}, the two geodesics $g_1|_{[t_1,\infty)}$ and $g_2|_{[t_1,\infty)}$ must coalesce, and the first point of intersection is the coalescence point. Since $g_1(t_2)  = g_2(t_2)$, this implies that $g_1(t) = g_2(t)$ for all $t > t_2$, a contradiction to the existence of $t_3$.
\newpage
\noindent  \textbf{Item \ref{itm:unif_coal} (Uniformity of coalescence):} Let $\dir \in \R$, $\sigg \in \{-,+\}$, and let the compact set $K$ be given. Let $S$ be the smallest integer greater than $\max\{s: (x,s) \in K\}$. Set 
\[
A := \inf\{g_{(x,s)}^{\dir \sig,L}(S): (x,s) \in K\},\qquad\text{and}\qquad B := \sup\{g_{(x,s)}^{\dir \sig,R}(S):(x,s) \in K\}.
\]
By Lemma \ref{lem:bounded_maxes}, $-\infty < A \le B < \infty$.  Then, by Theorem \ref{thm:g_basic_prop}\ref{itm:DL_SIG_mont_x}, whenever $g$ is a $\dir \sig$ geodesic starting from $(x,s) \in K$,
\[
g_{(A,S)}^{\dir \sig,L}(t) \le g(t) \le g_{(B,S)}^{\dir \sig,R}(t) \qquad\text{for all }t \ge S.
\]
To complete the proof, let $T$ be the time at which $g_{(A,S)}^{\dir \sig,L}$ and $g_{(B ,S)}^{\dir \sig,R}$ coalesce, which is guaranteed to be finite by Item \ref{itm:DL_allsigns_coal}.
\end{proof}

For two initial points on a horizontal level, as $\dir$ varies, a constant  Busemann process corresponds to a  constant    coalescence point of the geodesics.
The non-uniqueness of geodesics  
requires us to be careful about the choice of left and right geodesic. 
\begin{definition} \label{def:coal_pt}
For $s \in \R$ and $x < y$, let $\mbf z^{\dir \sig}(y,s;x,s)$ be the coalescence point of $g_{(y,s)}^{\dir \sig,L}$ and $g_{(x,s)}^{\dir \sig,R}$.
\end{definition}

\begin{theorem} \label{thm:DL_eq_Buse_cpt_paths}
On a single event of probability one, for all reals $\alpha < \beta$, $s$, and $x < y$, the following are equivalent.
\begin{enumerate}[label=\rm(\roman{*}), ref=\rm(\roman{*})]  \itemsep=3pt
\item \label{itm:DL_buse_eq}$\W_{\alpha +}(y,s;x,s) = \W_{\beta -}(y,s;x,s)$.
\item \label{itm:DL_coal_pt_equal} $\mbf z^{\alpha +}(y,s;x,s) = \mbf z^{\beta -}(y,s;x,s)$.
\item \label{itm:DL_paths} There exist $t > s$ and $z \in \R$ such that there are paths $g_1:[s,t] \to \R$ {\rm(}connecting $(x,s)$ and $(z,t)${\rm)} and $g_2:[s,t] \to \R$ {\rm(}connecting $(y,s)$ to $(z,t)${\rm)} such that for all $\dir \in (\alpha,\beta)$, $\sigg \in \{-,+\}$, and $u \in [s,t)$,
\be \label{124}\begin{aligned} 
g_1(u) &= g_{(x,s)}^{\dir \sig,R}(u) = g_{(x,s)}^{\alpha +,R}(u) = g_{(x,s)}^{\beta -,R}(u) \\
&< g_2(u) =g_{(y,s)}^{\dir \sig,L}(u) = g_{(y,s)}^{\alpha +,L}(u) = g_{(y,s)}^{\beta -,L}(u).
\end{aligned} \ee
\end{enumerate}
\end{theorem}

\begin{proof}
 
\ref{itm:DL_buse_eq}$\Rightarrow$\ref{itm:DL_coal_pt_equal} follows from Lemma \ref{lem:DL_LR_coal}. 

 \noindent \ref{itm:DL_coal_pt_equal}$\Rightarrow$\ref{itm:DL_buse_eq}: Assume $(z,t) := \mbf z^{\alpha +}(y,s;x,s) = \mbf z^{\beta -}(y,s;x,s)$. \\ By additivity (Theorem \ref{thm:DL_Buse_summ}\ref{itm:DL_Buse_add}) and Theorem \ref{thm:DL_SIG_cons_intro}\ref{itm:DL_all_SIG},
\begin{align*}
\W_{\alpha +}(y,s;x,s) &= \W_{\alpha +}(y,s;z,t) - \W_{\alpha +}(x,s;z,t) \\
&= \Ll(y,s;z,t) - \Ll(x,s;z,t) 
\\ &= \W_{\beta -}(y,s;z,t) - \W_{\beta -}(x,s;z,t)
= \W_{\beta -}(y,s;x,s).
\end{align*}

 \noindent \ref{itm:DL_coal_pt_equal}$\Rightarrow$\ref{itm:DL_paths}: Let $(z,t)$ be as in the proof of \ref{itm:DL_coal_pt_equal}$\Rightarrow$\ref{itm:DL_buse_eq}. By Theorem \ref{thm:DL_SIG_cons_intro}\ref{itm:DL_LRmost_geod}, the restriction of $g_{(x,s)}^{\alpha +,R}$ and $g_{(x,s)}^{\beta -,R}$ to the domain $[s,t]$ are both rightmost geodesics between $(x,s)$ and $(z,t)$, and therefore they agree on this restricted domain. Similarly, $g_{(y,s)}^{\alpha +,L}$ and $g_{(y,s)}^{\beta -,L}$ agree on the domain $[s,t]$. By the monotonicity of Theorem \ref{thm:g_basic_prop}\ref{itm:DL_mont_dir}, and since $(z,t)$ is the common coalescence point, \eqref{124} holds for $u \in [s,t)$, as desired. 

 \noindent \ref{itm:DL_paths}$\Rightarrow$\ref{itm:DL_coal_pt_equal} is immediate.
\end{proof}

\begin{theorem} \label{thm:Buse_pm_equiv}
On a single event of probability one, for all reals $s,\dir \in \R$, and $x < y$, the following are equivalent. 
\begin{enumerate} [label=\rm(\roman{*}), ref=\rm(\roman{*})]  \itemsep=3pt
    \item \label{itm:DL_pm_Buse_eq} $\W_{\dir -}(y,s;x,s) = \W_{\dir +}(y,s;x,s).$
    \item \label{itm:DL_pm_coal_pt} $\mbf z^{\dir -}(y,s;x,s) = \mbf z^{\dir +}(y,s;x,s)$.
    \item \label{itm:DL_disjoint_paths} $g_{(x,s)}^{\dir-,R}(t) = g_{(y,s)}^{\dir +,L}(t)$ for some $t > s$, i.e., the paths $g_{(x,s)}^{\dir -,R}$ and $g_{(y,s)}^{\dir +,L}$  intersect.
\end{enumerate}
\end{theorem}
\begin{remark}
In Item \ref{itm:DL_disjoint_paths}, if $\dir \in \DLBusedc$, then despite intersecting, the geodesics   $g_{(x,s)}^{\dir -,R}$ and $g_{(y,s)}^{\dir +,L}$ cannot coalesce. This follows from Theorem \ref{thm:DL_good_dir_classification}, which gives  a full classification of the directions in which all semi-infinite geodesics coalesce. 
\end{remark}
\begin{proof}[Proof of Theorem \ref{thm:Buse_pm_equiv}]
\ref{itm:DL_pm_Buse_eq}$\Rightarrow$\ref{itm:DL_pm_coal_pt}: If $\W_{\dir-}(y,s;x,s) = \W_{\dir+}(y,s;x,s)$, then \\ Theorem \ref{thm:DL_Buse_summ}\ref{itm:DL_unif_Buse_stick} implies that for some $\alpha < \dir < \beta$, $\W_{\alpha +}(y,s;x,s) = \W_{\beta -}(y,s;x,s)$. Then, we apply \ref{itm:DL_buse_eq}$\Rightarrow$\ref{itm:DL_paths} of Theorem \ref{thm:DL_eq_Buse_cpt_paths} to conclude that for some $t > s$ and $z \in \R$, 
\[
g_{(x,s)}^{\dir -,R}(u) = g_{(x,s)}^{\dir +,R}(u) < g_{(y,s)}^{\dir -,L}(u) = g_{(y,s)}^{\dir +,L}(u),\qquad\text{for }u \in [s,t),
\]
whereas for $u = t$, all terms above equal some common value $z$. Therefore, $(z,t) = \mbf z^{\dir -}(y,s;x,s) = \mbf z^{\dir +}(y,s;x,s)$.

 \noindent \ref{itm:DL_pm_coal_pt}$\Rightarrow$\ref{itm:DL_pm_Buse_eq}: 
Similarly as in the proof of \ref{itm:DL_coal_pt_equal}$\Rightarrow$\ref{itm:DL_buse_eq} of Theorem \ref{thm:DL_eq_Buse_cpt_paths}, if $(z,t) = \mbf z^{\dir -}(y,s;x,s) = \mbf z^{\dir +}(y,s;x,s)$, then
$
\W_{\dir -}(y,s;x,s) = 
\Ll(y,s;z,t) - \Ll(x,s;z,t) 
= \W_{\dir +}(y,s;x,s).
$

 \noindent \ref{itm:DL_pm_coal_pt}$\Rightarrow$\ref{itm:DL_disjoint_paths}: Assume $(z,t) = \mbf z^{\dir -}(y,s;x,s) = \mbf z^{\dir +}(y,s;x,s)$. Then, $g_{(x,s)}^{\dir -,R}(t) = z = g_{(y,s)}^{\dir +,L}(t)$. 

 \noindent \ref{itm:DL_disjoint_paths}$\Rightarrow$\ref{itm:DL_pm_coal_pt}: 
 Assume that $g_{(x,s)}^{\dir-,R}(t) = g_{(y,s)}^{\dir +,L}(t)$ for some $t > s$. Let $t$ be the minimal such time, and let $(z,t)$ be the point where the geodesics first intersect. By Theorem \ref{thm:g_basic_prop}, Items \ref{itm:DL_mont_dir} and \ref{itm:DL_SIG_mont_x}, for $u > s$,
 \be \label{406}
 g_{(x,s)}^{\dir-,R}(u) \le g_{(x,s)}^{\dir +,R}(u) \wedge g_{(y,s)}^{\dir -,L}(u) \le  g_{(x,s)}^{\dir +,R}(u) \vee g_{(y,s)}^{\dir -,L}(u)  \le  g_{(y,s)}^{\dir +,L}(u).
 \ee
 In particular, when $u = t$, all inequalities in \eqref{406} are equalities. Further, since  $g_{(x,s)}^{\dir -,R}$,$g_{(x,s)}^{\dir +,R}$ are rightmost geodesics between $(x,s)$ and $(z,t)$ (Theorem \ref{thm:DL_SIG_cons_intro}\ref{itm:DL_LRmost_geod}), $g_{(x,s)}^{\dir -,R}(u) = g_{(x,s)}^{\dir +,R}(u)$ for $u \in [s,t]$. Similarly, $g_{(y,s)}^{\dir -,L}(u) = g_{(y,s)}^{\dir +,L}(u)$ for $u \in [s,t]$. Since $t$ was chosen minimally for $g_{(x,s)}^{\dir-,R}(t) = g_{(y,s)}^{\dir +,L}(t)$, we have $(z,t) = \mbf z^{\dir -}(y,s;x,s) = \mbf z^{\dir +}(y,s;x,s)$. 
\end{proof}

\begin{proof}[Proof of Theorem \ref{thm:DL_good_dir_classification} (Classification of directions)]
\ref{itm:DL_good_dir}$\Rightarrow$\ref{itm:DL_LR_all_agree}: If $\dir \notin \DLBusedc$, then $W_{\dir -} = \W_{\dir +}$, so \ref{itm:DL_LR_all_agree} follows by the construction of the Busemann geodesics from the Busemann functions. 

 \noindent \ref{itm:DL_LR_all_agree}$\Rightarrow$\ref{itm:DL_good_dir_coal}: Since a geodesic in direction $\dir$ from $(x,s)$ must pass through each horizontal level $t > s$, it is sufficient to show that, for $s \in \R$ and $x < y$, whenever $g_1$ is a semi-infinite  geodesic from $(x,s)$ in direction $\dir$ and $g_2$ is a semi-infinite geodesic from $(y,s)$ in direction $\dir$, $g_1$ and $g_2$ coalesce. Assuming \ref{itm:DL_LR_all_agree} and using Theorem \ref{thm:all_SIG_thm_intro}\ref{itm:DL_LRmost_SIG}, for all $t > s$,
\[
g_{(x,s)}^{\dir +,L}(t) = g_{(x,s)}^{\dir -,L}(t) \le g_1(t) \wedge g_2(t) \le g_1(t) \vee g_2(t) \le g_{(y,s)}^{\dir +,R}(t).
\]
By Theorem \ref{thm:DL_all_coal}\ref{itm:DL_allsigns_coal}, $g_{(x,s)}^{\dir +,L}$ and $g_{(y,s)}^{\dir +,R}$ coalesce, so all inequalities above are equalities for large $t$, and $g_1$ and $g_2$ coalesce. 

 \noindent \ref{itm:DL_good_dir_coal}$\Rightarrow$\ref{itm:DL_good_dir}: We prove the contrapositive. If $\dir \in \DLBusedc$, then by Theorem \ref{thm:DL_Buse_summ}\ref{itm:DL_Buse_gen_mont}-\ref{itm:Buse_KPZ_description}, \\$\W_{\dir -}(y,s;x,s) < \W_{\dir +}(y,s;x,s)$ for some $x < y$ and $s \in \R$. By \ref{itm:DL_pm_Buse_eq}$\Leftrightarrow$\ref{itm:DL_disjoint_paths} of Theorem \ref{thm:Buse_pm_equiv}, $g_{(x,s)}^{\dir-,R}(t) < g_{(y,s)}^{\dir +,L}(t)$ for all $t > s$. In particular, $g_{(x,s)}^{\dir-,R}$ and $g_{(y,s)}^{\dir +,L}$ do not coalesce. 

 \noindent \ref{itm:DL_LR_all_agree}$\Rightarrow$\ref{itm:DL_good_dir_unique_geod}: By definition of $\NU$, whenever $p\notin \NU$, $g_p^{\dir \sig,L} = g_{p}^{\dir \sig,R}$ for $\dir \in \R$ and $\sigg \in \{-,+\}$. Hence, assuming $p \notin \NU$ and $g_{p}^{\dir -,R} = g_{p}^{\dir +,R}$, we also have $g_{p}^{\dir-,L} = g_{p}^{\dir +,R}$, so there is a unique  geodesic from $p$ in direction $\dir$ by Theorem \ref{thm:all_SIG_thm_intro}\ref{itm:DL_LRmost_SIG}.

 \noindent \ref{itm:DL_good_dir_unique_geod}$\Rightarrow$\ref{itm:DL_good_dir_pt_unique}: By Theorem \ref{thm:DLNU}\ref{itm:DL_NU_p0}, on the event $\Omega_2$, $\NU$ contains no points of $\Q^2$, and therefore, $\NU$ is not all of $\R^2$.

 \noindent \ref{itm:DL_good_dir_pt_unique}$\Rightarrow$\ref{itm:DL_good_dir_L_unique} and \ref{itm:DL_good_dir_pt_unique}$\Rightarrow$\ref{itm:DL_good_dir_R_unique} are direct consequences of Theorem \ref{thm:all_SIG_thm_intro}\ref{itm:DL_LRmost_SIG}: If there is a unique semi-infinite geodesic in direction $\dir$ from a point $p \in \R^2$, then $g_{p}^{\dir -,L} = g_{p}^{\dir +,L} = g_{p}^{\dir -,R} = g_{p}^{\dir +,R}$.

 \noindent \ref{itm:DL_good_dir_L_unique}$\Rightarrow$\ref{itm:DL_LR_all_agree}: Let $p$ be a point from which $g_{p}^{\dir-,L} = g_{p}^{\dir +,L}$, and call this common geodesic $g$. Let $q$ be an arbitrary point in $\R^2$. By Theorem \ref{thm:DL_all_coal}\ref{itm:DL_allsigns_coal}, $g_{q}^{\dir -,L}, g_{q}^{\dir +,L}, g_{q}^{\dir - ,R}$, and $g_{q}^{\dir +,R}$ each coalesce with $g$, so $g_{q}^{\dir -,L}$ and $g_{q}^{\dir +,L}$ coalesce. Since both geodesics are the leftmost geodesics between their points by Theorem \ref{thm:DL_SIG_cons_intro}\ref{itm:DL_LRmost_geod}, they must be the same. Similarly, $g_{q}^{\dir-,R} = g_{q}^{\dir +,R}$. 

 \noindent \ref{itm:DL_good_dir_R_unique}$\Rightarrow$\ref{itm:DL_LR_all_agree}: follows by the same proof.

 \noindent \textbf{Item \ref{itm:DL_allBuse}:} Let $\dir \in \R \setminus \DLBusedc$, and let $g$ be a  semi-infinite geodesic in direction $\dir$, starting from a point $(x,s) \in \R^2$. By Lemma \ref{lem:L_and_Buse_ineq} and Theorem \ref{thm:DL_SIG_cons_intro}\ref{itm:arb_geod_cons}, it is sufficient to show that for sufficiently large $t$, 
\be \label{647}
\Ll(x,s;g(t),t) = \W_{\dir}(x,s;g(t),t).
\ee
(we dropped the $\pm$ distinction since $\W_{\dir -} = \W_{\dir +}$). By Item \ref{itm:DL_good_dir_coal}, $g$ coalesces with $g_{(x,s)}^{\dir,R}$. Then, for sufficiently large $t$, $g(t) = g_{(x,s)}^{\dir,R}(t)$ and by Theorem \ref{thm:DL_SIG_cons_intro}\ref{itm:DL_all_SIG}, \eqref{647} holds. 
\end{proof}

\newpage
\section{Remaining proofs from Section \ref{sec:Buse_geod_results} and the proofs of the main theorems} \label{sec:Busegeod_finalproofs}

In this section, we complete the unfinished business of the chapter. In Section \ref{sec:last_proofs1}, we start by proving Items \ref{itm:BuseLim1}--\ref{itm:global_attract} of Theorem \ref{thm:DL_Buse_summ} and mixing in Theorem \ref{thm:Buse_dist_intro}\ref{itm:stationarity}. Then, we prove the first main theorem of the chapter, namely Theorem \ref{thm:DLSIG_main}. In Section \ref{sec:last_proofs2}, we prove Theorems \ref{thm:DLBusedc_description}\ref{itm:Busedc_t} and \ref{thm:Split_pts}, along with some necessary lemmas.

\subsection{Proof of Theorem \ref{thm:DLSIG_main} and related results} \label{sec:last_proofs1}

\begin{proof}[Proof of Items \ref{itm:BuseLim1}--\ref{itm:global_attract} of Theorem \ref{thm:DL_Buse_summ} and mixing in Theorem \ref{thm:Buse_dist_intro}\ref{itm:stationarity}]
We \\continue to work on  $\Omega_2$.

\noindent \textbf{Item \ref{itm:BuseLim1} of Theorem \ref{thm:DL_Buse_summ} (Busemann limits I):}  By Theorem \ref{thm:DL_good_dir_classification}\ref{itm:DL_allBuse}, if $\dir \notin \DLBusedc$, all $\dir$-directed semi-infinite geodesics are Busemann geodesics, and they all coalesce.  By Theorem \ref{thm:DL_all_coal}\ref{itm:unif_coal}, there exists a level $T$ such that all geodesics from points starting in the compact set $K$ have coalesced by time $T$. Let $(Z,T)$ denote the location of the point of the common geodesics at time $T$. Let $r_t = (z_t,u_t)_{t \in \R_{\ge 0}}$ be any net with $u_t \to \infty$ and $z_t/u_t \to \dir$. By Theorem \ref{thm:all_SIG_thm_intro}\ref{itm:finite_geod_stick}, for all sufficiently large $t$ and $p \in K$, all geodesics from $p$ to $r_t$ pass through $(Z,T)$. Then, for $p,q \in K$, 
\[
\Ll(p;r_t) - \Ll(q;r_t) = \Ll(p;Z,T) + \Ll(Z,T;r_t) - (\Ll(q;Z,T) + \Ll(Z,T;r_t)).
\]
By Theorems \ref{thm:DL_SIG_cons_intro}\ref{itm:arb_geod_cons}\ref{itm:weight_of_geod} and \ref{thm:DL_Buse_summ}\ref{itm:DL_Buse_add}, the right-hand side is equal to 
\[
\W_\dir(p;Z,T) - \W_\dir(q;Z,T) = \W_\dir(p;q). 
\]

\noindent
\textbf{Item \ref{itm:BuseLim2} of Theorem \ref{thm:DL_Buse_summ} (Busemann limits II):} By Theorem \ref{thm:DLBusedc_description}\ref{itm:DL_dc_set_count}, on the event $\Omega_2$, $\DLBusedc$ contains no rational directions. Then, for arbitrary $\dir \in \R$, $s \in \R$, $x < y \in \R$, $\alpha,\beta \in \Q$ with $\alpha < \dir < \beta$, and a net $(z_r,u_r)$ with $u_r \to \infty$ and $z_r/u_r \to \dir$, for sufficiently large $r$, $\alpha u_r < z_r < \beta u_r$. Theorem \ref{thm:DL_Buse_summ}\ref{itm:BuseLim1} gives the existence of the limits in the first and last lines below, while the monotonicity of Lemma \ref{lem:DL_crossing_facts}\ref{itm:DL_crossing_lemma} justifies the first and last inequalities:
\begin{align*}
\W_\alpha(y,s;x,s) &= \lim_{r \to \infty} \Ll(y,s;\alpha u_r,u_r) - \Ll(x,s;\alpha u_r,u_r) \\
&\le \liminf_{r \to \infty} \Ll(y,s;z_r,u_r) - \Ll(x,s;z_r,u_r) \\ 
&\le \limsup_{r \to \infty} \Ll(y,s;z_r,u_r) - \Ll(x,s;z_r,u_r)\\
&\le \lim_{r \to \infty} \Ll(y,s;\beta u_r,u_r) -\Ll(x,s;\beta u_r,u_r)= \W_{\beta}(y,s;x,s).
\end{align*}
Sending $\Q \ni \alpha \nearrow \dir$ and $\Q \ni \beta \searrow \dir$ and using Item \ref{itm:DL_unif_Buse_stick} completes the proof.

\noindent \textbf{Item \ref{itm:global_attract} of Theorem \ref{thm:DL_Buse_summ} (Global attractiveness):}   We follow a similar proof to the attractiveness in Theorem \ref{thm:invariance_of_SH}. Let $\dir \notin \DLBusedc$ and assume $\h \in \UC$ is a function satisfying the drift condition \eqref{eqn:drift_assumptions}. Recall that we define
\be \label{hst}
h_{s,t}(x) = \sup_{z \in \R}\{\Ll(x,s;z,t) + \h(z)\}.
\ee
For $a > 0$ and $s < t$, Theorems \ref{thm:DL_Buse_summ}\ref{itm:DL_unif_Buse_stick} and \ref{thm:DLBusedc_description}\ref{itm:DL_dc_set_count} allows us to choose $\ve = \ve(\dir) > 0$ small enough  so that $\dir \pm 2\ve \in \Q $ (and thus $\dir \pm 2\ve \notin \DLBusedc$),  and so for all $x \in [-a,a]$,
\be \label{pmeq}
\W_{\dir \pm 2\ve}(x,s;0,s) = \W_\dir(x,s;0,s).
\ee
By Theorem \ref{thm:DL_all_coal}\ref{itm:unif_coal}, there exists a random $T = T(a,\dir \pm \ve)$ such that all $\dir - 2\ve$ Busemann geodesics have coalesced by time $T$ and all $\dir + 2\ve$ Busemann geodesics have coalesced by time $T$. For $t > T$, let $g^{\dir \pm 2\ve}(t)$ be locations of these two common geodesics at time $t$. By Theorem \ref{thm:DL_SIG_cons_intro}\ref{itm:arb_geod_cons}\ref{itm:geo_dir}, $g^{\dir \pm 2\ve}(t)/t \to \dir \pm 2\ve$. By Equation \eqref{downexit} in Lemma \ref{lem:unq} applied to the temporally reflected version of $\Ll$, there exists $t_0(a,\ve(\dir),s)$ so that for $t > t_0$, whenever $x \in [-a,a]$ and $z$ is a maximizer in \eqref{hst},
$
 g^{\dir - 2\ve}(t) < z < g^{\dir + 2\ve}(t).
$
Then, by Lemma \ref{lem:DL_crossing_facts}\ref{itm:KPZ_crossing_lemma}, for such large $t$,
\[
\W_{\dir - 2\ve}(x,s;0,s) \le  h_{s,t}(x) - h_{s,t}(0) \le \W_{\dir + 2\ve}(x,s;0,s),
\]
while for $-a \le x \le 0$, the equalities reverse. Combined with \eqref{pmeq}, this completes the proof. 

\noindent \textbf{Item \ref{itm:stationarity} of Theorem \ref{thm:Buse_dist_intro} (Mixing):} 
This proof follows a similar idea as that in Lemma 7.5 of~\cite{Bakhtin-Cator-Konstantin-2014}, and the key is that, within a compact set, the Busemann functions are equal to differences of the directed landscape for large enough $t$. Then, we use Lemma~\ref{lm:horiz_shift_mix}, which states that, as a projection of $\{\Ll,\W\}$, the directed landscape $\Ll$ is mixing under the transformation $T_{z;a,b}$. Set $r_z = (az,bz)$.  By a standard $\pi-\lambda$ argument, it suffices to show that for $\dir_1,\ldots\dir_k \in \R$ (ignoring the sign $\sigg$ since $\dir_i \notin \DLBusedc$ a.s.), all compact sets $K := K_1 \times K_2^k \subseteq \Rup \times (\R^4)^k$, and all Borel sets $A,B \in C(K,\R)$,
\begin{align*}
&\lim_{z \to \infty} \Pp\Bigl(\{\Ll, \W_{\dir_{1:k}}\}\big|_K \in A, \{T_{z;a,b} \Ll, T_{z;a,b}\W_{\dir_{1:k}}\}\big|_K \in B\Bigr) \\
&\qquad\qquad= \Pp\bigl( \{\Ll, \W_{\dir_{1:k}}\}\big|_K \in A\bigr) \Pp\bigl(\{\Ll, \W_{\dir_{1:k}}\}\big|_K \in B \bigr),
\end{align*}
where we use the shorthand notation 
\[
\{\Ll, \W_{\dir_{1:k}}\}\big|_K := \{\Ll(v), \W_{\dir_i}(p;q):1 \le i \le k,(v,p,q) \in K\},
\]
and $T_{z;a,b}$ acts on $\Ll$ and $\W$ as projections of $\{\Ll,\W\}$.
By Theorem \ref{thm:DL_Buse_summ}\ref{itm:BuseLim1}, we may choose $t > 0$ sufficiently large so that 
\be \label{busc1}
\Pp(\W_{\dir_i}(p;q) = \Ll(p;(t\dir,t)) - \Ll(q;(t\dir,t)) \;\forall (p,q) \in K_2, 1 \le i \le k) \ge 1 - \ve. 
\ee
By stationarity of the process under space-time shifts, we also have that for such large $t$ and all $z \in \R$, 
\be \label{busc2}
\Pp(T_{z;a,b} \W_{\dir_i}(p;q) = T_{z;a,b} [\Ll(p ;(t\dir,t)) - \Ll(q;(t\dir,t))] \;\forall (p,q) \in K_2, 1 \le i \le k) \ge 1 - \ve
\ee
Let $C_{z,t}$ be the intersection of the events in \eqref{busc1} over $1 \le i \le k$ with the event \eqref{busc2}.  Then for large enough $t$, $\Pp(C_{z,t}) \ge 1 - 2\ve$, and
\begin{align*}
    &\Bigl|\Pp\Bigl(\{\Ll, \W_{\dir_{1:k}}\}|_K \in A, \{T_{z;a,b} \Ll, T_{z;a,b}\W_{\dir_{1:k}}\}|_K \in B\Bigr)  \\
    &\qquad- \Pp\Bigl( \{\Ll, \W_{\dir_{1:k}}\}|_K \in A\Bigr) \Pp\Bigl(\{\Ll, \W_{\dir_{1:k}}\}|_K \in B \Bigr) \Bigr|  \\
    &\le \Bigl|\Pp\Bigl(\{\Ll, \W_{\dir_{1:k}}\}|_K \in A, \{T_{z;a,b} \Ll, T_{z;a,b}\W_{\dir_{1:k}}\}|_K \in B, C_{z,t}\Bigr) \\
    &\qquad -\Pp\Bigl( \{\Ll, \W_{\dir_{1:k}}\}|_K \in A,C_{z,t}\Bigr) \Pp\Bigl(\{\Ll, \W_{\dir_{1:k}}\}|_K \in B,C_{z,t} \Bigr) \Bigr| + C\ve \\
    &= \Bigl|\Pp\Bigl(\{\Ll(v), \Ll(p;(t\dir_{1:k},t)) - \Ll(q;(t\dir_{1:k},t))\}|_K \in A, \\
    &\qquad\qquad \{T_{z;a,b} \Ll(v), T_{z;a,b}[\Ll(p;(t\dir_{1:k},t)) - \Ll(q;(t\dir_{1:k},t))]\}|_K \in B, C_{z,t}\Bigr) \\
    &\qquad -\Pp\Bigl( \{\Ll(v), \Ll(p;(t\dir_{1:k},t)) - \Ll(q;(t\dir_{1:k},t))\}|_K \in A,C_{z,t}\Bigr) \\ &\qquad\qquad\times \Pp\Bigl(\{\Ll(v), \Ll(p;(t\dir_{1:k},t)) - \Ll(q;(t\dir_{1:k},t))\}|_K \in B,C_{z,t} \Bigr) \Bigr| + C\ve  \\
    &\le \Bigl|\Pp\Bigl(\{\Ll(v), \Ll(p;(t\dir_{1:k},t)) - \Ll(q;(t\dir_{1:k},t))\}|_K \in A, \\
    &\qquad\qquad \{T_{z;a,b} \Ll(v), T_{z;a,b}[\Ll(p;(t\dir_{1:k},t)) - \Ll(q;(t\dir_{1:k},t))]\}|_K \in B\Bigr) \\
    &\qquad -\Pp\Bigl( \{\Ll(v), \Ll(p;(t\dir_{1:k},t)) - \Ll(q;(t\dir_{1:k},t))\}|_K \in A\Bigr) \\ &\qquad\qquad\times \Pp\Bigl(\{\Ll(v), \Ll(p;(t\dir_{1:k},t)) - \Ll(q;(t\dir_{1:k},t))\}|_K \in B \Bigr) \Bigr| + C'\ve,
\end{align*}
where the constants $C,C'$ came as the cost of adding and removing the high probability event $C_{z,t}$. The proof is complete by sending $z \to \infty$ and using the mixing of $\Ll$ under the shift $T_{z;a,b}$ (Lemma \ref{lm:horiz_shift_mix}).
\end{proof}

\begin{proof}[Proof of Theorem \ref{thm:DLSIG_main}]
\textbf{Item \ref{itm:all_dir} (All geodesics have a direction):} First, we show that, on $\Omega_2$, if $g$ is a semi-infinite geodesic starting from $(x,s)$, then
\be \label{594}
-\infty < \liminf_{t \to \infty} t^{-1}{g(t)} \le \limsup_{t \to \infty} t^{-1}{g(t)} < \infty.
\ee
We show the rightmost inequality, the leftmost being analogous. Assume, by way of contradiction, that $\limsup_{t \to \infty} g(t)/t = \infty$. By the directedness of Theorem \ref{thm:DL_SIG_cons_intro}\ref{itm:DL_all_SIG}, $\forall \dir \in \R$ there exists an infinite sequence $t_i \to \infty$ such that  $g(t_i) > g_{(x,s)}^{\dir +,L}(t_i)$ for all $i$. Since $g_{(x,s)}^{\dir +,L}$ is the leftmost geodesic between any two of its points (Theorem \ref{thm:DL_SIG_cons_intro}\ref{itm:DL_LRmost_geod}), we must have $g(t) \ge g_{(x,s)}^{\dir+,L}(t)$ $\forall \dir \in \R$ and  $t \in \R$. By Theorem \ref{thm:g_basic_prop}\ref{itm:limits_to_inf}, $g(t) = \infty$ $\forall t > s$, a contradiction. 

Having established \eqref{594},   assume by way of contradiction that
\[
\liminf_{t \to \infty} t^{-1} {g(t)} < \limsup_{t \to \infty} t^{-1}{g(t)}.
\]
Choose some $\dir$ strictly between the two values above. By the directedness of Theorem \ref{thm:DL_SIG_cons_intro}\ref{itm:DL_all_SIG}, there exists a sequence $t_i \to \infty$ such that $g_{(x,s)}^{\dir +,R}(t_i) < g(t_i)$ for $i$ even and $g_{(x,s)}^{\dir +,R}(t_i) > g(t_i)$ for $i$ odd. This cannot occur since  $g_{(x,s)}^{\dir +,R}$ is the rightmost geodesic between any two of its points.

By  Theorem \ref{thm:DL_SIG_cons_intro}\ref{itm:DL_all_SIG}, for each $\dir \in \R$ and $(x,s) \in \R^2$, $g_{(x,s)}^{\dir +,R}$, for example, is a semi-infinite geodesic from $(x,s)$ in direction $\dir$, justifying the claim that there is at least one semi-infinite geodesic from each point and in every direction. 

\noindent \textbf{Item \ref{itm:good_dir_coal} (Coalescence):} The first statement  follows from the equivalences \\\ref{itm:DL_good_dir}$\Leftrightarrow$\ref{itm:DL_good_dir_coal}$\Leftrightarrow$\ref{itm:DL_good_dir_unique_geod} of Theorem \ref{thm:DL_good_dir_classification}. 
By Theorem \ref{thm:DLNU}\ref{itm:DL_NU_p0}, $\Pp(p \in \NU)=0$ $\forall p \in \R^2$. This and Tonelli's theorem imply that the expected Lebesgue measure of the set $\NU$ is
\[
\Ee \int_{(x,s) \in \R^2} \ind(p \in \NU)\,dx\,ds = \int_{(x,s) \in \R^2} \Pp(p \in \NU)\,dx\,ds = 0,
\]
so $\NU$ almost surely has planar Lebesgue measure zero.

\noindent \textbf{Item \ref{itm:bad_dir_split} (Non-uniqueness in exceptional directions):} This follows from Remark \ref{rmk:split_from_all_p}. 
\end{proof}

\subsection{Proof of Theorem \ref{thm:Split_pts} and related results.} \label{sec:last_proofs2}

Recall the  functions $f_{s,\dir}(x) = \W_{\dir +}(x,s;0,s) - \W_{\dir -}(x,s;0,s)$ defined in \eqref{fsdir} and the sets $\Split_{s,\dir}$ from \eqref{Split_sdir}:
\be\label{eqn:gen_split_set4}\begin{aligned}
    \Split_{s,\dir} &:= \{x \in \R:  \exists \text{ 
    \textbf{disjoint}}   \text{   }\text{semi-infinite  geodesics from  }(x,s) \text{ in direction }\dir\}\\
    \Split &:= \bigcup_{s \tspb\in\tspb \R, \, \dir\tspb \in\tspb \DLBusedc} \Split_{s,\dir} \times \{s\}. 
    \end{aligned}\ee

\begin{theorem} \label{thm:random_supp}
On a single event of full probability, the function $f_{s,\dir}$ is nondecreasing simultaneously for all $s \in \R$ and $\dir \in \DLBusedc$.  Denote the set of local variation of $f_{s,\dir}$  by  
\be \label{Dsdir}  \D_{s,\dir} =  \{ x\in\R:  f_{s,\dir}(x - \ve) < f_{s,\dir}(x + \ve)\; \forall \ve > 0 \}.
\ee
Then, on a single event of full probability, simultaneously for each $s \in \R$ and $\dir \in \DLBusedc$,
\be \label{eqn:supp_set}
\D_{s,\dir} = \Split_{s,\dir}^L \cup \Split_{s,\dir}^R \;\subseteq  \;\Split_{s,\dir},
\ee
where for $S \in \{L,R\}$,
\be \label{eqn:split_LR_sdir}
\Split_{s,\dir}^S := \{x \in \R: g_{(x,s)}^{\dir -,S} \text{ and } g_{(x,s)}^{\dir +,S} \text{ are disjoint}\}.
\ee
 Furthermore, $(\Split_{s,\dir} \setminus \D_{s,\dir}) \times \{s\}$ is contained in the countable set $\NU^{\dir -} \cap \tspb\NU^{\dir +} \cap \,\Hh_s$. 
\end{theorem}
\begin{remark} \label{rmk:NUsupp}
Presently, we do not know if $\D_{s,\dir}$ equals $\Split_{s,\dir}$.  Since  $\NU^{\dir -} \cap \NU^{\dir +} \subseteq \NU$ and  $\NU \cap \,\Hh_s$ is at most countable (Theorem \ref{thm:DLNU}\ref{itm:DL_NU_count}), $\Split_{s,\dir}$ and $\D_{s,\dir}$ have the same Hausdorff dimension for all $s \in \R$ and $\dir \in \DLBusedc$. 
\end{remark}

\begin{remark} \label{rmk:splitsetseq}
In contrast with  $\Split$ in \eqref{eqn:gen_split_set4},   the sets $\Split^S$ are concerned only with leftmost ($S=L$) and rightmost ($S=R$) Busemann geodesics.
In  BLPP, the analogues of $\Split^L$ and $\Split^R$ are both equal to the set of initial points from which some geodesic travels initially vertically (Theorems 2.10 and 4.30 in \cite{Seppalainen-Sorensen-21b}). Furthermore, in BLPP,  the analogue of this set contains $\NU$. We do not presently know whether either is true in DL. 
\end{remark}

\begin{proof}[Proof of Theorem \ref{thm:random_supp}]
The full-probability event is $\Omega_2$ in \eqref{omega2}.  The monotonicity of the function $f_{s,\dir}$ follows from \eqref{801}.  We now prove that $\D_{s,\dir} = \Split_{s,\dir}^L \cup \Split_{s,\dir}^R$. Assume that $y \notin \D_{s,\dir}$. Then, there exist $a < y < b$ such that $f_{s,\dir}$ is constant on $[a,b]$. Hence, for $a \le x < y$,
\[
\W_{\dir +}(x,s;0,s) - \W_{\dir -}(x,s;0,s) = \W_{\dir +}(y,s;0,s) - \W_{\dir -}(y,s;0,s),
\]
and by additivity (Theorem \ref{thm:DL_Buse_summ}\ref{itm:DL_Buse_add}), $\W_{\dir-}(y,s;x,s) = \W_{\dir +}(y,s;x,s)$. Choose $t > s$ sufficiently small so that $g_{(x,s)}^{\dir +,R}(t) < g_{(y,s)}^{\dir -,L}(t)$. By Lemma \ref{lem:Buse_equality_coal}, $g_{(y,s)}^{\dir -,L}(u) = g_{(y,s)}^{\dir +,L}(u)$ for $u \in [s,t]$.  By a symmetric argument, instead choosing a point $x > y$, $g_{(y,s)}^{\dir -,R}$ and $g_{(y,s)}^{\dir +,R}$ agree near the starting point $(y,s)$. Hence, $y \notin \Split_{s,\dir}^L \cup \Split_{s,\dir}^R$.

Next, assume that $y \in \D_{s,\dir}$. Then, for all $x < y < z$,
\[
\W_{\dir +}(x,s;0,s) - \W_{\dir -}(x,s;0,s) < \W_{\dir +}(z,s;0,s) - \W_{\dir -}(z,s;0,s),
\]
Since the difference is a monotone function,
either (i)  $\W_{\dir -}(y,s;x,s) < \W_{\dir +}(y,s;x,s)$    for all $x < y$  or (ii) $\W_{\dir -}(z,s;y,s) < \W_{\dir +}(z,s;y,s)$  for all $z > y.$

We show that $g_{(y,s)}^{\dir -,L}$ and $g_{(y,s)}^{\dir +,L}$ are disjoint in the first case.     A symmetric proof shows that  $g_{(y,s)}^{\dir -,R}$ and $g_{(y,s)}^{\dir +,R}$ are disjoint in the second case. So 
assume $\W_{\dir -}(y,s;x,s) < \W_{\dir +}(y,s;x,s)$ for all $x < y$.  Sending $x \nearrow y$, $g_{(x,s)}^{\dir -,R}$ converges to $g_{(y,s)}^{\dir -,L}$ by Theorem \ref{thm:g_basic_prop}\ref{itm:DL_SIG_conv_x}.   Assume, by way of contradiction, that $g_{(y,s)}^{\dir -,L}(u) = g_{(y,s)}^{\dir +,L}(u)$ for some $u > s$. This implies then $g_{(y,s)}^{\dir -,L}(t) = g_{(y,s)}^{\dir +,L}(t)$ for all $t \in [s,u]$ since both paths are the leftmost geodesic between any two of their points (Theorem \ref{thm:DL_SIG_cons_intro}\ref{itm:DL_LRmost_geod}). For $t \ge s$, the convergence $g_{(x,s)}^{\dir -,R}(t) \to g_{(y,s)}^{\dir -,L}(t)$ is monotone by Theorem \ref{thm:g_basic_prop}\ref{itm:DL_SIG_mont_x}. Since geodesics are continuous paths, Dini's theorem implies that, as $x \nearrow y$, $g_{(x,s)}^{\dir -,R}(t)$ converges to $g_{(y,s)}^{\dir -,L}(t) = g_{(y,s)}^{\dir + ,L}(t)$ uniformly in $t \in [s,u]$. Lemma \ref{lm:BGH_disj} implies that, for sufficiently close $x < y$, $g_{(x,s)}^{\dir -,R}$ and $g_{(y,s)}^{\dir +,L}$ are not disjoint. This contradicts \ref{itm:DL_pm_Buse_eq}$\Leftrightarrow$\ref{itm:DL_paths} of Theorem \ref{thm:Buse_pm_equiv} since we assumed $\W_{\dir -}(y,s;x,s) < \W_{\dir +}(y,s;x,s)$ for all $x < y$.

Lastly, we show that $(\Split_{s,\dir} \setminus \D_{s,\dir}) \times \{s\} \subseteq \NU^{\dir - } \cap \NU^{\dir +} \cap\, \Hh_s$. Let $x \in \Split_{s,\dir} \setminus \D_{s,\dir}$. By Theorem \ref{thm:all_SIG_thm_intro}\ref{itm:DL_LRmost_SIG}, $g_{(x,s)}^{\dir -,L}$ is the leftmost $\dir$-directed geodesic from $(x,s)$, and $g_{(x,s)}^{\dir +,R}$ is the rightmost. Since $x \in \Split_{s,\dir}$, these two geodesics must be disjoint. Since $x \notin \D_{s,\dir}$, $g_{(x,s)}^{\dir -,L}$ and $g_{(x,s)}^{\dir +,L}$ are not disjoint, and $g_{(x,s)}^{\dir-,R}$ and $g_{(x,s)}^{\dir +,R}$ are not disjoint. Since the leftmost/rightmost semi-infinite geodesics are leftmost/rightmost geodesics between their points (Theorem \ref{thm:DL_SIG_cons_intro}\ref{itm:DL_LRmost_geod}), there exists $\ve > 0$ such that for $t \in (s,s + \ve)$,
\[
g_{(x,s)}^{\dir -,L}(t) = g_{(x,s)}^{\dir +,L}(t) < g_{(x,s)}^{\dir -,R}(t) = g_{(x,s)}^{\dir +,R}(t),
\]
and therefore, $(x,s) \in \NU^{\dir -} \cap \NU^{\dir +} \cap\, \Hh_s$. Recall this set is at most countable by Theorem \ref{thm:DLNU}\ref{itm:DL_NU_count}.
\end{proof}

\begin{lemma} \label{lem:rm_geod}
Given $\omega \in \Omega_2$ and  $(x,s;y,u) \in \Rup$, let $g:[s,u] \to \R$ be the leftmost {\rm(}resp.\ rightmost{\rm)} geodesic between $(x,s)$ and $(y,u)$. Then, $(g(t),t) \in \Split^L$ ${\rm(} \text{resp.\ } \Split^R {\rm)}$ for some $t \in [s,u)$. Furthermore, among the directions $\dir$ for which $g_{(x,s)}^{\dir-,L}$ and $g_{(x,s)}^{\dir+,L}$ separate at some $t \in [s,u)$, there is a unique direction $\wh \dir$ such that 
\[
g_{(x,s)}^{\wh \dir-,L}(u) \le y < g_{(x,s)}^{\wh \dir+,L}(u).
\]
The same  holds with $L$ replaced by $R$ and the strict and weak inequalities swapped.
\end{lemma}
\begin{proof}
We prove the statement for leftmost geodesics.  The proof  for rightmost geodesics is analogous. Set
\be \label{468}
\wh \dir := \sup \{\dir \in \R: g_{(x,s)}^{\dir \sig,L}(u) \le y\} = \inf \{\dir \in \R: g_{(x,s)}^{\dir \sig,L}(u) > y\}.
\ee
The monotonicity of Theorem \ref{thm:g_basic_prop}\ref{itm:DL_mont_dir} guarantees that the second equality holds, and that the definition is independent of the choice of $\sigg \in \{-,+\}$. Theorem \ref{thm:g_basic_prop}\ref{itm:limits_to_inf} guarantees that $\wh \dir \in \R$. By definition of $\wh \dir$ and the monotonicity of Theorem \ref{thm:g_basic_prop}\ref{itm:DL_mont_dir}, $g_{(x,s)}^{\alpha \sig,L}(u) \le y = g(u) < g_{(x,s)}^{\beta \sig,L}(u)$  whenever $\alpha < \wh \dir < \beta$ and $\sigg \in \{-,+\}$.  
But by Theorem \ref{thm:g_basic_prop}\ref{itm:DL_SIG_unif}, the $\beta\sigg$ and $\wh\xi+$ geodesics agree locally when  $\beta$ is close enough to $\wh\xi$. We can conclude  that 
\be \label{423}
g_{(x,s)}^{\wh \dir -,L}(u) \le y = g(u) < g_{(x,s)}^{\wh \dir +,L}(u).
\ee
Since all three are leftmost geodesics (recall Theorem \ref{thm:DL_SIG_cons_intro}\ref{itm:DL_LRmost_geod}), 
\be \label{523}
g_{(x,s)}^{\wh \dir -,L}(t)\le g(t) \le g_{(x,s)}^{\wh \dir +,L}(t)\qquad\text{for } t \in [s,u].
\ee
By \eqref{423} the paths $g_{(x,s)}^{\wh \dir -,L}$ and $g_{(x,s)}^{\wh \dir +,L}$ must separate at some time $t \in [s,u)$. Furthermore, once $g_{(x,s)}^{\wh \dir -,L}$ splits from $g_{(x,s)}^{\wh \dir +,L}$ at a point $(z_1,t_1)$, the geodesics must stay apart. Otherwise, they would meet again at a point $(z_2,t_2)$, and Theorem \ref{thm:DL_SIG_cons_intro}\ref{itm:DL_LRmost_geod} implies that both paths are the leftmost geodesic between $(z_1,t_1)$ and $(z_2,t_2)$. See Figure \ref{fig:splitting}.  Set $\hat t = \inf\{t > s: g_{(x,s)}^{\wh \dir-,L}(t) < g_{(x,s)}^{\wh \dir +,L}(t)\}$. Then, $g_{(x,s)}^{\wh \dir-,L}(t) < g_{(x,s)}^{\wh \dir +,L}(t)$ for all $t > \hat t$. By \eqref{523} and continuity of geodesics, $g_{(x,s)}^{\wh \dir-,L}(t) = g(t) = g_{(x,s)}^{\wh \dir +,L}(t)$ for $t \in [s,\hat t\tspa]$, and so $(g(\hat t),\hat t\tspa) \in \Split^L$. 
\end{proof}

\begin{figure}[t]
    \centering    \includegraphics[height = 1.5 in]{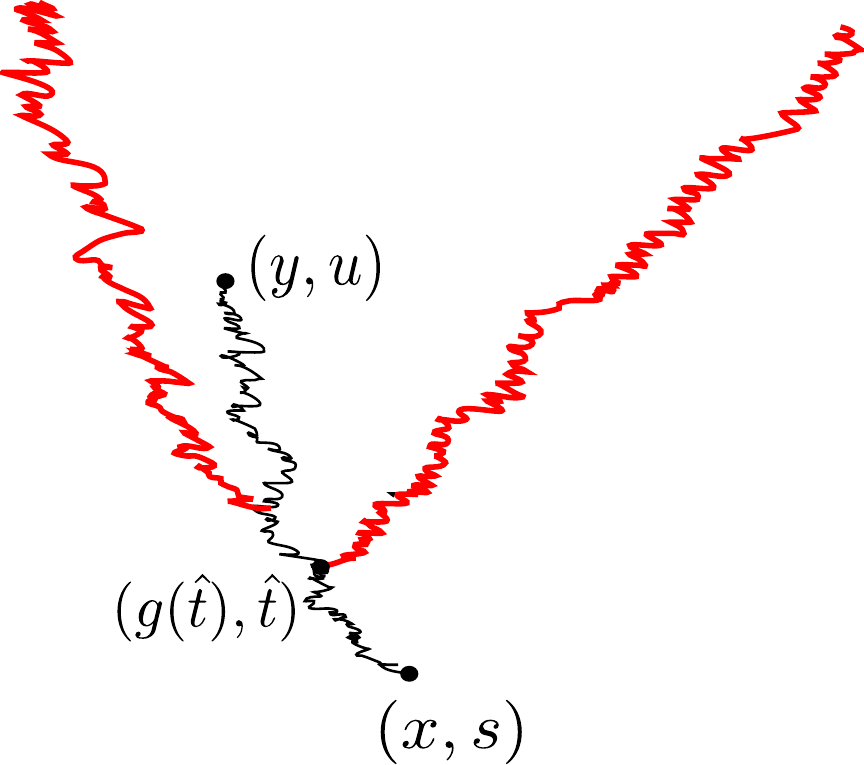}
    \caption{\small The black/thin path is the path $g$. The red/thick paths are the semi-infinite geodesics $g_{(x,s)}^{\wh \dir-,L}$ and $g_{(x,s)}^{\wh \dir+,L}$ after they split from $g$. Once the red paths split, they cannot return, or else there would be two leftmost geodesics from $(g(\hat t),\hat t)$ to the point where they come back together.}
    \label{fig:splitting}
\end{figure}

It remains to prove Theorems \ref{thm:DLBusedc_description}\ref{itm:Busedc_t} and \ref{thm:Split_pts}. Recall the definition of the function from \eqref{fsdir}: 
$
f_{s,\dir}(x) = \W_{\dir+}(x,s;0,s) - \W_{\dir -}(x,s;0,s). 
$
\be \label{omega3} \begin{aligned} 
&\text{Let $\Omega_3$ be the subset of $\Omega_2$ on which the following holds: for each $T \in \Z$,}\\ 
&\text{whenever  $\dir \in \R$ is such that $f_{T,\,\dir} \neq 0$, then}   
\lim_{x \to \pm \infty} f_{T,\,\dir}(x) = \pm \infty.
\end{aligned} \ee
 By Theorem \ref{thm:Buse_dist_intro}\ref{itm:SH_Buse_process} and Corollary \ref{cor:dcLR}, $\Pp(\Omega_3) = 1$. 
\begin{proof}[Proof of Theorem \ref{thm:DLBusedc_description}\ref{itm:Busedc_t}] We work on the full-probability event  $\Omega_3$.   The statement \eqref{bad_ub} to be proved  is  $\dir \in \DLBusedc \iff \forall s\in\R:\lim_{x \to \pm \infty} f_{s,\dir}(x) = \pm \infty$.
 If, for \textit{any} $s$, $f_{s,\dir} \to \pm \infty$ as $x \to \pm \infty$, then $\W_{\dir-}(x,s;0,s) \neq \W_{\dir +}(x,s;0,s)$ for $|x|$ sufficiently large, and $\dir \in \DLBusedc$. It remains to prove the converse statement.
From \eqref{881}, 
\[
\DLBusedc = \bigcup_{T \in \Z} \{\dir \in \R: \W_{\dir - }(x,T;0,T) \neq \W_{\dir +}(x,T;0,T) \text{ for some }x \in \R\}.
\]
To finish the proof of \eqref{bad_ub}, by definition of $\Omega_3$, it suffices to show these two statements:
\begin{enumerate}[label=\rm(\roman{*}), ref=\rm(\roman{*})]  \itemsep=3pt
    \item If $f_{s,\dir} \neq 0$ for some $s,\dir \in \R$ then $f_{T,\,\dir} \neq 0$ for all $T > s$. 
    \item For $T \in \Z, \dir \in \R$, if $f_{T,\,\dir} \neq 0$, then for all $s < T$, $\lim_{x \to \pm \infty} f_{s,\dir}(x) = \pm \infty$.
\end{enumerate}

Part  (i) follows from the equality below.   By \eqref{880}, for $s < T$,
\be \label{883}\begin{aligned} 
\W_{\dir \sig}(x,s;0,s) &= \sup_{z \in \R}\{\Ll(x,s;z,T) + \W_{\dir \sig}(z,T;0,T)\}\\[-3pt]
&\qquad\qquad 
- \sup_{z \in \R} \{\Ll(0,s;z,T) + \W_{\dir \sig}(z,T;0,T)\}. 
\end{aligned} \ee

To prove (ii), we show the limits as $x \to +\infty$, and the limits as $x \to -\infty$ follow analogously. Let $T \in \Z,\dir \in \R$ be such that $f_{T,\,\dir} \neq 0$, and let $R > 0$. By definition of the event $\Omega_3$, we may choose $Z > 0$ sufficiently large so that $\inf_{z \ge Z}\{f_{T,\,\dir}(z)\} \ge R$. Then, by Equation \eqref{373} of Theorem \ref{thm:g_basic_prop}\ref{itm:DL_SIG_conv_x}, for all sufficiently large $x$ and $\sigg \in \{-,+\}$,
\[
\sup_{z \in \R}\{\Ll(x,s;z,T) + \W_{\dir \sig}(z,T;0,T)\} = \sup_{z \ge Z}\{\Ll(x,s;z,T) + \W_{\dir \sig}(z,T;0,T)\}.
\]
Let 
\[
A := \sup_{z \in \R}\{\Ll(x,s;z,T) + \W_{\dir +}(z,T;0,T)\}- \sup_{z \in \R}\{\Ll(x,s;z,T) + \W_{\dir -}(z,T;0,T)\},
\]
and note that this does not depend on $x$. Then,
by \eqref{883},
\begin{align*}
    -f_{s,\dir}(x) &= \sup_{z \ge Z}\{\Ll(x,s;z,T) + \W_{\dir -}(z,T;0,T)\} \\
    &\qquad\qquad 
    - \sup_{z \ge Z}\{\Ll(x,s;z,T) + \W_{\dir +}(z,T;0,T)\} + A \\
    &\le \sup_{z \ge Z}\{\W_{\dir-}(z,T;0,T) - \W_{\dir +}(z,T;0,T) \} +A \\&
    = -\inf_{z \ge Z}\{f_{T,\,\dir}(z)\} + A \le -R + A,
\end{align*}
so that $f_{s,\dir}(x) \ge R - A$. Since $A$ is constant in $x$ and $R$ is arbitrary, the desired result follows. 

Note that \eqref{bad_ub} immediately proves \eqref{eqn:dcset_union1} in the case $x = 0$. The general case follows from additivity of the Busemann functions (Theorem \ref{thm:DL_Buse_summ}\ref{itm:DL_Buse_add}) and \eqref{bad_ub}.
\end{proof}
\begin{proof}[Proof of Theorem \ref{thm:Split_pts}]

\noindent \textbf{Item \ref{itm:split_dense} ($\Split$ is dense):} Work on the full-probability event   $\Omega_2$. Since $\Split \supseteq \Split^L \cup \Split^R$, it suffices to show that for $(x,s) \in \R^2$   there is  a sequence $(y_n,t_n) \in \Split^L$ converging to $(x,s)$. Let $g$ be the leftmost geodesic from $(x,s)$ to $(x,s + 1)$. Then $\forall n \ge 1$, $g|_{[s,s + n^{-1}]}$ is the leftmost geodesic from $(x,s)$ to $(x,s + n^{-1})$. By Lemma \ref{lem:rm_geod},  $\forall n \in \Z_{> 0}$  $\exists (x_n,t_n) \in \Split^L$ such that $x_n = g(t_n)$ and $s \le t_n \le s + n^{-1}$. The proof is complete by continuity of geodesics.

\noindent \textbf{Item \ref{itm:splitp0} ($\Pp(p \in \Split) = 0$ for all $p \in \R^2$):} If there exist disjoint semi-infinite geodesics from $(x,s)$, then for each level $t > s$, there exist disjoint geodesics from $(x,s)$ to some points $(y_1,t),(y_2,t)$. For each fixed $(x,s)$, with probability one, this occurs for no such points by \cite[Remark 1.12]{Bates-Ganguly-Hammond-22}.

\noindent \textbf{Item \ref{itm:Hasudorff1/2} (Hausdorff dimension of $\Split_{s,\dir})$:} Since $s$ is fixed, it suffices to take $s = 0$. By Theorem \ref{thm:Buse_dist_intro}\ref{itm:SH_Buse_process}, $\{\W_{\dir +}(\abullet,0;0,0)\}\deq G^{\sqrt 2}$, and by Theorem \ref{thm:DLBusedc_description}\ref{itm:Busedc_t}, $\dir \in \DLBusedc$ if and only if $f_{0,\dir} \neq 0$. Therefore, Corollary \ref{cor:SHHaus1/2} implies that, with probability one, $\dim_H(\D_{0,\dir}) = \f{1}{2}$ for all $\dir \in \DLBusedc$. Then, as noted in Remark \ref{rmk:NUsupp}, $\Pp(\dim_H(\Split_{0,\dir}) = \f{1}{2} \;\forall \dir \in \DLBusedc)  = 1$.

\noindent \textbf{Item \ref{itm:nonempty} ($\Split_{s,\dir}$ is nonempty and unbounded for all $s$):} 
By Theorem \ref{thm:DLBusedc_description}\ref{itm:Busedc_t}, on the event $\Omega_3$, whenever $\dir \in \DLBusedc$, for all $s \in \R$, $f_{s,\dir}(x) \to \pm \infty$ as $x \to \pm \infty$. Since $f_{s,\dir}$ is continuous (Theorem \ref{thm:DL_Buse_summ}\ref{itm:general_cts}), the set $\D_{s,\dir}$ is unbounded in both directions. The proof is complete since $\D_{s,\dir} \subseteq \Split_{s,\dir}$ by definition.
\end{proof}

\chapter{Scaling limit of the TASEP speed process} \label{chap:TASEP}
\section{Introduction}

\subsection{The totally asymmetric simple exclusion process}
In this final chapter, we show how the SH appears as a scaling limit for coupled initial data for a particle system know as the totally asymmetric simple exclusion process (TASEP). In the simplest TASEP dynamics each site of $\Z$  contains either a particle or a hole.  Each site has an  independent rate 1 Poisson clock. If at time $t$ the clock rings at site $x\in\Z$ the following happens. If there is a particle at site $x$ and no particle at site $x+1$ then the particle at site $x$ jumps to site $x+1$, while the other sites  remain unchanged. If there is no particle at site $x$ or there is a particle at site $x+1$ then the jump is suppressed. In other words, a particle can jump to the right only if the target site has no particle at the time of the jump attempt. This is the \textit{exclusion rule}. TASEP is a Markov process on the compact  state space $\{0,1\}^\Z$. Generic elements of $\{0,1\}^\Z$, or \textit{particle configurations}, are denoted by $\eta=\{\eta(x)\}_{x\in\Z}$, where $\eta(x)=1$  means that site $x$ is occupied by a particle and $\eta(x)=0$ that site $x$ is occupied by a hole, in other words, is empty. 

	For each density $\rho\in[0,1]$, the i.i.d.\ Bernoulli  distribution $\nu^\rho$ on $\{0,1\}^\Z$ with density $\rho$ is  the unique translation-invariant extremal stationary distribution of particle density $\rho$ under the TASEP dynamics. 
	
	There is a natural way to couple multiple  TASEPs from different initial conditions but with the same driving dynamics. Let $ \{\cN_x:x \in \Z\}$ be a $\Z$-indexed  collection of independent  rate $1$ Poisson  processes on $\R$. The clock at location $x$ rings at the times that correspond to points in $\cN_x$. One can then take two densities $0 <  \rho^1 <  \rho^2 <  1$ and ask whether there exists a coupling measure $\pi^{\rho_1,\rho_2}$ on $\{0,1\}^\Z\times\{0,1\}^\Z$ with Bernoulli marginals $\nu^{\rho_1}$ and $\nu^{\rho_2}$  that  is stationary under the joint TASEP dynamics
	and ordered. In other words, the twin requirements are    that if initially  $(\eta^1,\eta^2)\sim \pi^{\rho_1,\rho_2}$,  then $(\eta^1_t,\eta^2_t)\sim\pi^{\rho_1,\rho_2}$ at all subsequent times $t\ge0$, and $\eta^1(x)\leq \eta^2(x)$ for all $x\in\Z$ with $\pi^{\rho_1,\rho_2}$-probability one. Such a two-component  stationary distribution exists and is unique \cite{Liggett76}.
	
	One reason for the  interest in stationary measures of more than one density comes from the connection between the TASEP dynamics on $k$ coupled profiles in the state space $(\{0,1\}^\Z)^k$ and the TASEP dynamics on particles with classes in $\lzb1,k\rzb=\{1,\ldots,k\}$, called multiclass or multitype  dynamics. In the $k$-type dynamics, each  particle has a class in $\lzb1,k\rzb$ that remains the same for all time. A  particle  jumps to the right, upon the ring of a Poisson clock, only if there is either a hole or a particle of lower class (higher label) to the right. If this happens, the lower class particle moves left.  The state space of  $k$-type dynamics is  $\{1,\ldots,k,\infty\}^\Z$, with generic configurations denoted again by $\eta=\{\eta(x)\}_{x\in\Z}$. A value   $\eta(x)=i\in\lzb1,k\rzb$ means that site $x$ is occupied by  a particle of class $i$, and  $\eta(x)=\infty$ means  that site $x$ is empty, equivalently, occupied by a   hole.  Denoting a hole by $\infty$ is convenient now because holes can be equivalently viewed as particles of the absolute  lowest class.  
	For $k=1$ the multitype dynamics is the same as basic TASEP.
	
	The next question is whether we can  couple all the invariant multiclass distributions  so that the resulting construction is still invariant under TASEP dynamics. This was achieved by \cite{Amir_Angel_Valko11}:    such couplings can be realized by applying   projections to an object they constructed and named the \textit{TASEP speed process}. We describe briefly the construction. To start,  each site  $i\in\Z$ is occupied by a particle of class $i$.  This creates the initial   profile $\eta_0\in \Z^\Z$ such that $\eta_0(i)=i$. Let $\eta_t$ evolve under  TASEP dynamics, now interpreted so that a  particle switches places with the particle to its right only if the particle to the right  is of lower class, that is, has a higher label. Note that now each site is always occupied by a particle of some integer label.   The limit  from \cite{mountford2005motion}  
	implies  that each  particle has a well-defined limiting    speed:  if $X_t(i)$ denotes  the time-$t$ position of the particle initially at site $i$,  then the following random limit exists almost surely: 
	\begin{equation}\label{speed8}
	U_i= \lim_{t\to\infty} t^{-1}  X_t(i). 
	\end{equation}      
	The process $\{U_i\}_{i\in\Z}$ is   the TASEP speed process. It is a random element of the space  $[-1,1]^\Z$. 

\begin{theorem}[{\cite[Theorem 1.5]{Amir_Angel_Valko11}}]  \label{thm:AAV}
    The TASEP speed process $\{U_i\}_{i\in\Z}$ is the unique invariant distribution of TASEP that is ergodic under translations of the lattice $\Z$ and such that each $U_i$ is uniformly distributed on $[-1,1]$. 
\end{theorem}
In the context of the theorem above, the TASEP state  $\eta=\{\eta(i)\}_{i\in\Z}$ is a real-valued sequence, but the meaning of the dynamics is the same as before.  Namely, at each pair $\{i,i+1\}$ of nearest-neighbor sites, at the rings of a rate one exponential clock, the variables $\eta(i)$ and $\eta(i+1)$ are swapped if $\eta(i)<\eta(i+1)$, otherwise left unchanged.

 A key point  is that the
		 TASEP speed process projects to multitype stationary distributions. 
		
	
 \begin{theorem}[{\cite[Theorem 2.1]{Ferrari-Martin-2007}, \cite[Theorem 2.1]{Amir_Angel_Valko11}}] \label{thm:TASEPprod}
	Let $k\in\N$ be the number of \\ classes. Let  $\bar\rho=(\rho_1,\ldots,\rho_k)\in(0,1)^k$ be a parameter vector such that  $\sum_{i=1}^k  \rho_i \le 1$. Then there is a  translation-invariant  stationary distribution $\widecheck\pi^{\bar\rho}$ for the $k$-type TASEP which is unique under the conditions {\rm(i)} and {\rm(ii)}, and also under the conditions {\rm(i)} and {\rm(ii')} below:  
 
 {\rm(i)}  $\widecheck\pi^{\bar\rho}\{\eta\in \{1,\ldots,k,\infty\}^\Z:\eta(x) = j\} = \rho_j$ for each site $x\in\Z$ and class $j\in\lzb1,k\rzb$; 
 
 {\rm(ii)} under $\widecheck\pi^{\bar\rho}$, for each $\ell\in\lzb1,k\rzb$, the distribution of the sequence  $\{\ind[\eta(x) \le \ell]\}_{x \in \Z}
 $ of indicators is the i.i.d.\ Bernoulli measure $\nu^{\sum_{j=1}^{\ell} \rho_j}$ of intensity  $\sum_{j=1}^{\ell} \rho_j$;

 {\rm(ii')}  $\widecheck\pi^{\bar\rho}$ is  ergodic under the translation of the lattice $\Z$.
 
 Furthermore, $\widecheck\pi^{\bar\rho}$ is   extreme among translation-invariant stationary   measures of the $k$-type dynamics with jumps to the right. 
	\end{theorem}

 Theorem \ref{thm:TASEPprod} is not stated exactly in this form in either reference.  It can be proved with the techniques of Section VIII.3 of Liggett \cite{liggett1985interacting}. 
	
	\begin{lemma}[{\cite[Corollary 5.4]{Amir_Angel_Valko11}}]\label{lem:FU}
		Let $F:[-1,1] \rightarrow \{1,\dotsc,k, \infty\}$ be a  nondecreasing function and  $\lambda_j= \f{1}{2} {\rm Leb}\big(F^{-1}(j)\big)$, i.e.,  one-half the Lebesgue measure of the interval  mapped to the value $j\in\{1,\dotsc,k, \infty\}$.
		Then  the distribution of the $\{1,\dotsc,k, \infty\}$-valued sequence  $\{F(U_i)\}_{i\in\Z}$ is   the stationary measure $\widecheck\pi^{(\lambda_1,\dotsc,\lambda_k)}$ described in Theorem \ref{thm:TASEPprod} for  the $k$-type TASEP with jumps to the right. 
	\end{lemma} 
 For example, the case $k=1$ of  Lemma \ref{lem:FU} tells us that  to produce a particle configuration with Bernoulli distribution  $\nu^\rho$ from the TASEP speed process,  assign a particle to each site $x$ such that  $U_x\le2\rho-1$.   Lemma \ref{lem:FU} follows readily from Theorems \ref{thm:AAV} and  \ref{thm:TASEPprod} because the nondecreasing projection $F$ commutes with the pathwise dynamics. 

\begin{remark}[Direction of the jumps] Throughout this section, jumps in TASEP go to the right.  Later in Section \ref{sec:fdd}, we use the convention from \cite{Ferrari-Martin-2007} whereby TASEP jumps proceed left. This is convenient because then discrete  time in the queuing setting agrees with the order on $\Z$.  Notationally,  $\widecheck\pi^{\bar\rho}$ denotes the multiclass stationary measure under rightward jumps, as in Theorem \ref{thm:TASEPprod} and Lemma \ref{lem:FU} above, while 
  $\pi^{\bar\rho}$ will denote the stationary measure under leftward jumps. These  measures are simply reflections of each other (see Theorem \ref{thm:FM}). 
\end{remark}

	\subsection{Scaling limit of the speed process} \label{sec:scsp}
The space $\{0,1\}^\Z$ of TASEP particle configurations $\eta$ can be mapped bijectively  onto the space of  continuous  interfaces $f:\R\to\R$ such that $f(0)=0$, $|f(x)-f(x+1)|=1$ for all $x\in\Z$,  and $f(x)$  interpolates   linearly between integer points.  Define $\mathcal{P}:\{0,1\}^\Z\rightarrow C(\R)$  by stipulating that  on integers $i$ the image function $\mathcal{P}[\eta]$ is given by
	\begin{subequations} \label{eqn:TASEPh}
	\be
		\mathcal{P}[\eta](i)=
		\begin{cases}
			\sum_{j=0}^{i-1}(2\eta(j)-1), & i\in\N\\
			0, & i=0\\
			-\sum_{j={i}}^{-1}(2\eta(j)-1), & i\in -\N
		\end{cases}
		\ee
	and then extend $\mathcal{P}[\eta]$ to the reals by linear interpolation: 
	\be \text{for} \quad x\in \R\setminus \Z, \quad  \mathcal{P}[\eta](x)=(\ce x-x)\mathcal{P}[\eta](\fl x)+ (x-\fl x)\mathcal{P}[\eta](\ce x).
	\ee
	\end{subequations}
	TASEP can therefore be thought of as dynamics on continuous interfaces $f:\R \to \R$ such that for all $x \in \Z$, $f(x) \in\Z$ and  $f(x\pm1) \in \{f(x)-1,f(x) + 1\}$. When a particle at location $x$ lies immediately to the left of a hole at location $x + 1$, the interface  has a local maximum at location $x + 1$. When the particle changes places with the hole, the  local maximum becomes a local minimum.

	Let	$U=\{U_j\}_{j\in\Z}$ be the TASEP speed process and for $s\in\R$,  $\ind_{U < s}=\{\ind_{U_j< s}\}_{j\in\Z}$  a shorthand for the $\{0,1\}$-valued  sequence of indicators.   For each value of the centering $v \in (-1,1)$ and a scaling parameter $N\in\N$, use
	the mapping \eqref{eqn:TASEPh} to define from the speed process a    $C(\R)$-valued process indexed by $\dir\in\R$: 
 \begin{equation} \label{H138}
     H^N_\dir(x) = H^{\sigma,v,N}_\dir(x) =N^{-1/2}\,\mathcal{P}\big[\ind_{U \le v + \dir(1 -v^2) N^{-1/2}}\big]\Bigl(\f{\sigma^2 x}{1 - v^2}N\Bigr) - \f{\sigma^2 vx}{1-v^2}N^{1/2}, \qquad \dir, x\in\R.
 \end{equation}
The main theorem of this chapter is the process-level  weak limit of $H^{\sigma,v,N} = \{H^{\sigma,v,N}_\dir\}_{\dir \in \R}$.   The path space of $\dir\mapsto H^{\sigma,v,N}_\dir$ is the Skorokhod space $D(\R,C(\R))$ of $C(\R)$-valued cadlag paths on $\R$, with its usual Polish topology. This topology is  discussed in more detail in \cite{Busa-Sepp-Sore-22b}. 
	
	\begin{theorem}\label{thm:conv}
		Let $\sigma > 0$ and $G^{\sigma}$ be the stationary horizon. Then, for each $v \in (-1,1)$, as $N\rightarrow \infty$,  the  distributional limit $H^{\sigma,v,N} \Rightarrow G^\sigma$   holds on the path space  $D(\R, C(\R))$. 
		
	\end{theorem}

  Theorem \ref{thm:conv} is taken from the author's joint work with Busani and Sepp\"al\"ainen \cite{Busa-Sepp-Sore-22b}. In that work, we focus on the case $\sigma = \sqrt 2$, but the extension to general $\sigma$ follows the same proof (or alternatively from the scaling relations of Theorem \ref{thm:SH_dist_invar}\ref{itm:SHscale}). As is typical, the proof splits into two main steps: 
  (i) weak convergence of finite-dimensional distributions of  $H^{N}$ to the limiting object  and (ii) tightness of $\{H^{N}\}_{N\in\N}$ on $D(\R,C(\R))$.   Both parts  use  the Ferrari-Martin  queuing representation of the  multitype stationary measures. 
  The first part shows that, in the limit, the queuing representation recovers  the queuing structure that defines the SH \cite{Ferrari-Martin-2007}. In this dissertation, we only show the finite-dimensional convergence (Proposition \ref{prop:fdd}), and refer the reader to \cite{Busa-Sepp-Sore-22b} for the proof of tightness. 

 \section{Finite-dimensional convergence} \label{sec:fdd}

\subsection{The space $D(\R,C(\R))$}\label{sec:D(R)}	

	

We observe here why the path $\dir\mapsto H^N_\dir$ defined in \eqref{H138} lies in $D(\R,C(\R))$. Restriction of $x\mapsto H^N_\dir(x)$ to a  bounded interval $[-x_0, x_0]$ is denoted by  $H^{N,x_0}_\dir=H^N_\dir\vert_{[-x_0, x_0]}$.  Then note that for $\dir<\rho$, $H^{N,x_0}_\dir\ne H^{N,x_0}_\rho$	if and only if $U_j\in(v + \dir(1 - v^2) N^{-1/2},v  + (1 - v^2)\rho N^{-1/2}]$ for some  
$j\in\lzb\,\fl{-\,\f{\sigma^2 x_0}{1 - v^2} N}, \ce{\f{\sigma^2 x_0}{1 - v^2} N}\,\rzb$.  Since this range of indices is finite, for each $\dir \in \R$ and $x_0 > 0$ there exists $\ve>0$ such that $H^{N,x_0}_\dir= H^{N,x_0}_\rho$ for $\rho\in[\dir, \dir+\ve]$ and $H^{N,x_0}_\rho= H^{N,x_0}_\sigma$ for $\rho, \sigma\in[\dir-\ve, \dir)$. 

\subsection{Ferrari-Martin representation of multiclass measures} \label{sec:FM}

This section describes the queuing construction of stationary multiclass measures from \cite{Ferrari-Martin-2007}. We use the convention of \cite{Ferrari-Martin-2007} that TASEP particles jump to the left rather than to the right, because this choice leads to the more natural queuing set-up where time flows on $\Z$ from left to right. This switch is then accounted for when we apply the results of this section.  

\subsubsection{Queues with a single customer stream}
	Let $\Qs_1:=\{1,\infty\}^\Z$ 
	be the space of configurations of particles on $\Z$ with the following interpretation:  a configuration  $\bq{x}=\{x(j)\}_{j\in\Z}\in \Qs_1$  has a particle at time $j\in\Z$ if $x(j)=1$, otherwise   $\bq{x}$ has a hole at time $j\in\Z$.  
	Let $\arrv,\srvv\in\Qs_1$. Think of $\arrv$ as arrivals of customers to a queue, and of $\srvv$ as the available  services in  the queue. For $i\le j\in \Z$ let $a^{\leq1}[i,j]$  be the number of customers, that is, the number of $1$'s in $\bq{a}$, that  arrive to the queue during time interval $[i,j]$. Similarly, let $s[i,j]$ be the number of services available  during time interval $[i,j]$. The queue length at time $i$ is then given by 
	\begin{equation}\label{Q}
		\Qd_i=\sup_{j: j\leq i}\big(a^{\leq1}[j,i]-s[j,i]\tspb\big)^+.
	\end{equation}
In principle this makes sense for arbitrary sequences $\arrv$ and $\srvv$ if one allows infinite queue lengths $\Qd_i=\infty$. However, in our treatment $\arrv$ and $\srvv$ are always such that 
queue lengths are finite.  We will not repeat this point in the sequel.

	The departures from the queue come from the mapping $\depv=\Dd(\arrv,\srvv):\Qs_1\times\Qs_1 \rightarrow \Qs_1$,  given by 
	\begin{equation}\label{eq6}
		\dep(i)=
		\begin{cases}
		1 & \text{$s(i)=1$ and either $Q_{i-1}>0$ or $a(i)=1$},\\
		\infty & \text{otherwise}.
		\end{cases}
	\end{equation}
	In other words, a customer leaves the queue at time $i$ (and $d(i)=1$) if there is a service  at time $i$ and either the queue is not empty or a customer just arrived at time $i$.
			The sequence $\bq{u}:=\Ud(\arrv,\srvv)$ of unused services  is given by a mapping  $\Ud:\Qs_1\times\Qs_1\rightarrow \Qs_1$  defined by 
	\begin{equation}\label{defU}
	u(j)=
	\begin{cases}
	1 & \text{if $\srv(j)=1$,  $\Qd_{j-1}=0$, and $\arr(j)=\infty$},\\
	\infty & \text{otherwise}.
	\end{cases}
	\end{equation}
	Last, we define the map $\Rd:\Qs_1\times\Qs_1\rightarrow \Qs_1$ as $\bq{r}=\Rd(\arrv,\srvv)$ with 
	\begin{equation}\label{defR} 
		r(j)=
		\begin{cases}
		1 & \text{ if either $a(j)=1$ or $u(j)=1$},\\
		\infty & \text{ if $a(j)=u(j)=\infty$}.  
		\end{cases}.
	\end{equation}
Extend the departure operator $\Dd$ to  queues in tandem.   Let  $\Dd^1(\bq{x})=\bq{x}$ be the identity, and for $n\ge 2$,  
\begin{equation}\label{eq18}
\begin{aligned}
\Dd^{(2)}(\bq{x}_1,\bq{x}_2)&=\Dd(\bq{x}_1,\bq{x}_2)\\
\Dd^3(\bq{x}_1,\bq{x}_2,\bq{x}_3)&=\Dd\big(\Dd^{(2)}(\bq{x}_1,\bq{x}_2),\bq{x}_3\big)\\
&	\ \; \vdots \\
\Dd^{(n)}(\bq{x}_1,\bq{x}_2,\dotsc,\bq{x}_n)&=\Dd\big(\Dd^{(n-1)}(\bq{x}_1,\bq{x}_2,\dotsc,\bq{x}_{n-1}),\bq{x}_n\big).
\end{aligned}
\end{equation}

We use the notation $\Dd$ (with the subscript $d$) to denote discrete, in comparison to the operators on continuous functions in \eqref{Ddef},\eqref{Diter}.


	\subsubsection{Queues with priorities}
	Now consider queues with customers of different classes.  For $m\in\N$, let  \\$\Qs_m:=\{1,2,\dotsc,m,\infty\}^\Z$ be the space of configurations of particles on $\Z$ with classes in $\lzb1,m\rzb=\{1,2,\dotsc,m\}$. A lower label indicates higher class and, as before,  the value $\infty$ signifies an empty time slot.  To illustrate the notation for  an arrival sequence $\arrv\in \Qs_m$,   the  value  $a(j)=k\in\lzb1,m\rzb$ means that    a customer of class $k$ arrives at time $j\in\Z$,  while   $a(j)=\infty$ means   no  arrival  at time $j$. Define 
	\begin{equation*}
		a^{\leq k}[j]=
		\begin{cases}
		1 &\text{if $a(j)\leq k$},\\
		0 &\text{if $a(j)>k$}.
		\end{cases}
	\end{equation*}
	 Consistently with earlier definitions,  $a^{\leq k}[i,j]=\sum_{l=i}^{j}a^{\leq k}[l]$ is the number of customers in classes $\lzb1,k\rzb$ that arrive to the queue in the time interval $[i,j]$. Let $\srvv\in \Qs_1$ be the sequence of available services. The number of customers in classes $\lzb1,k\rzb$ in the queue at time $i$ is then 
	\begin{equation*}
		Q^{\leq k}_i(\arrv,\srvv)=\sup_{j: j\leq i}\big(a^{\leq k}[j,i]-s[j,i]\big)^+, \qquad i\in \Z.
	\end{equation*}
	 The multiclass departure map $\depv = F_m(\arrv,\srvv):\Qs_m\times \Qs_1 \rightarrow \Qs_{m+1}$ is defined so that  customers of  higher class (lower label) are served first.  These are the rules: 
	\begin{equation}\label{fm}
		\begin{cases}
		d(i)\leq k & \text{for $k\in\lzb1,m\rzb$ if $s(i)=1$ and either $Q^{\leq k}_{i-1}>0$ or $a(i)\leq k$}, \\
		d(i)=m+1 & \text{if $s(i)=1$,  $Q^{\leq m}_{i-1}=0$, and $a(i)= \infty$},\\
		d(i)=\infty &\text{if }  s(i)=\infty .
		\end{cases}
	\end{equation}
	The map $F_m$ works as follows.  The queue is fed with arrivals $\arrv\in\Qs_m$ of customers in classes $1$ to $m$. Suppose a service is available at time $i\in\Z$ ($s(i)=1$). Then  the customer of the highest class (lowest label in $\lzb1,m\rzb$)   in the queue at time $i$,  or  just arrived at time $i$, is served at time $i$, and its label becomes the value of $d(i)$. If no customer arrived at time $i$ ($a(i)=\infty$) and the queue is empty ($Q^{\leq m}_{i-1}=0$), then the unused service  $s(i)=1$ is converted into a departing  customer of class $m+1$:  $\dep(i)=m+1$. If there is no  service available at time $i\in\Z$ ($s(i)=\infty$), then no customer leaves at time $i$ and  $\dep(i)=\infty$.
	
	In particular, for $m=1$, the output    $\depv=F_1(\arrv,\srvv)$ satisfies 
	\be\label{Q670} 
	d(i)=\begin{cases} 1, &[\Dd(\arrv,\srvv)]_i=1\\ 2, &[\Ud(\arrv,\srvv)]_i=1 \\ \infty, &s(i)=\infty.   \end{cases}
	\ee
 	
	For  $n\in\N$  define the space  $\ml_n=\Qs_1^n=\{1,\infty\}^{\Z\times \{1,\dotsc,n\}}$ of $n$-tuples of sequences. Let $\bar{\lambda}=(\lambda_1,\dotsc,\lambda_n)\in(0,1)^n$ be a parameter vector such that  $\sum_{r=1}^n \lambda_r\le1$. Define  the product measure $\nu^{\bar{\lambda}}$ on  $\ml_n$ so that if   $\bar{\bq{x}}=(\bq{x}_1,\dotsc,\bq{x}_n)\sim\nu^{\bar{\lambda}}$ then the sequences  $\bq{x}_k$ are independent and  each   $\bq{x}_k$ has the i.i.d.\ product  Bernoulli distribution  with intensity $\sum_{i=1}^{k}\lambda_i$.  From this input we define a new process $\bar{\bq{v}}=(\bq{v}_1,\dotsc,\bq{v}_n)$ such that each $\bq{v}_m\in \Qs_m$   by the iterative  formulas 
	\begin{equation}\label{eq1}
		\begin{aligned}
		\bq{v}_1&=\bq{x}_1 \qquad\text{and} \\
		\bq{v}_m&=F_{m-1}(\bq{v}_{m-1},\bq{x}_m) \qquad \text{for } \ m= 2, \dotsc,n.
		\end{aligned}
	\end{equation}
	We denote this map  by $\bar{\bq{v}}=\Vmap(\bar{\bq{x}})=(\Vmap_1(\bar{\bq{x}}),\dotsc,\Vmap_n(\bar{\bq{x}}))$. For a vector $\bar{\lambda}=(\lambda_1,\dotsc,\lambda_n)$ and $\bar{\bq{x}}\sim \nu^{\bar{\lambda}}$, define the distribution  $\pi^{\bar{\lambda}}$ as the image of $\nu^{\bar{\lambda}}$ under this map:  
	\begin{equation}\label{mu}
\pi^{\bar{\lambda}}=\nu^{\bar{\lambda}}\circ \Vmap_n^{-1}  \ \iff \ 	\Vmap_n(\bar{\bq{x}})\sim 	\pi^{\bar{\lambda}}.
	\end{equation}
	\begin{theorem}\cite[Theorem 2.1]{Ferrari-Martin-2007}\label{thm:FM}   
 For each $m\in\lzb1,n\rzb$, 
		the distribution  of $\bq{v}_m$   under $\pi^{\bar{\lambda}}$ is 
    the unique translation-ergodic
  stationary distribution of the $m$-type TASEP on $\Z$ with leftward jumps and  with density $\lambda_r$ of particles of class $r\in\lzb1,m\rzb$. 
   The distribution of the reversed configuration $\{\bq{v}_m(-i)\}_{i \in \Z}$ is the unique  distribution $\widecheck\pi^{\bar\lambda}$ described in Theorem \ref{thm:TASEPprod}, in other words, the unique translation-ergodic stationary distribution of the $m$-type TASEP on $\Z$ with rightward  jumps, with density $\lambda_r$ of particles of class $r \in \lzb 1,m \rzb$. 
\end{theorem}
\begin{remark}
The statement about the TASEP with rightward jumps is not included in \cite{Ferrari-Martin-2007}, but its proof is straightforward: reflecting the index does not change the density of the particles, so the values $\lambda_r$ are preserved. Consider an $m$-type TASEP with left jumps $\{\eta_t\}_{t \ge 0}$ defined by the  Poisson clocks  $\{\cN_i\}_{i \in \Z}$ and started from initial profile $\eta_0 \sim \bq{v}_m$. Let $\{\widecheck \eta_t\}_{t \ge 0}$ be TASEP with right jumps defined by the   Poisson clocks $\{\cN_{-i}\}_{i \in \Z}$ and started from initial profile $\{\widecheck \eta_0(i)\}_{i \in \Z} := \{\eta_0(-i)\}_{i \in \Z}$, which has distribution $\{\bq{v}_m(-i)\}_{i \in \Z}$. Then, in the process $\eta_t$, a particle jumps from site $i$ to site $i - 1$ exactly when a particle in the process $\widecheck \eta_t$ jumps from site $-i$ to site $-i + 1$. By the invariance of $\eta_0$ under TASEP with left jumps,
\[
\{\widecheck \eta_t(i)\}_{i \in \Z} = \{\eta_t(-i)\}_{i \in \Z} \deq \{\eta_0(-i)\}_{i \in \Z} = \{\widecheck \eta_0(i)\}_{i \in \Z}, 
\]
so $\{\bq{v}_m(-i)\}_{i \in \Z}$ is the invariant measure for TASEP with right jumps and densities $\lambda_r$. 
\end{remark}

	For  $\bq{x}\in\Qs_n$, define 
\begin{equation*}
\Cls^{[i,j]}_m(\bq{x})=\#\{l\in [i,j]:{x}(l)\leq m\}, \qquad m\in\lzb1,n\rzb.
\end{equation*}  
$\Cls^{[i,j]}_m(\bq{x})$  records the number of customers in classes  $\lzb1,m\rzb$   during time interval  $[i,j]$  in the sequence  $\bq{x}$. Note that a customer of class $m$ appears in $[i,j]$ if and only if $\Cls^{[i,j]}_m(\bq{x})>\Cls^{[i,j]}_{m-1}(\bq{x})$, with the convention $\Cls_0\equiv0$. The key technical lemma is that the iteration in \eqref{eq1} can be represented by tandem queues.

\begin{lemma}\label{lem:Cl}
Let $n\in\N$ and $\bar{\bq{x}}=(\bq{x}_1,\dotsc,\bq{x}_n)\in\ml_n$. 
	Let $\bq{v}_n=\Vmap_n(\bar{\bq{x}})$, where $\Vmap_n$ is given in \eqref{eq1}.   Define
	 \begin{equation*}
	\begin{aligned}
	\depv^{n,i}:=\Dd^{(n - i +1)}(\bq{x}_{i},\bq{x}_{i+1},\dotsc,\bq{x}_n) \quad \text{for } \ i=1,\dotsc,n-1,  \quad \text{and} \quad 
	\depv^{n,n}:=\bq{x}_n.
	\end{aligned}
	\end{equation*}
	Then for all time intervals $[i,j]$, 
	\begin{equation}\label{ind}
		\big(\Cls^{[i,j]}_1(\bq{v}_n),\dotsc,\Cls^{[i,j]}_n(\bq{v}_n)\big)
		=  \big(d^{n,1}[i,j],\dotsc,d^{n,n}[i,j]\big).
	\end{equation}
\end{lemma}
\begin{proof}
	The proof goes by induction on $n$, with  base case $n=2$. From   \eqref{Q670}, $\depv^{2,1}=\Dd(\bq{x}_1,\bq{x}_2)$ registers the first class departures out of the queue $F_1(\bq{x}_1,\bq{x}_2)$ while $\depv^{2,2}=\bq{x}_2$ is the combined number of first and second class customers coming out of the queue.  The case $n=2$ of \eqref{ind}  has been verified.

	Assume   \eqref{ind} holds for some $n=k\ge2$. 	This means that  for each $m\in\lzb1,k\rzb$, $\depv^{k,m}$ registers the   customers in classes $\lzb1,m\rzb$ in $\bq{v}_k$. In the next step, $\bq{v}_{k+1}=F_k(\bq{v}_k, \bq{x}_{k+1})$, and 
	\begin{equation*}
		\big(\depv^{k+1,1},\dotsc,\depv^{k+1,k+1}\big)=\big(\Dd(\depv^{k,1},\bq{x}_{k+1}),\dotsc,\Dd(\depv^{k,k},\bq{x}_{k+1}),\bq{x}_{k+1}\big). 
	\end{equation*}
	Since the same service process $\bq{x}_{k+1}$ acts in both queuing maps, the outputs match in the sense that for each $m\le k$,  $\depv^{k+1,m}=\Dd(\depv^{k,m},\bq{x}_{k+1})$ registers the customers in classes $\lzb1,m\rzb$ in $\bq{v}_{k+1}$.   
	Under $F_k(\bq{v}_k, \bq{x}_{k+1})$, unused services become departures of class $k+1$. Hence,  every service event of $\bq{x}_{k+1}$ becomes a departure of some class in $\lzb1,k+1\rzb$. This verifies the equality $\Cls^{[i,j]}_{k+1}(\bq{v}_{k+1})={x}_{k+1}[i,j]=d^{k+1,k+1}[i,j]$ of the last coordinate. Thereby, the validity of \eqref{ind} has been extended from $k$ to $k+1$. 
\end{proof}

\subsection{Convergence of queues}

 \noindent 
This section shows the finite-dimensional weak  convergence of the TASEP speed process,  using the representation of stationary distributions in terms of  queuing mappings. To do this, we derive a convenient representation for the random walk defined by the departure mapping $\Dd$ \eqref{DQ}. Consistently with the count notation $x[i,j]$ introduced above for  $\bq x \in \Qs_1$, abbreviate 
   $x[i] =  x[i,i]=\ind_{x(i) = 1}$. With this convention,  configurations $\bq{x}$ can also be thought of as members of the sequence space  $\{0,1\}^\Z$. Recall the operation $\mathcal 
 P$ from \eqref{eqn:TASEPh}.
   \begin{lemma} 
   For $i \in \Z$,
   \be \label{DQ}
    \mathcal P[\Dd(\arrv,\srvv)](i) = \mathcal P[\srvv](i) + \sup_{-\infty < j \le 0}[\mathcal P[\srvv](j) -\mathcal P[\arrv](j)] - \sup_{-\infty < j \le i}[\mathcal P[\srvv](j) -\mathcal P[\arrv](j)].
   \ee
   \end{lemma}
  \begin{proof}
      Recall the definition of $D$ from \eqref{eq6}. Observe that
      \begin{equation} \label{gen:Qa}
    \Dd(\arrv,\srvv)[j,i]=Q_{j-1}-Q_i+a[j,i],
\end{equation}
because any arrival that  cannot be accounted for in $Q_i$ must have left by time $i$.
Use also the empty interval convention  $x[i + 1,i] = 0$. Then, from \eqref{Q}, we can equivalently write 
   \be \label{Qialt}
   Q_i = \sup_{j:j \le i + 1}\bigl(a^{\le 1}[j,i] - s[j,i]\bigr). 
   \ee
Now, observe that for $\bq{x} \in \Qs_1$,
\be \label{xincs}
2x[j,i] - (i - j + 1) = \sum_{k = j}^i (2x[k] - 1) = \mathcal P[\bq{x}](i + 1) - \mathcal P[\bq{x}](j)
\ee
Combining \eqref{gen:Qa}--\eqref{xincs} and the definition  $\mathcal P[\bq{x}](0) = 0$, gives for $i > 0$,
\begin{align*}
\mathcal P[\Dd(\arrv,\srvv)](i) &= \sum_{k = 0}^{i - 1} (2\Dd(\arrv,\srvv)[k] - 1) = 2\Dd(\arrv,\srvv)[0,i-1] - i \\
&=  2a[0,i-1] - i + 2Q_{-1} - 2Q_{i -1} \\
&\overset{\eqref{xincs}}{=} \mathcal P[\arrv](i) - \mathcal P[\arrv](0) + 2\sup_{-\infty < j \le 0}[a^{\le 1}[j,-1] - s[j,-1]]\\
&\qquad\qquad\qquad\qquad\qquad\qquad- 2\sup_{-\infty < j \le i }[a^{\le 1}[j,i-1] - s[j,i-1]] \\
&= \mathcal P[\arrv](i) + \sup_{-\infty < j \le 0}[2a^{\le 1}[j,-1] +j - (2s[j,-1] + j)] \\
&\qquad\qquad - \sup_{-\infty < j \le i}[2a^{\le 1}[j,i-1] - (i - j) - (2s[j,i-1] - (i - j))] \\
&\overset{\eqref{xincs}}{=} \mathcal P[\arrv](i) + \sup_{-\infty < j \le 0}[\mathcal P[\arrv](0) - \mathcal P[\arrv](j) - \mathcal P[\srvv](0) + \mathcal P[\srvv](j) ] \\
&\qquad\qquad\qquad - \sup_{-\infty < j \le i}[\mathcal P[\arrv](i) - \mathcal P[\arrv](j) - \mathcal P[\srvv](i) + \mathcal P[\srvv](j) ] \\
&= \mathcal P[\srvv](i) + \sup_{-\infty < j \le 0}[\mathcal P[\srvv](j) -\mathcal P[\arrv](j)] - \sup_{-\infty < j \le i}[\mathcal P[\srvv](j) -\mathcal P[\arrv](j)].
\end{align*}
The case $i < 0$ follows an analogous proof. 
  \end{proof}

\begin{proposition}\label{prop:fdd}   Fix the diffusion coeffieient $\sigma > 0$ the centering $v \in (-1,1)$. 
Then the scaled TASEP speed process  $H^N$ of  \eqref{H138} satisfies the weak convergence 
$(H_{\dir_1}^N,\dotsc,H_{\dir_k}^N)$ $\Rightarrow$ $(G^\sigma_{\dir_1},\dotsc,G^\sigma_{\dir_k})$  on $C(\R)^k$ for any finite sequence 
$(\dir_1,\dotsc,\dir_k)\in \R^k$.  
\end{proposition}
\begin{proof}

Without loss of generality, take $\dir_1 < \dir_{2} < \cdots < \dir_k$.  
For  $N > |\dir_1|^{3} \vee |\dir_k|^{3}$, consider the following nondecreasing map $F:[-1,1] \to \{1,\ldots,k\} \cup \{\infty\}$:
 \[
 \text{For } U \in [-1,1],\; F(U) =
 \begin{cases}
 1, & U \le v + \dir_1 (1 - v^2) N^{-1/2} \\
 2,& v + \dir_1(1 - v^2) N^{-1/2} < U \le v + \dir_2(1 - v^2) N^{-1/2} \\
 \vdots & \\
 k, &v + \dir_{k - 1}(1 - v^2)N^{-1/2} < U \le v +  \dir_k(1 -v^2) N^{-1/2} \\
 \infty, &U > v + \dir_k(1 -v^2) N^{-1/2}.
\end{cases}
 \]
 By considering the output of this map as classes, Lemma \ref{lem:FU} implies that $\{F(U_i)\}_{i \in \Z}$
 is distributed as the stationary distribution for $k$-type TASEP with right jumps and densities
 \[
 \bar \lambda = \biggl(\f{1 + v + \dir_1(1 - v^2) N^{-1/2}}{2},\f{(\dir_2 - \dir_1)(1 - v^2)N^{-1/2}}{2},\ldots,\f{(\dir_{k} - \dir_{k-1})(1 -v^2)N^{-1/2}}{2}\biggr).  \] 
 The reflection of Theorem \ref{thm:FM} and translation invariance then imply that $\{F(U_{-i - 1})\}_{i \in \Z}$
 has the stationary distribution $\pi^{\bar\lambda}$ for TASEP with left jumps.    
Lemma \ref{lem:Cl} implies that
 \be \label{bercoup}\begin{aligned} 
&\bigl(\ind_{U_{-i - 1} \le v + \dir_1(1 - v^2) N^{-1/2}},\ldots,  \ind_{U_{-i-1} \le v + \dir_{k-1} (1 -v^2) N^{-1/2}}\,, \ind_{U_{-i-1} \le v + \dir_k (1 -v^2) N^{-1/2}}\bigr)_{i \in \Z}\\
&\qquad\qquad\qquad
\deq \bigl(\Dd^{(k)}(\bq{x}_1^N,\ldots,\bq{x}_k^N)[i]\,,\ldots,\Dd^{(2)}(\bq{x}^N_{k - 1},\bq{x}_k^N)[i]\,,\bq{x}_k^N[i]\bigr)_{i \in \Z},
\end{aligned} \ee
 where $(\bq{x}_1^N,\ldots,\bq{x}_k^N) \sim \nu^{\bar \lambda}$.
 Remark \ref{rmk:L/R8} at the end of the section gives an alternative way to justify the index reversal on the left-hand side above when $v = 0$.  

Before proceeding with the proof, we give a roadmap. First, by definition of $\mathcal P$,  if $\widecheck{U}_i = U_{-i - 1}$, then for $x \in \R$,
\be \label{chU}
\begin{aligned}
H_{\dir}^N(x) &=  N^{-1/2}\mathcal P[\ind_{U \le v + \dir(1 - v^2) N^{-1/2}}]\Bigl(\f{\sigma^2 x}{1 - v^2}N\Bigr) -\f{\sigma^2 vx}{1 -v^2}N^{1/2} \\
&= -N^{-1/2}\mathcal P[\ind_{\widecheck U \le v + \dir(1 - v^2) N^{-1/2}}]\Bigl(-\f{\sigma^2 x}{1-v^2}N\Bigr) - \f{\sigma^2 vx}{1- v^2}N^{1/2}.
\end{aligned}
\ee
Our goal is to show the weak limit 
\be \label{convflip}
\begin{aligned}
&\biggl(-N^{-1/2}\mathcal P[\bq{x}_k^N]\Bigl(\f{\sigma^2 \tsp\aabullet}{1 - v^2}N\Bigr) + \f{\sigma^2 v\tsp\aabullet}{1 - v^2}N^{1/2}\,, \\[3pt]
&\qquad -N^{-1/2}\mathcal P[\Dd(\bq{x}^N_{k-1},\bq{x}_k^N)]\Bigl(\f{\sigma^2 \tsp\aabullet}{1 - v^2}N\Bigr) + \f{\sigma^2 v\tsp\aabullet}{1 - v^2}N^{1/2},\ldots, \\[3pt] 
&\qquad\qquad 
-N^{-1/2}\mathcal P[\Dd^k(\bq{x}_1^N,\ldots,\bq{x}_k^N)]\Bigl(\f{2\tsp\aabullet}{1 - v^2}N\Bigr) + \f{2v\tsp\aabullet}{1 - v^2}N^{1/2}\biggr) \\
&\qquad\qquad   
\Longrightarrow (G^\sigma_{-\dir_k},\ldots, G^\sigma_{-\dir_1}).
\end{aligned}
\ee
Once we show \eqref{convflip}, from \eqref{convflip},  \eqref{bercoup} and \eqref{chU} it follows that $(H_{\dir_1}^N,\ldots,H_{\dir_k}^N) \Longrightarrow (G^\sigma_{-\dir_1}(-\tsp\aabullet),\ldots, G^\sigma_{-\dir_k}(-\tsp\aabullet))$. This limit has the same distribution as $(G^\sigma_{\dir_1},\ldots,G^\sigma_{\dir_k})$ by Theorem \ref{thm:SH_dist_invar}\ref{itm:SH_reflinv}. As mentioned previously,  these reflections   in the proof are a consequence of having   time flow left to right in  the queuing setting. 

Now, we prove  \eqref{convflip}.
 By construction, for $j \in \lzb 1,k\rzb$, $\{{x}^N_j[i]\}_{i \in \Z}$ is an i.i.d. Bernoulli sequence with intensity $\sum_{\ell \le j} \lambda_\ell = \f12{(1 +v +  \dir_j(1 -v^2)N^{-1/2})}$.  Hence, for $j \in \lzb 1,k \rzb$, 
 \be \label{xconv}
 -N^{-1/2} \mathcal P[\bq{x}_j^N]\Bigl(\f{\sigma^2 \aabullet}{1 - v^2}N\Bigr) + \f{\sigma^2 v\tsp\aabullet}{1 - v^2}N^{1/2}
 \ee
 converges in distribution  to a Brownian motion with diffusivity $\sigma$ and drift $-\sigma^2 \dir_j$. 
 To elevate this to the  joint convergence of \eqref{convflip}, we utilize the queuing mappings in \eqref{bercoup} and the transformations $D^{(k)}$ from \eqref{Diter} that construct the SH.

By Skorokhod representation (\cite[Thm.~11.7.2]{dudl}, \cite[Thm.~3.1.8]{ethi-kurt}), we may couple $\{\bq{x}_j^N\}_{j = 1,\ldots,k}$ and independent Brownian motions $\{Z^j\}_{j \in \lzb 1,k \rzb}$ with diffusivity $\sigma$ and drift $-\sigma^2 \dir_j$ so that, with probability one, for $j  \in \lzb 1,k \rzb$, \eqref{xconv}
converges uniformly on compact sets to $Z^j$. Let $\Pp$ be the law of this coupling (To be precise, the sequences  $\bq{x}_j^N$ are  functions of the  converging processes \eqref{xconv}, which  Skorokhod representation couples with their  limiting  Brownian motions).


By Proposition \ref{prop:SH_cons}\ref{itm:SH_dist} and definition of $D^{(k)}$ \eqref{Diter}, for reals $\dir_1  < \cdots < \dir_k$, the valued marginal  $(G^\sigma_{-\dir_k},\dotsc,G^\sigma_{-\dir_1})$ of the SH can be constructed as follows:
\begin{align*}
 G^\sigma_{-\dir_k}&=D(Z^k)=Z^k, \qquad 
 G^\sigma_{-\dir_{k-1}}=D^{(2)}(Z^k, Z^{k-1})=D(Z^{k - 1}, Z^{k}), \ldots  \\
  G^\sigma_{-\dir_1}&=D^{(k)}(Z^1, Z^2, \dotsc,  Z^k) =  D(D^{(k - 1)}(Z^1,\ldots,Z^{k - 1}),Z^k). 
\end{align*}
We recall \eqref{Ddef} that The map $D$ is
\[
D(Z,B)(y) = B(y) + \sup_{-\infty <x \le y }\{Z(x) - B(x)\} - \sup_{-\infty < x \le 0}\{Z(x) - B(x)\}. 
\]

 By a union bound, it suffices to show that, under this coupling,  for each $\ve > 0, a > 0$, and each $j = 0,\ldots,k - 1$, 
\be \label{convPhi}\begin{aligned} 
\limsup_{N \to \infty}\Pp\biggl(\;\sup_{x \in [-a,a]}\Bigl|-N^{-1/2}\mathcal P[\Dd^{(j + 1)}(\bq{x}_{k - j}^N,\ldots, \bq{x}_k^N)] & \Bigl(\f{2 x}{1 - v^2}N\Bigr) + \f{2vx}{1 - v^2}N^{1/2}  \\
&- D^{(j+1)}(Z^{k - j},\ldots,Z^k)(x)\Bigr| > \ve  \biggr) = 0,
\end{aligned} \ee
where the argument of the discrete operator is understood via linear interpolation.
We prove \eqref{convPhi} by induction on $j$. The base case $j = 0$ follows by the almost sure uniform convergence on compact sets of \eqref{xconv} to $Z^j$. Now, assume the statement holds for some $j - 1 \in \{0,\ldots,k -2\}$. The proof is completed by Lemma \ref{lem:convD} below, once we recall \eqref{eq18} that \[
\Dd^{(j + 1)}(\bq{x}_{k - j}^N,\ldots,\bq{x}_{k}^N) = \Dd(\Dd^{(j)}(\bq{x}_{k - j}^N,\ldots,\bq{x}_{k - 1}^N),\bq{x}_k^N). \qedhere \]  
\end{proof}

 \begin{lemma} \label{lem:convD}
Let $v \in (-1,1)$. For each $N>0$, let $\arrv^N$ and $\srvv^N$ be $\{0,1\}^\Z$-valued i.i.d. sequences such that  the intensity of $\srvv^N$ is strictly greater than the intensity of $\arrv^N$. Assume further that these sequences are coupled together with Brownian motions $Z^1,Z^2$ with diffusivity $\sigma$ and drifts $-\sigma^2\dir_1 > -\sigma^2 \dir_2$ so that, for each $\ve > 0$ and $a > 0$,
\be \label{ipcon}
\begin{aligned}
&\limsup_{N \to \infty} \Pp\Biggl(\sup_{y \in [-a,a]}\Bigl|-N^{-1/2} \mathcal P[\arrv^N]\Bigl(\f{\sigma^2 y}{1 - v^2}N\Bigr) + \f{\sigma^2 vy}{1 - v^2}N^{1/2} - Z^1(y) \Bigr| \\
&\qquad \qquad
\vee \Bigl|-N^{-1/2} \mathcal P[\srvv^N]\Bigl(\f{\sigma^2 y}{1 - v^2}N\Bigr) + \f{\sigma^2 vy}{1 - v^2}N^{1/2} - Z^2(y) \Bigr| > \ve \Biggr) = 0.
\end{aligned}
\ee
 Then, for every $\ve > 0$ and $a > 0$, 
\begin{multline} \label{2ptprob}
\limsup_{N \to \infty} \Pp\Biggl(\sup_{y \in [-a,a]}\Bigl|-N^{-1/2} \mathcal P[\Dd(\mbf a^N,\mbf s^N)]\Bigl(\f{\sigma^2 y}{1 - v^2}N \Bigr) + \f{\sigma^2 vy}{1 - v^2}N^{1/2} \\
- D(Z^1,Z^2)(y) \Bigr| > \ve  \Biggr) = 0.
\end{multline}
 \end{lemma}
 \begin{proof}
 From \eqref{DQ}, we have, for $\f{\sigma^2 N x}{1 - v^2} \in \Z$,
 \be \label{Prew}
 \begin{aligned}
 &-N^{-1/2} \mathcal P[\Dd(\arrv^N,\srvv^N)]\Bigl(\f{\sigma^2 x}{1 - v^2}N\Bigr) + \f{\sigma^2 vx}{1 - v^2}N^{1/2} \\
 &=-N^{-1/2} \mathcal P[\srvv^N]\Bigl(\f{\sigma^2 x}{1 - v^2} N\Bigr) + \f{\sigma^2 vx}{1 -v^2}N^{1/2} \\
&\qquad\qquad + N^{-1/2}\sup_{-\infty < j \le \sigma^2 Nx/(1 - v^2)}[\mathcal P[\srvv^N](j) - \mathcal P[\arrv^N](j)] 
\\ &\qquad\qquad\qquad\qquad\qquad\qquad 
- N^{-1/2}\sup_{-\infty < j \le 0}[\mathcal P[\srvv^N](j) - \mathcal P[\arrv^N](j)] \\
 &= -N^{-1/2} \mathcal P[\srvv^N]\Bigl(\f{\sigma^2 x}{1 - v^2}N\Bigr)  +\f{\sigma^2 vx}{1 -v^2}N^{1/2} \\
 &\qquad\qquad  
 + \sup_{-\infty < j \le \sigma^2 Nx/(1 - v^2)}[- N^{-1/2} \mathcal P[\arrv^N](j)- (-N^{-1/2}\mathcal P[\srvv^N](j)) ] \\
&\qquad\qquad\qquad
-\sup_{-\infty < j \le 0}[- N^{-1/2} \mathcal P[\arrv^N](j)- (-N^{-1/2}\mathcal P[\srvv^N](j))].
 \end{aligned}
 \ee 
Hence, from \eqref{Prew} and the assumed convergence of $N^{-1/2} \mathcal P[\srvv^N](\f{\sigma^2N}{1 - v^2}\tsp\aabullet)$ in probability \eqref{ipcon}, to prove \eqref{2ptprob}, it suffices to show that, for each $a > 0$ and $\ve > 0$,
\begin{multline}\label{3004}  
\limsup_{N \to \infty}\Pp\biggl(\,\sup_{y \in [-a,a]}\Bigl|\sup_{-\infty < x \le [\sigma^2 N y/(1 - v^2)]}\bigl[- N^{-1/2} \mathcal P[\arrv^N](x)- (-N^{-1/2}\mathcal P[\srvv^N](x )) \bigr] \\
 - \sup_{-\infty < x \le y}[Z^1(x) - Z^2(x)]  \Bigr| > \ve\biggr) = 0.
\end{multline}
For $x \in \R$, recall that $[x]$ denotes the integer closest to $x$ with $|[x]| \le |x|$. Note that there is a drift term for both the walks $\mathcal P[\arrv^N]$ and $\mathcal P[\srvv^N]$ that cancels when they are subtracted. 
For shorthand, let 
\[
X^N(x) = - N^{-1/2} \mathcal P[\arrv^N](x)- (-N^{-1/2}\mathcal P[\srvv^N](x )) 
\]
For the $a$ in the hypothesis of the lemma and arbitrary $S > a$, let $E_{N,a,S}$  be the event where the following three  conditions all hold: 
\begin{enumerate}[label={\rm(\roman*)}, ref={\rm\roman*}]   \itemsep=3pt
\item \[\sup_{-\infty < x \le [-\sigma^2 N a/(1 -v^2)]}\bigl[X^N(x) \bigr] = \sup_{[-\sigma^2 NS/(1 -v^2)]\le x \le [-\sigma^2 Na/(1 - v^2)]}\bigl[X^N(x) \bigr].\]
\item \[
\sup_{-\infty < x \le -a}[Z^1(x) - Z^2(x)] = \sup_{-S \le x \le -a}[Z^1(x) - Z^2(x)].
\]
\item \[\sup_{y \in [-a,a]}\Bigl|\sup_{[-\sigma^2 NS/(1 - v^2)] \le x \le [\sigma^2 N x/(1 - v^2)]} X^N(y) - \sup_{-S \le x \le y}[Z^1 (x) - Z^2(x)] \Bigr| \le \ve.\]
\end{enumerate}
For every $S > a$, the event in \eqref{3004} is contained in $E_{N,a,S}^c$. By assumption \eqref{ipcon} and Lemma \ref{lem:convprob} (applied to the random walk $-\mathcal P[\arrv^N] + \mathcal P[\srvv^N]$ with $m = \dir_2 - \dir_1$, $\sigma = \sqrt{2\sigma^2}$, $\varphi(N) = N^{-1/2}, \xi(N) = \sigma^2 N/(1 - v^2)$, and $B = Z^1(\aabullet/(2\sigma^2)) - Z^2(\aabullet/(2\sigma^2))$),  $\lim_{S \to \infty} \limsup_{N \to \infty} \Pp(E_{N,a,S}^c) = 0$, completing the proof. 
\end{proof}

\begin{remark}\label{rmk:L/R8} 
For $v = 0$, one can alternatively arrive at \eqref{bercoup}  by   considering the speed process for TASEP with left jumps. As in \eqref{speed8}, let  $X_i(t)$ be the position of the right-going particle with label $i$ that starts at $X_i(0)=i$ and define the right-going speed process by $U_i=\lim_{t\to\infty} t^{-1}X_i(t)$.  To flip the space direction, define left-going particles $\wt X_i(t)=-X_{-i}(t)$ and the corresponding speed process $\wt U_i=\lim_{t\to\infty} t^{-1}\wt X_i(t)=-U_{-i}$. Reversing the lattice direction reversed the priorities of the labels (for the walks $\wt X$, lower label means lower priority), so the non-decreasing  projection $F$ to left-going multiclass stationary measures has to be applied to speeds $-\wt U_i=U_{-i}$. 

The distributional equality  $\{\wt U_i\}_{i\in\Z}\deq\{U_{i}\}_{i\in\Z}$ from \cite[Proposition 5.2]{Amir_Angel_Valko11}  implies that both speed processes have the same SH limit. This can also be verified by replacing $U$ with $\wt U$ in \eqref{H138}, rearranging, taking the limit, and using the reflection property (Proposition \ref{prop:SH_cons}\ref{itm:SH_reflinv}) of the SH.

\end{remark}

\singlespacing
\appendix
\chapter{Auxiliary technical inputs} \label{appA}
\section{Maximizers of continuous functions}
Recall the definitions of $f(x,y)$ and $f \li g$ from Section \ref{sec:notat}. 
\begin{lemma} \label{lemma:max_monotonicity}
Let $f,g:\R \to \R$ be continuous functions satisfying $f(x)\vee g(x) \to -\infty$ as $x \to \pm \infty$ and $f \li g$. Let $x_f^L$ and $x_f^R$ be the leftmost and rightmost maximizers of $f$ over $\R$, and similarly defined for $g$. Then,  $x_f^L \le x_g^L$ and $x_f^R \le x_g^R$.
\end{lemma}
\begin{proof}
By the definition of  $x_g^R$   and by the assumption $f \li g$, for all $x > x_g^R$
\[
f(x_g^R,x) \le g(x_g^R,x) < 0.
\]
Hence, the rightmost maximizer of $f$ must be less than or equal to $x_g^R$. We get the statement for leftmost maximizers by considering the functions $x \mapsto f(-x)$ and $g \mapsto g(-x)$.
\end{proof}

\begin{lemma} \label{lem:ext_mont}
Assume that $f,g:\R \to \R$ satisfy $f \li g$. Then, for $a \le x \le y \le b$,
\[
0 \le g(x,y) - f(x,y) \le g(a,b) - f(a,b).
\]
\end{lemma}
\begin{proof}
The first inequality follows immediately from the assumption $f \li g$. The second follows from the inequality
\[
f(a,b) - f(x,y) = f(a,x) + f(y,b) \le g(a,x) + g(y,b) = g(a,b) - g(x,y). \qedhere
\]
\end{proof}

\begin{lemma} \label{lemma:convergence of maximizers from converging functions}
Let $S \subseteq \R^n$, and let  $f_n:S \rightarrow \R$ be  continuous functions that  converge uniformly to   $f:S \rightarrow \R$. Let $c_n$ be a maximizer of $f_n$ and assume $c_n\to c\in S$.   Then $c$ is a maximizer of $f$. 
\end{lemma}
\begin{proof}
$f_n(c_n) \ge f_n(x)$ for all $x \in S$, so  it suffices to show that $f_n(c_n) \rightarrow f(c)$. This follows from the uniform convergence of $f_n$ to $f$, the continuity of $f$, and
\[
|f_n(c_n) - f(c)| \le |f_n(c_n) - f(c_n)| +|f(c_n) - f(c)|. \qedhere
\]
\end{proof}

The following is a well-known fact, but is often stated without proof, so we include full justification for the sake of completeness. For example, it appears as Equation (1.4) in \cite{Pitman1975} and Equation (13) in \cite{brownian_queues}.
\begin{lemma} \label{pitman Representation Lemma}
Let $f:\R \rightarrow \R$ be a continuous function such that 
\[
\lim_{x \rightarrow \pm \infty} f(x) = \pm \infty.
\]
Set $F(y) = \sup_{-\infty < x \le y} f(x)$. Then,
\begin{equation} \label{pitman Representation Equation}
\inf_{y \le x < \infty } (2 F(x) - f(x)) = F(y).
\end{equation}
\end{lemma}
\begin{proof}
The left-hand side of \eqref{pitman Representation Equation} is
\[
\inf_{y \le x < \infty} (2F(x) - f(x)) = \inf_{y \le x < \infty} \bigl(2\sup_{-\infty < u \leq x} f(u) - f(x)\bigr).
\]
For all $x \geq y$, $\sup_{-\infty < u \le x} f(u)$ is greater than or equal to both $f(x)$ and $\sup_{-\infty < u \le y} f(u)$. Therefore,
\[
2\sup_{-\infty < u \leq x} f(u) - f(x)\geq \sup_{-\infty < u \leq y} f(u)+ f(x) - f(x) = F(y).
\]
 This establishes one direction of \eqref{pitman Representation Equation}. To show the other direction, we show that there exists $x \geq y$ such that 
\[
2F(x) - f(x) = F(y). 
\]
Note that $f(y) \leq F(y)$ recall the assumption $\lim_{x \rightarrow \infty} f(x) = \infty$. By the intermediate value theorem,  $f(x) = F(y)$ for some $x \geq y$. Let 
\[
x^\star = \inf\{x \geq y: f(x) = F(y)\}.
\]
Then, by continuity, $f(x^\star) = F(x^\star) = F(y)$, and
\[
2F(x^\star) - f(x^\star) = 2F(y) - F(y) = F(y). \qedhere
\]
\end{proof}

\section{Miscellaneous probabilistic facts}
\begin{lemma} \label{Identity for Laplace Transform}
For a nonnegative random variable $Z$ on $(\Omega,\F,\Pp)$ and $\alpha \in \R$, the following holds 
\[
\int_\Omega e^{-\alpha Z}\, d\Pp  = 1 - \alpha \int_0^\infty e^{-\alpha z}\Pp(Z > z)\, dz. 
\]
\end{lemma}
\begin{proof}
We use Tonelli's Theorem below.
\begin{align*}
    &-\alpha \int_0^\infty  e^{-\alpha z} \Pp(Z > z)\, dz = -\alpha \int_0^\infty \int_{\Omega}e^{-\alpha z}1(Z > z)\, d\Pp dz \\
    = &-\alpha \int_{\Omega}\int_0^\infty  e^{-\alpha z} 1(Z >z)\, dzd\Pp = -\alpha \int_\Omega \int_0^Z  e^{-\alpha z}\, dzd\Pp \\
    = &\int_\Omega (e^{-\alpha Z} - 1)\, d\Pp = \int_\Omega e^{-\alpha Z}d\Pp - 1.
\end{align*}
Rearranging this equation completes the proof.
\end{proof}

\noindent The following is a classical result. We provide a proof for the specific formulation we use. 
\begin{theorem}[The Cameron-Martin-Girsanov Theorem] \label{cameron-martin}
Let $X:[0,\infty) \to \R$ be a standard Brownian motion with drift $\lambda \in \R$ on the space $(\Omega,\Ff,\Pp)$. Then, for $t > 0$ and $\rho \in \R$ consider the measure on $(\Omega,\Ff)$ defined by
\[
\wh \Pp(A) = \E\Bigl[\exp\Bigl((\rho - \lambda) X(t) + \f{\lambda^2 - \rho^2}{2}t \Bigr) \ind_A \Bigr].
\]
Then, under $\wh \Pp$, $\{X(s):0\le s \le t\}$ is a standard Brownian motion with drift $\rho$. 
\end{theorem}
\begin{proof}
The measures $\Pp$ and $\wh \Pp$ are mutually absolutely continuous, so continuity is preserved. It remains to show that, under $\wh \Pp$, for $0 = t_0 < t_1 < \cdots < t_{n - 1} < t_n = t$, $X(t_1),X(t_2) - X(t_1),\ldots,X(t_n) - X(t_{n - 1})$ are independent Gaussian random variables with mean $\rho(t_i - t_{i - 1})$ and variance $t_i - t_{i - 1}$. We prove this via moment generating functions: For $1 \le i \le n$, set $s_i = t_{i} - t_{i - 1}$ and $Y_i = X(t_i) - X(t_{i - 1})$. Let $\alpha_1,\ldots,\alpha_n \in \R$. Then, using the fact that $t = s_1+ \cdots + s_n$ and $X(t) = Y_1 + \cdots + Y_n$,
\begin{align*}
    &\quad \;\wh \Ee\Bigl[\exp\Bigl(\sum_{i = 1}^n \alpha _i Y_i\Bigr)     \Bigr] = \Ee\Bigl[\exp\Bigl((\rho - \lambda)(Y_1 + \cdots + Y_{n}) + \f{\lambda^2 - \rho^2}{2} t 
 + \sum_{i = 1}^n \alpha _i Y_i\Bigr) \Bigr)    \Bigr] \\
    &= \int_{\R^n} \prod_{i = 1}^n \f{1}{\sqrt{2\pi s_i}} \exp\bigl(-\f{(y_i - \lambda s_i)^2}{2s_i} + (\rho - \lambda) y_i + \alpha_i y_i +\f{\lambda^2 - \rho^2}{2} s_i\bigr) \,dy_i \\
    &= \prod_{i = 1}^n \int_\R \f{1}{\sqrt{2\pi s_i}} \exp\bigl(-\f{(y_i - \lambda s_i)^2}{2s_i} + (\rho - \lambda) y_i + \alpha_i y_i +\f{\lambda^2 - \rho^2}{2} s_i\bigr) \,dy_i \\
    &= \prod_{i = 1}^n \exp\bigl(\lambda s_i(\rho - \lambda + \alpha_i) + \f{s_i(\rho - \lambda + \alpha_i)^2}{2} + \f{\lambda^2 - \rho^2}{2}s_i  \Bigr) = \prod_{i = 1}^n \exp(\rho \alpha_i s_i + s_i\alpha_i^2/2). \;\;\qedhere
\end{align*}
\end{proof}
\noindent The following is derived from a classical result. 
\begin{lemma}\cite[Equation 1.1.4 (1) on pg 251 ]{BM_handbook} \label{lemma:sup of BM with drift}
For a standard Brownian motion $B$ and $\dir > 0$,
\[
\sup_{-\infty < u \le 0}\{\sqrt 2 B(u) + \dir u\} \sim \operatorname{Exp}(\dir).
\]
\end{lemma}
\begin{proof}
The statement given in \cite{BM_handbook} is 
\[
\sup_{0 \le u \le \infty}\{B(u) - \dir u\} \sim \operatorname{Exp}(2\dir).
\]
To get to the result we use, observe that 
\[
\sup_{-\infty < u \le 0}\{\sqrt 2 B(u) + \dir u\} = \sqrt 2\sup_{0 \le u < \infty}\bigl\{B(-u) - \f{\dir}{\sqrt 2} u\bigr\} \sim \sqrt 2 \operatorname{Exp}(\sqrt 2 \dir) \sim \operatorname{Exp}(\dir). \qedhere
\]
\end{proof}

\noindent The following is also well-known (see, for example, \cite[Equation 1.2.4 on page 251]{BM_handbook}; this result is stated for the running infimum, but can be modified by reflection invariance of Brownian motion). We provide a short proof for the specific choice of parameterization we use 
\begin{lemma} \label{lem:BM_drift_sup}
    Let $B$ be a standard Brownian motion. Then, for all $x  \ge 0$,
    \[
    \Pp\bigl(\sup_{0 \le u \le 1}\{\sqrt 2 B(u) + \dir u\} \le x\bigr)
=  \Phi\Bigl(\frac{x - \dir}{\sqrt{2}}\Bigr) - e^{\dir x}\Phi\Bigl(\frac{-x - \dir}{\sqrt{2}}\Bigr).
\]
\end{lemma}
\begin{proof}
    We prove that 
    \be \label{xisup}
    \Pp\bigl(\sup_{0 \le u \le 1}\{ B(u) + \dir u\} \le x\bigr)
=  \Phi(x - \dir) - e^{2 \dir x}\Phi(-x - \dir),
    \ee
    and the statement given follows by replacing $x$ with $x/\sqrt 2$ and $\dir$ with $\dir/\sqrt 2$. Let 
    \[
    W^\dir(u) = B(u) + \dir u, \quad  M^\dir = \sup_{0 \le u \le 1}\{W^\dir(u)\},
    \]
    and denote $M = M^0$ and $W = W^0$. A classical fact due to the reflection principle (see for example, \cite[Exercise 7.4.3]{Durrett}), is that $M$ and $W(1)$ have joint density
    \[
    f(m,w) = \f{2(2m - w)}{\sqrt{2\pi}}e^{-(2m - w)^2/2} \ind(m \ge 0, -\infty < w \le m).
    \]
    By Theorem \ref{cameron-martin}, $M^\dir$ and $W^\dir(1)$ have joint density
    \[
    f^\dir(m,w) = \f{2(2m - w)}{\sqrt{2\pi}}e^{-(2m - w)^2/2 + \dir w - \dir^2/2} \ind(m > 0, -\infty < w < m).
    \]
    Hence,
    \begin{align*}
        \Pp(M^\dir \le x) &= \int_{-\infty}^x \int_w^x \f{2(2m - w)}{\sqrt{2\pi}}e^{-(2m - w)^2/2 + \dir w - \dir^2/2} dm\,dw \\
        &= \int_{-\infty}^x \int_{w}^{2x - w} \f{u}{\sqrt{2\pi}} e^{-u^2/2 + \dir w - \dir^2/2}\,du\,dw \\
        &= \int_{-\infty}^x \f{e^{\dir w - \dir^2/2}}{\sqrt{2\pi}} \Bigl[e^{-w^2/2} - e^{-(2x - w)^2/2}  \Bigr]\,dw \\
        &= \int_{-\infty}^x \f{e^{-(w - \dir)^2/2}}{\sqrt{2\pi}}\,dw - e^{2x \dir}\int_{-\infty}^x \f{e^{-(w - (2x + \dir))^2/2}}{\sqrt{2\pi}}\,dx,
    \end{align*}
    and this equals \eqref{xisup} after a simple change of variables. 
\end{proof}

\section{Maximum of random walk }
We first state a random walk lemma that comes from p.~519--520 in \cite{resnick}. 
See also Chapter VIII, Section 6 in \cite{Asmussen-1987}. 

\begin{lemma}
[\cite{resnick}, Proposition 6.9.4. See also Chapter VIII, Section 6 in \cite{Asmussen-1987}] 
\label{lem:resnick}
Let $\mu_N$ be a sequence of strictly positive numbers with $\mu_N \to 0$. Let $\sigma_N$ be a sequence satisfying $\sigma_N \to 1$. For each $N$, let $\{X_{N,i}:i \in \Z\}$ be a collection of i.i.d. random variables with mean $\mu_N$ and variance $\sigma_N^2$. Further, suppose that the sequence $\{X_{N,0}^2:N \ge 1\}$ is uniformly integrable, meaning that 
\[
\lim_{b \to \infty} \sup_{N \ge 1} \mathbb E\bigl[ X_{N,0}^2 \ind_{|X_{N,0}| > b} \bigr] = 0.
\]

Let $B$ be a Brownian motion with diffusion coefficient $1$ and zero drift. Then, the following convergence in distribution holds: 
\[
\sup_{-\infty < x \le 0} \mu_N S^N(\lceil x\rceil) \underset{N \to \infty}{\Longrightarrow} \sup_{-\infty < x \le 0} \{B(x) + x\}.
\]
\end{lemma}

Next, we reformulate the statement slightly to suit our purposes. 
 
\begin{lemma} \label{lem:genres}
Let $\mu_N$ be a sequence of strictly positive numbers with $\mu_N \to 0$. Let $\sigma_N$ be a sequence satisfying $\sigma_N \to \sigma > 0$. Let $\varphi(N)$ be a sequence satisfying $\mu_N/\varphi(N) \to m > 0$. For each $N$, let $\{X_{N,i}:i \in \Z\}$ be a collection of i.i.d. random variables with mean $\mu_N$ and variance $\sigma_N^2$. Further, suppose that the sequence $\{X_{N,0}^2:N \ge 1\}$ is uniformly integrable. Let $S^N(m)$ be defined as
\be \label{SNm}
S^N(m) = \begin{cases}
 -\sum_{i = m}^{-1} X_{N,i} &m \le 0 \\
  \sum_{i = 0}^{m - 1} X_{N,i} &m \ge 0
\end{cases}
\ee
with   $S^N(0) = 0$. Let $B$ be a Brownian motion with diffusion coefficient $1$ and zero drift. Then, the following convergence in distribution holds:
\be \label{convBMgen}
\sup_{-\infty < x \le 0} \varphi(N) S^N(\ce{x}) \underset{N \to \infty}{\Longrightarrow} \sup_{-\infty < x \le 0} \{\sigma B(x) + m x\}
\ee
\end{lemma}
\begin{remark}
It is immediate that on the left-hand side of \eqref{convBMgen}, one can replace $x$ with $\lceil \xi(N) x\rceil$ for any strictly positive sequence $\xi(N)$.  
\end{remark}
\begin{proof}
The random variables $X_{N,i}/\sigma_N$ satisfy the conditions of Lemma \ref{lem:resnick} with $\mu_N$ replaced by $\mu_N/\sigma_N$. Then,
\begin{align*}
&\,\sup_{-\infty < x \le 0} \varphi(N) S^N(x)
= \f{\varphi(N)\sigma_N^2}{ \mu_N} \sup_{-\infty < x \le 0} \f{\mu_N}{\sigma_N}\cdot \f{S^N(x)}{\sigma_N} \\
&= \f{\varphi(N)\sigma_N^2}{ \mu_N} \sup_{-\infty < x \le 0} \f{\mu_N}{\sigma_N} \cdot \f{S^N(\lceil \sigma_N^2 x/\mu_N^2\rceil)}{\sigma_N} 
\ \overset{\rm Lemma \ref{lem:resnick}}{\Longrightarrow} \  \f{\sigma^2}{m} \sup_{-\infty < x \le 0}\{B(x) + x\} \\
&= \f{\sigma^2}{m}\sup_{-\infty < x \le 0}\Bigl\{B\Bigl(\f{m^2}{\sigma^2}x\Bigr) + \f{m^2}{\sigma^2}x\Bigr\} 
\deq \sup_{-\infty < x \le 0}\{\sigma B(x) + m x\},
\end{align*}
so we may now apply Lemma \ref{lem:resnick}.
\end{proof}

For $x \in \R$, recall that $[x]$ denotes the integer closest to $x$ with $|[x]| \le |x|$. 

\begin{lemma}\label{lem:convprob}
Consider the setting of Lemma \ref{lem:genres}. Let $\xi(N)$ be a sequence satisfying $\varphi(N)^2 \xi(N) \to R > 0$. Then, for each $S < T \in \R$,
\be \label{rwconv}
\begin{aligned}
&\,\lim_{N \to \infty} \Pp\Bigl[\,\sup_{-\infty < x \le [S \xi(N)]} \varphi(N) S^N(x) > \sup_{[S \xi(N)] \le x \le [T \xi(N)]} \varphi(N) S^N(x)\Bigr] \\
&\qquad\qquad 
= \Pp\bigl[\sup_{-\infty < x \le S}\{\sigma B(R x) + mRx\} > \sup_{S \le x \le T}\{\sigma B(Rx) + mRx\}\bigr].
\end{aligned}
\ee

\end{lemma}
\begin{proof}
Observe that 
\begin{align*}
&\Pp\Bigl[\,\sup_{-\infty < x \le [S \xi(N)]} \varphi(N) S^N(x) > \sup_{[S \xi(N)] \le x \le [T \xi(N)]} \varphi(N) S^N(x)\Bigr] \\
&= \Pp\Bigl[\,\sup_{-\infty < x \le [S \xi(N)]} \varphi(N) S^N([S\xi(N)],x) > \sup_{[S \xi(N)] \le x \le [T \xi(N)]} \varphi(N) S^N([S\xi(N)],x)\Bigr].
\end{align*}
Now, note that \[
\sup_{-\infty < x \le [S \xi(N)]} \varphi(N) S^N([S\xi(N)],x),\quad \text{and} \quad \sup_{[S \xi(N)] \le x \le [T \xi(N)]} \varphi(N) S^N([S\xi(N)],x)\]
are independent. By convergence of random walk to Brownian motion with drift (with respect to  the topology of uniform convergence on compact sets) , we get that 
\[
\sup_{[S \xi(N)] \le x \le [T \xi(N)]} \varphi(N) S^N(RS,[S\xi(N)],x) \Longrightarrow \sup_{S \le x \le T} \{\sigma B(RS, Rx) + R(m-S)x\}.
\]
By shift invariance of random walk and Lemma \ref{lem:genres},
\begin{align*}
&\,\sup_{-\infty < x \le [S \xi(N)]} \varphi(N) S^N([S\xi(N)],x) \deq \sup_{-\infty < x \le 0} \varphi(N) S^N(x)  \\
&\Longrightarrow \sup_{-\infty < x \le 0}\{\sigma B(x) + mx\} = \sup_{-\infty < x \le 0} \{\sigma B(Rx) + mRx\} \\
&\deq \sup_{-\infty < x \le S} \{\sigma B(RS,Rx) + m(R-S)x\}. 
\end{align*}
By independence, we have shown the following joint convergence:
\be \label{jc}
\begin{aligned}
&\, \Bigl(\sup_{-\infty < x \le [S \xi(N)]} \varphi(N) S^N([S \xi(N)],x), \sup_{[S \xi(N)] \le x \le [T \xi(N)]} \varphi(N) S^N([S \xi(N)],x)\Bigr)  \\
&\Longrightarrow \Bigl(\sup_{-\infty < x \le S}\{\sigma B(RS,R x) + m(R-S)x\}, \sup_{S \le x \le T}\{\sigma B(RS,Rx) + m(R-S)x\}\Bigr).
\end{aligned}
\ee

The right-hand side of \eqref{jc} consists of two independent random variables with continuous distribution. Therefore, 
\begin{align*}
&\quad \lim_{N \to \infty} \Pp\Bigl[\,\sup_{-\infty < x \le [S \xi(N)]} \varphi(N) S^N(x) > \sup_{[S \xi(N)] \le x \le [T \xi(N)]} \varphi(N) S^N(x)\Bigr] \\
&= \lim_{N \to \infty}\Pp\Bigl[\,\sup_{-\infty < x \le [S \xi(N)]} \varphi(N) S^N([S \xi(N)],x) > \sup_{[S \xi(N)] \le x \le [T \xi(N)]} \varphi(N) S^N([S \xi(N)] ,x)\Bigr] \\
& =  \Pp\bigl[\sup_{-\infty < x \le S}\{\sigma B(RS,R x) + m(R-S)x\} > \sup_{S \le x \le T}\{\sigma B(Rx) + m(R-S)x\}\bigr]
\\ &
= \Pp\bigl[\sup_{-\infty < x \le S}\{\sigma B(R x) + mRx\} > \sup_{S \le x \le T}\{\sigma B(Rx) + mRx\}\bigr],
\end{align*}
with the second equality holding because the event on the right-hand side is a continuity set for the joint vector on the right in \eqref{jc}. 
\end{proof}

\chapter{The directed landscape} \label{appB}
\section{Bounds and distributional invariances}
The directed landscape satisfies the following symmetries.
\begin{lemma}\cite[Lemma 10.2]{Directed_Landscape},\cite[Proposition 1.23]{Dauvergne-Virag-21} \label{lm:landscape_symm}
As a random continuous function of $(x,s;y,t) \in \Rup$, the directed landscape $\Ll$ satisfies the following distributional symmetries, for all  $r,c \in \R$ and $q > 0$.
\begin{enumerate} [label=\rm(\roman{*}), ref=\rm(\roman{*})]  \itemsep=4pt
    \item {\rm(Space-time stationarity)}  \label{itm:time_stat} \ \ $\Ll(x,s;y,t) \deq \Ll(x+c,s + r;y+c,t + r).
    $
    \item {\rm(Skew stationarity)} \label{itm:skew_stat}
    \ \ $
    \Ll(x,s;y,t) \deq \Ll(x + cs,s;y + ct,t) -2c(x - y) + (t- s)c^2.  
    $
    \item \label{itm:DL_reflect} {\rm(Spatial and temporal reflections)} 
    \ \ $
    \Ll(x,s;y,t) \deq \Ll(-x,s;-y,t) \deq \Ll(y,-t;x,-s).
    $
    \item \label{itm:DL_rescaling} {\rm(Rescaling)} 
    \ \ $
    \Ll(x,s;y,t) \deq q\Ll(q^{-2}x,q^{-3}s;q^{-2}y,q^{-3}t).
    $
\end{enumerate}
\end{lemma}

\begin{lemma}\cite[Corollary 10.7]{Directed_Landscape}\label{lem:Landscape_global_bound}
There exists a \\random constant $C$ such that for all $v = (x,s;y,t) \in \Rup$, we have 
\[
\Bigl|\Ll(x,s;y,t) + \f{(x - y)^2}{t - s}\Bigr| \le C (t - s)^{1/3} \log^{4/3} \Bigl(\f{2(\|v\| + 2)}{t - s}\Bigr)\log^{2/3}(\|v\| + 2),
\]
where $\|v\|$ is the Euclidean norm.
\end{lemma}
\begin{lemma}\cite[Proposition 2.6]{Dauvergne-22}\label{lem:DL_erg_coup}
For every $i = 1,\ldots,k$ and $\ve > 0$, let 
\be \label{Kie}
K_{i,\ve} = \{(x,s;y,t) \in \Rup: s,t \in [0,\ve],x,y \in [i -1/4,i + 1/4]\}.
\ee
Then, there exists a coupling of $k + 1$ copies of the directed landscape $\Ll_0,\Ll_1,\ldots,\Ll_k$ so that $\Ll_1,\ldots,\Ll_k$ are independent, and almost surely, there exists a random $\ve > 0$ such that for $1 \le i \le k$, $\Ll_0|_{K_{i,\ve}} = \Ll_i|_{K_{i,\ve}}$.
\end{lemma}

On a measure space $(\Omega,\F,\Pp)$, a measure-preserving transformation $T$ satisfies $T^{-1}E \in \F$ and $\Pp(T^{-1}E) = \Pp(E)$ for all $E \in \F$. Such a transformation is said to be \textit{ergodic} if $\Pp(E) \in\{0,1\}$ whenever $T^{-1}E = E$. The transformation $T$ is said to be \textit{mixing} if, for all $A,B \in \F$, $\Pp(A \cap T^{-k}B) \to \Pp(A) \Pp(B)$ as $k \to \infty$. By setting $A = B$, one sees that mixing implies ergodicity. 
\begin{lemma} \label{lm:horiz_shift_mix}
For $a,b \in \R$, not both $0$ and $z >0$, consider the shift operator $T_{z;a,b}$ acting on the directed landscape $\Ll$ as
\[
T_{z;a,b} \Ll(x,s;y,t) = \Ll(x + az,s + bz;y + az;t + bz),
\]
where both sides are understood as a process on $\Rup$.  Then, $\Ll$ is mixing under this transformation. That is, for all Borel subsets $A,B$ of the space $C(\Rup,\R)$,
\[
\Pp(\Ll \in A, T_{z;a,b}\Ll \in B) \overset{z \to \infty}{\longrightarrow} \Pp(\Ll \in A)\Pp(\Ll \in B).
\]
In words, the directed landscape is mixing (and therefore ergodic) under space-time shifts in any planar direction.
\end{lemma}
\begin{proof}
This key to the proof is Lemma \ref{lem:DL_erg_coup}, and I thank Duncan Dauvergne for pointing this out to me. We prove the case $a \neq 0$, and the case $a = 0$ and $b \neq 0$ will be proven separately. Further, since we send $z \to \infty$, it suffices to show the result for $a = 1$, for then the result also holds for arbitrary $a$ and $b = ab$. With this simplification, we use the shorthand notation $T_{z;b} = T{z:1,b}$. By Lemma \ref{lm:landscape_symm}\ref{itm:time_stat}, $\Ll$ is stationary under the shift $T_{z;b}$. By Dynkin's $\pi$-$\lambda$ theorem, it suffices to show that for a compact set $K\subseteq \Rup$ and Borel sets $A\subseteq C(K,\R)$ and $B \subseteq C(K,\R)$,  
\[
\Pp(\Ll|_{K} \in A,(T_{z;b}\Ll)|_{K} \in B) \overset{z \to \infty}{\longrightarrow} \Pp(\Ll|_{K} \in A)\Pp(\Ll|_{K} \in B).
\]
Further, by temporal stationarity, it suffices to assume that $\inf\{s: (x,s;y,t) \in K\} \ge 0$. Consider the coupling $\Ll_0,\Ll_1,\Ll_2$ of Lemma \ref{lem:DL_erg_coup} with $k = 2$. Then, using the rescaling and spatial stationarity of Lemma \ref{lm:landscape_symm}, 
\begin{align}
   &\quad\;\Pp(\Ll|_{K} \in A,(T_{z;b}\Ll)|_{K} \in B) \nonumber \\
   &= \Pp(\Ll_0(x,s;y,t)|_{K} \in A, \Ll_0(x + z,s + bz;y + z,t + bz)|_{K} \in B)\nonumber \\
   &= \Pp(z^{1/2}\Ll_0(z^{-1} x,z^{-3/2}s  ;z^{-1} y,z^{-3/2} t  )|_{K} \in A, \nonumber \\ 
   &\qquad\qquad z^{1/2}\Ll_0(z^{-1} x + 1,z^{-3/2}s + z^{-1/2} b;z^{-1} y + 1,z^{-3/2} t + z^{-1/2}b)|_{K} \in B) \nonumber \\
  &= \Pp(z^{1/2}\Ll_0(z^{-1} x + 1,z^{-3/2}s  ;z^{-1} y + 1 ,z^{-3/2} t )|_{K} \in A, \nonumber \\
  &\qquad\qquad z^{1/2}\Ll_0(z^{-1} x + 2,z^{-3/2}s + z^{-1/2}b  ;z^{-1} y + 2,z^{-3/2} t + z^{-1/2}b  )|_{K} \in B). \label{433}
\end{align}
Specifically, we used the rescaling property with $q = z^{1/2}$ in the second equality, and in the third equality, we shifted the entire process by $1$ in the spatial direction.
Above, the restrictions $|_{K_j}$ mean that $(x,s;y,t) \in K_j$.  
Since $K$ is compact and we assumed $s \ge 0$ for all $(x,s;y,t) \in K$, for any $\ve > 0$, there exists $Z > 0$ such that for $z > Z$,
\be \label{small_cpct}
\begin{aligned}
  &\{(z^{-1} x + 1,z^{-3/2}s;z^{-1} y + 1,z^{-3/2} t): (x,s;y,t) \in K\} \subseteq   K_{1,\ve}, \quad \text{and}\\
  &\{(z^{-1} x + 2,z^{-3/2}s + z^{-1/2}b;z^{-1} y + 2,z^{-3/2} t + z^{-1/2}b): (x,s;y,t) \in K\} \subseteq K_{2,\ve},
  \end{aligned}
\ee
where $K_{i,\ve}$ are defined in \eqref{Kie}. Let $C_z$ be the event where both these containments hold for the random $\ve > 0$ in Lemma \ref{lem:DL_erg_coup}, and let $D_z$ be the event in \eqref{433}. Then, $\Pp(C_z) \to 1$ as $z \to +\infty$. Then, for any $\delta > 0$, whenever $z$ is sufficiently large so that $1 - \Pp(C_z) = \delta > 0$,
\begin{align*}
    &\quad \;\Big|\Pp(D_z) - \Pp(\Ll|_{K} \in A)\Pp(\Ll|_{K} \in B)\Big| \\
    &\le \Big|\Pp(D_z \cap C_z) -\Pp(\Ll|_{K} \in A)\Pp(\Ll|_{K} \in B)\Big| + \delta \\
    &=\Big|\Pp(z^{1/2}\Ll_1(z^{-1} x + 1,z^{-3/2}s;z^{-1} y + 1,z^{-3/2} t)|_{K} \in A,  \\
  &\qquad\qquad z^{1/2}\Ll_2(z^{-1} x + 2,z^{-3/2}s + z^{-1/2}b;z^{-1} y + 2,z^{-3/2} t + z^{-1/2}b)|_{K} \in B, C_z) \\
  &\qquad\qquad\qquad- \Pp(\Ll_1|_{K} \in A)\Pp(\Ll_1|_{K} \in B)\Big| + \delta \\
  &\le \Big|\Pp(z^{1/2}\Ll_0(z^{-1} x + 1,z^{-3/2}s;z^{-1} y + 1,z^{-3/2} t)|_{K} \in A,  \\
  &\qquad\qquad z^{1/2}\Ll_0(z^{-1} x + 2,z^{-3/2}s + z^{-1/2}b;z^{-1} y + 2,z^{-3/2} t + z^{-1/2}b)|_{K} \in B) \\
  &\qquad\qquad\qquad- \Pp(\Ll_1|_{K} \in A)\Pp(\Ll_2|_{K} \in B)\Big| + 2\delta \\
  &= \Big|\Pp(\Ll_1|_{K} \in A, \Ll_2|_{K} \in B) - \Pp(\Ll_1|_{K} \in A)\Pp(\Ll_2|_{K} \in B)\Big| + 2\delta = 2\delta,
\end{align*}
completing the proof since $\delta$ is arbitrary. Specifically, in the two inequalities, we added and removed the event $C_z$ at the cost of $\delta$. In the first equality above, we used the fact that the containments \eqref{small_cpct} hold on $C_z$ and  $\Ll_0|_{K_i,\ve} = \Ll_i|_{K_i,\ve}$ for $i = 1,2$. In the last line, we reversed the application of the rescaling and spatial stationarity, then used the independence of $\Ll_1$ and $\Ll_2$ from Lemma \ref{lem:DL_erg_coup}.

The statement for the vertical shift operator when $a = 0$ is simpler. For a compact set $K$, the processes $\Ll|_{K}$ and $T_{z;0,b} \Ll|_{K}$ are independent for sufficiently large $b$ by the independent increments property of the directed landscape, and the desired result follows. 
\end{proof}


Recall the definition of the state space $\UC$ \eqref{UCdef} for the KPZ fixed point. Recall that the KPZ fixed point $h_\Ll(t,\abullet;s,\h)$ with initial data $\h$ at time $s$ can be represented as
\be \label{992}
h_\Ll(t,y;s,\h) = \sup_{x \in R}\{\h(x) + \Ll(x,s;y,t)\} \qquad\text{for }t > s.
\ee

\begin{lemma}\cite{Directed_Landscape,Basu-Ganguly-Hammond-21,Ganguly-Hegde-2021,Pimentel-21b} \label{lem:DL_crossing_facts}
Let $\Ll:\Rup \to \R$ be a continuous function satisfying the metric composition law \eqref{eqn:metric_comp} and such that maximizers in \eqref{eqn:metric_comp} exist. Then,
\begin{enumerate} [label=\rm(\roman{*}), ref=\rm(\roman{*})]  \itemsep=3pt
    \item \label{itm:DL_crossing_lemma} Whenever $s < t$, $x_1 < x_2$, $y_1 < y_2$,
\[
\Ll(x_2,s;y_1,t) - \Ll(x_1,s;y_1,t) \le \Ll(x_2,s;y_2,t) - \Ll(x_1,s;y_2,t).
\]
\end{enumerate} 
Let $\h^1,\h^2 \in \UC$, and For $i = 1,2$ and $t > s$, define $h_\Ll(t,y;s,\h^i)$ by \eqref{992}. 
Then, assuming that maximizers in \eqref{992} exist, the following hold.
\begin{enumerate}[resume, label=\rm(\roman{*}), ref=\rm(\roman{*})]  \itemsep=3pt
    \item \label{itm:KPZ_attractiveness} If $\h^1 \li \h^2$, then $h_\Ll(t,\aabullet;s,\h^1)\li h_\Ll(t,\aabullet;s,\h^2)$ for all $t > s$.
    \item \label{itm:KPZ_crossing_lemma} For $t > s$ and $i = 1,2$, set 
\[
Z_\Ll(t,y;s,\h^i) = \max \argmax_{x \in \R}\{\h^i(x) + \Ll(x,s;y,t)\}.
\]
Then, if $x < y$ and $Z_\Ll(t,y;s,\h^1) \le Z_\Ll(t,x;s,\h^2)$,
\[
h_\Ll(t,y;s,\h^1) - h_\Ll(t,x;s,\h^1) \le h_\Ll(t,y;s,\h^2) - h_\Ll(t,x;s,\h^2).
\]
\end{enumerate}
\end{lemma}

\section{Geodesics in the directed landscape}
 We start by citing some results from \cite{Bates-Ganguly-Hammond-22} and \cite{Dauvergne-Sarkar-Virag-2020}. 
\begin{lemma}\cite[Theorem 1.18]{Bates-Ganguly-Hammond-22} \label{lm:BGH_disj}
There exists a single event of full probability on which, for any compact set $K \subseteq \Rup$, there is a random $\ve > 0$ such that the following holds. If $v_1 = (x,s;y,u) \in K$ and $v_2 = (z,s;w,u) \in K$ admit geodesics $\gamma_1$ and $\gamma_2$ satisfying $|\gamma_1(t) - \gamma_2(t)| \le \ve$ for all $t \in [s,u]$, then $\gamma_1$ and $\gamma_2$ are not disjoint, i.e., $\gamma_1(t) = \gamma_2(t)$ for some $t \in [s,u]$. 
\end{lemma}

Let $g$ be a geodesic from $(x,s)$ to $(y,u)$. Define the graph of this geodesic as
\[
\graph g := \{(g(t),t): t \in [s,u]\}.
\]

\begin{lemma}\cite[Lemma 3.1]{Dauvergne-Sarkar-Virag-2020} \label{lem:precompact}
The following holds on a single event of full probability. Let $(p_n;q_n) \to (p,q) = (x,s;y,t) \in \Rup$, and let $g_n$ be any sequence of geodesics from $p_n$ to $q_n$. Then, the sequence of graphs $\graph g_n$ is precompact in the Hausdorff metric, and any subsequential limit of $\graph g_n$ is the graph of a geodesic from $p$ to $q$.  
\end{lemma}

\begin{lemma}\cite[Lemma 3.3]{Dauvergne-Sarkar-Virag-2020} \label{lem:overlap}
The following holds on a single event of full probability. Let $(p_n;q_n) = (x_n,s_n;y_n,u_n) \in \Rup \to (p;q) = (x,s;y,u) \in \Rup$, and let $g_n$ be any sequence of geodesics from $p_n$ to $q_n$. Suppose that either
\begin{enumerate} [label=\rm(\roman{*}), ref=\rm(\roman{*})]  \itemsep=3pt
    \item \label{uniqn} For all $n$, $g_n$ is the unique geodesic from $(x_n,s_n)$ to $(y_n,u_n)$ and $\graph g_n \to \graph g$ for some geodesic $g$ from $p$ to $q$, or
    \item \label{uniqueg} There is a unique geodesic $g$ from $p$ to $q$.
    \end{enumerate}
    Then, the \textbf{overlap}
    \[
    O(g_n,g) := \{t \in [s_n,u_n]\cap [s,u]: g_n(t) = g(t)\}    
    \]
    is an interval for all $n$ whose endpoints converge to $s$ and $u$. 
\end{lemma}
\begin{remark}
We note that condition \ref{uniqn} is slightly different from that stated in \cite{Dauvergne-Sarkar-Virag-2020}. There, it is assumed instead that $(x_n,s_n;y_n,u_n) \in \Q^4 \cap \Rup$ for all $n$. The only use of this requirement in the proof is to ensure that there is a unique geodesic from $(x_n,s_n)$ to $(y_n,u_n)$ for all $n$, so there is no additional justification needed for the statement we use here. 
\end{remark}

\begin{lemma} \label{lem:geod_pp}
On the intersection of the full probability events from Lemmas \ref{lem:precompact} and \ref{lem:overlap}, the following holds. For all ordered triples $s < t < u$ and compact sets $K \subseteq \R$, the set  
\be \label{distinct}
\{g(t): g \text{ is the unique geodesic between }(x,s) \text{ and }(y,u) \text{ for some } x \in K, y \in K\}
\ee
is finite. 
\end{lemma}
Lemma \ref{lem:geod_pp} is known. Its derivation from  Lemma \ref{lm:BGH_disj} and some results of \cite{Dauvergne-Sarkar-Virag-2020} are  shown in \cite{Busa-Sepp-Sore-22arXiv}.  Lemma 3.12 in \cite{Ganguly-Zhang-2022a} (posted after the first version of \cite{Busa-Sepp-Sore-22arXiv})  provides a stronger quantitative statement, but we do not need it for our purposes. This stronger estimate can be traced back to the work of \cite{SlyNonexistenceOB} in exponential LPP using integrable methods.
\begin{proof}[Proof of Lemma \ref{lem:geod_pp}]
Assume, without loss of generality, that $K$ is a closed interval $[a,b]$. We observe that by planarity, all geodesics from $(x,s)$ to $(y,u)$ for $x,y \in K$ lie between the leftmost geodesic from $(a,s)$ to $(a,u)$ and the rightmost geodesic from $(b,s)$ to $(b,u)$. Hence, the set \eqref{distinct} is contained in a compact set, and it suffices to show that the set has no limit points.

Assume, to the contrary, that there exists a point $(\hat x,s;\hat y,u) \in (K \times \{s\}\times K \times \{u\}) \cap \Rup$ with unique geodesic $\hat g$ such that there exists a sequence of points $x_n,y_n \in K$ such that for all $n$, the geodesic $g_n$ from $(x_n,s)$ to $(y_n,t)$ is unique and so that $g_n(t) \to \hat g(t)$ but $g_n(t) \neq \hat g(t)$ for all $n$.
By compactness, there exists a convergent subsequence $(x_{n_k},y_{n_k}) \to (x,y)$. By Lemma \ref{lem:precompact}, there exists a further subsequence $(x_{n_{k_\ell}},y_{n_{k_\ell}})$ such that the geodesic graphs $\G g_{n_{k_\ell}}$ converge to the graph of some geodesic $\G g$ from $(x,s)$ to $(y,u)$ in the Hausdorff metric. Since $g_n(t) \to \hat g(t)$, we have $g(t) = \hat g(t)$. By Lemma \ref{lem:overlap}, the overlap $O(g_{n_{k_\ell}},g)$ is an interval whose endpoints converge to $s$ and $u$, so $g_{n_{k_\ell}}(t) = g(t) = \hat g(t)$ for sufficiently large $\ell$, contradicting the definition of the sequence $g_n$. 
\end{proof}

\singlespacing
\printbibliography

\end{document}